\documentclass{amsart}
\usepackage[utf8]{inputenc}
\usepackage{amsmath,amsfonts,amssymb,amscd,amsthm,amsbsy,upref, siunitx}
\usepackage[dvipsnames]{xcolor}
\usepackage{amsmath}
\usepackage{hyperref,cleveref,amssymb,mathtools}
\usepackage{bbm}
\usepackage{enumitem}
\definecolor{darkgreen}{rgb}{0.0, 0.27, 0.13}

\usepackage{tikz}
\usetikzlibrary{hobby}

\definecolor{ForestGreen}{RGB}{34,139,34}

\numberwithin{equation}{section}
\def\hangbox to #1 #2{\vskip3pt\hangindent #1\noindent \hbox to #1{#2}$\!\!$}

\usepackage{hyperref,cleveref,amssymb,mathtools}

\theoremstyle{plain}
\newtheorem{theorem}{Theorem}[section]
\newtheorem{proposition}[theorem]{Proposition}
\newtheorem{corollary}[theorem]{Corollary}
\newtheorem{lemma}[theorem]{Lemma}

\theoremstyle{definition}
\newtheorem{remark}[theorem]{Remark}

\newtheorem{problem}[theorem]{Problem}

\newtheorem{definition}[theorem]{Definition}

\DeclareSymbolFont{bbold}{U}{bbold}{m}{n}
\DeclareSymbolFontAlphabet{\mathbbold}{bbold}

\def\N{{\mathbb N}}
\def\R{{\mathbb R}}

\def\bH{{\mathbf H}}

\def\sfrac#1#2{\kern.1em\raise.5ex\hbox{$#1$}
	\kern-.1em/\kern-.05em\lower.25ex\hbox{$#2$}}

\newcommand{\fw}{\text{\fw}}

\newcommand{\dist}{{\rm dist}}

\newcommand{\tg}{\tilde{g}}

\DeclareMathOperator{\curl}{curl}

\def\bH{{\mathbf H}}

\begin{document}

\title[Well-posedness for the incompressible free boundary Euler equations]{Sharp Hadamard local well-posedness, enhanced uniqueness and pointwise continuation criterion for the incompressible free boundary Euler equations
}

\author{ Mihaela Ifrim}
\address{Department of Mathematics\\
University of Wisconsin - Madison
} \email{ifrim@wisc.edu}
\author{Ben Pineau}
\address{Courant Institute for Mathematical Sciences\\
New York University
} \email{brp305@nyu.edu}
\author{Daniel Tataru}
\address{Department of Mathematics\\
University of California at Berkeley
} \email{tataru@math.berkeley.edu}
\author{Mitchell~A.\ Taylor}
\address{Department of Mathematics\\
ETH Z\"urich, Ramistrasse 101, 8092 Z\"urich, Switzerland.
} \email{mitchell.taylor@math.ethz.ch}

 \begin{abstract}
 We provide a complete local well-posedness theory in $H^s$ based Sobolev spaces for the free boundary incompressible Euler equations with zero surface tension on a connected fluid domain. Our well-posedness theory includes: (i) Local well-posedness in the Hadamard sense, i.e., local existence, uniqueness, and the first proof of continuous dependence on the data, all in low regularity Sobolev spaces; (ii) Enhanced uniqueness: Our uniqueness result holds at the level of the Lipschitz norm of the velocity and the $C^{1,\frac{1}{2}}$ regularity of the free surface; (iii) Stability bounds:  We construct a nonlinear functional which measures, in a suitable sense, the distance between two solutions (even when defined on different domains) and we show that this distance is propagated by the flow; (iv) Energy estimates: We prove refined, essentially scale invariant energy estimates for solutions,  relying on   a newly constructed family of elliptic estimates; (v) Continuation criterion: We give the first proof of a sharp continuation criterion in the physically relevant pointwise norms, at the level of scaling. In essence, we show that solutions can be continued as long as the velocity is in $L_T^1W^{1,\infty}$ and the free surface is in $L_T^1C^{1,\frac{1}{2}}$, which is at the same level as the Beale-Kato-Majda criterion for the boundaryless case; (vi) A  novel proof of the construction of regular solutions. 
 
 Our entire approach is in the Eulerian framework and can be adapted to work in more general fluid domains.  

\end{abstract} 


\dedicatory{To my dad$^0$: Thank you for patiently waiting  for the answer to the question you casually asked me \\ years ago as I poured myself a glass of water.   
}

\footnotetext[0]{This dedication is to the first author’s father, Manolache Ifrim.} 
 
 \maketitle

\tableofcontents

\section{Introduction}
In this article, we study  the dynamics of an inviscid fluid droplet in the absence of surface tension. At time $t$, our fluid occupies a compact, connected, but not necessarily simply connected region $\overline{\Omega}_t\subseteq \mathbb{R}^d$, and its motion  is governed by the incompressible Euler equations
\begin{equation}\label{Euler}
  \begin{cases}
    &\partial_tv+v\cdot \nabla v=-\nabla p-ge_d,
    \\  
    &\nabla\cdot v=0.
    \end{cases}
\end{equation}
Here, $v$ is the fluid velocity, $p$ is the pressure, $g\geq 0$ is the gravitational constant, and $e_d$ is the standard vertical basis vector. In the local theory of the droplet problem, the gravity can be freely neglected.  However, it becomes important in the 
case of an unbounded fluid domain and in the case of a domain with a rigid bottom, so we  retain it in \eqref{Euler} for completeness.

\begin{figure}[h]
\begin{tikzpicture}
\path[draw,use Hobby shortcut,closed=true, fill=cyan]
(0,1) .. (.4,1) .. (1,3) .. (.3,4) .. (-3,2) .. (-3,.1);
\node at (-1, 2.2) {$\Omega_t$}; 
\node at (.3, .7) {$\Gamma_t$} ;
\node at (-6, 3) {$-ge_d$} ;
\draw[->, thick] (-5,4)--(-5,2) ;
\end{tikzpicture}
\caption{Inviscid fluid droplet in a gravitational field, where $\Omega_t$ is the fluid domain and $\Gamma_t$ is its moving boundary.}
\end{figure}

An essential role in the analysis  of the droplet problem is played by the vector field 
\begin{equation*}
D_t : = \partial_t+v\cdot \nabla,
\end{equation*}
which is called the \emph{material derivative} and describes the particle trajectories. On the free boundary,  we require  the kinematic boundary condition
\begin{equation}\label{bb}
    D_t  \ \ \text{is tangent to} \ \bigcup_t \{t\}\times\partial\Omega_t\subseteq\mathbb{R}^{d+1},
\end{equation} 
which says that the domain $\Omega_t$ is transported along the 
material derivative (or equivalently, the particle trajectories), 
and that the normal velocity of $\Gamma_t:=\partial\Omega_t$ is given by $v\cdot n_{\Gamma_t}$. Additionally, we require the dynamic  boundary condition
\begin{equation}\label{BC1}
p|_{\Gamma_t}=0,
\end{equation}
which represents the balance of forces at the fluid interface in the absence of  surface tension. Using the above boundary conditions, it is easy to see that   the energy
\begin{equation*}\label{energy}
    E:=\int_{\Omega_t}\left(\frac{|v|^2}{2}+gx\cdot e_d\right)\, dx
\end{equation*}
is formally conserved. 
Throughout the article, we will refer to the system \eqref{Euler}-\eqref{BC1} as the \textit{free boundary (incompressible) Euler equations}. 
\\

As is the case with  all Euler flows, an important  role in the above evolution is played by   the \emph{vorticity}, $\omega$, defined by 
\begin{equation*} 
\omega_{ij} = \partial_i v_j - \partial_j v_i.
\end{equation*}
By taking the curl of \eqref{Euler}, the vorticity is easily seen to solve the following transport equation along the flow:
\begin{equation}\label{transport vort}
D_t\omega=-(\nabla v)^*\omega-\omega\nabla v.
\end{equation}
If initially $\omega = 0$, then \eqref{transport vort} guarantees that this condition is propagated dynamically.
Such velocity fields are called \emph {irrotational}, and the corresponding solutions to the free boundary incompressible Euler equations are called \emph{water waves}.
\\

By taking the divergence of \eqref{Euler},  we obtain the following Laplace equation for the pressure:
\begin{equation}\label{Euler-pressure}
  \begin{cases}
    &\Delta p=-\text{tr}(\nabla v)^2 \ \ \text{in} \ \Omega_t,
    \\  
    &p=0 \ \ \text{on} \ \Gamma_t.
    \end{cases}
\end{equation}
For regular enough $v$ on  sufficiently regular   $\Omega_t$, the equation \eqref{Euler-pressure} uniquely determines the pressure from the velocity and domain.
A key role in the study of the free boundary Euler equations is played by the \emph{Taylor coefficient}, $a$, which is defined on the boundary $\Gamma_t$ by  
\begin{equation}\label{def-a}
a :=  -\nabla p \cdot n_{\Gamma_t}.
\end{equation}
Indeed, a classical result of Ebin \cite{MR886344} provides a specific ill-posedness scenario with $a \leq 0$
 and
gives strong evidence that the free boundary Euler equations are ill-posed
unless $a \geq 0$; this is widely believed to be true, though we are not aware of any work proving it in full generality.
For this reason, we will always assume that the initial data for the free boundary Euler equations verifies the following:
\medskip

{\bf Taylor sign condition.} There is a $c_0>0$ such that $a_0:=-\nabla p_0\cdot n_{\Gamma_0}>c_0$ on $\Gamma_0$.  
\medskip

 For irrotational data on compact simply connected domains, the Taylor sign condition is automatic by the strong maximum principle \cite{MR2178961}. See also \cite{MR3535894, MR1471885} for similar results on unbounded domains when $g>0$ and \cite{agrawal2023uniform,MR4656809,MR3987173} for results when $a$ is also allowed to vanish at certain points. Geometrically, enforcing $a_0>0$ ensures that the initial pressure $p_0$ is a non-degenerate defining function for the initial boundary hypersurface $\Gamma_0$, and thus can be used to describe the regularity of the boundary.
As part of our well-posedness theorem below, we prove that the Taylor sign condition is propagated by the flow on some non-trivial time interval.
\\

Another important role in this paper is played by the \emph{material derivative  of the Taylor coefficient}, $D_t a$,  which turns out to be closely related to (a derivative of) the normal component of the velocity $v \cdot n_{\Gamma_t}$. We will elaborate further on this relation  shortly when we discuss our choice of control parameters and good variables.
 
\subsection{The Cauchy problem: scaling, Sobolev spaces
and control parameters}

A state for the free boundary Euler equations consists 
of a domain $\Omega$ and a velocity field $v$ on $\Omega$.
A  bounded connected domain $\Omega$ can be equally described by its
boundary $\Gamma$. Hence, in the sequel, by a state we mean a pair $(v,\Gamma)$. 
\\

Describing the time evolution of $(v,\Gamma)$ along the 
free boundary incompressible Euler flow is most naturally done in 
a functional setting described via appropriate Sobolev norms.
To understand the proper setting, it is very helpful to consider the scaling properties of our problem. 
The boundaryless incompressible Euler flow admits a two parameter scaling group. However, when considering the 
free boundary flow there is an additional constraint;
namely, that the pointwise property $a \approx 1$ rests unchanged. At a technical level, this is reflected in the fact that the Taylor coefficient appears as a weight in the Sobolev norms which are used on $\Gamma$. Imposing this constraint leaves us with a one parameter family of
scaling laws, which have the form
\begin{equation*}
    \begin{split}
        v_\lambda(t,x)&=\lambda^{-\frac{1}{2}}v\left(\lambda^\frac{1}{2}t,\lambda x\right),
        \\
         p_\lambda(t,x)&=\lambda^{-1}p\left(\lambda^\frac{1}{2}t,\lambda x\right),
         \\
          \left(\Gamma_{\lambda}\right)_t&=\{\lambda^{-1}x : x\in \Gamma_{\lambda^\frac{1}{2}t}\}.
    \end{split}
\end{equation*}
As noted earlier, the above transformations  have the property that 
the Taylor coefficient has the dimensionless scaling,
\begin{equation*}
  a_\lambda(t,x)= a\left(\lambda^\frac{1}{2}t,\lambda x\right).  
\end{equation*}
A first benefit we derive from the scaling law is to understand what are the matched Sobolev regularities for $v$  
and $\Gamma$. This leads us to the following definition.
\begin{definition}[State space] The
\emph{state space} $\mathbf{H}^s$ is the set of all pairs $(v,\Gamma)$ such that $\Gamma$ is the boundary of a bounded, connected domain $\Omega$ and such that the following properties are satisfied:
\begin{enumerate}
    \item (Regularity). $v\in H_{div}^s(\Omega)$ and $\Gamma\in H^s$, where $H_{div}^s(\Omega)$ denotes the space of divergence free vector fields in $H^s(\Omega)$.
    \item (Taylor sign condition). $a:=-\nabla p\cdot n_{\Gamma}>c_0>0$, where $c_0$ may depend on the choice of $(v,\Gamma)$, and the pressure $p$ is obtained from $(v,\Gamma)$ by solving the  elliptic equation \eqref{Euler-pressure} associated to \eqref{Euler} and \eqref{BC1}.
\end{enumerate}
\end{definition}
For states $(v,\Gamma)$ as above, we define their size by
\[
\|(v,\Gamma)\|_{\mathbf{H}^s}^2:=\|\Gamma\|_{H^s}^2+\|v\|^2_{H^s(\Omega)}.
\]
Note, however, that  $\mathbf{H}^s$ is not  a linear space, so $\|\cdot\|_{\mathbf{H}^s}$ does not induce a norm topology in the usual sense. Heuristically, the state space $\bH^s$ may be thought of as an infinite dimensional manifold, though a precise interpretation of this is beyond the scope of this paper. For our purposes, it suffices to define a consistent notion of topology on $\bH^s$. 
Although we will not describe the precise topology in the introduction, this topology will allow us to define the space $C([0,T];\mathbf{H}^s)$  of continuous functions with values in $\mathbf{H}^s$, as well as an appropriate notion of $\mathbf{H}^s$ continuity of the data-to-solution map $(v_0,\Gamma_0)\mapsto (v(t),\Gamma_t)$. Armed with these notions, it makes sense to talk about the Cauchy problem.
\begin{problem}[Cauchy problem for the free boundary  Euler equations] Given an initial state $(v_0,\Gamma_0) \in \mathbf H^s$, find the unique solution $(v,\Gamma) \in 
C([0,T];\mathbf{H}^s)$ in some time interval $[0,T]$.
\end{problem}

A natural question to ask is what are the exponents $s$ 
for which the Cauchy problem is well-posed in $\mathbf H^s$.
Our first clue in this direction comes from scaling, which 
leads us to the critical exponent 
\[
s_c = \frac{d+1}2,
\]
and implicitly the lower bound $s \geq s_c$.
However, this does not tell the entire story, as even in the boundaryless case a result of Bourgain-Li~\cite{MR3359050}
shows that well-posedness holds only in the more restricted range
\[
s > \frac{d}2+1,
\]
which is heuristically connected to another scaling law 
of the boundaryless problem; namely,
\[
v(t,x) \mapsto \lambda^{-1}v(t,\lambda x).
\]
This latter exponent range $s > \frac{d}2+1$ is   exactly what we consider in our work. Specifically, in this article we  solve the Cauchy problem for the free boundary incompressible Euler equations at the same  regularity level as the  incompressible Euler equations on a fixed domain.
\\

The reader who is more familiar with the boundaryless case
may ask at this point why we confine ourselves to $L^2$ based Sobolev spaces, instead of using the full range of indices $L^p$ 
as in the boundaryless case. The reason for this is
precisely the boundary, where a portion of the dynamics is concentrated.  In particular, as a subset of our problem we have the irrotational case $\omega = 0$, when the flow may be fully interpreted as the flow of the free boundary. This case, 
commonly identified as water waves, yields a dispersive flow,
where $L^p$ based Sobolev spaces are disallowed if $p \neq 2$;  see also the discussion in \cite{MR3465379}. This is not to say that exponents $p \neq 2$ do not play a central role in our analysis. Instead, we use them, particularly 
the case $p = \infty$, in the definition of our \emph{control parameters}, which control the size and growth of our energy functionals. Precisely, our analysis involves two such control parameters, which ideally should be appropriately scale 
invariant, as follows:

\begin{enumerate}
    \item An ``elliptic" control parameter $A^\sharp$, 
used to control implicit constants in fixed time elliptic estimates, given by
\begin{equation}\label{Asharp}
A^\sharp = \| v\|_{\dot C^\frac12(\Omega)} + \| \Gamma\|_{Lip},    
\end{equation}
which is exactly invariant under scaling.

\item A ``dynamical" control parameter $B^\sharp$, used to control the growth of energy in time, given by 
\begin{equation}\label{Bsharp}
B^\sharp = \| v\|_{Lip(\Omega)} + \| \Gamma\|_{\dot C^{1,\frac12}}.    
\end{equation}
This latter control parameter is $1/2$ derivatives above scaling, and instead the scale invariant quantity is $\|B^\sharp\|_{L^1_t}$, which is  what will actually appear in our continuation criterion later on.
\end{enumerate}

With these control parameters in hand, we would like to have energy estimates in the scale invariant form
\begin{equation}\label{id EE}
    \frac{d}{dt} E^k(v,\Gamma) \lesssim_{A^\sharp} B^\sharp
 E^k(v,\Gamma),
\end{equation}
where $E^k$ denotes a suitable energy at the $\mathbf{H}^k$ regularity. As noted earlier, these are our ideal choices, but 
for our results we need to make some small adjustments
and relax them a bit, as follows:
\begin{enumerate}[label=\alph*)]
\item\label{a} Working with $A^\sharp$ would require edge case elliptic estimates in Lipschitz domains, bringing forth a broad host of issues 
which are less central to our problem, if even possible to overcome.  So, instead,
we will simply add $\epsilon$ derivatives to the norms in $A^\sharp$.
\item\label{b} In the case of $B^\sharp$, we do not want to lose the sharp scaling, which is  exactly as in the Beale-Kato-Majda criteria
in the boundaryless case. Therefore, we do not want to add extra derivatives as we did with $A^\sharp$. However, as we shall soon see,   the quantity $\|D_ta\|_{L^\infty(\Gamma)}$  appears as a control parameter  in  the $L^2$ estimate for the linearized equation.  As it turns out, in order to propagate our low regularity difference bounds,  control of $\|D_ta\|_{L^\infty(\Gamma)}$ will be needed. However, for the energy estimates, a careful analysis will show that  the control parameter $B^\sharp$ is sufficient, if we slightly modify the form of the estimate \eqref{id EE}. In both cases, maintaining the sharp top order control parameter is non-trivial. In the difference estimates, it requires  a careful analysis on intersections of  domains (and hence, in particular, performing elliptic theory on Lipschitz domains) and in the energy estimates  it requires (amongst several other things) finding a way to appropriately absorb the  logarithmic divergences occurring in the endpoint elliptic estimates when attempting to control $\|D_ta\|_{L^\infty(\Gamma)}$ by $B^\sharp$. To deal with this latter issue, we will take some inspiration from the proof of Beale-Kato-Majda \cite{MR763762}. 
\end{enumerate}
   The issues mentioned above have well-known counterparts in the boundaryless Euler flow. In fact,   \emph{strong ill-posedness} of the boundaryless Euler equations  has been recently proven  in the ``ideal" pointwise spaces $C^1$ and $Lip$ \cite{MR3320889,MR4065655}. 


\subsection{Historical comments}  The local well-posedness problem for the free boundary Euler equations has a long history. For irrotational flows, the first rigorous local existence result in Sobolev spaces was obtained by Wu \cite{MR1471885,MR1641609}, in the late 1990s. Since then, various  methods have been introduced to shorten the proofs, lower the regularity threshold and allow for more complicated geometries. For a small sample of such results we cite Beyer and G\"unther in \cite{MR1637554}, Lannes in \cite{MR2138139},  Alazard, Burq and Zuily in \cite{MR2805065,MR3260858}, Hunter, Ifrim and Tataru in \cite{MR3535894},  Ai in \cite{Ai1, Ai2} and Ai, Ifrim and Tataru in \cite{ai2019two}.  Although physically restrictive, the irrotationality assumption allows one to reduce the dynamics to a system of equations on the free boundary. Depending on the choices made, this  typically culminates in either the Zakharov-Craig-Sulem formulation of the water waves problem used in \cite{Ai1, Ai2,MR2805065,MR3260858,MR2138139}, or the holomorphic coordinates formulation used in \cite{ai2019two,MR3535894}. In either case, the reduction to a system of equations on $\mathbb{R}^{d-1}$ greatly simplifies the analysis. 
\\

For the free boundary Euler equations with non-trivial vorticity, certain generalized systems based on the above irrotational reductions have been proposed \cite{MR3385789,zz}. However, historically, the most successful approach has been to use Lagrangian coordinates to fix the domain.  For an execution of this  approach to proving local existence, the reader may consult the papers of Christodoulou and Lindblad \cite{MR1780703}, Coutand and Shkoller \cite{MR2291920} and Lindblad \cite{MR2178961}. One may also compare with the article \cite{MR3762281} of  Kukavica and Tuffaha, which uses the so-called \emph{arbitrary Lagrangian-Eulerian  change of variables}, as well as the more recent advances in the Lagrangian analysis presented in \cite{aydin2023construction,MR4019189}.
\\

In contrast to the above articles, we will utilize a \emph{fully Eulerian} strategy to prove the local well-posedness of the free boundary Euler equations. In other words, we will work directly with the physical equations  \eqref{Euler}-\eqref{BC1}, and avoid the use of any non-trivial coordinates changes. On time-independent domains, both the Lagrangian and Eulerian approaches have been widely successful in analyzing fluid equations. However, for  free boundary problems, the Eulerian approach has seen relatively little attention, due to the obvious difficulty in having the domain of the fluid itself serve as a time-dependent unknown. Our aim in this article is to directly confront this issue. Corollaries of our newly obtained insights include:
\begin{enumerate}
    \item \label{i} The first proof of the continuity of the data-to-solution map for this problem. 
    \item \label{ii} An enhanced uniqueness result, requiring only pointwise norms of very limited regularity.
    \item \label{iii} Refined low regularity energy estimates with geometrically natural pointwise control parameters.
    \item \label{iiii} A new, direct proof of existence for regular solutions.
    \item \label{iiiii} A method to obtain rough solutions as unique limits of regular solutions at a Sobolev regularity that matches the optimal result for the Euler equations on $\mathbb{R}^d$.
    \item \label{vi} An essentially scale invariant continuation criterion akin to that of Beale-Kato-Majda  for the incompressible Euler equations on the whole space.
\end{enumerate}
We will elaborate further on the ideas for obtaining the above results in \Cref{OOTP}. For now, it is important to note that we are not  the first to utilize an Eulerian approach to analyze the well-posedness  of fluid equations in the  free boundary setting. The pioneering work in this regard is the remarkable series of papers by Shatah and Zeng \cite{sz,MR2400608,MR2763036}. However, Shatah and Zeng primarily consider the free boundary Euler equations with surface tension. While they are able to produce a solution to the pure gravity problem in the zero surface tension limit, it seems that their construction at least requires bounded curvature, which corresponds to greater regularity assumptions on the data than we need here. For this reason,  the overlap between their analysis and ours tends to be on a more philosophical level, which we will elaborate on  further in \Cref{OOTP}. 
A more direct comparison is with the memoir \cite{MR4263411} of Wang, Zhang, Zhao and Zheng. In \cite{MR4263411}, the authors construct solutions to the free boundary Euler equations in an unbounded graph domain at the same \emph{Sobolev} regularity that we achieve here. That is, they prove existence and uniqueness of solutions in $H^s$ for $s>\frac{d}{2}+1.$ The approach in  \cite{MR4263411}  is  in the style of Alazard, Burq and Zuily \cite{MR2805065,MR3260858}, though the addition of vorticity makes the execution much more technical. Our approach is completely different to the one that they follow and works well in more complicated fluid domains. Additionally, we  prove properties  \eqref{i}-\eqref{vi} above. We also remark that all other fully Eulerian approaches (see,~e.g., \cite{MR4157561,ming2021local,MR4194622}) follow  Shatah and Zeng, and hence require the regularizing effect of surface tension and higher regularity.  The one step towards  a fully Eulerian proof without surface tension is the work \cite{MR3925531} of de Poyferr\'{e}, who proves energy estimates for the pure gravity shoreline problem. However, the energy estimates in \cite{MR3925531} have $H^s$ based control norms and no well-posedness proof is presented.
\\

The goal of our paper is twofold. First, we intend to present a comprehensive, Hadamard style well-posedness theory, with an aim towards proving sharp results. At the same time, we provide a novel, geometric analysis, which we argue is more direct and streamlined than previous works.  For instance, our proofs do not require  paralinearization or Chemin-Lerner spaces as in \cite{MR4263411}. Moreover, our existence scheme is new and direct - it does not use Nash-Moser, the approach in \cite{MR4263411}, or go through the zero surface tension limit as in \cite{sz,MR2400608,MR2763036}. For this reason, we believe that the techniques introduced in this paper will have a wide range of applicability.  In particular, in the work \cite{ifrim2024sharp} we are able to extend the results in this paper to the significantly more challenging free boundary MHD equations,  achieving the first sharp well-posedness result for this problem. 
\\

Finally, we mention that the analysis we present here is for the case of a compact fluid domain. In the study of the free boundary Euler equations, it is also common to consider the case of an infinite ocean of either finite or infinite depth. The choice of compact fluid domain emphasizes the geometric nature of our problem, and removes the temptation to flatten the domain into a strip or a half-space. Although some changes need to be made, as with the analysis of the capillary problem \cite{sz,MR2400608,MR2763036} by Shatah and Zeng, the general strategy we use here  can be adapted to  all three geometries. That being said, to streamline the exposition, we do allow  some of our estimates to depend on the domain volume, which is a conserved quantity for the droplet problem. 

\subsection{An overview of the main results}\label{OOTP}  
In a nutshell, our main result asserts that the free boundary incompressible Euler equations are well-posed in $\mathbf H^s$ 
for $s > \frac{d}2+1$. However, simply stating this fails 
to convey the full strength of both the result and of its various aspects and consequences. Instead, it is more 
revealing to divide the result in a modular way into four independently interesting parts;
namely, (a) uniqueness and stability, (b) well-posedness,
(c) energy estimates and (d) the continuation criteria.
\\

To set the stage for our results, let $\Omega_*$ be a bounded, connected domain with smooth boundary $\Gamma_*$. Given $\epsilon,\delta>0$, consider the collar neighborhood $\Lambda_*:=\Lambda(\Gamma_*,\epsilon,\delta)$ consisting of all hypersurfaces $\Gamma$ which are  $\delta$-close to $\Gamma_*$ in the $C^{1,\epsilon}$ topology. As long as $\delta>0$ is small enough,  hypersurfaces in $\Lambda_*$ can be written as graphs over $\Gamma_*$. This permits us to define Sobolev and H\"older norms on these hypersurfaces in a consistent fashion. To state our results, we will  assume  that a collar neighborhood $\Lambda_*$ has been fixed, and consider solutions with initial data $(v_0,\Gamma_0)$ having $\Gamma_0 \in \Lambda_*$.  A more precise description of the functional setting will be given later, in \Cref{AOMD}. For now, we remark that, while the collar neighborhood is very useful in order to uniformly define the $\mathbf H^s$ norms, it is not needed at all for the definition of our control parameters.

\subsubsection{Uniqueness and stability} 

We start by stating our uniqueness result, which requires the least in terms of notations and preliminaries.  Here, of crucial importance are the  control parameters 
\begin{equation}\label{ACONT}
A:=A_\epsilon:=\|v\|_{C^{\frac{1}{2}+\epsilon}_x(\Omega_t)}+\|\Gamma_t\|_{C_x^{1,\epsilon}}, \ \epsilon>0,
\end{equation}
and
\begin{equation}\label{BCONT}
B_{\text{diff}}:=\|v\|_{W^{1,\infty}_x(\Omega_t)}+\|D_tp\|_{W^{1,\infty}_x(\Omega_t)}+\|\Gamma_t\|_{C_x^{1,\frac{1}{2}}},
\end{equation}
which represent slight adjustments of the ideal control parameters $A^\sharp$ and $B^\sharp$, as discussed earlier. 
Using these control parameters, our main uniqueness result is as follows:

\begin{theorem}[Uniqueness]\label{Uniqueness intro} Let $\epsilon, T>0$ and let $\Omega_0$ be a domain with boundary $\Gamma_0$ of $C^{1,\frac{1}{2}}$ regularity. Then for every divergence free initial data $v_0\in W^{1,\infty}(\Omega_0)$, the free boundary Euler equations with the Taylor sign condition admit at most one solution $(v,\Gamma_t)$  with $\Gamma_t\in \Lambda_*$  and
\begin{equation*}
\sup_{0\leq t\leq T}A_\epsilon(t)+\int_{0}^{T}B_{\emph{diff}}(t) \, dt<\infty. 
\end{equation*}
\end{theorem}
To the best of our knowledge, \Cref{Uniqueness intro} is the first uniqueness result for the free boundary Euler equations which involves only low regularity pointwise norms. Indeed, as far as we are aware, all other papers on this subject are content to prove uniqueness in the same class of $H^s$ spaces for which they prove existence.
\\


While uniqueness is a fundamental property in its own right, in our work 
it can be seen as a corollary of a  far more useful  stability result, which we now explain. Let $(v,\Gamma_t)$ and $(v_h,\Gamma_{t,h})$ be two solutions to the free boundary Euler equations  with corresponding domains $\Omega_t$ and $\Omega_{t,h}$. An obvious objective is to show that if $(v,\Gamma_t)$ and $(v_h,\Gamma_{t,h})$ are ``close" at time zero, then they remain close on a suitable timescale. However, since the domains $\Omega_t$ and $\Omega_{t,h}$ are evolving in time, we cannot  compare the solutions $(v,\Gamma_t)$ and $(v_h,\Gamma_{t,h})$ in a linear way. To resolve this issue, we construct a nonlinear functional  which quantifies the distance between solutions and is  propagated by the flow. 
\\

To avoid comparing solutions whose corresponding domains 
are very different, we harmlessly restrict ourselves 
to solutions $(v,\Gamma_t)$ and $(v_h,\Gamma_{t,h})$ evolving in the same collar neighborhood $\Lambda_*$. For such solutions we define the nonlinear distance functional
\begin{equation}\label{diff functional candidate1}
\begin{split}
   D((v,\Gamma),(v_h,\Gamma_h)):= \frac{1}{2}\int_{\tilde{\Omega}_t}|v-v_h|^2dx+\frac{1}{2}\int_{\tilde{\Gamma}_t}b|p-p_h|^2\, dS.
\end{split}
\end{equation}
Here, $p$ and $p_h$ are the  pressures, $\tilde{\Gamma}_t$ is the boundary of $\tilde{\Omega}_t:=\Omega_t\cap\Omega_{t,h}$ and $b$ is a suitable weight function. Morally speaking, the first term on the right-hand side of \eqref{diff functional candidate1} measures the $L^2$ distance between $v$ and $v_h$. On the other hand, by the Taylor sign condition, $p$ and $p_h$ are non-degenerate defining functions for $\Gamma_t$ and $\Gamma_{t,h}$, so the second term on the right-hand side of \eqref{diff functional candidate1} gives a measure of the distance between $\Gamma_t$ and $\Gamma_{t,h}$. In \Cref{DEAU}, we prove that \eqref{diff functional candidate1} does indeed act as a proper  measure of distance between solutions. More crucially, we prove that this distance is propagated by the flow, in the sense that 
\begin{equation}\label{Stab est1}
\frac{d}{dt} D((v,\Gamma),(v_h,\Gamma_h)) \lesssim_{A,A_h} (B_{\text{diff}}+B_{\text{diff},h}) D((v,\Gamma),(v_h,\Gamma_h)).
\end{equation}
Here, $A_h$ and $B_{\text{diff},h}$ are  the control parameters \eqref{ACONT} and \eqref{BCONT} corresponding to the solution $(v_h, \Gamma_{t,h})$. An immediate corollary of the stability estimate  \eqref{Stab est1} is the aforementioned \Cref{Uniqueness intro}. However, \eqref{Stab est1}  will also prove to be useful  in various other scenarios. For example, we will use it in our proof of the continuity of the data-to-solution map, as well as in the construction of rough solutions as unique limits of regular solutions. 

\subsubsection{Well-posedness}

 Our second main result is concerned with the well-posedness 
 problem. To fix the notations, we start with a collar neighborhood $\Lambda_*$ and $s>\frac{d}{2}+1$.  We then consider initial data $(v_0,\Gamma_0) \in \mathbf H^s$ with  $\Gamma_0 \in \Lambda_*$.  Viewing $\Gamma_0$ as a graph over $\Gamma_*$, we may unambiguously define its $H^s$ 
 norm.  
 With this setup, we may state our well-posedness theorem as follows:

\begin{theorem}[Hadamard local well-posedness]\label{MT}
    Fix $s>\frac{d}{2}+1$ and a collar $\Lambda_*$. For any $(v_0,\Gamma_0)$ in $\mathbf{H}^s$ 
    with $\Gamma_0 \in \Lambda_*$ there exists a time $T>0$,  depending only on $\|(v_0,\Gamma_0)\|_{\mathbf{H}^s}$ and the lower bound in the Taylor sign condition, for which there exists a unique  solution $(v(t),\Gamma_t)\in C([0,T];\mathbf{H}^s)$ to the free boundary Euler equations satisfying a proportional  uniform lower bound in the Taylor sign condition. Moreover, the data-to-solution map is continuous with respect to the $\mathbf{H}^s$ topology.
\end{theorem}
    The regularity of the velocity in \Cref{MT} matches the optimal Sobolev regularity for the Euler equations on $\mathbb{R}^d$. Indeed, as shown by Bourgain and Li  \cite{MR3359050}, the Euler equations are ill-posed in $H^s(\mathbb{R}^d)$ when $s=\frac{d}{2}+1$.
    \\
    
    We note crucially that our article is not the first to reach the $s>\frac{d}{2}+1$ Sobolev threshold for the free boundary Euler equations. Indeed, this threshold was achieved  for the first time  in the recent memoir  \cite{MR4263411}, in the case of an unbounded fluid domain with graph geometry. However, it is important to note that the approach in \cite{MR4263411} is very different from ours, as it passes through a paralinearization and utilizes properties of strip-like domains and Chemin-Lerner spaces. In particular, the approach in \cite{MR4263411} cannot be easily modified to the droplet problem, whereas our approach applies equally well in unbounded domains. Moreover, there is no mention of the continuity of the data-to-solution map  in \cite{MR4263411}. To the best of our knowledge, \Cref{MT} gives the \emph{first} proof of this important property for the free boundary Euler equations. In addition, our  approach  significantly refines the well-posedness theory by adding properties \eqref{ii}-\eqref{vi} above as well as introduces  an  entirely new set of techniques that we believe will have broad applications. 
    \\

  
  When it comes to free boundary problems, the continuity of the data-to-solution map -- if justified -- is usually proven by reformulating the problem on  a fixed domain and then working with the standard notion of continuous dependence on  fixed domains \cite{MR3535894,MR3593707}. As far as we are aware, the only exception to this appears in the work \cite{sz,MR2400608,MR2763036} of Shatah and Zeng, where continuous dependence is proven for the free boundary Euler equations with surface tension directly in the Eulerian setting. The drawback of Shatah and Zeng's proof, however, is that it relies crucially on the regularizing effect of surface tension,  which imposes a higher Sobolev regularity of the boundary relative
  to the interior velocity. This in turn allows for a bounded decoupling of the 
  velocity into rotational and irotational components, an idea which
  is not applicable to the pure gravity problem. In particular, Shatah and Zeng do not construct a distance functional, as we do here. For this reason, our robust proof which simultaneously avoids domain flattenings  and works on a  quasilinear problem without regularizing effects   can be seen as one of the main novelties of our paper.
\subsubsection{Energy estimates}\label{Subsection 1.2.2}
Controlling the  growth of solutions to our boundary value problem is essential for both local well-posedness and  understanding  potential blowup. This control is achieved via energy estimates. Due to the complex geometry of our problem, the first challenge  is to construct good 
energy functionals.
\\

Fix an integer $k\geq 0$. In light of \Cref{Uniqueness intro} and the stability estimate \eqref{Stab est1}, it is natural to try to construct an energy functional $E^k=E^k(v,\Gamma)$ satisfying $E^{k}(v,\Gamma)\approx_A \|(v,\Gamma)\|_{\mathbf{H}^k}^2$ and the estimate
\begin{equation*}
    \frac{d}{dt}E^k(v,\Gamma)\lesssim_A B_{\text{diff}}E^k(v,\Gamma).
\end{equation*}
Indeed, by Gr\"onwall's inequality, this would yield the bound
\begin{equation*}
    \|(v,\Gamma)(t)\|_{\mathbf{H}^k}^2\lesssim \exp\left({\int_{0}^{t}C_AB_{\text{diff}}(s)\,ds}\right) \|(v,\Gamma)(0)\|_{\mathbf{H}^k}^2,
\end{equation*}
for some constant $C_A$ depending only on $A$, the collar, and the verification of the Taylor sign condition. Morally speaking, such an estimate would  then allow one to conclude that solutions to the free boundary Euler equations with the Taylor sign condition can be continued as long $A$ remains bounded and $B_{\text{diff}}\in L^1_t$.
\\

However, there is one issue with the above estimates. Note that the control parameter $A$ in \eqref{ACONT} depends only on the H\"older norms of our main variables (the surface and the velocity) at (nearly) the correct scale. However,  the control parameter $B_{\text{diff}}$ in \eqref{BCONT} depends also on the auxiliary variable $D_tp.$ From the point of view of the analysis of the free boundary Euler equations, this is completely natural. Indeed, even at the level of the linearized equation, one sees that the uniform norm of $\nabla D_tp$ (or more specifically  the uniform norm of $D_ta$, but these are essentially equivalent) appears as a control parameter  for the $L^2$ energy estimates in \Cref{EE}.  On the other hand, for the purpose of providing a clear and physical description of how solutions to the free boundary Euler equations break down,  we would ultimately like to use the control parameter $B:=B^\sharp$ defined in \eqref{Bsharp}, which  depends only on the H\"older norms of $\Gamma$ and $v$. To achieve this, our key observation is that, as long as $k>\frac{d}{2}+1$, we can use a log of the  energy to absorb endpoint losses, and hence prove an estimate of the form
\begin{equation}\label{Remove loss}
\|D_tp\|_{W_x^{1,\infty}(\Omega_t)}\lesssim_A \log(1+E^k)B.
\end{equation}
 An estimate akin to \eqref{Remove loss} is not to be expected in the difference estimates, as the distance functional is too low of regularity to absorb the logarithmic divergences inevitably arising from $C^1$ and $W^{1,\infty}$ elliptic estimates. With the above discussion in mind, the actual energy estimates we prove can be essentially stated as follows.
\begin{theorem}[Energy estimates]\label{Energy est. thm intro}
Fix a collar neighborhood $\Lambda_*$,  let $s\in\mathbb{R}$ with  $s>\frac{d}{2}+1$ and let $k> \frac{d}{2}+1$ be an integer. Then for $\Gamma$ restricted to $\Lambda_*$ there exists an energy functional $\mathbf H^k \ni (v,\Gamma)\mapsto E^k(v,\Gamma)$ such that
\begin{enumerate}
\item (Energy coercivity).
\begin{equation}\label{Coercivity bound on integers1}
E^k(v,\Gamma)\approx_A \|(v,\Gamma)\|_{\mathbf{H}^k}^2.
\end{equation}
\item  (Energy propagation). If, in addition to the above, $(v,\Gamma)=(v(t),\Gamma_t)$ is a solution to the free boundary incompressible Euler equations, then $E^k(t):=E^k(v(t),\Gamma_t)$ satisfies 
\begin{equation}\label{EPFS}
\frac{d}{dt}E^k\lesssim_A B\log(1+\|(v,\Gamma)\|_{\mathbf{H}^s})E^k.
\end{equation}
\end{enumerate}
Here, $A$ is as in \eqref{ACONT} and $B=B^\sharp$.
\end{theorem}
By Gr\"onwall's inequality, \eqref{Coercivity bound on integers1} and \eqref{EPFS}  yield the following single and double exponential bounds of the type
\begin{equation}
\begin{split}
\label{doubleexpbound1}
\|(v(t),\Gamma_t)\|^2_{\mathbf{H}^k}&\lesssim_A\exp\left(\int_{0}^{t}C_AB\log(1+\|(v,\Gamma)\|_{\mathbf{H}^s})ds\right)\|(v_0,\Gamma_0)\|_{\mathbf{H}^k}^2,
\\
\|(v(t),\Gamma_t)\|^2_{\mathbf{H}^k}&\lesssim_{A} 
 \exp\left(\log(1+C_A\|(v_0,\Gamma_0)\|_{\mathbf{H}^k}^2)\exp{\int_{0}^{t}C_AB\,ds}\right),
 \end{split}
\end{equation}
for all integers $k>\frac{d}{2}+1$.  We do not directly prove the analogue of Theorem~\ref{Energy est. thm intro} for noninteger exponents $k$. Nevertheless, as a consequence of our analysis in the last section of the paper, we do obtain
the bounds \eqref{doubleexpbound1} also for noninteger $k$.
This is achieved by using frequency envelopes in order to combine the distance functional and the energy estimates akin to a nonlinear Littlewood-Paley type theory. It is also worth noting that a similar double exponential growth rate for the $L^1_TL^\infty_x$ norm of the vorticity appears in the classical  Beale-Kato-Majda \cite{MR763762} criteria as a consequence of trying to weaken the natural control parameters of the problem.
\\

In order to understand the form of the energy functionals used in \Cref{Energy est. thm intro},  a key step is to identify Alinhac style 
\emph{good variables} for the problem, which are as follows:

\begin{enumerate}[label=(\roman*)]
    \item The vorticity $\omega$, which is measured in $H^{k-1}(\Omega)$.
    \item The Taylor coefficient $a$, which is measured
    in $H^{k-1}(\Gamma)$.
    \item The material derivative $D_t a$ of the Taylor coefficient, which is measured
    in $H^{k-\frac32}(\Gamma)$.
    \end{enumerate}
 Our energy functionals are  constructed as certain combinations 
of well-chosen norms of the above good variables. The general strategy for constructing these norms is to  apply appropriate vector fields and elliptic operators to $\omega$, $a$ and $D_ta$ at the $\mathbf{H}^k$ regularity in such a way that the resulting variables solve the linearized equation to leading order. After this, the nonlinear energy $E^k$  may be essentially defined as the linear energy evaluated at these good variables. As it turns out, after completing  this process, we arrived at  essentially the same energy as \cite{MR3925531}, which was derived by different means. However, as can be immediately inferred from our control norms, the way we treat the energy is very different from \cite{MR3925531}. Indeed, without going into details, we mention that the proof  of \Cref{Energy est. thm intro} requires not only a delicate analysis of the fine structure and cancellations present in the free boundary Euler equations, but also the use of a new family of refined elliptic estimates. Although we refrain from stating them here in the introduction, these elliptic estimates  serve as an important part of the paper. Moreover, since they are quite general, we believe that they will prove to be useful in other problems as well. 

\bigskip

\subsubsection{Low regularity continuation criterion}\label{lrcc}

A very natural objective in the study of the  Euler equations is to find a geometric characterization of how solutions break down. For the Euler equations without free boundary, this direction traces back to the famous paper of Beale, Kato and Majda \cite{MR763762}. In recent years, interest in sharp blow up criterion for the free boundary Euler equations has risen, and progress has been made by de Poyferr\'e \cite{de2014blow}, Ginsberg \cite{MR4272912}, Wang and Zhang \cite{MR3585049} and  Wang, Zhang, Zhao and Zheng \cite{MR4263411}.  Here, we explain our  rather definitive answer to this question, which is essentially a consequence of our local well-posedness result 
in Theorem~\ref{MT} and   the energy estimates in Theorem~\ref{Energy est. thm intro}. However,  to avoid topological issues, we must first introduce a notion of thickness for the fluid domain.
\begin{definition}\label{d:thickness}
The fluid domain $\Omega$ has thickness at least   $R > 0$ if  for each $x \in \Gamma$,  $B(x,R) \cap \Gamma$ is the graph of a $C^{1,\epsilon}$ function which separates $B(x,R)$ into two connected components.
\end{definition}

With this notion in hand, our continuation criterion reads as follows:

\begin{theorem}[Continuation criterion]\label{cont crit intro}
A solution $(v,\Gamma) \in C(\mathbf H^s)$, $s>\frac{d}{2}+1$, of the free boundary incompressible Euler equations with the Taylor sign condition can be continued
for as long as the following properties hold:
\begin{enumerate}[label=\alph*)]
\item\label{a2} (Uniform bound from below for the Taylor coefficient). There is a $c>0$ such that 
\[
a \geq c > 0.
\]
\item\label{b2} (Uniform thickness). There is an $R>0$ such that $\Omega_t$ has thickness at least $R$.
\item (Control parameter bounds). The control parameters satisfy
\[
A \in L^\infty_t, \qquad B \in L^1_t.
\]
\end{enumerate}

\end{theorem}

One may compare our continuation criteria for the free boundary problem with the classical Beale-Kato-Majda criteria 
for the boundaryless problem and note that they are essentially at the same level, with the natural addition of 
the $C^{1,\frac12}$ boundary regularity bound. Another minor difference is that we use the Lipschitz bound on the velocity $v$ rather than the uniform bound on the vorticity $\omega$.
One may ask whether it is possible to further relax our 
criterion in order to use only the vorticity bound. The major obstruction is that while in fixed domains the vorticity uniquely determines 
the velocity, in our case an appropriate boundary condition is also needed, which is best described via the $D_t a$ good variable. So, a potential conjecture might be that in order 
to use only the vorticity bound in the interior, one might have to compensate by adding a uniform bound on $D_t a$, as seen in the linear control parameter $B_{lin}$ and in the difference estimates. That being said, in this paper we have opted for a continuation criteria   involving only the natural variables $v$ and $\Gamma$ and no auxiliary pressure related terms.
\\


As mentioned above, several recent articles \cite{de2014blow, MR4272912, MR3585049, MR4263411} have focused on  obtaining improved continuation criterion for the free boundary Euler equations. The most significant of these contributions is the  memoir \cite{MR4263411}, which proves that $H^k$ solutions to the free boundary Euler equations with the Taylor sign condition can be continued after $t=T$ as long as properties \ref{a2} and \ref{b2} in \Cref{cont crit intro} hold and
\begin{equation}\label{mem conti}
    \sup_{t\in [0,T]}\left(\|\kappa(t)\|_{(L^p\cap L^2)(\Gamma_t)}+\|v(t)\|_{W^{1,\infty}(\Omega_t)}\right)<\infty \ \text{for some} \  p>2d-2.
\end{equation} Here, $\kappa$ denotes the mean curvature of the surface. To motivate their result, \cite{MR4263411}  recalls a question of Craig and Wayne \cite{MR2355420}, which asks  one to find (in the context of the irrotational water waves problem) the lowest H\"older regularity of the surface and velocity potential whose boundedness on $[0,T]$ implies that one can continue the solution past $t=T$.  Although \eqref{mem conti} makes significant progress on this question, it fails to achieve purely pointwise norms and is far from scale invariant. Moreover, the criterion \eqref{mem conti} only applies to solutions which a priori live in integer based Sobolev spaces $H^k$. This limits the applicability of \eqref{mem conti} to solutions with at least a half derivative of excess  regularity. In contrast, \Cref{cont crit intro} replaces the criterion  $v\in L^\infty_TW^{1,\infty}_x$ by the sharp and scale invariant criterion $v\in L^1_TW^{1,\infty}_x$, and  only requires control of H\"older norms of the free surface at the correct scale.  In particular, \Cref{cont crit intro}  gives a rather definitive answer to Craig and Wayne's question for the full free boundary Euler equations. For the state-of-the-art result for the two-dimensional irrotational water waves problem, see \cite{ai2019two}. Also, note that \Cref{cont crit intro}  applies to solutions in all Sobolev spaces $\mathbf{H}^s$ with $s>\frac{d}{2}+1$, not just to those in integer spaces. This improvement is by no means trivial; rather, it follows from a careful usage of our distance functional.
\\

\subsection{Outline of the paper} The article has a modular structure, where, for the essential part, only the main results of each section are used later.
\subsubsection{The linearized equations}
The starting point for our analysis, in \Cref{linearized section}, is to derive the linearization of our problem in Eulerian coordinates. The linearized system will serve as a guide to several of the choices made in our nonlinear analysis. In particular, it will suggest the correct variables to use, as well as the form of our distance functional. Moreover, when  proving energy estimates, the Alinhac style good variables we construct will be shown to  solve the linearized equations to leading order. This is also where 
the control parameters $A$ and $B_{lin}$ (an enhanced version of $B$) make their first appearance.  

\subsubsection{Function spaces and the geometry of moving domains} \Cref{AOMD}  describes the appropriate functional setting for our analysis. We begin by setting up a basic framework for our problem, including introducing low regularity control neighborhoods which will allow us to establish uniform control over constants in  Sobolev and elliptic estimates in certain topologies for an appropriate family of domains. After  defining the function spaces and norms that we will be using, we define the state space $\mathbf{H}^s$ where we will seek solutions to the free boundary Euler equations. Unlike in problems on fixed domains, the state space $\mathbf{H}^s$ will not be linear. However, it will be equipped with an appropriate notion of convergence, allowing us to define continuity of functions with values in $\mathbf{H}^s$ as well  continuity of the data-to-solution map.

\subsubsection{Stability estimates and uniqueness} The aim of \Cref{DEAU} is to construct a nonlinear distance functional which will allow us to track the distance between two solutions at very low regularity. The general scheme  is akin to the difference bounds in a weaker topology which are common in the study of quasilinear problems on fixed domains. However, here there are fundamental difficulties to overcome, as we are seeking to not only compare functions on different domains, but also  track the evolution in time of this distance. These difficulties are embedded into the nonlinear character of our distance functional; both careful choices and delicate estimates are required to propagate this distance forward in time. To the best of our knowledge, this is only the second time difference estimates have been successfully proven  in the free boundary setting. The other successful execution, which conceptually inspired the present approach, was  in the case of a compressible gas \cite{disconzi2020relativistic,ifrim2020compressible}, which is very different from the incompressible liquid we consider here. In particular, unlike in the gas case, the boundary of our fluid contains non-trivial energy,   requiring  interesting geometric insights to understand. 
\\

As a consequence of our stability estimates, we  deduce uniqueness of solutions at very low regularity. Also, as we shall see in later sections, the low regularity distance bounds we prove will serve both as  an essential building block in our construction of rough solutions as unique limits of regular solutions as well as in the proof of the continuity of the data-to-solution map.

\subsubsection{Elliptic theory} The main goal of \Cref{BEE} is to introduce a new family of refined elliptic estimates which will be crucial for obtaining the sharp pointwise control norms in the higher energy bounds. The secondary objective of \Cref{BEE} is to define a relevant Littlewood-Paley theory, collect various ``balanced" product, Moser and Sobolev type estimates, and note several identities for operators and functions defined on moving domains. For the most part, the material  in \Cref{BEE}  does not rely on any   specific  structure of the Euler equations, so should be applicable to other free boundary problems as well. In \Cref{SSRO}, we construct the regularization operators which we will need for our existence scheme and the frequency envelopes for states $(v,\Gamma)\in \mathbf{H}^s$ that we will use to establish the refined properties of the data-to-solution map.

\subsubsection{Energy estimates}  In \Cref{HEB} we establish energy estimates within the $\mathbf{H}^k$ scale of spaces.  As a first step, we  construct a coercive energy functional $(v,\Gamma)\mapsto E^k(v,\Gamma)$  associated to each integer $k >\frac{d}{2} +1$. The scheme here is to identify  Alinhac style ``good variables” $(w_k,s_k)$ which solve the linearized equation modulo perturbative source terms. We then define our energy as the sum of the rotational energy and the linearized energy evaluated at these good variables. To prove the energy estimates, we  split the argument in a modular fashion into two parts. First, we prove the coercivity of our energy functional; that is, we show that $E^k(v,\Gamma)\approx \|(v,\Gamma)\|_{\mathbf{H}^k}^2$. After this, we track the time evolution of the energy, establishing control of $E^k(v,\Gamma)$ in terms of the initial data, with growth  dictated by the pointwise control parameters $A$ and $B$. Both steps of this argument are delicate. In particular, the former makes extensive use  of the refined elliptic estimates from \Cref{BEE}, and the latter requires us to identify and exploit various structural properties and fine cancellations present in the Euler equations.

\subsubsection{Construction of regular solutions} \Cref{Existence section} is devoted to the construction of regular solutions to the free boundary Euler equations. The  overarching scheme we utilize is similar to \cite{ifrim2020compressible}, which analyzed the case of a compressible gas. However, we stress that the main difficulties in the incompressible liquid case are quite different than for the gas, especially near the free boundary, as the surface of a liquid carries a non-trivial energy. As a general overview, the scheme we utilize  is constructive, employing a time discretization via an Euler type method together with a separate transport step to produce good approximate solutions. However,
a na\"ive implementation of Euler’s method loses derivatives.
To overcome this, we  ameliorate the derivative loss by an initial regularization of each iterate in our discretization. To ensure that the uniform energy bounds survive, such a regularization needs to be chosen carefully. For this, we employ  a modular approach and try to decouple this process into two steps, where we regularize individually the domain and  the velocity. 
We believe that this modular approach will serve as a recipe for a new and relatively simple method for constructing solutions to various free boundary problems. That being said, the execution of this scheme is still quite subtle, requiring several novel ideas in addition to those coming from \cite{ifrim2020compressible}.

\subsubsection{Rough solutions and continuous dependence}
 The last section of the paper aims to construct rough solutions as strong limits of smooth solutions. This is achieved by considering a family of dyadic regularizations of the initial data, which generate corresponding smooth solutions. For these smooth solutions we control on one hand higher Sobolev norms $\mathbf{H}^k$, using our energy estimates, and on the other hand the $L^2$ type distance between consecutive ones, from our difference estimates. Combining the high and the low regularity bounds directly yields rapid convergence in all $\mathbf{H}^l$ spaces for  $l<k$. To gain strong convergence in $\mathbf{H}^k$, we use frequency envelopes to more accurately control both the low and the high Sobolev norms above. This allows us to bound differences in the strong $\mathbf{H}^k$ topology. Interpolation and a similar argument yields local existence in fractional Sobolev spaces as well as continuous dependence of the solutions in terms of the initial data  in the strong topology. Finally,  our main continuation result in \Cref{cont crit intro} follows along similar lines, given the careful treatment of our control norms in the energy and difference estimates. 


 \medskip
 
 For problems on $\mathbb{R}^d$, the scheme outlined above for obtaining rough solutions from smooth solutions, good energy estimates and difference estimates is more classical; see the expository article \cite{IT-primer}, the article \cite{alazard2024nonlinear} for a more abstract formulation of the method and \cite{MR4331023,MR4830552} for some recent applications. However, as we shall see, the fact that  solutions are all defined on different domains leads to some new subtleties in our free boundary setting.



\subsection{Acknowledgments}
The first author was supported  by the  NSF CAREER grant DMS-1845037, by the Sloan Foundation, by the Miller Foundation and by a Simons Fellowship. The other three authors were supported by the NSF grant DMS-2054975 as well as by a Simons Investigator grant from the Simons Foundation. 
Some of this work was carried out while all four authors were in residence at the Simons Laufer Mathematical Sciences Institute (formerly MSRI) in Berkeley, California, during the summer of 2023, 
participating in the program ``Mathematical problems in fluid dynamics, Part II", which
 was supported by the National Science Foundation under Grant No.~DMS-1928930. This work forms part of the second authors' PhD thesis \cite{MR4820290}.

\section{The linearized equation}\label{linearized section}
The first goal of this section is to formally derive the linearization of our problem, working entirely in Eulerian coordinates;
this is the system of equations \eqref{DM lin1}.
Then, we prove Theorem~\ref{t:wp-lin}, which asserts that the linearized system is well-posed in $L^2$, with energy bounds 
determined by our sharp control parameters.
The key elements here are the linearized energy \eqref{linearized en} and the basic energy estimate \eqref{EE for DMlin}. 
\\

 Conceptually, the linearized system is 
 an essential piece of the puzzle. On a practical level, however, it is not immediately useful in proving well-posedness,
 as it is not clear that $C^1$ one parameter families of solutions exist in the first place. It is only a posteriori, after well-posedness
 is established, that the linearized energy estimates may be used to derive bounds for differences of solutions.
Instead, we will use
our understanding of the linearized system 
to guide us in our choice of distance functional
in  \Cref{DEAU} and later in our choice of  energy functionals  in \Cref{HEB}.
\\

To derive the linearized system, we take a one parameter family of solutions $(v_h,p_h)$ defined on domains $\Omega_{t,h}$, with $(v_0,p_0):=(v,p)$ and $\Omega_{t,0}:=\Omega_t$. We define $w=\partial_hv_h |_{h=0}$ and $q=\partial_hp_h|_{h=0}$.
\\

In $\Omega_t$, the linearized equation is rather standard: 
\begin{equation*}\label{Euler linearized}
  \begin{cases}
    &\partial_tw+w\cdot \nabla v+v\cdot\nabla w =-\nabla q,
    \\  
    &\nabla\cdot w=0.
    \end{cases}
\end{equation*}
However, we also need to linearize the kinematic and dynamic boundary conditions on the surface $\Gamma_t$. For this, let us denote by $\Gamma_{t,h}$ the free surface at time $t$ for the solution $(v_h,p_h)$, so $\Gamma_{t,0}:=\Gamma_{t}$. Fix a one parameter family of diffeomorphisms $\phi_h(t):\Gamma_t\to \Gamma_{t,h}$, with $\phi_0(t)=Id_{\Gamma_t}$. The dynamic boundary condition \eqref{BC1} asserts that for every point $x\in \Gamma_t$,
$$p_h(t,\phi_h(t)(x))=0.$$
Differentiating in $h$ and evaluating at $h=0$ gives
$$q|_{\Gamma_t}=-\nabla p|_{\Gamma_t}\cdot \psi(t),$$
where $\psi(t):=\frac{\partial}{\partial h}\phi_h(t)|_{h=0}$.
Using that $\nabla p|_{\Gamma_t}$ is normal to $\Gamma_{t}$  we deduce that
\begin{equation}\label{Def of s}
q|_{\Gamma_t}=-\nabla p|_{\Gamma_t}\cdot n_{\Gamma_t} \psi(t)\cdot n_{\Gamma_t}=:as.\end{equation}
Here, we define $s:=\psi(t)\cdot n_{\Gamma_t}$ which we loosely interpret as the normal velocity in the parameter $h$ of the family $\Gamma_{t,h}$ at $h=0$. We will use this as one of our linearized variables. Note that since $a>0$, $s$ does not depend on the choice of diffeomorphisms $\phi_h(t)$.
\\

Next, we linearize the kinematic boundary condition. Analogously to $v\cdot n_{\Gamma_t}$ describing the normal velocity of the free surface, we expect $w\cdot n_{\Gamma_t}$ to describe the ``normal velocity" of our linearized variable $s$. Therefore, up to a perturbative error, $D_ts$ should agree with $w\cdot n_{\Gamma_t}$. In fact, we obtain the  relation
\begin{equation}\label{Transport equation 1}
D_ts-w\cdot n_{\Gamma_t}=s(n_{\Gamma_t}\cdot\nabla v)\cdot n_{\Gamma_t}.
\end{equation}
To derive \eqref{Transport equation 1}, we note that \eqref{bb} and \eqref{BC1} imply that
\begin{equation}\label{Mat pressure}
    D_tp=0 \ \ \ \text{on} \ \Gamma_t.
\end{equation}
This is the equation that we will linearize to obtain \eqref{Transport equation 1}. As before, let $\phi_h(t):\Gamma_t\to\Gamma_{t,h}$ be a diffeomorphism. We then have for $x\in \Gamma_t,$
\begin{equation*}
    [(\partial_t+v_h\cdot \nabla)p_h](t,\phi_h(t)(x))=0.
\end{equation*}
Taking $h$ derivative and evaluating at $h=0$ yields,
\begin{equation}\label{Deriving 2.3}
    w\cdot\nabla p+D_tq+\nabla D_tp\cdot \psi=0 \ \ \text{on} \ \Gamma_t.
\end{equation}
Using \eqref{Def of s}, and that $\nabla D_tp$ is normal to $\Gamma_t$ by \eqref{Mat pressure}, we deduce \eqref{Transport equation 1} from \eqref{Deriving 2.3} after some simple algebraic manipulation. Indeed, we have $\nabla p|_{\Gamma_t}=-an_{\Gamma_t}$. Then using the relation $q_{|\Gamma_t}=as$, we compute $D_tq=aD_ts+sD_ta$. This reduces \eqref{Deriving 2.3} to
\begin{equation}\label{Deriving 2.3 2}
    -aw\cdot n_{\Gamma_t}+aD_ts+sD_ta+s\nabla D_tp\cdot n_{\Gamma_t}=0.
\end{equation}
After division by $a$, the first two terms in \eqref{Deriving 2.3 2} evidently align with the left-hand side of \eqref{Transport equation 1}. The right-hand side of \eqref{Transport equation 1} appears by  commuting the gradient with the material derivative in the last term of \eqref{Deriving 2.3 2}, and by using the fact that  $\nabla p\cdot D_t n_{\Gamma_t}=0$ to rewrite $sD_ta=-sD_t(\nabla p\cdot n_{\Gamma_t})=-sD_t\nabla p\cdot n_{\Gamma_t}$. 
\\

Putting everything together, the linearized system takes the form:
\begin{equation}\label{DM lin1}
\left\{
\begin{aligned}
 &D_t w+ \nabla q = - w\cdot \nabla v \  \ \text{in} \ \Omega_t,
    \\  
    &\nabla\cdot w=0 \ \  \text{in} \  \Omega_t,  
\\
&     D_t s -     w \cdot n_{\Gamma_t} = s (n_{\Gamma_t}\cdot \nabla v)\cdot n_{\Gamma_t}\ \ \text{on} \ \Gamma_t,
\\
&    q = as \ \ \text{on} \ \Gamma_t,
\end{aligned}
\right. 
\end{equation}
where the terms on the right-hand side can be viewed as perturbative source terms. 
\\

In order to study the well-posedness of the linearized system \eqref{DM lin1}, we introduce an enhanced version $B_{lin}$ of the control parameter $B^\sharp$:
\begin{equation}\label{Lin control par}
\begin{split}
B_{lin}(t):=\|a^{-1}D_ta\|_{L^{\infty}(\Gamma_t)}+\|\nabla v\|_{L^{\infty}(\Omega_t)}.
\end{split}
\end{equation}
Using this, we may state our main linearized well-posedness result as follows.
\begin{theorem} \label{t:wp-lin}
Let $(v,\Gamma)$ be a solution to the 
free boundary incompressible Euler equations
in a time interval $[0,T]$
so that $a>0$, $A^\sharp$ stays uniformly bounded 
and $B_{lin} \in L^1_T$. Then the linearized 
system \eqref{DM lin1} for $(w,s)$ is 
well-posed in $L^2(\Omega) \times L^2(\Gamma)$
in $[0,T]$.
\end{theorem}

Here we recall that $\Omega$ and $\Gamma$ are
time dependent. The rest of this section is devoted to the proof of this very simple theorem. The  basic strategy  is to construct a suitable energy functional and  prove corresponding energy estimates. Once this is done,  well-posedness follows via a standard duality argument, which is left for the reader. To execute this argument, one simply notes that the adjoint system is essentially identical to the direct
system \eqref{DM lin1}, modulo perturbative terms, and that the energy estimates are time reversible.
\\

Below, we will work with a slightly more general system, since this is what will appear in the higher order energy bounds later on. We define the \emph{generalized linearized system} as follows:
\begin{equation}\label{DM lin}
\left\{
\begin{aligned}
 &D_t w+ \nabla q = f \  \ \text{in} \ \Omega_t,
    \\  
    &\nabla\cdot w=0 \ \  \text{in} \  \Omega_t,  
\\
&     D_t s -     w \cdot n_{\Gamma_t} = g\ \ \text{on} \ \Gamma_t,
\\
&    q = as \ \ \text{on} \ \Gamma_t,
\end{aligned}
\right. 
\end{equation}
where we allow for arbitrary source terms $f$ and $g$ on the right-hand side of the first and third equation.
\\
\\
It remains to prove a suitable energy estimate for the system \eqref{DM lin}. The natural energy associated to this system is
\begin{equation}\label{linearized en}
E_{lin}(w,s)(t)=\frac{1}{2}\int_{\Omega_t}|w|^2\, dx+\frac{1}{2}\int_{\Gamma_t}as^2\, dS.
\end{equation}
Using \eqref{linearized en}, the main energy estimate 
for the generalized linear system is as follows:

\begin{proposition}\label{EE}
Suppose $a>0$. Then the system \eqref{DM lin} satisfies the energy estimate
\begin{equation}\label{EE for DMlin}
\frac{d}{dt}E_{lin}(w,s)(t)\leq B_{lin}E_{lin}(w,s)(t)+\langle as,g\rangle_{L^2(\Gamma_t)}+\langle w,f\rangle_{L^2(\Omega_t)}.
\end{equation}
\end{proposition}
We note that the energy functional \eqref{linearized en} is also the energy functional for the linearized system \eqref{DM lin1}, and that this proposition yields energy 
estimates for \eqref{DM lin1}, thereby concluding the proof of Theorem~\ref{t:wp-lin}.
\begin{proof}

We will make use of the following standard Leibniz type formulas  (see; for example, \cite[Appendix A]{MR2317005}).
\begin{proposition}\label{Leibniz}
\begin{enumerate}
\item Assume that the time-dependent domain $\Omega_t$ flows with Lipschitz velocity $v$. Then the time derivative of the time-dependent volume integral is given by
\begin{equation*}\label{Dtuint}
    \frac{d}{dt}\int_{\Omega_t}f(t,x)\, dx=\int_{\Omega_t}D_tf+f\nabla\cdot v\, dx.
\end{equation*}
\item Assume that the time-dependent hypersurface $\Gamma_t$ flows with divergence free velocity $v$. Then the time derivative of the time-dependent surface integral is given by
\begin{equation*}\label{Leibniz derivative under int general}
    \frac{d}{dt}\int_{\Gamma_t}f(t,x)\, dS=\int_{\Gamma_t}D_tf-f(n_{\Gamma_t}\cdot\nabla v)\cdot n_{\Gamma_t}\, dS.
\end{equation*}

\end{enumerate}

\end{proposition}
Now, to prove the energy estimate \eqref{EE for DMlin}, we apply \Cref{Leibniz} to  obtain
\begin{equation}\label{E1}
\begin{split}
\frac{d}{dt}E_{lin}(w,s)(t)&=\int_{\Omega_t}D_t w\cdot w\,dx+\int_{\Gamma_t}asD_ts\, dS+\frac{1}{2}\int_{\Gamma_t}D_tas^2\, dS-\frac{1}{2}\int_{\Gamma_t}[n_{\Gamma_t}\cdot\nabla v\cdot n_{\Gamma_t}]as^2\, dS
\\
&\leq \int_{\Omega_t}D_t w\cdot w\, dx+\int_{\Gamma_t}asD_ts\, dS+B_{lin}E_{lin}(w,s)(t).
\end{split}    
\end{equation}
Integrating by parts, we obtain
\begin{equation*}
\begin{split}
\int_{\Omega_t}D_t w\cdot w\, dx+\int_{\Gamma_t}asD_ts\, dS&=\int_{\Omega_t}w\cdot f\,dx+\int_{\Gamma_t}asD_ts\, dS-\int_{\Gamma_t}qw\cdot n_{\Gamma_t}\,dS    
\\
&=\langle as,g\rangle_{L^2(\Gamma_t)}+ \langle w,f\rangle_{L^2(\Omega_t)}.
\end{split}
\end{equation*}
Combining this with (\ref{E1}) completes the proof.
\end{proof}

\section{Analysis on moving domains}\label{AOMD} One difficulty when working directly on moving domains is that many of the standard Sobolev and elliptic  estimates have domain dependent constants. It is therefore necessary to work in a framework which allows for uniform control of these constants in certain topologies. This section is devoted to dealing with this issue. Our approach in this regard is somewhat analogous to that of Shatah and Zeng \cite{sz,MR2400608,MR2763036} and   de Poyferr\'e \cite[Section 3]{MR3925531}, but with the key difference being that our control neighborhoods will only be uniform in the pointwise $C^1$ or $C^{1,\epsilon}$ topologies as opposed to the stronger $L^2$ based topologies considered in those papers. This will be essential for establishing the pointwise continuation criterion for solutions.

\subsection{Function spaces}
To begin, we precisely define the function spaces and norms that  we will be using. Throughout, $\Omega\subseteq \mathbb{R}^d$ will denote a bounded, connected domain. We define $H^s(\Omega)$, $s\geq 0$, as the set of all $f\in L^2(\Omega)$ such that
\begin{equation}\label{def of Hs norm}
    \|f\|_{H^s(\Omega)}:=\inf \left\{\|F\|_{H^s(\mathbb{R}^d)} :F\in H^s(\mathbb{R}^d), \ F|_{\Omega}=f \right\}
\end{equation}
is finite. Here, $\|\cdot\|_{H^s(\mathbb{R}^d)}$ is defined in the standard way, via the Fourier transform. We let $H^s_0(\Omega)$ denote the closure of $C_0^\infty(\Omega)$ in $H^s(\Omega)$ and identify $H^{-s}(\Omega)$ isometrically with the dual space $(H_0^s(\Omega))^*$. Importantly, with this definition of the $H^s$ norm, the constants in Sobolev embedding theorems (either $H^s\to L^p$ or $H^s\to C^\alpha$) are independent of $\Omega$. For regular enough domains and integer $s$, the norm defined in \eqref{def of Hs norm} is equivalent to the standard one. We will precisely quantify this equivalence later.
\\

We next define the regularity of the boundary of a connected domain $\Omega$, which is characterized in terms of the regularity of local coordinate parameterizations of $\partial\Omega$. Indeed, in general, an $m$-dimensional manifold $\mathcal{M}\subseteq \mathbb{R}^d$ is said to be of class $C^{k,\alpha}$ or $H^s$, $s>\frac{d}{2}$, if,  locally in linear frames, $\mathcal{M}$ can be represented by graphs with the same regularity. 
\\
\\
If $s>\frac{d+1}{2}$, then given  $\Omega$ as above with boundary of class $H^{s}$, we can define what it means to be an $H^r$ function on $\partial \Omega$ for $s\geq r\geq -s$. Indeed, these are simply the functions whose coordinate representatives are locally in $H^r(\mathbb{R}^{d-1})$. It is easy to see that the space of $H^r$ functions on $\partial\Omega$, $s\geq r\geq -s$, can be made into a Banach space. Indeed, a norm can be chosen by taking a covering of $\partial\Omega$ by a finite number of coordinate patches and an adapted partition of unity. However, there is one problem with this approach. Although such a norm is well-defined up to equivalence, the precise value of the norm is dependent on the choice of local coordinates. Since we will be dealing with a family of domains, we need to make sure that we define norms on their boundaries in a consistent and uniform way. 
\subsection{Collar coordinates}
As a first step towards resolving the above issue, we  fix a bounded, connected reference domain $\Omega_*$ with smooth boundary $\Gamma_*:=\partial\Omega_*$. We define $H^s$ and $C^{k,\alpha}$ based norms on $\Gamma_*$ by making an appropriate choice of local parameterizations of $\Gamma_*$. Letting $\delta>0$ be a small positive constant, we define $N(\Gamma_*,\delta)$ to be the collection of all $C^1$ hypersurfaces $\Gamma$ such that there exists a $C^{1}$ diffeomorphism $\Phi_{\Gamma}:\Gamma_*\to \Gamma$ with
\begin{equation*}
\|\Phi_{\Gamma}-id_{\Gamma_*}\|_{C^1(\Gamma_*)}<\delta.
\end{equation*} 
 If $\delta>0$ is small enough, we can  represent hypersurfaces $\Gamma\in N(\Gamma_*,\delta)$ as graphs over $\Gamma_*$. Indeed, we denote the outward unit normal to $\Gamma_*$ by $n_{\Gamma_*}$.  Following \cite[Section 2.1]{MR2763036}, if we have a smooth unit vector field $\nu:\Gamma_*\to\mathbb{S}^{d-1}$ which is suitably transversal to $\Gamma_*$ (that is, $\nu\cdot n_{\Gamma_*}>1-c$ for some small $c>0$),  it follows from the implicit function theorem that there  exists a $\delta>0$, determined by $\Gamma_*$ and $\nu$, such that the map 
 \begin{equation*}\label{diff}
     \varphi:\Gamma_*\times[-\delta,\delta]\to\mathbb{R}^d, \ \ \varphi(x,\mu)=x+\mu\nu(x)
 \end{equation*}
 is a $C^1$ diffeomorphism from its domain to a collar neighborhood of $\Gamma_*$. If $\delta>0$ is small enough, the above coordinate system associates each hypersurface $\Gamma\in N(\Gamma_*,\delta)$ with a unique function $\eta_\Gamma:\Gamma_*\to\mathbb{R}$ such that 
 \begin{equation}\label{PhiGamma}
    \Phi_\Gamma(x):=\varphi(x,\eta_\Gamma(x))=x+\eta_\Gamma(x)\nu(x)
 \end{equation}
 is a diffeomorphism in $C^{1}(\Gamma_*,\Gamma\subseteq \mathbb{R}^d)$. We can think of the map $\Phi_\Gamma$ as a way to represent  $\Gamma$ as a (global) graph over $\Gamma_*$.  With this notation in hand, we can now define what it means to be a $H^s$ hypersurface which is close to $\Gamma_*$. 
\begin{definition}\label{Def2}
For $\delta>0$ small enough and $\alpha\in [0,1)$, define the control neighborhood $\Lambda(\Gamma_*,\alpha,\delta)$ as the collection of all hypersurfaces $\Gamma\in N(\Gamma_*,\delta)$ such that the associated map $\eta_\Gamma:\Gamma_*\to\mathbb{R}$ satisfies \[
\|\eta_\Gamma\|_{C^{1,\alpha}(\Gamma_*)}<\delta.
\]
\end{definition}
\begin{definition}
 Suppose $s\geq 0$, $\Gamma\in N(\Gamma_*,\delta)$ for  $\delta>0$  small enough, and the associated map $\eta_\Gamma:\Gamma_*\to\mathbb{R}$ satisfies   $\eta_{\Gamma}\in H^s(\Gamma_*)$. We then define  the $H^s$ norm of $\Gamma$ by
\begin{equation*}
\|\Gamma\|_{H^s}:=\|\eta_{\Gamma}\|_{H^s(\Gamma_*)}.
\end{equation*}
\end{definition}
In the above definitions, $\|\eta_\Gamma\|_{C^{1,\alpha}(\Gamma_*)}$ and $\|\eta_{\Gamma}\|_{H^s(\Gamma_*)}$ are computed with respect to fixed, independent of $\Gamma$, local coordinates on $\Gamma_*$.  In an analogous way, we define for $\gamma\in [0,1)$ and integers $k\geq 0$, the $C^{k,\gamma}$ norm, $\|\Gamma\|_{C^{k,\gamma}}$.  As was essentially noted in \cite[Section 2.1]{MR2763036}, when $0<\delta\ll 1$, each $\Gamma\in \Lambda(\Gamma_*,\alpha,\delta)$ is associated to a well-defined domain $\Omega$.

\begin{remark}\label{r:epsilon}
One key point in \Cref{Def2} is that we only require $\Gamma$ be close to $\Gamma_*$ in the $C^{1,\alpha}$ topology, as opposed to  the stronger $L^2$ based topologies used in \cite{MR3925531,sz,MR2400608,MR2763036}. In practice, we will want the control topology to be as weak as possible. For our purposes, we will typically take $\alpha=\epsilon>0$ for some arbitrarily small (but fixed) constant $\epsilon>0$. 
\end{remark}

\begin{remark}\label{r:delta}
A second key point in \Cref{Def2} concerns the choice of the small parameter $\delta$. This will not be arbitrarily small,
but instead its size may also be chosen to depend on weaker 
topologies; namely, (i) the $C^{1,\epsilon}$ norm of $\Gamma_*$ and (ii)
the thickness (see Definition~\ref{d:thickness}) of the domain $\Omega$. This will serve two purposes:
\begin{itemize}
    \item To allow us to place any rough $H^s$ boundary $\Gamma$
within a suitable control neighborhood $\Lambda(\Gamma_*,\epsilon,\delta)$.
 \item To allow us to obtain the robust continuation result in Theorem~\ref{cont crit intro}, which does not require any reference to control neighborhoods.
\end{itemize}
\end{remark}
Following the discussion in the above two remarks, throughout  the article we will often abbreviate $\Lambda(\Gamma_*,\epsilon,\delta)$ by $\Lambda_*,$ where the suppressed parameters $\varepsilon>0$ and $\delta>0$ are understood to be small but fixed universal parameters, which depend only on $s$ and on the thickness of $\Omega$.

\subsection{State space}
Fix a collar neighborhood $\Lambda_*$ and $s>\frac{d}{2}+1$. We define $\mathbf{H}^s$ as the set of all pairs $(v,\Gamma)$ such that $\Gamma\in\Lambda_*$ is the boundary of a bounded, connected domain $\Omega$ and such that the following properties are satisfied:
\begin{enumerate}
    \item (Regularity). $v\in H_{div}^s(\Omega)$ and $\Gamma\in H^s$, where $H_{div}^s(\Omega)$ denotes the space of divergence free vector fields in $H^s(\Omega)$.
    \item (Taylor sign condition). $a:=-\nabla p\cdot n_{\Gamma}>c_0>0$, where $c_0$ may depend on the choice of $(v,\Gamma)$, and the pressure $p$ is obtained from $(v,\Gamma)$ by solving the standard elliptic equation \eqref{Euler-pressure} associated to \eqref{Euler} and \eqref{BC1}.
\end{enumerate}
Given initial data $(v_0,\Gamma_0)$ in the state space $\mathbf{H}^s$, our eventual goal will be to construct local solutions $(v(t),\Gamma_t)$ that evolve continuously in $\mathbf{H}^s$. To accomplish this, we must define a suitable notion of topology on our state space. This  will enable us to establish two key properties of our flow; namely,
\begin{enumerate}
    \item Continuity of solutions with values in $\mathbf{H}^s$.
    \item Continuous dependence of solutions $(v(t),\Gamma_t)$ as functions of the initial data $(v_0,\Gamma_0)$.
\end{enumerate}
Note that since $\mathbf{H}^s$ is not a linear space, the above two continuity properties require some explanation. To measure the size of individual states $(v,\Gamma)\in \mathbf{H}^s$, we define $\|(v,\Gamma)\|_{\mathbf{H}^s}^2:=\|\Gamma\|_{H^s}^2+\|v\|^2_{H^s(\Omega)}$. However, since  $\mathbf{H}^s$ is not  a linear space, $\|\cdot\|_{\mathbf{H}^s}$ does not induce a norm topology in the usual sense. Hence, we  still  need  an appropriate way of comparing different states.  Motivated by \cite{disconzi2020relativistic,ifrim2020compressible}, we define convergence in $\mathbf{H}^s$  as follows.
\begin{definition}\label{Def of convergence} We say that a sequence $(v_n,\Gamma_n)\in \mathbf{H}^s$ converges to $(v,\Gamma)\in \mathbf{H}^s$ if 
\begin{enumerate}
\item (Uniform Taylor sign condition). For some $c_0>0$ independent of $n$, we have
\begin{equation*}
a_n, a>c_0>0.     
\end{equation*}
\item (Domain convergence). $\Gamma_n\to \Gamma$ 
 in $H^s$. That is, $\eta_{\Gamma_n}\to \eta_{\Gamma}$ in $H^s(\Gamma_*)$ where $\eta_{\Gamma_n}$ and $\eta_\Gamma$ correspond to the collar coordinate representations of $\Gamma_n$ and $\Gamma$, respectively. 
 \item (Norm convergence). For every $\epsilon>0$ there exists a smooth divergence free function $\tilde{v}$ defined on a neighborhood $\tilde{\Omega}$ of $\overline{\Omega}$ with   $\|\tilde{v}\|_{H^s(\tilde\Omega)}<\infty$ and satisfying 
 \begin{equation*}
\|v-\tilde{v}\|_{H^s(\Omega)}\leq \epsilon
 \end{equation*}
 and
 \begin{equation*}
 \limsup_{n\to \infty}\|v_n-\tilde{v}\|_{H^s(\Omega_n)}\leq\epsilon.    
 \end{equation*}
 \end{enumerate}
\end{definition}
With the above notion of convergence, it makes sense to define $C([0,T];\mathbf{H}^s)$. We remark, however, that in \cite{MR3925531,sz,MR2400608,MR2763036}, $C([0,T];\mathbf{H}^s)$ is defined in a slightly different way, via the existence of an extension to a continuous function with values in $H^s(\mathbb{R}^d)$. In \Cref{cont of domain ext}, we  construct a family of extension operators which depend continuously in a suitable sense on the domain, making the above two notions of continuity essentially interchangeable.



\section{Difference estimates and uniqueness}\label{DEAU}

Comparing different solutions is key to any well-posedness result.
Since our problem is quasilinear, such a comparison cannot be achieved uniformly in 
the leading $\mathbf H^s$ topology, but instead  only in weaker topologies. The main result of  this section provides a Lipschitz bound for the 
distance between two solutions in the $L^2$ topology, akin to our bounds for the linearized equation. Notably, our distance bounds propagate at the level  of our control parameters, which require for instance a Lipschitz bound on the velocity but no higher regularity. This is what will allow us to establish uniqueness of solutions under very weak regularity assumptions. Moreover, as we shall see shortly, these low regularity distance bounds  also serve as an essential building block in our construction of rough solutions as unique limits of smooth solutions, as well as in our proof of the continuity of the data-to-solution map; this broadly follows a scheme first introduced in \cite{ifrim2020compressible} but is quite different on the technical level.
\\


The fundamental difficulty in achieving our  distance bounds is the need to compare states which live on different domains.
To overcome this difficulty, we  construct a ``distance functional" which \emph{simultaneously} captures the 
distance between (functions on) different domains and admits a time evolution that we are able to track.
To the best of our knowledge, no such low regularity difference bounds or even uniqueness results were previously known  for any incompressible free boundary Euler model. Instead, we take our cue from the work \cite{ifrim2020compressible} of the first and the third authors, which considers a similar free boundary problem but for a compressible Euler model. We note, however, that the similarity between the uniqueness argument here and its counterpart in \cite{ifrim2020compressible} is only at the conceptual level, as the two flows have very different behaviors both inside the domain and near the free boundary.

\subsection{The distance functional}\label{DF}
 Our first objective is to use the linearized energy as a guide to construct a distance functional which will be suitable for comparing nearby solutions. We begin by fixing a collar neighborhood $\Lambda(\Gamma_*,\epsilon,\delta)$, where  $\epsilon>0$ and $\delta>0$ are small. We then suppose that we have two states $(v,\Gamma)$, $(v_h,\Gamma_h)$ with respective domains $\Omega$, $\Omega_h$. We let $\eta_{\Gamma}$ and $\eta_{\Gamma_{h}}$ be the corresponding representations of $\Gamma$ and $\Gamma_{h}$ as graphs over $\Gamma_*$. Following the linearized energy estimate, we aim to define analogues of the linearized variables $w$ and $s$, which heuristically should measure the  $L^2$ distance between $v$ and $v_h$ and the distance between $\Gamma$ and $\Gamma_{h}$, respectively. One technical caveat is that $v$ and $v_h$ are not defined on the same domain. For this reason, we define $\widetilde{\Omega}=\Omega\cap\Omega_{h}$. We can represent the free boundary $\widetilde{\Gamma}$ for $\widetilde{\Omega}$ as a graph over $\Gamma_*$ via the function $\eta_{\widetilde{\Gamma}}=\eta_{\Gamma}\wedge \eta_{\Gamma_h}$. Note that although the graph representation $\eta_{\widetilde{\Gamma}}$ is well-defined, $\widetilde{\Gamma}$ is only Lipschitz in general, so will not be in $\Lambda(\Gamma_*,\epsilon,\delta)$. 
\\
\\
To measure the (signed) distance between $\Gamma$ and $\Gamma_h$, we define $s_h^*:\Gamma_*\to \mathbb{R}$ by
\begin{equation}\label{distancefunction}
    s_h^*(x)=\eta_{\Gamma_h}(x)-\eta_{\Gamma}(x).
\end{equation}
As will become evident below, although $s_h^*$ correctly measures the distance between the free hypersurfaces, it has the ``wrong" domain. To fix this, we define the variable $s_h:\widetilde{\Gamma}\to \mathbb{R}$ by pushing  $s_h^*$ forward to the hypersurface $\widetilde{\Gamma}$. In other words, for $x\in \widetilde{\Gamma},$ we define $s_h(x)=s_h^*(\pi(x))$, where $\pi$ denotes the canonical projection, mapping the image of $\Gamma_*\times [-\delta,\delta]$ under $\varphi$ back to $\Gamma_*$. For convenience, we also extend $\nu$ to a vector field $X$ defined  on the image of $\varphi$ via $X(x)=\nu(\pi(x)).$ We will not actually use the displacement function $s_h$ directly in the difference estimates below. In particular, it will not act as our desired analogue of the linearized variable $s$. This is because  its dynamics are somewhat awkward to work with. Instead of using $s_h$, it is far more convenient (and geometrically natural) to use the the pressure difference $p-p_h$ (along with a suitable weight to be defined below) to measure the distance between $\Gamma$ and $\Gamma_{h}$. To motivate this, recall that for solutions to the free boundary Euler equations,  the Taylor sign condition implies that  $p$ and $p_h$ are non-degenerate defining functions for $\Gamma_t$ and $\Gamma_{t,h}$ within a suitable collar neighborhood. Therefore, on the boundary of $\tilde{\Omega}_t=\Omega_t\cap\Omega_{t,h}$, $p-p_h$ should be proportional to the displacement function $s_h$. The dynamics of $p-p_h$ turn out to be much easier to work with than those of $s_h$, as terms involving $p-p_h$ will appear naturally when we use the free boundary Euler equations to compare solutions.
\\
\\
With the above motivation in mind and using the linearized equation as a guide, we define our distance functional as follows:
\begin{equation}\label{diff functional candidate}
\begin{split}
D((v,\Gamma), (v_h,\Gamma_h)):=    D(v,v_h):=\frac{1}{2}\int_{\widetilde{\Omega}}|v-v_h|^2\,dx+\frac{1}{2}\int_{\tilde{\Gamma}}b|p-p_h|^2\, dS,
\end{split}
\end{equation}
where the weight function $b$ is defined by
\begin{equation*}
b:=a^{-1}1_{\tilde{\Gamma}\cap\Gamma}+a_h^{-1}1_{\tilde{\Gamma}\cap\Gamma_h}.
\end{equation*}
As $p-p_h$ vanishes on $\Gamma\cap\Gamma_{h}$, we may rewrite the distance functional in the slightly more convenient form
\begin{equation*}
D(v,v_h)=\frac{1}{2}\int_{\tilde{\Omega}}|v-v_h|^2\,dx+\frac{1}{2}\int_{\mathcal{A}}a^{-1}|p-p_h|^2\, dS+\frac{1}{2}\int_{\mathcal{A}_{h}}a_h^{-1}|p-p_h|^2\, dS,
\end{equation*}
where $\mathcal{A}:=\tilde{\Gamma}\cap\Gamma-\Gamma\cap\Gamma_{h}$ and $\mathcal{A}_{h}:=\tilde{\Gamma}\cap\Gamma_{h}-\Gamma\cap\Gamma_{h}$. 
\\

Letting $\overline{F}$ denote the average of $F$ along the flow $\varphi$ between the free surfaces, the  fundamental theorem of calculus implies that for $x\in \widetilde{\Gamma}$,
\begin{equation}\label{FTC pressure average}
p_h(x)-p(x)=
\begin{cases}
&-\overline{\nabla p_h\cdot X}s_h(x) \hspace{10mm}\text{if $x\in \mathcal{A}$},
\\
&-\overline{\nabla p\cdot X}s_h(x)\hspace{12mm}\text{if $x\in \mathcal{A}_{h}$}.
\end{cases}
\end{equation}
Therefore, thanks to the Taylor sign condition and assuming the regularity $p,p_h\in C^{1,\epsilon}$, we should have $|p-p_h|\approx |s_h|$ on $\widetilde{\Gamma}$ within a tight enough collar neighborhood. The precise manner in which we have this proportionality will be made clear shortly. Finally, note that, for solutions to the free boundary Euler equations, a simple computation yields the following equation for $v-v_h$ in $\tilde{\Omega}_t$:
\begin{equation}\label{veqn}
\begin{cases}
&D_t(v-v_h)+\nabla (p-p_h)=(v_h-v)\cdot\nabla v_h,
\\
&\nabla\cdot (v-v_h)=0.
\end{cases}
\end{equation}
\begin{remark}
    Although it is not particularly important for the difference estimates, we note that the distance functional \eqref{diff functional candidate} makes sense for general (not necessarily dynamical) states $(v,\Gamma)$ and $ (v_h,\Gamma_h)$. Indeed, given suitable states $(v,\Gamma)$ and $(v_h,\Gamma_h)$,  we can always  associate pressures $p$ and $p_h$ by solving the standard elliptic equation associated to \eqref{Euler} and \eqref{BC1}. As we will see in \Cref{HEB}, it is very important that our energy functional for the $\mathbf{H}^k$ energy bounds be defined for general states $(v,\Gamma)\in \mathbf{H}^k$.
\end{remark}
\subsection{Difference estimates}
We are now ready to propagate difference bounds for two solutions to the free boundary Euler equations.  
\begin{theorem}[Difference Bounds]\label{Difference} 
 Let $0<\epsilon,\delta\ll 1$ and let $\Lambda_*=\Lambda(\Gamma_*,\epsilon,\delta)$ be a collar neighborhood. Suppose that $(v,\Gamma_t)$ and $(v_h,\Gamma_{t,h})$ are solutions to the free boundary Euler equations that evolve in the collar in a time interval $[0,T]$ and satisfy  $a$,$a_h>c_0>0$.  Then we have the estimate
\begin{equation*}
\frac{d}{dt}D(v,v_h)\lesssim_{A,A_h} (B+B_h)D(v,v_h)
\end{equation*}
where
\begin{equation*}
B:=\|v\|_{W^{1,\infty}(\Omega_t)}+\|\Gamma_t\|_{C^{1,\frac{1}{2}}}+\|D_tp\|_{W^{1,\infty}(\Omega_t)},\hspace{5mm}A:=\|v\|_{C^{\frac{1}{2}+\epsilon}(\Omega_t)}+\|\Gamma_t\|_{C^{1,\epsilon}},
\end{equation*}
 $B_h$ and $A_h$ are the analogous quantities corresponding to $v_h$, $p_h$, $D_t^hp_h$ and $\Gamma_{t,h}$ and we  have implicitly assumed that our solutions have regularity $B,B_h\in L^1_T$ and $A,A_h\in L^\infty_T$.
\end{theorem}

\begin{remark}
It is worth remarking that all of the results in this section hold equally well if the control parameter $B$ is replaced by 
\begin{equation*}
B_{\epsilon}=\|v\|_{C^{1,\epsilon}(\Omega_t)}+\|\Gamma_t\|_{C^{1,\frac{1}{2}}},   
\end{equation*}
which depends solely on the regularity of $v$ and $\Gamma_t$. This is because we will later prove  an elliptic estimate of the form 
\begin{equation*}
\|D_tp\|_{W^{1,\infty}(\Omega_t)}\lesssim_A B_{\epsilon}.    
\end{equation*}
See \Cref{Linfest2} and \Cref{logremovablremark} for details. We prefer, however, to work with the control parameter $B$ defined above as its $L_T^1$ norm is scale invariant.
\end{remark}
\begin{proof}
For simplicity of notation, we drop the $t$ subscript for the domains below. We also use $\lesssim_A$ as a shorthand for $\lesssim_{A,A_h}$. To ensure that we can estimate expressions involving the pressure in terms of the control parameters $A$ and $B$ above, we need the  bounds
\begin{equation}\label{pressurebounddiff}
\|p\|_{C^{1,\epsilon}(\Omega)}\lesssim_A 1,\hspace{5mm}\|p\|_{C^{1,\frac{1}{2}}(\Omega)}\lesssim_A B   ,
\end{equation}
as well as the analogous bounds for $p_h$. The proof that these bounds hold will be postponed until later when the requisite elliptic estimates are developed. See \Cref{Linfest} and \Cref{Linfest2} for details.
Now, to proceed with the difference estimate, we recall the identity
\begin{equation}\label{dt-dist} 
\begin{split}
\frac{d}{dt}D(v,v_h)&=\frac{1}{2}\frac{d}{dt}\int_{\tilde{\Omega}}|v-v_h|^2\, dx+\frac{1}{2}\frac{d}{dt}\int_{\mathcal{A}}a^{-1}|p-p_h|^2\,dS+\frac{1}{2}\frac{d}{dt}\int_{\mathcal{A}_h}a_h^{-1}|p-p_h|^2\,dS.
\end{split}
\end{equation}
To compute the first term, we would like to use Reynolds'
transport theorem, as in Proposition~\ref{Leibniz}. However, here we do not have a good velocity field $\tilde{v}$ so that $\tilde{\Omega}$ flows with velocity $\tilde{v}$. Constructing such a field seems to be at the very least impractical, so we will instead allow for a correction term which is a boundary integral.  For this purpose,  suppose that  $D(t)$ is a time-dependent domain for which  we may define at almost every point of the boundary  a normal velocity $v_b$ for the boundary. Note that if $D(t)$ were  flowing with velocity $v$, then $v_b = v \cdot n_{\partial D(t)}$, where $n_{\partial D(t)}$ is the outward unit normal. For more general velocity fields  $v$ on  $D(t)$, we have the following proposition.
\begin{proposition}\label{Reynolds}
Given a  velocity field $v$
defined on a time-dependent domain $D(t)$ with Lipschitz boundary flowing with normal velocity $v_b$, we have
\begin{equation*}
\frac{d}{dt}\int_{D(t)}f\,dx=\int_{D(t)}D_t f + \nabla \cdot v f\, dx+  \int_{\partial D(t)}f  ( v_b - v \cdot n_{\partial D(t)}) \,dS.
\end{equation*}
\end{proposition}
The proof is a straightforward application of the divergence theorem.
\\

In our setting, we need to make a vector field choice on $\tilde{\Omega}_t$; this will simply 
be the velocity $v$, though we could have equally chosen $v_h$. We remark that in the corresponding argument in \cite{ifrim2020compressible}
the average of the two was used, in order to better symmetrize the problem. However, the argument here is slightly more robust, and such a choice is not needed.
\\

For this choice of $v$, we examine the boundary weight 
$v \cdot n_{\partial D(t)} - v_b$ appearing in the above formula.
For this we use the disjoint boundary decomposition
\[
\tilde{\Gamma} = \mathcal{A} \cup \mathcal{A}_h \cup (\Gamma \cap \Gamma_h) ,
\]
where the normal $n_{\tilde{\Gamma}}$ is given a.e.~by 
\[
n_{\tilde{\Gamma}} = \left\{
\begin{aligned} 
&n_\Gamma \ \ \text{in } \mathcal{A} \cup (\Gamma \cap \Gamma_h), \\
&n_{\Gamma_h} \ \text{in } \mathcal{A}_h \cup (\Gamma \cap \Gamma_h),
\end{aligned} 
\right.
\]
with the two normals agreeing a.e.~on $\Gamma \cap \Gamma_h$.
Correspondingly, for almost every point on $\tilde{\Gamma}$ we have $|v_b-v\cdot n_{\tilde{\Gamma}}|\leq |v-v_h|$, as can be seen by working with the collar parameterization $\eta_\Gamma\wedge\eta_{\Gamma_h}$ for $\tilde{\Gamma}$ and the kinematic boundary conditions for $\Gamma$ and $\Gamma_h$.
\\

We now use Proposition~\ref{Reynolds} and the incompressibility of $v$
for each of the three terms in \eqref{dt-dist}. We begin by studying the first term, where we obtain
\begin{equation}\label{first term}
\begin{split}
\frac{1}{2}\frac{d}{dt}\int_{\tilde{\Omega}}|v-v_h|^2\, dx\leq \frac{1}{2}\int_{\tilde{\Omega}}D_t|v-v_h|^2\, dx +\frac{1}{2}\int_{\tilde{\Gamma}}|v-v_h|^3\,dS.
\end{split}    
\end{equation}
We note that, unlike in the case of the linearized equation, here we obtain a nonzero boundary term.  However, this term has the redeeming feature that it is cubic in the difference $v-v_h$. To estimate it, we use a simple variant of the trace theorem. Indeed, as $\Gamma,\Gamma_{h}\in\Lambda_*$, we may find a smooth vector field $X$ defined on $\mathbb{R}^d$ with $C^k$ bounds uniform in $\Lambda_*$ which is also uniformly transverse to $\tilde{\Gamma}$. By the divergence theorem, we then have
\begin{equation}\label{trace1}
\begin{split}
\frac{1}{2}\int_{\tilde{\Gamma}}|v-v_h|^3\,dS&\lesssim \int_{\tilde{\Gamma}}X\cdot n_{\tilde{\Gamma}}|v-v_h|^3\,dS\lesssim (B+B_h)\|v-v_h\|_{L^2(\tilde{\Omega})}^2
\\
&\lesssim (B+B_h)D(v,v_h).
\end{split}
\end{equation}
 Now, for the remaining term in (\ref{first term}), we use (\ref{veqn}) and integrate by parts to obtain
\begin{equation}\label{interior1}
\begin{split}
\frac{1}{2}\int_{\tilde{\Omega}}D_t|v-v_h|^2\,dx&=\int_{\tilde{\Omega}}(v-v_h)D_t(v-v_h)\,dx
\\
&=-\int_{\tilde{\Gamma}}(p-p_h)(v-v_h)\cdot n_{\tilde{\Gamma}}\,dS+\int_{\tilde{\Omega}}(v-v_h)\cdot [(v_h-v)\cdot\nabla v_h]\,dx   
\\
&\leq -\int_{\tilde{\Gamma}}(p-p_h)(v-v_h)\cdot n_{\tilde{\Gamma}}\,dS+(B+B_h)D(v,v_h).
\end{split}
\end{equation}
Using the decomposition $\tilde{\Gamma}=\mathcal{A}\cup \mathcal{A}_h\cup (\Gamma\cap \Gamma_{h})$ and using that $p-p_h=0$ on $\Gamma\cap \Gamma_{h}$ by the dynamic boundary condition \eqref{BC1}, we can write
\begin{equation*}
\begin{split}
-\int_{\tilde{\Gamma}}(p-p_h)(v-v_h)\cdot n_{\tilde{\Gamma}}\,dS&=-\int_{\mathcal{A}}(p-p_h)(v-v_h)\cdot n_{\Gamma}\,dS-\int_{\mathcal{A}_h}(p-p_h)(v-v_h)\cdot n_{\Gamma_{h}}\,dS
\\
&=\int_{\mathcal{A}}a^{-1}(p-p_h)(v-v_h)\cdot\nabla p\,dS+\int_{\mathcal{A}_h}a_h^{-1}(p-p_h)(v-v_h)\cdot\nabla p_h \,dS.
\end{split}
\end{equation*}
Now, define
\begin{equation*}
J:=\int_{\mathcal{A}}a^{-1}(p-p_h)(v-v_h)\cdot\nabla p\,dS+\frac{1}{2}\frac{d}{dt}\int_{\mathcal{A}}a^{-1}|p-p_h|^2\,dS,    
\end{equation*}
and
\begin{equation*}
J_{h}:=\int_{\mathcal{A}_h}a_h^{-1}(p-p_h)(v-v_h)\cdot\nabla p_h \,dS+\frac{1}{2}\frac{d}{dt}\int_{\mathcal{A}_h}a_h^{-1}|p-p_h|^2\,dS.    
\end{equation*}
Combining (\ref{trace1}) and (\ref{interior1}), we  obtain
\begin{equation*}
\frac{d}{dt}D(v,v_h)\lesssim (B+B_h)D(v,v_h)+J+J_{h}.    
\end{equation*}
It remains  to show that
\begin{equation*}
J+J_{h}\lesssim_A (B+B_h)D(v,v_h). 
\end{equation*}
We show the details for $J$. The treatment of $J_{h}$ will be virtually identical. We begin by using \Cref{Leibniz}  to expand
\begin{equation}\label{Int on s_h>0}
\begin{split}
\frac{1}{2}\frac{d}{dt}\int_{\mathcal{A}}a^{-1}|p-p_h|^2\,dS&=-\frac{1}{2}\int_{\mathcal{A}}a^{-2}D_ta|p-p_h|^2\,dS-\frac{1}{2}\int_{\mathcal{A}}a^{-1}|p-p_h|^2 [n_{\Gamma}\cdot\nabla v\cdot n_{\Gamma}]\,dS
\\
&+\int_{\mathcal{A}}a^{-1}(p-p_h)D_t(p-p_h)\,dS.
\end{split}    
\end{equation}

The validity of the identity \eqref{Int on s_h>0} is justified by noting that $|p-p_h|^2$ vanishes to second order on $\Gamma\cap \Gamma_h$, so one can extend by zero to  write the integral on the left-hand side as an integral over $\Gamma$, apply standard identities there, and then return to an integral over $\mathcal{A}$. From \eqref{Int on s_h>0} and adding the first term in the definition of $J$, we obtain (noting that by the kinematic and dynamic boundary conditions, we have $D_tp=0$ on $\mathcal{A}$),
\begin{equation*}
J\lesssim_A -\int_{\mathcal{A}}a^{-1}(p-p_h)D_t^hp_hdS+\int_{\mathcal{A}}a^{-1}(p-p_h)(v-v_h)\cdot\nabla (p-p_h)dS+BD(v,v_h).
\end{equation*}
In the above, we used the standard identity \eqref{Moving normal} to control $D_ta$. For the first term on the right-hand side we use that $D_t^hp_h$ vanishes on $\Gamma_{h}$, \eqref{FTC pressure average}, the fundamental theorem of calculus, the Taylor sign condition and (\ref{pressurebounddiff}), to estimate
\begin{equation*}
\begin{split}
    |D_t^hp_h|&\lesssim_A \|\nabla D_t^hp_h\|_{L^{\infty}}|s_h|\approx_A \|\nabla D_t^hp_h\|_{L^{\infty}}|p-p_h|\lesssim_A (B+B_h)|p-p_h|.
\end{split}
\end{equation*}
Hence,
\begin{equation*}
\int_{\mathcal{A}}a^{-1}(p-p_h)D_t^hp_hdS\lesssim_A (B+B_h)D(v,v_h).
\end{equation*}
It remains to estimate the cubic term,  and show that
\begin{equation}\label{last-cubic}
\begin{split}
\left| \int_{\mathcal{A}}a^{-1}(p-p_h)(v-v_h)\cdot\nabla (p-p_h)\, dS \right|
\lesssim_A (B+B_h)D(v,v_h).
\end{split}
\end{equation}
We will need to perform a more careful analysis here, so that only the pointwise control terms appear in the estimate. Note that if we had instead settled for $L^2$ based control parameters, this cubic term could be handled relatively easily.
\\

We recall that $\mathcal A \subseteq \Gamma$. Given a point 
$x \in \mathcal A$, its distance 
to $\Gamma_h$ is proportional 
to $|(p-p_h)(x)|$. We consider
 a locally finite Vitali type covering of the set $\mathcal{A}$ with countably many balls $B_j = B(x_j,r_j)$ of radius $r_j$ proportional to $|(p-p_h)(x_j)|$, so that in particular we have $B_j \subseteq \Omega_h$. We denote by $D_j$ the energy of the difference in the region $B_j$, i.e., the integral in \eqref{diff functional candidate} restricted to $B_j$. Then 
 \[
\sum_j D_j \lesssim  D((v,\Gamma),(v_h,\Gamma_h)).
\]
 Hence, by the uniform bound on $a^{-1}$, it would suffice 
to show that
\begin{equation}\label{cubic-err}
\int_{\mathcal{A} \cap B_j}\left|(p-p_h)(v-v_h)\cdot\nabla (p-p_h)\right|\,dS \lesssim_A (B+B_h) D_j.
\end{equation}
We will indeed show that this bound holds for the bulk of the 
expression on the left. However, for the remaining part we will return to a global argument. For $\mathcal{A}$ we just use the uniform Lipschitz bound
in this analysis. We first note that in $\tilde \Omega \cap B_j$ we have 
\[
|p-p_h|\approx_A r_j,
\]
which after integration yields a good bound for $r_j$ within $B_j$:
\begin{equation}\label{rj}
\int_{\mathcal{A} \cap B_j}|p-p_h|^2 \, dS \approx_A r_j^{d+1} \lesssim_A D_j.
\end{equation}
Next we consider $v - v_h$, for which we use the $C^\frac12$
norm, which is part of our control norm $A$, in order 
to estimate the surface integral by the ball integral. This yields
\begin{equation}\label{vj}
\int_{\mathcal{A} \cap B_j}|v-v_h|^2 \, dS \lesssim_A 
r_j^{-1} \int_{\tilde \Omega \cap B_j}|v-v_h|^2\, dx + r_j^d A^2  
\lesssim_A r_j^{-1} D_j + r_j^d A^2 \lesssim_A r_j^{-1}D_j.
\end{equation}

It remains to consider $\nabla (p-p_h)$. Our starting point is the 
global bound 
\begin{equation}\label{nablap}
\| \nabla p \|_{C^\frac12(\Omega)} +
\| \nabla p_h \|_{C^\frac12(\Omega_h)} 
\lesssim_A B+B_h,
\end{equation}
which is noted in (\ref{pressurebounddiff}). This allows us to replace $\nabla(p-p_h)$ with its average $\overline{\nabla (p-p_h)}_{j}$ in 
any smaller ball $\tilde B_j \subseteq \tilde \Omega \cap B_j$ of comparable size, because 
\[
\| \nabla(p-p_h) - \overline{\nabla (p-p_h)}_{j}\|_{L^\infty(\tilde \Omega \cap B_j)} \lesssim_A r_j^\frac12 (B+B_h).
\]
Putting everything together we arrive at 
\[
\int_{\mathcal{A} \cap B_j}|(p-p_h)(v-v_h)\cdot(\nabla(p-p_h) - \overline{\nabla (p-p_h)}_j)|  \, dS \lesssim_A
(B+B_h) D_j,
\]
which represents the bulk of \eqref{cubic-err}.
\\

It remains to estimate the contribution of the local average of 
$\nabla (p-p_h)$. Here we view $p-p_h$ as a solution to the following Laplace equation in $\tilde \Omega$:
\[
\left\{
\begin{aligned}
&\Delta (p-p_h) = -\text{tr}(\nabla v)^2 +\text{tr}(\nabla v_h)^2,
\\
& {p-p_h}_{| \tilde{\Gamma}} = \tg:= p 1_{\mathcal{A}_h} - p_h 1_{\mathcal{A}}  .    
\end{aligned}
\right.
\]
We  split the problem for $p-p_h$ into an  inhomogeneous one  with homogeneous boundary condition, and a homogeneous one  with inhomogeneous boundary condition,
\[
p-p_h = (p-p_h)_{inh} + (p-p_h)_{hom}.
\]

For the   inhomogeneous problem we can write the source term in divergence form to estimate
\[
\| \text{tr}(\nabla v)^2 -\text{tr}(\nabla v_h)^2\|_{H^{-1}(\tilde \Omega)} 
\lesssim (B+B_h)D^\frac12,
\]
which by a simple energy estimate gives a global $L^2$ bound
\[
\| \nabla (p-p_h)_{inh}\|_{L^2(\tilde{\Omega})} \lesssim_A (B+B_h)D^\frac12.
\]
This in turn yields a  bound for the corresponding averages by H\"older's inequality,
\[
\sum_j r_j^{d} |\overline{\nabla(p-p_h)_{inh,j}}|^2 \lesssim_A (B+B_h)^2 D.
\]
The contribution of this into \eqref{last-cubic} is then estimated using  \eqref{rj} and \eqref{vj} as follows:
\[
\begin{aligned}
J_{inh}:= & \ \sum_j \int_{\mathcal{A} \cap B_j} |p-p_h||v-v_h||\overline{\nabla (p-p_h)_{inh,j}}| \,dS  
\\
\lesssim_A & \  \sum_j r_j^{\frac{d+1}2} \|v-v_h\|_{L^2(\mathcal{A} \cap B_j)} |\overline{\nabla (p-p_h)_{inh,j}}|
\\
\lesssim_A & \  \sum_j   D_j^\frac12  r_j^{\frac{d}2}  |\overline{\nabla (p-p_h)_{inh,j}}|
\\
\lesssim_A & \ (B+B_h)D,
\end{aligned}
\]
where in the last step we have used Cauchy-Schwarz with respect to $j$.
\\

For the homogeneous term, on the other hand, we need to carefully examine the regularity of the Dirichlet data $\tg$. On one hand, by the definition 
of the distance $D$ we have the $L^2$ bound
\begin{equation}\label{g-2}
 \|\tg\|_{L^2(\tilde \Gamma)}^2 \lesssim_A D.    
\end{equation}
On the other hand, by \eqref{nablap}, on each of the two  regions $\mathcal A_h$ respectively $\mathcal A$, we have formally 
\begin{equation}\label{g-holder}
 \|\tg\|_{C^{1,\frac12}(\mathcal A_h)}
 + \|\tg\|_{C^{1,\frac12}(\mathcal A)}
 \lesssim_A B+B_h.    
\end{equation}
This bound has to be carefully interpreted,
which we do within the proof of Lemma~\ref{l:dg-6}
below. 
\\

A formal interpolation between \eqref{g-2}
and \eqref{g-holder} would yield a $W^{1,6}(\tilde{\Gamma})$ bound 
for $\tg$. We make this bound rigorous in the following.

\begin{lemma}\label{l:dg-6}
The function $\tg$ above satisfies the bound
\begin{equation}\label{dtg-6}
\|\tg\|_{W^{1,6}(\tilde \Gamma)}  \lesssim (B+B_h)^\frac23 D^\frac16.
\end{equation}
\end{lemma}
\begin{proof}
We begin by noting that the two components  $g := p 1_{\mathcal A_h}$ and $g_h := -p_h 1_{\mathcal A}$
of $\tg$ are nonzero on  disjoint sets $\mathcal A_h$ respectively $\mathcal A$, and vanish 
on the corresponding boundaries $\partial \mathcal A_h$, respectively $\partial \mathcal A$. Hence,
we can prove the bound \eqref{dtg-6} separately for the two components. We consider $g$, which 
lives on $\mathcal A_h \subseteq \Gamma_h$. 
Here not only is $\Gamma_h$ a Lipschitz surface,
but it also has a $C^{1,\frac12}$ bound of $B_h$
(which is not the case for $\tilde \Gamma$).
\\

Using a standard partition of unity we can reduce the problem to the case when $\Gamma_h$ is a graph,
\[
\Gamma_h = \{ x_d = \phi(x')\},
\]
where 
\begin{equation}\label{phi-reg}
\|\phi\|_{Lip} \lesssim_A 1, \qquad \| \phi\|_{C^{1,\frac12}}
\lesssim B_h.
\end{equation}
We denote the Lipschitz projection of $\mathcal A_h$ by 
$\mathcal {PA}_h \subseteq \R^{d-1}$. We can equivalently 
consider $g$ as a function on $\mathcal {PA}_h$, in which case
the bound \eqref{dtg-6} becomes
\begin{equation}\label{dg-6}
\| \nabla g\|_{L^6(\mathcal {PA}_h)}  \lesssim_A (B+B_h)^\frac23 D^\frac16.
\end{equation}
We now summarize the information that we have on $g$ as a function on $\mathcal {PA}_h$:
\begin{enumerate}[label=(\roman*)]
    \item ($L^2$ control).
\begin{equation*}\label{pg-2}
 \|g\|_{L^2(\mathcal {PA}_h)}^2 \lesssim_A D,   
\end{equation*}
 which comes from \eqref{g-2}.
 \item (H\"older control).
 \begin{equation*}
 \|\nabla g\|_{C^{\frac12}(\mathcal {PA}_h)} \lesssim_A B+B_h, 
 \end{equation*}
which is a consequence of \eqref{nablap}, \eqref{phi-reg} 
and chain rule.

\item (Zero boundary data).
\begin{equation*}
 g = 0 \ \  \text{on} \  \ \partial  \mathcal {PA}_h.    
\end{equation*}
 \end{enumerate}
We will prove that these three properties imply the desired bound \eqref{dg-6}. The difficulty here is that we do not know that $\nabla g = 0$ on $\partial  \mathcal {PA}_h$; 
else we could simply extend $g$ by $0$ outside $\mathcal {PA}_h$ and this becomes a standard interpolation bound.
Further, we do not a priori control the regularity 
of the boundary $\partial  \mathcal {PA}_h$. 
\\

Without any loss of generality we assume that $g > 0$
on $\mathcal 
{PA}_h$; else we split this set into connected components
where $g$ has constant sign, modulo a set where $\nabla g = 0$ a.e.~To prove the desired bound we will use a well-chosen Vitali covering of the set $S = \mathcal {PA}_h \setminus \{ \nabla g = 0\}$ with balls. This choice is as follows: For each $x \in S$ we consider 
a ball $B_x= B(x,r_x)$ with radius $r_x = c^2 (B+B_h)^{-2}|\nabla g(x)|^2$ where $c>0$ is a small universal constant,  chosen so that $|\nabla g|$ is nearly constant on $B_x$, i.e.,
\[
|\nabla g(y)-\nabla g(x)| \lesssim c | \nabla g(x)|
\ll |\nabla g(x)|, \qquad y \in B_x.
\]
The union of the balls $B_x$ with $x \in S$ clearly covers $S$, so Vitali's lemma allows us to extract 
a countable disjoint subfamily of such balls $B_j= B_{x_j}$ so that 
\[
S \subseteq \bigcup 5 B_j.
\]

Since $\nabla g$ is almost constant on $B_x$ and $g(x) > 0$,
a key observation is that there must exist a nontrivial sector
$C_x \subseteq B_x$ where 
\[
g > 0 \quad \text{in}\  C_x, \qquad |C_x|\approx |B_x|.
\]
Since $g=0$ on $\partial \mathcal{PA}_h$,
it follows that we must have $C_x \subseteq S$; this is what allows us to bypass the lack of geometric information on the set
$\mathcal {PA}_h$.
\\

On $C_x$, the function $g$ is almost linear with slope approximately $|\nabla g(x)|$.
Therefore, we must have
\begin{equation*}
    \| g\|_{L^2(C_x)}^2 \gtrsim r_x^{d+1} |\nabla g(x)|^2.
\end{equation*}
We will use this bound to estimate from above the $L^6$ norm of $\nabla g$ in each $5B_j$ as follows:
\[
\begin{aligned}
\| \nabla g\|_{L^6(5B_j)}^6 \lesssim & \ r_{x_j}^{d-1} |\nabla g(x_j)|^6
\\
\lesssim & \ \| g\|_{L^2(C_j)}^2 r_{x_j}^{-2}|\nabla g(x_j)|^4 
\\
\approx & \ \| g\|_{L^2(C_j)}^2 (B+B_h)^4.
\end{aligned}
\]
Now, we sum over $j$, using the disjointness of the balls $B_j$ and thus of $C_j$. This gives
\[
\sum_{j}\| \nabla g\|_{L^6(5B_j)}^6 \lesssim 
\| g\|_{L^2(S)}^2 (B+B_h)^4 \lesssim_A D (B+B_h)^4,
\]
which concludes the proof of the lemma.

\end{proof}

Now we use the bound in \Cref{l:dg-6} to solve the homogeneous Dirichlet problem
in $\tilde \Omega$ and to obtain the estimate
\[
\| \nabla (p-p_h)_{hom}^*\|_{L^6(\tilde \Gamma)} \lesssim (B+B_h)^\frac23 D^\frac16 ,
\]
where $*$ stands for the nontangential maximal function. This bound is due to Verchota~\cite{MR0769382}, but see also the further 
discussion by Jerison-Kenig \cite[Theorem 5.6]{MR1331981} as well as the case of $C^1$ boundaries 
considered earlier by  Fabes-Jodeit-Rivi\`ere~\cite{FJR}. 
\\

The exponent $6$ is allowed above provided that the 
Lipschitz norm of the boundary is sufficiently small. Precisely, the upper limit of the allowed exponents goes to infinity as the corner size decreases to $0$.
The smallness of the intersection angle between $\Gamma$ and $\Gamma_h$ is a consequence of the $C^{1,\epsilon}$ common regularity bound together 
with the use of a sufficiently refined collar region.
\\

To use the nontangential maximal function bound, 
within the ball $B_j = B(x_j,r_j)$ we consider a smaller ball 
\[
\tilde B_j = B(x_j - \frac12 r_j n_j, \frac14 r_j).
\]
For $y \in \tilde B_j$ we have 
\[
|\nabla (p-p_h)_{hom}(y)| \lesssim 
|\nabla (p-p_h)_{hom}^* (z)|, \qquad z \in \tilde \Gamma \cap \frac14 B_j.
\]
Taking averages on the left and integrating on the right, we arrive at 
\[
r_j^{d-1} |\overline{\nabla (p-p_h)_{hom,j}}|^6 \lesssim_A 
\|\nabla (p-p_h)_{hom}^*\|_{L^6(\tilde  \Gamma \cap \frac14 B_j)}^6.
\]
Since the balls $B_j$ are disjoint, summation in $j$ yields 
\begin{equation}\label{phom-last}
\sum_j r_j^{d-1} |\overline{\nabla (p-p_h)_{hom,j}}|^6 \lesssim 
(B+B_h)^4 D .
\end{equation}

On the other hand, for $v-v_h$ we use the interpolation bound \eqref{trace1}, which gives
\begin{equation}\label{v-last}
\| v-v_h\|_{L^3(\tilde \Gamma)} \lesssim (B+B_h)^\frac13 D^{\frac13}.
\end{equation}

We are now ready to estimate the corresponding contribution to
\eqref{last-cubic} using also \eqref{rj} and \eqref{vj} as follows: 

\[
\begin{aligned}
J_{hom}:= & \ \sum_j \int_{\mathcal{A} \cap B_j} |p-p_h||v-v_h||\overline{\nabla (p-p_h)_{hom,j}}|\, dS  
\\
\lesssim_A & \  \sum_j r_j  (r_j^{\frac{2(d-1)}{3}} \|v-v_h\|_{L^3(\mathcal{A} \cap B_j)})   |\overline{\nabla (p-p_h)_{hom,j}}|
\\
\lesssim_A & \  \sum_j   r_j^{\frac{d+1}{2}} \|v-v_h\|_{L^3(\mathcal{A} \cap B_j)}  (r_j^{\frac{d-1}6}  |\overline{\nabla (p-p_h)_{hom,j}}|)
\\
\lesssim_A & \ (B+B_h)D.
\end{aligned}
\]

At the last step we have applied H\"older's inequality in $j$ 
with exponents $2$, $3$ and $6$,
using \eqref{rj}, \eqref{v-last} and \eqref{phom-last}.
 This completes the proof of \eqref{cubic-err} and therefore the proof of \Cref{Difference}.

\end{proof}
One consequence of the difference bounds is the following uniqueness result.
\begin{theorem}[Uniqueness] \label{t:unique} Let $\epsilon>0$ and let $\Omega_0$ be a bounded domain with boundary $\Gamma_0\in\Lambda(\Gamma_*,\epsilon,\delta)$. Then for $\Gamma_0\in C^{1,\frac{1}{2}}$ and divergence free $v_0\in W^{1,\infty}(\Omega_0)$ satisfying the Taylor sign condition, the free boundary Euler equations admit at most one solution $(v,\Gamma_t)$ on a time interval $[0,T]$ with $\Gamma_t\in \Lambda(\Gamma_*,\epsilon,\delta)$ and
\begin{equation*}
\sup_{0\leq t\leq T}\|v\|_{C^{\frac{1}{2}+\epsilon}_x(\Omega_t)}+\int_{0}^{T}\|v\|_{W^{1,\infty}_x(\Omega_t)}+\|D_tp\|_{W^{1,\infty}_x(\Omega_t)}+\|\Gamma_t\|_{C_x^{1,\frac{1}{2}}}\,dt<\infty.    
\end{equation*}
\end{theorem}
\begin{proof}
Suppose $(v,\Omega_t)$ and $(v_h,\Omega_{t,h})$ are a pair of solutions satisfying the conditions of the theorem with the same initial data. From the differences estimates, we immediately obtain $v=v_h$ on $\Omega_t\cap\Omega_{t,h}$. Next, we argue that the domain $\Omega_t$ coincides with $\Omega_{t,h}$. First, we note that the intersection is non-empty if $\delta>0$ is small enough. We now show $\Omega_t\subseteq\Omega_{t,h}$. It suffices to show $\Omega_t\subseteq\overline{\Omega}_{t,h}$. If this is not true, then there is $x\in\Gamma_{t,h}$ such that $x\in \Omega_t$. Such a point must lie on $\partial(\Omega_t\cap\Omega_{t,h})$. Therefore, from the estimate for the distance functional, we have $p(x)=0$. However, within a small enough collar neighborhood, the Taylor sign condition tells us that the level set $\{p=0\}$ corresponds exactly to the free surface $\Gamma_t$. This is a contradiction to $x$ being an interior point of $\Omega_t$. Therefore $\Omega_t\subseteq\Omega_{t,h}$. The reverse inclusion follows by an identical argument.
\end{proof}


\section{Balanced elliptic estimates}\label{BEE}
In this section, we prove a collection of  refined elliptic estimates which will be crucial for obtaining the sharp pointwise control norms in the higher energy bounds. These estimates will turn out to be quite general and should be applicable to other free boundary problems.  In a sense, they can be seen as significant refinements of the so-called \emph{tame estimates} which have been fundamental in the analysis of many water waves problems (see the discussion in \cite{MR3460636,MR2138139}), but are not nearly sufficient for our purposes. Indeed, as we will soon see, our proofs of the higher energy bounds   require estimates for various elliptic operators which more precisely balance the contributions of the input function and the domain regularity, simultaneously, in both pointwise and $L^2$ based norms. This simultaneous balance cannot be achieved with the known tame estimates, which often only seem to balance the contributions in $L^2$ based norms or involve domain dependent constants in pointwise norms which are significantly off scale. The  technical utility of our balanced estimates  will become readily apparent in \Cref{HEB}, where they will be used to efficiently dispatch with expressions involving relatively complicated iterated applications of the Dirichlet-to-Neumann operator and various other elliptic operators.
\\

In the following, we will always assume that $\Omega$ is a bounded domain with boundary $\Gamma\in\Lambda_*:=\Lambda(\Gamma_*,\epsilon_0,\delta)$ for suitably small (but fixed) constants $\epsilon_0,\delta>0$. Most of the bounds in this section do not make reference to a particular velocity function, and so, the implicit constants in many of the estimates will only depend on the surface component of the control parameter $A$; namely, $A_{\Gamma}:=\|\Gamma\|_{C^{1,\epsilon_0}}$. Hence, for this section, by the relation $X\lesssim_A Y$, we mean $X\leq C(A_{\Gamma})Y$ for some constant $C$ depending exclusively on $A_{\Gamma}$. The only exception to this rule (which we will make note of explicitly) will be  in \Cref{Movingsurfid}, where we will use the full control parameter $A$ to establish estimates for commutators of various elliptic operators with $D_t$.  We will also harmlessly let $A$ depend on the domain volume throughout, as the volume of the domain will be conserved in the dynamic problem.
\\

Throughout the section, by a slight abuse of notation, we will follow the convention that a parameter  $\epsilon$ may vary from line to line by a fixed scalar factor. Generally speaking, we will take $\epsilon>0$ to be any positive constant with $\epsilon\ll\epsilon_0$.

\subsection{Extension operators in \texorpdfstring{$\Lambda_*$}{} and product type estimates on \texorpdfstring{$\Omega$}{}}
To establish the desired elliptic estimates, it will be convenient to have an extension operator which is bounded from $H^s(\Omega)\to H^s(\mathbb{R}^d)$ for $s\geq 0$, and $C^{k,\alpha}(\Omega)\to C^{k,\alpha}(\mathbb{R}^d)$ for a suitable range of $k$ and $\alpha$ with bounds depending only on the implicit constant $A$. Among other things, this will enable us to recover many of the standard product type estimates which are well-known on $\mathbb{R}^d$. To this end, let $\varphi:\mathbb{R}^{d-1}\to \mathbb{R}$ be a Lipschitz function with Lipschitz constant $M$. Let $\Omega=\{(x,y)\in \mathbb{R}^d: y>\varphi(x)\}$. Moreover, for $1\leq p\leq\infty$ and an integer $k\geq 0$, let $W^{k,p}(\Omega)$ denote the usual Sobolev space consisting of distributions whose derivatives up to order $k$ belong to $L^p(\Omega)$. It is a classical result of Stein \cite[Theorem 5', p.~181]{MR0290095} that there exists a linear operator $\mathcal{E}$ mapping functions on $\Omega$ to functions on $\mathbb{R}^d$ with the property that $\mathcal{E}: W^{k,p}(\Omega)\to W^{k,p}(\mathbb{R}^d)$ is well-defined and continuous for all $1\leq p\leq \infty$ and integers $k$. Moreover, the norm of $\mathcal{E}: W^{k,p}(\Omega)\to W^{k,p}(\mathbb{R}^d)$ depends only on the dimension $d$, the order of differentiability $k$ and the Lipschitz constant $M$. The operator $\mathcal{E}$ is called \emph{Stein's extension operator}.  As one can see directly from its definition \cite[Equation (24), p.~182]{MR0290095}, $\mathcal{E}$ also maps $C^1(\Omega)\to C^1(\mathbb{R}^d)$.
\\

As explained in Section 3.3 of \cite{MR0290095}, a partition of unity argument allows one to construct an extension operator $\mathcal{E}=\mathcal{E}_\Omega$ on all Lipschitz domains $\Omega$, with constant depending only on $d,k,p$, the number and size of the balls needed to cover the boundary, and the Lipschitz constant of the defining function on each ball. Since for a tight enough collar $ \Lambda_*$ one can use the same balls to cover all elements of $ \Lambda_*$ with control of the Lipschitz constant on each ball, this shows that Stein's extension operator has norm bounds that are uniform for domains with boundary in $\Lambda_*$.
\\

In the above discussion, the definition of the $W^{k,p}$ norm was the usual one, defined by requiring the first $k$ weak-derivatives to be in $L^p$. However, as noted earlier, we also define the $H^s$ norm of a function $f$ as the infimum of the $H^s$ norms of all possible extensions of $f$ to $\mathbb{R}^d$. Clearly, $\|\cdot\|_{W^{k,2}}\lesssim\|\cdot\|_{H^{k}}$ with constant independent of the domain. However, by the above, for domains with boundary in $\Lambda_*$, the reverse inequality also holds, with implicit constant depending on $A_{\Gamma}$.
\\

From \cite[Theorem B.8]{MR1742312} we know that for any non-empty open subset $\Omega$ of $\mathbb{R}^d$ and any $s_0,s_1\in \mathbb{R}$ we have the identification
\begin{equation*}
    \left(H^{s_0}(\Omega),H^{s_1}(\Omega)\right)_{\theta,2}=H^s(\Omega), \ \text{where} \ s=(1-\theta)s_0+\theta s_1 \ \text{and} \ 0<\theta<1,
\end{equation*}
with equivalent norms uniform in the collar. Thus, by interpolation, we have the following result.
\begin{proposition}\label{Stein} 
 Let $\Omega$ be a bounded domain with boundary $\Gamma\in \Lambda_*$. Then for every $s\geq 0$ and $0\leq\alpha\leq 1+\epsilon_0$, Stein's extension operator $\mathcal{E}$ satisfies
\begin{equation*}
\|\mathcal{E}\|_{C^{\alpha}(\Omega)\to C^{\alpha}(\mathbb{R}^d)},\hspace{5mm}\|\mathcal{E}\|_{H^s(\Omega)\to H^s(\mathbb{R}^d)}\lesssim_A 1
\end{equation*}
uniformly in $ \Lambda_*$. 
\end{proposition}
\begin{proof}
    The $H^s$ case follows from interpolation between integer powers. For $C^{\alpha}$, we first note from \cite[Theorem A.1]{MR1742312} (and higher order variants, c.f.~\cite[Lemma 6.37]{MR1814364}) that there are extension operators with the above $C^{\alpha}\to C^{\alpha}$ bound. That Stein's operator has this property then follows by making use of such extensions and interpolating, similar to \cite[p.~11-12]{MR3753604}. 
\end{proof}
\begin{remark}
    As mentioned in \cite[Proposition 2.17]{MR1331981}, by an interpolation argument, one can also prove that Stein's extension operator maps  the Besov space $B_\alpha^{p,q}(\Omega)$ to $B_\alpha^{p,q}(\mathbb{R}^d)$ for all $\alpha>0$, $1\leq p,q\leq \infty$ and  Lipschitz domains $\Omega$. However, we will not require anything this precise.
\end{remark}
\subsection{Littlewood-Paley decomposition and paraproducts on \texorpdfstring{$\Omega$}{}} Using the Stein extension operator, many of the standard paraproduct estimates on $\mathbb{R}^d$ pass over to $\Omega$. 
\subsubsection{Littlewood-Paley decomposition} For a distribution $u$ on $\mathbb{R}^d$, we will make use of the standard Littlewood-Paley decomposition
\begin{equation*}
u=\sum_{k\geq 0}P_ku,  
\end{equation*}
where for $k>0$, $P_k$ corresponds to a Fourier multiplier with smooth symbol supported in the dyadic frequency region $|\xi|\approx 2^k$ and $P_0$ corresponds to a multiplier localized to the unit ball. The notation $P_{<k}$, $P_{\leq k}$, $P_{\geq k}$ and $P_{>k}$ will have the usual meaning. Using the Stein extension operator, we may  also consider Littlewood-Paley projections when $u$ is defined only on $\Omega$. In this case, we abuse notation, and  write $P_ku$ instead of $P_k\mathcal{E}u$, with corresponding definitions for $P_{<k}$, $P_{\leq k}$, etc. We will also often write $u_k$, $u_{<k}$, etc.~as shorthand for the above operators applied to $u$.
\subsubsection{Paraproducts on $\Omega$} The above decomposition allows us to make use of some of the standard tools of paradifferential calculus (see e.g. \cite{bony1981calcul} and  \cite{metivier2008differential}) on $\mathbb{R}^d$ and apply them to functions defined on $\Omega$. For bilinear expressions, we will make heavy use of the Littlewood-Paley trichotomy (now defined for functions on $\Omega$ with suitable regularity),
\begin{equation*}
f\cdot g=T_fg+T_gf+\Pi(f,g)  ,  
\end{equation*}
where the above three terms correspond to the respective ``low-high", ``high-low" and ``high-high" frequency interactions between $f$ and $g$. More specifically, $T_fg$ is defined as 
\begin{equation*}
T_fg:=\sum_{k}f_{<k-k_0}g_k    ,
\end{equation*}
where $k_0$ is some universal parameter independent of $k$. We will be able to take, e.g., $k_0=4$ for most purposes.
\subsubsection{Bilinear estimates on $\Omega$}
One important consequence of the bounds for $\mathcal{E}$ and the corresponding inequality on $\mathbb{R}^d$ is the following algebra property for $H^s(\Omega)$, $s\geq 0$,
\begin{equation}\label{algebraprop}
\|fg\|_{H^s(\Omega)}\lesssim_A \|f\|_{H^s(\Omega)}\|g\|_{L^{\infty}(\Omega)}+\|g\|_{H^s(\Omega)}\|f\|_{L^{\infty}(\Omega)}.    
\end{equation}
In our estimates for the elliptic problems below, the bilinear terms above will frequently appear in the form $\partial_if\partial_jg$ where $f$ is some function defined on $\mathbb{R}^d$ encoding the regularity of the domain and the desired uniform bound for $g$ is below $C^1$. For this reason, in order to avoid negative H\"older norms inside a domain, we will need the following paraproduct type estimate, which we will use in the sequel.
\begin{proposition}[Bilinear paraproduct type estimate on $\Omega$]\label{productestref}
Let either i) $s> 0$ and $\alpha_1,\alpha_2,\beta\in [0,1]$ or ii) $s=0$, $\alpha_1=\alpha_2=1$ and $\beta\in[0,1]$. Then we have for any $r\geq 0$,
\begin{equation*}
\begin{split}
\|\partial_if\partial_jg\|_{H^s(\Omega)}&\lesssim_A\|g\|_{H^{s+2-\alpha_1}(\Omega)}\|f\|_{C^{\alpha_1}(\Omega)}+\| f\|_{H^{s+r+1}(\Omega)}\sup_{k>0}2^{-k(r+\alpha_2-1)}\|g_k^1\|_{C^{\alpha_2}(\Omega)}
\\
&\qquad +\|f\|_{C^{1,2\epsilon}(\Omega)}\sup_{k>0}2^{k(s+\beta-\epsilon)}\|g_k^2\|_{H^{1-\beta}(\Omega)},
\end{split}
\end{equation*}
where $g=g_k^1+g_k^2$ is any sequence of partitions of $g$ in $C^{\alpha_2}(\Omega)+H^{1-\beta}(\Omega)$.
\end{proposition}
\begin{proof}
 By \Cref{Stein}, it suffices to prove these  estimates for $f,g$ defined on $\mathbb{R}^d$. We prove the estimate for $0<\alpha_1,\alpha_2<1$ and $s>0$ as the other cases are more easily dealt with. We recall that for $0<\alpha<1$, the $C^{\alpha}$ norm on $\mathbb{R}^d$ can be characterized by the equivalent Besov norm,
\begin{equation}\label{Besov}
\|u\|_{C^{\alpha}(\mathbb{R}^d)}\approx \|P_{\leq 0} u\|_{L^{\infty}(\mathbb{R}^d)}+\sup_{j>0}2^{\alpha j}\|P_ju\|_{L^{\infty}(\mathbb{R}^d)}.
\end{equation}
We now decompose $\partial_if\partial_jg$ into paraproducts,
\begin{equation}\label{parad}
\partial_if\partial_jg=T_{\partial_if}\partial_jg+T_{\partial_jg}\partial_i f+\Pi(\partial_if,\partial_jg).
\end{equation}
We then have the standard estimate
\begin{equation*}
\begin{split}
\|T_{\partial_if}\partial_jg\|_{H^s(\mathbb{R}^d)}&\lesssim \|f\|_{C^{\alpha_1}(\mathbb{R}^d)}\|\partial_j g\|_{H^{s+1-\alpha_1}(\mathbb{R}^d)},
\end{split}
\end{equation*}
which follows by shifting $1-\alpha_1$ derivatives off of the low frequency factor and onto the high frequency factor in each term. Using the hypothesis $s>0$, the high-high paraproduct may be estimated by the same term. For the remaining low-high interaction, we write 
\begin{equation*}
T_{\partial_j g}\partial_i f=\sum_{k}P_{<k-4}\partial_j gP_k\partial_i f=\sum_{k}P_{<k-4}\partial_j (g_k^1)P_k\partial_i f+\sum_{k}P_{<k-4}\partial_j (g_k^2)P_k\partial_i f.    
\end{equation*}
Using standard Bernstein type inequalities and square summing, the first term on the right can be easily controlled by
\begin{equation*}
\|\partial_i f\|_{H^{s+r}(\mathbb{R}^d)}\sup_{k>0}2^{-k(r+\alpha_2-1)}\|g_k^1\|_{C^{\alpha_2}(\mathbb{R}^d)},
\end{equation*}
while the latter can be controlled by
\begin{equation*}
\|f\|_{C^{1,2\epsilon}(\mathbb{R}^d)}\sup_{k>0}2^{k(s+\beta-\epsilon)}\|g_k^2\|_{H^{1-\beta}(\mathbb{R}^d)}.    
\end{equation*}
\end{proof}
The following corollary of the above proposition will be used heavily in the higher energy bounds to control product terms on $\Omega$ with suitable pointwise control norms.
\begin{corollary}\label{productest} Let $s$ and $\alpha_1,\alpha_2$ be as in  \Cref{productestref}. Assume that $f\in H^{s+2-\alpha_2}(\Omega)\cap C^{\alpha_1}(\Omega)$ and $g\in H^{s+2-\alpha_1}(\Omega)\cap C^{\alpha_2}(\Omega)$. Then we have
\begin{equation*}
 \|\partial_if\partial_jg\|_{H^s(\Omega)}\lesssim_A \|g\|_{H^{s+2-\alpha_1}(\Omega)}\|f\|_{C^{\alpha_1}(\Omega)}+\|f\|_{H^{s+2-\alpha_2}(\Omega)}\|g\|_{C^{\alpha_2}(\Omega)}.   
\end{equation*}    
\end{corollary}
\begin{proof}
 This follows immediately from \Cref{productestref} by taking $g_j^2=0$ and $r=1-\alpha_2$.   
\end{proof}
\subsubsection{Generalized Moser type estimate} Next, we prove a Moser type estimate with the same flavor as the above bilinear estimate. The main purpose of this estimate will be to suitably control (extensions of) compositions of functions on $\Omega$ with diffeomorphisms of $\mathbb{R}^d$. This will be important for obtaining more refined elliptic estimates where we need to use such diffeomorphisms to flatten the boundary.
\begin{proposition}[Balanced Moser estimate]\label{Moservariant} Let $d\geq 1$ be an integer and let $G:\mathbb{R}^d\to\mathbb{R}^d$ be a diffeomorphism with $\|DG\|_{C^{\epsilon}}, \|DG^{-1}\|_{C^{\epsilon}}\lesssim_A 1$. Let $s\geq 0,r\geq 0$ and $\alpha,\beta\in [0,1]$. Then for every $F\in H^s(\mathbb{R}^d)$ and partition $F=F_j^1+F_j^2\in C^{\alpha}(\mathbb{R}^d)+H^{1-\beta}(\mathbb{R}^d)$, we have
\begin{equation*}
\|F(G)\|_{H^s(\mathbb{R}^d)}\lesssim_A \|F\|_{H^{s}(\mathbb{R}^d)}+\|G-Id\|_{H^{s+r}}\sup_{j>0}2^{-j(\alpha+r-1)}\|F_j^1\|_{C^\alpha(\mathbb{R}^d)}+\sup_{j>0}2^{j(s+\beta-1-\epsilon)}\|F_j^2\|_{H^{1-\beta}(\mathbb{R}^d)}.
\end{equation*}
\end{proposition}
\begin{remark}
The same estimate holds for $F\in H^s(\Omega)$ by replacing $F$ with its Stein extension.
\end{remark}
\begin{proof} The case $0\leq s\leq 1$ is a consequence of the following standard fact.
\begin{proposition}[Theorem 3.23 of \cite{MR1742312}]\label{McLean}
Let $0\leq s\leq 1$ and let $G: \mathbb{R}^d\to\mathbb{R}^d$ be a diffeomorphism with $\|DG\|_{L^{\infty}}\lesssim_A 1$ and $\|DG^{-1}\|_{L^{\infty}}\lesssim_A 1$. Then for every $F\in H^s(\mathbb{R}^d)$, we have
\begin{equation*}
\|F(G)\|_{H^s(\mathbb{R}^d)}\approx_A \|F\|_{H^s(\mathbb{R}^d)}.
\end{equation*}
\end{proposition}
Now, assume $s>1$. We begin by performing a Littlewood-Paley decomposition,
\begin{equation*}
\|F(G)\|_{H^s(\mathbb{R}^d)}^2\lesssim_{j_0} \|F(G)\|_{L^2(\mathbb{R}^d)}^2+\sum_{j>j_0}2^{2js}\|P_j(F(G))\|_{L^2(\mathbb{R}^d)}^2,
\end{equation*}
where $j_0>0$ is some fixed constant depending only on $A$, to be chosen later. We have
\begin{equation*}
2^{js}\|P_j(F(G))\|_{L^2(\mathbb{R}^d)}\lesssim 2^{js}\|P_j(F_{<j'}(G))\|_{L^2(\mathbb{R}^d)}+ 2^{js}\|P_j(F_{\geq j'}(G))\|_{L^2(\mathbb{R}^d)},
\end{equation*}
where $F_{<j'}:=P_{<j'}F$, $F_{\geq j'}:=F-F_{<j'}$ and $j':=j-j_1$ with $j_1$ being some parameter depending only on $s$ which will also be chosen later. For the latter term, by a change of variables and since $s>0$, we have
\begin{equation*}
    \sum_{j>j_0} 2^{2js}\|P_j(F_{\geq j'}(G))\|_{L^2(\mathbb{R}^d)}^2\lesssim_A \sum_{j>j_0}\sum_{k\geq j'} 2^{2(j-k)s}2^{2ks}\|P_kF\|_{L^2(\mathbb{R}^d)}^2\lesssim_A \|F\|_{H^s(\mathbb{R}^d)}^2.
\end{equation*}
On the other hand, using the fundamental theorem of calculus, we obtain
\begin{equation}\label{paratwoterms}
    \begin{split}
        2^{js}\|P_j(F_{<j'}(G))\|_{L^2(\mathbb{R}^d)}\lesssim 2^{js} \sup_{\tau\in [0,1]}\|P_j\left(DF_{<j'}(G_\tau)P_{\geq j'}G\right)\|_{L^2(\mathbb{R}^d)}+2^{js}\|P_j\left(F_{<j'}(P_{<j'}G)\right)\|_{L^2(\mathbb{R}^d)},
    \end{split}
\end{equation}
where
\begin{equation*}
 G_\tau=\tau P_{<j'}G+(1-\tau)G.   
\end{equation*}
Now, as $\|DG\|_{\dot{C}^{\epsilon}}, \|DG^{-1}\|_{\dot{C}^{\epsilon}}\lesssim_A 1$, it follows that $P_{<j'}G$ and $G_{\tau}$ (for $\tau\in [0,1])$ are invertible with $\|P_{<j'}DG\|_{L^{\infty}},\|DG_{\tau}\|_{L^{\infty}}\lesssim_A 1$ as long as $j_0$ is large enough (depending only on $A$ and the collar). Now, to control the first term on the right-hand side of (\ref{paratwoterms}), we split $F_{<j'}=(F_j^1)_{<j'}+(F_j^2)_{<j'}$ and estimate (using the estimate for $G_{\tau}^{-1}$), 
\begin{equation}\label{moserstep}
\begin{split}
 2^{js} \sup_{\tau\in [0,1]}\|P_j\left(DF_{<j'}(G_\tau)P_{\geq j'}G\right)\|_{L^2(\mathbb{R}^d)}&\lesssim_A 2^{-j(r+\alpha-1)}\|F_j^1\|_{C^{\alpha}(\mathbb{R}^d)}2^{j(s+r)}\|P_{\geq j'}G\|_{L^2(\mathbb{R}^d)}
 \\
 &\qquad +2^{j(s-1+\beta-\epsilon)}\|F_j^2\|_{H^{1-\beta}(\mathbb{R}^d)}.   
\end{split}
\end{equation}
Square summing (and possibly relabelling $\epsilon$) gives
\begin{equation*}
\begin{split}
\left(\sum_{j>j_0}2^{2js} \sup_{\tau\in [0,1]}\|P_j\left(DF_{<j'}(G_\tau)P_{\geq j'}G\right)\|_{L^2(\mathbb{R}^d)}^2\right)^{\frac{1}{2}}&\lesssim_A \sup_{j>0}2^{-j(r+\alpha-1)}\|F_j^1\|_{C^{\alpha}(\mathbb{R}^d)}\|G-Id\|_{H^{s+r}}
\\
&\qquad +\sup_{j>0}2^{j(s-1+\beta-\epsilon)}\|F_j^2\|_{H^{1-\beta}(\mathbb{R}^d)}.     
\end{split}
\end{equation*}
Next, we control the second term on the right-hand side of (\ref{paratwoterms}), which is a bit easier. Let $k$ be the largest integer strictly less than $s$ so that $0<s-k\leq 1$. If $j_1:=j-j'$ is large enough (depending only on $k$), we have by the chain rule and straightforward paraproduct analysis,
\begin{equation*}
2^{js}\|P_jF_{<j'}(P_{<j'}G)\|_{L^2(\mathbb{R}^d)}\lesssim_A 2^{j(s-k)}\|\tilde{P}_j(D^kF_{<j'}(P_{<j'}G))\|_{L^2(\mathbb{R}^d)},
\end{equation*}
where $\tilde{P}_j$ is a slightly fattened Littlewood-Paley projection. We then use the fundamental theorem of calculus to obtain
\begin{equation*}\label{para1}
\begin{split}
2^{j(s-k)}\|\tilde{P}_j(D^kF_{<j'}(P_{<j'}G))\|_{L^2(\mathbb{R}^d)}&\lesssim_A 2^{j(s-k)}\sup_{\tau\in [0,1]}\|\tilde{P}_j(D^{k+1}F_{<j'}(G_{\tau})P_{\geq j'}G)\|_{L^2(\mathbb{R}^d)}
\\
&\qquad +2^{j(s-k)}\|\tilde{P}_j(D^kF_{<j'}(G))\|_{L^2(\mathbb{R}^d)}.   
\end{split}
\end{equation*}
For the first term, we have simply
\begin{equation*}
2^{j(s-k)}\sup_{\tau\in [0,1]}\|\tilde{P}_j(D^{k+1}F_{<j'}(G_{\tau})P_{\geq j'}G)\|_{L^2(\mathbb{R}^d)}\lesssim_A 2^{j(s-k-1-\epsilon)}\|D^{k+1}F_{<j'}\|_{L^2(\mathbb{R}^d)}\lesssim 2^{-j\epsilon}\|F\|_{H^s(\mathbb{R}^d)}.    
\end{equation*}
For the second term, we have
\begin{equation*}
2^{j(s-k)}\|\tilde{P}_j(D^kF_{<j'}(G))\|_{L^2(\mathbb{R}^d)}\lesssim_A 2^{j(s-k)}\|D^kF_{\geq j'}\|_{L^2(\mathbb{R}^d)}+\|\tilde{P}_j((D^kF)(G))\|_{H^{s-k}(\mathbb{R}^d)}   .
\end{equation*}
Since $0<s-k\leq 1$, we obtain from \Cref{McLean},
\begin{equation*}
\left(\sum_{j>j_0}2^{2j(s-k)}\|\tilde{P}_j(D^kF_{<j'}(P_{<j'}G))\|_{L^2(\mathbb{R}^d)}^2\right)^{\frac{1}{2}}\lesssim_A \|F\|_{H^s(\mathbb{R}^d)}    ,
\end{equation*}
where we used that $s-k\leq 1$ to control $\|(D^kF)(G)\|_{H^{s-k}(\mathbb{R}^d)}$ and that $s-k>0$ to control the $l^2$ sum of $2^{j(s-k)}\|D^kF_{\geq j'}\|_{L^2(\mathbb{R}^d)}$. Combining everything together completes the proof.
\end{proof}
We also note a much cruder variant of the above proposition where we measure $G$ only in pointwise norms and $F$ in Sobolev based norms. This will only be needed in our construction of regularization operators later on. 
\begin{proposition}[Crude Moser estimate]\label{crudermoser}
Under the assumptions of \Cref{Moservariant}, the following bound holds for every $F\in H^s(\mathbb{R}^d)$,
\begin{equation*}
\|F(G)\|_{H^s(\mathbb{R}^d)}\lesssim_A \|F\|_{H^s(\mathbb{R}^d)}+\|G-Id\|_{C^{s+r+\epsilon}(\mathbb{R}^d)}\|F\|_{H^{1-r}(\mathbb{R}^d)}.   
\end{equation*}
\end{proposition}
\begin{proof}
The proof follows almost identical reasoning to \Cref{Moservariant}. The only difference is that we do not partition $F$ in (\ref{moserstep}) and instead estimate 
\begin{equation*}
\|(D F_{<j'})(G_{\tau})\|_{L^2(\mathbb{R}^d)}\lesssim_A \|D F_{<j'}\|_{L^2(\mathbb{R}^d)}\lesssim 2^{jr}\|F\|_{H^{1-r}(\mathbb{R}^d)}.  
\end{equation*}
We then invoke Bernstein's inequality to obtain
\begin{equation*}
2^{j(r+s)}\|P_{\geq j}G\|_{L^{\infty}(\mathbb{R}^d)}\lesssim 2^{-j\epsilon}\|G-Id\|_{C^{s+r+\epsilon}}  ,  
\end{equation*}
and  conclude by summing in $j$.  
\end{proof}
\subsection{Local coordinate parameterizations and Sobolev norms in \texorpdfstring{$\Lambda_*$}{}} With the above estimates in hand, we can begin the process of proving refined versions of the various elliptic, trace and product type estimates on $\Gamma$ that will be important for establishing our higher energy estimates. Our goal in this subsection is to construct a family of coordinate neighborhoods  for $\Gamma_*$ which will act as a ``universal" set of coordinate neighborhoods which we can use to flatten the boundary of nearby hypersurfaces $\Gamma\in\Lambda_*$. We will also use these local coordinates to define Sobolev type norms on $\Gamma$ which are suitable for proving uniform estimates later in this section. To achieve this, we  slightly modify the construction from \cite[Appendix A]{sz} (but note the difference in our definitions of $\Lambda_*$). 
\subsubsection{Local coordinates and partition of unity}
As in \cite[Appendix A]{sz}, since $\Gamma_*$ is compact, for any $\sigma>0$ we can choose $x_i\in \mathbb{R}^d$ and $r,r_i\in (0,\frac{1}{2}]$, $i=1,\dots,m$, such that we have the following two properties:
 \begin{enumerate}
     \item\label{prop1} $B(\Gamma_*,r)\subseteq\cup_{i=1}^m R_i(r_i)$, where $B(S,\epsilon)$ denotes the $\epsilon$ neighborhood of $S$ and $R_i(\cdot):=\tilde{R}_i(\cdot)\times I_i(\cdot)\subseteq\mathbb{R}^d$ is a rotated cylinder with perpendicular vertical segment centered at $x_i$ with the given equal radius and length.
     \item\label{prop2} For each $i$, $z=(\tilde{z},z_d)$ being the natural Euclidean coordinates on $R_i$, there exists a function $f_{*i}:\tilde{R}_i(2r_i)\to I_i$ such that
     \begin{equation}\label{smallness}
         \|f_{*i}\|_{C^0}<\sigma r_i,\ \ \|Df_{*i}\|_{C^0}<\sigma \ \ \text{and} \ \Omega_*\cap R_i(2r_i)=\{z_d>f_{*i}(\tilde{z})\}.
     \end{equation}
 \end{enumerate}
When $\delta>0$ is small enough, for every $\Gamma\in\Lambda_*$ with corresponding bounded domain $\Omega$,  (\ref{prop1}) holds with $\Gamma_*$ replaced by $\Gamma$. Moreover, there exist functions $f_i:\tilde{R}_i(2r_i)\to I_i$ satisfying (\ref{prop2}) with $\Omega_*$ replaced by $\Omega$ such that we can control the Sobolev and H\"older type norms of $f_i$ by the corresponding norms of $\Gamma$. Specifically, we have
\begin{equation*}
\|f_i\|_{H^{s}}\lesssim_A 1+\|\Gamma\|_{H^s},\hspace{10mm}\|f_i\|_{C^{k,\alpha}}\lesssim_A 1+\|\Gamma\|_{C^{k,\alpha}}    
\end{equation*}
for $s\geq 0$, integer $k\geq 0$ and $\alpha\in [0,1)$.
Indeed, by performing a computation in local coordinates, the above Sobolev bound follows from the Moser estimate in \Cref{Moservariant} and the pointwise bound can be verified directly from the chain rule and interpolation. Using these coordinate representations, we intend to construct local coordinate maps on each $\tilde{R}_i(2r_i)$ for $\Omega$ which flatten $\Gamma$ and have uniform estimates in $\Lambda_*$. In some of the estimates in this section, by a slight abuse of notation, we write $\|\Gamma\|$ when we really mean $1+\|\Gamma\|$ in order to declutter the notation. This will not affect any of the analysis for the dynamic problem.
\\
\\
On each $\tilde{R}_i(2r_i)$, let $\phi_i=\gamma_if_i$, where $\gamma_i(\tilde{z})=\overline{\gamma}\left(\frac{|\tilde{z}|}{r_i}\right)$ and $\overline{\gamma} : [0,\infty)\to [0,1]$ is a smooth cutoff supported on $[0,\frac{3}{2}]$ and equal to $1$ on $[0,\frac{5}{4}]$. We can extend $\phi_i$ to a function on $\mathbb{R}^d$ which gains half a degree of regularity in $H^s$ norms and is bounded in suitable pointwise norms. Indeed, let $\tilde{z}\in \mathbb{R}^{d-1}$ and $s\geq\frac{1}{2}$. We define an extension $\Phi_i$ of $\phi_i$ by
    \begin{equation*}\label{fullspaceext}
       \Phi_i(z)=\int_{\mathbb{R}^{d-1}}\widehat{\phi_i}(\xi')e^{-(1+|\xi'|^2)z_d^2}e^{2\pi i\xi'\cdot \tilde{z}} d\xi' \ \text{for} \ z=(\tilde{z},z_d)\in \mathbb{R}^d.
    \end{equation*}
We first observe that for each integer $k\geq 0$ and $\alpha\in [0,1)$, $\|\Phi_i\|_{C^{k,\alpha}(\mathbb{R}^d)}\lesssim_{k,\alpha} \|\phi_i\|_{C^{k,\alpha}(\mathbb{R}^{d-1})}$. One also has the same bounds for $W^{k,\infty}$ for each $k\geq 0$. To see this, we observe that $\Phi_i$ can be rewritten as the convolution
\begin{equation*}
\Phi_i(z)=c_de^{-z_d^2}\int_{\mathbb{R}^{d-1}}\phi_i(\tilde{z}+z_dy)e^{-|y|^2}dy   ,
\end{equation*}
where $c_d$ is a dimensional constant. In this form, the above bounds are easily checked. We also have $\|\Phi_i\|_{H^{s+\frac{1}{2}}(\mathbb{R}^d)}\approx_s\|\phi_i\|_{H^s(\mathbb{R}^{d-1})}$ for every $s\geq 0$,  which follows from inspecting the Fourier transform of $\Phi_i$, in a similar fashion as  \cite[Lemma 3.36]{MR1742312}. 
\\
\\
From the above, we see that if $\sigma>0$ from (\ref{smallness}) is small enough, then the map 
\begin{equation*}
    H_i(\tilde{z},z_d):=(\Tilde{z},z_d+\Phi_i(\tilde{z},z_d))
\end{equation*}
is a diffeomorphism from $\mathbb{R}^d\to\mathbb{R}^d$ with $\|H_i-Id\|_{C^{k,\alpha}}\lesssim_A \|\Gamma\|_{C^{k,\alpha}}$ and $\|H_i-Id\|_{H^{s+\frac{1}{2}}}\lesssim_A \|\Gamma\|_{H^s}$ for $s\geq 0$, integer $k\geq 0$ and $\alpha\in [0,1)$.  Moreover, for the  inverse function $G_i:=H_i^{-1}$, the same bounds hold for $G_i-Id$ and its $d'$th component $g_i$ satisfies the bounds $|\partial_{z_d}g_i|+|(\partial_{z_d}g_i)^{-1}|\lesssim_A 1$. Finally, if $\sigma>0$ is small enough and $\Lambda_*$ is a tight enough collar neighborhood we have, in the $C^1$ topology,
\begin{equation*}
\|H_i-Id\|_{C^1}+\|G_i-Id\|_{C^1}\lesssim_A \rho,
\end{equation*}
where $\rho>0$ is some positive constant which can be made as small as we like (depending on $\sigma$ and $\Lambda_*$). We then have for some uniform $\delta_*>0$,
\begin{equation*}
    \left(\tilde{R}_i\left(\frac{5}{4}r_i\right)\times I_i\left(\frac{5}{4}\delta_*r_i\right)\right)\cap \Omega=\left(\tilde{R}_i\left(\frac{5}{4}r_i\right)\times I_i\left(\frac{5}{4}\delta_*r_i\right)\right)\cap\{g_i>0\}.
\end{equation*}
\textbf{Partition of unity.} Here, we construct a partition of unity for $\Omega$ with bounds uniform in $\Lambda_*$. We follow essentially the procedure from \cite[Appendix A]{sz}. Let $\gamma$ be a smooth cutoff defined on $[0,\infty)$ satisfying $0\leq \gamma\leq 1$ with $\gamma$ supported in $[0,\frac{5}{4})$ and equal to $1$ on $[0,\frac{9}{8}]$. Moreover, let $\zeta$ be a smooth function defined on $[0,\infty)$ taking values in $[\frac{1}{3},\infty)$ with $\zeta=\frac{1}{3}$ on $[0,\frac{1}{3}]$ and $\zeta(x)=x$ for $x\geq \frac{2}{3}$. Define
\begin{equation*}
 \tilde{\gamma}_{*i}(z):=\gamma(\frac{|\tilde{z}|}{r_i})\gamma(\frac{|z_d|}{\delta_*r_i}),\hspace{10mm}\eta=\zeta\circ\sum_i(\tilde{\gamma}_{*i}\circ G_i).
\end{equation*}
We then define a partition of unity via
\begin{equation}\label{Part of unity gamma*}
\gamma_{*i}:=\frac{\tilde{\gamma}_{*i}(G_i)}{\eta},\hspace{10mm}  \gamma_{*0}:=(1-\sum_i\gamma_{*i})\mathbbm{1}_{\Omega}.  
\end{equation}
We see that $\sum_{i\geq 0}\gamma_{*i}=1$  on $\Omega$ and $0\leq \gamma_{*i}\leq 1$ for each $i\geq 0$. Moreover, by the Moser and Sobolev product estimates, we have
\begin{equation*}
\|\gamma_{*i}\|_{H^{s+\frac{1}{2}}}\lesssim_A \|\Gamma\|_{H^s}    
\end{equation*}
for $s\geq 0$.
\subsubsection{Sobolev spaces on hypersurfaces in $\Lambda_*$} We can use the above partition of unity to define $C^{k,\alpha}$ and $H^s$ spaces on hypersurfaces $\Gamma\in\Lambda_*$. Indeed, if $\Gamma$ is $C^1$ and in $H^s$, we may define what it means to be in $H^r(\Gamma)$ for $0\leq r\leq s$ through the inner product,
\begin{equation*}
\langle f,g\rangle_{H^r(\Gamma)}:=\sum_{i\geq 1}\langle \phi_if_i,\phi_ig_i\rangle_{H^r(\mathbb{R}^{d-1})}   ,
\end{equation*}
where $\phi_i:=\gamma_{*i}\circ H_i(\tilde{z},0)$ (note that this is not the same $\phi_i$ as in the previous subsection), $f_i:=f\circ H_i(\tilde{z},0)$ and $g_i:=g\circ H_i(\tilde{z},0)$. If $\Gamma$ is $C^{k,\alpha}$ we may also define
\begin{equation*}
\|f\|_{C^{k,\alpha}(\Gamma)}:=\sup_{i\geq 1}\|\phi_if_i\|_{C^{k,\alpha}(\mathbb{R}^{d-1})}.    
\end{equation*}
Finally, for a function $v$ defined on $\Omega$, we write $v_i=\gamma_{*i}v$ and $u_i=v_i(H_i).$ 
\\

Using  the above and the full generality afforded by \Cref{Moservariant}, we prove a refined product type estimate on the boundary $\Gamma$. Precisely, we have the following.
\begin{proposition}[Product estimates on the boundary]\label{boundaryest}
Let $\Omega$ be a bounded domain with boundary $\Gamma\in\Lambda_*$. If $f,g$ are functions on $\Gamma$ and $g=g_j^1+g_j^2$ is any sequence of partitions, then for $s\geq 0$ and $r\geq 1$ we have
\begin{equation*}
\begin{split}
\|fg\|_{H^s(\Gamma)}\lesssim_A \|f\|_{L^{\infty}(\Gamma)}\|g\|_{H^s(\Gamma)}&+(\|f\|_{H^{s+r-1}(\Gamma)}+\|f\|_{L^{\infty}(\Gamma)}\|\Gamma\|_{H^{s+r}})\sup_{j>0}2^{-j(r-1)}\|g_j^1\|_{L^{\infty}(\Gamma)}
\\
&+(1+\|f\|_{C^{2\epsilon}(\Gamma)})\sup_{j>0}2^{j(s-\epsilon)}\|g_j^2\|_{L^2(\Gamma)}.
\end{split}
\end{equation*}
\end{proposition}
\begin{remark}
If we take $r=1$ and $g_j^1=g$, we recover something resembling the standard algebra property,
\begin{equation}\label{prodremark}
\|fg\|_{H^s(\Gamma)}\lesssim_A \|f\|_{L^{\infty}(\Gamma)}\|g\|_{L^{\infty}(\Gamma)}\|\Gamma\|_{H^{s+1}}+\|f\|_{H^s(\Gamma)}\|g\|_{L^{\infty}(\Gamma)}+\|g\|_{H^s(\Gamma)}\|f\|_{L^{\infty}(\Gamma)},
\end{equation}
but with the twist being the additional explicit presence of the $H^{s+1}$ norm of the surface on the right-hand side. We also remark that the proof below will allow for the first term on the right of (\ref{prodremark}) to be replaced by $(\|f\|_{W^{1,\infty}(\Gamma)}\|g\|_{L^{\infty}(\Gamma)}+\|f\|_{L^{\infty}(\Gamma)}\|g\|_{W^{1,\infty}(\Gamma)})\|\Gamma\|_{H^{s}}$, which is perhaps more natural, but we will never actually need this.
\end{remark}
\begin{proof}
Let $(\gamma_{*i})_{i}$ be the partition of unity for $\Omega$ defined in \eqref{Part of unity gamma*}. As before, we write  $\phi_i(\tilde{z}):=\gamma_{*i}(H_i(\tilde{z},0))$, which is smooth with domain independent bounds since $G_i$ and $H_i$ are inverse. Similarly, we write $f_i=f(H_i(\tilde{z},0))$ and $g_i=g(H_i(\tilde{z},0))$, which are functions defined on the support of $\phi_i$. By definition, it suffices to control $\|\phi_if_ig_i\|_{H^s(\mathbb{R}^{d-1})}$ for each $i\geq 1$. To begin with, let $j'=j-4$ and let $P_j$ and $P_{<j'}$ denote Littlewood-Paley projections on $\mathbb{R}^{d-1}$. Moreover, define $\tilde{\phi}_i$ to be a smooth compactly supported function equal to $1$ on the support of $\gamma_{*i}$ with support properties chosen so that $\tilde{\phi}_i$ is supported in the region where $f_i$ is well-defined. Then a simple paraproduct estimate using the Littlewood-Paley trichotomy gives 
\begin{equation*}
\begin{split}
\|\phi_if_ig_i\|_{H^s(\mathbb{R}^{d-1})}&\lesssim_A \|f\|_{L^{\infty}(\Gamma)}\|\phi_ig_i\|_{H^s(\mathbb{R}^{d-1})}+\left(\sum_{j>0}2^{2js}\|P_{<j'}(\phi_ig_i)P_j(f_i\tilde{\phi}_i)\|_{L^2(\mathbb{R}^{d-1})}^2\right)^{\frac{1}{2}}.
\end{split}
\end{equation*}
For the latter term in the above, we estimate
\begin{equation*}
\begin{split}
\left(\sum_{j>0}2^{2js}\|P_{<j'}(\phi_ig_i)P_j(f_i\tilde{\phi}_i)\|_{L^2(\mathbb{R}^{d-1})}^2\right)^{\frac{1}{2}}&\lesssim_A \|f_i\tilde{\phi}_{i}\|_{H^{s+r-1}(\mathbb{R}^{d-1})}\sup_{j>0}2^{-j(r-1)}\|g_j^1\|_{L^{\infty}(\Gamma)}
\\
&\qquad +(1+\|f\|_{C^{2\epsilon}(\Gamma)})\sup_{j>0}2^{j(s-\epsilon)}\|g_j^2\|_{L^2(\Gamma)}.
\end{split}
\end{equation*}
We are then reduced to showing
\begin{equation*}
\|f_i\tilde{\phi}_{i}\|_{H^{s+r-1}(\mathbb{R}^{d-1})}\lesssim_A \|f\|_{H^{s+r-1}(\Gamma)}+\|f\|_{L^{\infty}(\Gamma)}\|\Gamma\|_{H^{s+r}}.
\end{equation*}
For this, we note that
\begin{equation*}
\|f_i\tilde{\phi}_{i}\|_{H^{s+r-1}(\mathbb{R}^{d-1})}\leq\sum_{j\geq 1} \|\tilde{\phi}_i\gamma_{*j}(H_i(\tilde{z},0))f_i\|_{H^{s+r-1}(\mathbb{R}^{d-1})}.
\end{equation*}
Let us write $\varphi_{ij}:=G_j\circ H_i$. Then we have 
\begin{equation*}
\begin{split}
\|\tilde{\phi}_i\gamma_{*j}(H_i(\tilde{z},0))f_i\|_{H^{s+r-1}(\mathbb{R}^{d-1})}&=\|(\phi_jf_j)(\varphi_{ij}(\tilde{z},0))\tilde{\phi}_i\|_{H^{s+r-1}(\mathbb{R}^{d-1})}.
\end{split}
\end{equation*}
We note that $\varphi_{ij}$ is a diffeomorphism having the same bounds as $G_j$ and $H_i$. By using the extension $\Phi$ from earlier, we may assume  that $\phi_jf_j$ is defined on $\mathbb{R}^d$ with $\|\phi_jf_j\|_{H^{s+r-\frac{1}{2}}(\mathbb{R}^{d})}\lesssim \|\phi_jf_j\|_{H^{s+r-1}(\mathbb{R}^{d-1})}$ and $\|\phi_jf_j\|_{L^{\infty}(\mathbb{R}^d)}\lesssim \|\phi_jf_j\|_{L^{\infty}(\mathbb{R}^{d-1})}$. Therefore, by the trace estimate on $\mathbb{R}^{d-1}$, the fact that $\varphi_{ij}$ is a diffeomorphism and the balanced Moser estimate, we have 
\begin{equation*}
\|\tilde{\phi}_i(\phi_jf_j)(\varphi_{ij}(\tilde{z},0))\|_{H^{s+r-1}(\mathbb{R}^{d-1})}\lesssim_A \|(\phi_jf_j)\circ\varphi_{ij}\|_{H^{s+r-\frac{1}{2}}(\mathbb{R}^{d})}\lesssim_A\|\phi_jf_j\|_{H^{s+r-1}(\mathbb{R}^{d-1})}+\|\Gamma\|_{H^{s+r}}\|f\|_{L^{\infty}(\Gamma)}.
\end{equation*}
Since, by definition, we have
\begin{equation*}
\|\phi_jf_j\|_{H^{s+r-1}(\mathbb{R}^{d-1})}\leq \|f\|_{H^{s+r-1}(\Gamma)},
\end{equation*}
the proof is complete.
\end{proof}
\subsubsection{Trace estimates}
Now, we prove a refined version of the trace theorem for $\Gamma$. 
\begin{proposition}[Balanced trace estimate]\label{baltrace} Let $\Omega$ be a bounded domain with boundary $\Gamma\in\Lambda_*$. For every $s>\frac{1}{2}$, $r\geq 0$, $\alpha,\beta\in [0,1]$ and every sequence of partitions $v=v_j^1+v_j^2$, we have
\begin{equation*}
    \|v_{|\Gamma}\|_{H^{s-\frac{1}{2}}(\Gamma)}\lesssim_A\|v\|_{H^s(\Omega)}+\|\Gamma\|_{H^{s+r-\frac{1}{2}}}\sup_{j>0}2^{-j(r+\alpha-1)}\|v_j^1\|_{C^{\alpha}(\Omega)}+\sup_{j>0}2^{j(s-1+\beta-\epsilon)}\|v_j^2\|_{H^{1-\beta}(\Omega)}.
\end{equation*}
\end{proposition}
\begin{proof} For $i\geq 1$, define $\tilde{v}_i=\gamma_{*i}\mathcal{E}v$ where $\mathcal{E}$ is the Stein extension operator for $\Omega$. It suffices  to prove the estimate with the left-hand side replaced by $\|\tilde{v}_i(H_i(\tilde{z},0))\|_{H^{s-\frac{1}{2}}(\mathbb{R}^{d-1})}$. Using the trace theorem on $\mathbb{R}^{d-1}$, we have
\begin{equation*}
   \|\tilde{v}_i(H_i(\tilde{z},0))\|_{H^{s-\frac{1}{2}}(\mathbb{R}^{d-1})}\lesssim \|\tilde{u}_i\|_{H^s(\mathbb{R}^d)},
\end{equation*}
where $\tilde{u}_i:=\tilde{v}_i\circ H_i$. We then use  \Cref{Moservariant} and the operator bounds for $\mathcal{E}$ in \Cref{Stein} to conclude.
\end{proof}
\subsubsection{An extension operator depending continuously  on the domain}\label{cont of domain ext} Another use of the above local coordinates is to construct a family of extension operators which depend continuously in a suitable sense on the domain.  This will be important for establishing our continuous dependence result later on. Potentially, something akin to the Stein extension operator could work here, but we opt for the following simpler construction where the dependence on the domain is more transparent.
\begin{proposition}\label{continuosext}
Fix a collar neighborhood $\Lambda_*$ and let $s>\frac{d}{2}+1$. For each bounded domain $\Omega$ with $H^s$ boundary $\Gamma\in\Lambda_*$ there exists an extension operator $E_{\Omega}:H^s(\Omega)\to H^s(\mathbb{R}^d)$ such that for all $v\in H^s(\Omega)$,
\begin{equation}\label{otherextensionbounds}
\|E_{\Omega}v\|_{H^s(\mathbb{R}^d)}+\|\Gamma\|_{H^s}\approx_{A,\|v\|_{C^{\frac{1}{2}}(\Omega)}}\|(v,\Gamma)\|_{\mathbf{H}^s},\hspace{5mm}\|E_{\Omega}v\|_{H^s(\mathbb{R}^d)}\lesssim_{A} \|\Gamma\|_{H^{s-\frac{1}{2}}}\|v\|_{H^s(\Omega)},  
\end{equation}
where the dependence on $\|v\|_{C^{\frac{1}{2}}(\Omega)}$ is polynomial. Moreover, if $\Omega_n$ is a sequence of domains with $\Gamma_n\to \Gamma$ in $H^s$, then for every $v\in H^s(\mathbb{R}^d)$, there holds
\begin{equation}\label{SoT}
\|E_{\Omega_n}v_{|\Omega_n}-E_{\Omega}v_{|\Omega}\|_{H^s(\mathbb{R}^d)}\to 0.    
\end{equation}
\end{proposition}
\begin{remark}
One can loosely think of \eqref{SoT} as a  strong operator topology convergence for this family of extensions. 
\end{remark}
\begin{proof}
 Given a family of domains $\Omega_n$ and $\Omega$ with boundaries $\Gamma_n,\Gamma\in \Lambda_*$, denote by $\gamma_{*i}^n$ and $\gamma_{*i}$ the corresponding partitions of unity, so that
\begin{equation*}
    v=\sum_i \gamma_{*i}^nv \ \ \text{on} \ \Omega_n \ \text{and}  \  v=\sum_i \gamma_{*i}v \ \ \text{on} \ \Omega.
\end{equation*}
Define $u_i^n=(\gamma_{*i}^nv)\circ H_i^n$ on  $\mathbb{R}^d_+$. Let $k$ be the largest integer less than or equal to $s$, and define the half-space extension 
\begin{equation*}
\begin{cases}
&\tilde{u}_i^n(\tilde{z},z_d)=\sum_{j=1}^{k+1}c_ju_i^n(\tilde{z},-\frac{z_d}{j})\  \ \text{if} \ z_d<0,
\\
&\tilde{u}_i^n(\tilde{z},z_d)= u_i^n(\tilde{z},z_d) \hspace{18mm} \text{if} \ z_d\geq 0,
\end{cases}
\end{equation*}
where $c_1,\dots,c_{k+1}$ are gotten as in  \cite[Lemma 6.37]{MR1814364} by solving an appropriate Vandermonde system. It is standard to verify that we have $\tilde{u}_i^n\in H^s(\mathbb{R}^d)$. 
\\

We define the $\Omega_n$ extension of $v$ by
\begin{equation*}
    \tilde{v}_n=\sum_i \tilde{u}_i^n\circ G_i^n,
\end{equation*}
and similarly let $\tilde{v}$ by the $\Omega$ extension of $v$.  To verify the continuous dependence property, we want to verify that if $\Gamma_n\to \Gamma$ in $H^s$, then $\tilde{v}_n\to\tilde{v}$ in $H^s(\mathbb{R}^d).$ For this, it suffices to prove that $\tilde{u}_i^n\circ G_i^n\to \tilde{u}_i\circ G_i$ in $H^s(\mathbb{R}^d)$ for each $i$. We note that
\begin{equation}\label{Delta ineq}
    \|\tilde{u}_i^n\circ G_i^n- \tilde{u}_i\circ G_i\|_{H^s(\mathbb{R}^d)}\leq \|(\tilde{u}_i^n-\tilde{u}_i)\circ G_i^n\|_{H^s(\mathbb{R}^d)}+\|\tilde{u}_i\circ G_i^n-\tilde{u}_i\circ G_i\|_{H^s(\mathbb{R}^d)}.
\end{equation}
The first term on the right-hand side of \eqref{Delta ineq} can be shown to go to zero by using standard Moser estimates. The latter term goes to zero by arguing similarly to the proof that translation is continuous in $L^p$ spaces (using a simple density argument to replace $\tilde{u}_i$ by a smooth function).
\\

Finally, the bounds (\ref{otherextensionbounds}) follow from the definition of the extension and \Cref{Moservariant}.

\end{proof}
\subsection{Pointwise elliptic estimates}
Here we establish variants of the $C^{2,\alpha}$ and $C^{1,\alpha}$ estimates for the Dirichlet problem which adequately track the dependence on the domain regularity. In our analysis later, we will mostly use the $C^{1,\alpha}$ estimates with  $\alpha=\frac{1}{2}$ or $\alpha=\epsilon$. However, the $C^{2,\alpha}$ estimates will be relevant for proving bounds for our regularization operators, which are defined in \Cref{SSRO}.
\\

As will become apparent later, to obtain the desired pointwise elliptic estimates, it is crucial to use a domain flattening map whose Jacobian has determinant $1$. This will be  especially necessary for the $C^{1,\alpha}$ estimate, as we must  preserve the divergence form of the equation. For this reason, instead of the map $H_i$, we will   use the more familiar domain flattening map
\begin{equation}\label{classicalflattening}
F_i(z)=(\tilde{z},z_d+\phi_i(\tilde{z})),
\end{equation}
whose Jacobian has determinant 1.  The tradeoff when using the flattening $F_i$ is that  it does not exhibit a $\frac{1}{2}$ gain in regularity for the $H^s$ norm on the interior compared to the boundary, but this will not matter for this section because all domain dependent coefficients will be placed in $L^{\infty}$ based norms. We let $\Psi_i:=F_i^{-1}$, and begin with the $C^{2,\alpha}$ estimates.
\begin{proposition}[$C^{2,\alpha}$ estimates for the inhomogeneous Dirichlet problem]\label{C2est} Let $0<\alpha<1$ and let $\Omega$ be a bounded domain with boundary $\Gamma\in \Lambda_*$  having $C^{2,\alpha}$ regularity. Consider the boundary value problem
\begin{equation*}
\begin{cases}
&\Delta v=g\hspace{5mm}\text{in }\Omega,
\\
&\hspace{3mm}v=\psi\hspace{5mm}\text{on }\Gamma.
\end{cases}
\end{equation*}
Then $v$ satisfies the estimate
\begin{equation*}\label{C2est2}
\|v\|_{C^{2,\alpha}(\Omega)}\lesssim_A \|\Gamma\|_{C^{2,\alpha}}\|v\|_{W^{1,\infty}(\Omega)}+\|g\|_{C^{\alpha}(\Omega)}+\|\psi\|_{C^{2,\alpha}(\Gamma)}.
\end{equation*}
\end{proposition}
\begin{proof}
We write $v_i=\gamma_{*i}v$, $h_i=\Delta v_i$, $f_i=h_i\circ F_i$ and $v_i=u_i\circ\Psi_i$. Omitting some of the subscripts for notational convenience, we see that $u:=u_i$ satisfies the equation
\begin{equation}\label{frozeneqn}
\begin{cases}
&\Delta u=\partial_k((\delta^{jk}-a^{jk})\partial_ju)+f,
\\
&u_{|_{z_d=0}}=(\gamma_{*i}\psi)(H_i(\tilde{z},0)),
\end{cases}
\end{equation}
where $a^{jk}=(\Psi^j_{x_l}\Psi^k_{x_l})(F_i)$ with repeated indices summed over. Note that to compute the boundary term in \eqref{frozeneqn}  we used that $F_i(\tilde{z},0)=H_i(\tilde{z},0)$. By the well-known Schauder estimates for the half-space, we obtain
\begin{equation}\label{C2alpha e1}
\|u\|_{C^{2,\alpha}}\lesssim_A \|(\delta^{jk}-a^{jk})\partial_ju\|_{C^{1,\alpha}}+\|f\|_{C^{\alpha}}+\|(\gamma_{*i}\psi)(H_i(\tilde{z},0))\|_{C^{2,\alpha}}.
\end{equation}
Using the Besov characterization (\ref{Besov}) and the  paradifferential expansion (\ref{parad}), it is straightforward to estimate
\begin{equation}\label{RHS}
\begin{split}
\|(\delta^{jk}-a^{jk})\partial_ju\|_{C^{1,\alpha}}\lesssim \|\delta^{jk}-a^{jk}\|_{C^{\epsilon}}\|u\|_{C^{2,\alpha}}+\|\Gamma\|_{C^{2,\alpha}}\|v\|_{W^{1,\infty}(\Omega)}.
\end{split}
\end{equation}
As $a^{ij}$ is close to the identity in $C^{\epsilon}$, this simplifies the estimate \eqref{C2alpha e1} to
\begin{equation}\label{Chained rule}
\|u\|_{C^{2,\alpha}}\lesssim_A \|\Gamma\|_{C^{2,\alpha}}\|v\|_{W^{1,\infty}(\Omega)}+\|f\|_{C^{\alpha}}+\|(\gamma_{*i}\psi)(H_i(\tilde{z},0))\|_{C^{2,\alpha}}.
\end{equation}
Clearly, we have $\|f\|_{C^{\alpha}}\lesssim_A \|h\|_{C^{\alpha}(\Omega)}$. On the other hand, we have
\begin{equation}\label{chainruleterms}
\|u(\Psi_i)\|_{\dot{C}^{2,\alpha}}\lesssim_A \|(D\Psi_i)^{*}(D^2u)(\Psi_i)D\Psi_i\|_{\dot{C}^{\alpha}}+\|(Du)(\Psi_i)D^2\Psi_i\|_{\dot{C}^{\alpha}}.
\end{equation}
We can estimate both terms above by the right-hand side of (\ref{Chained rule}). We show how to do this for the first term, as the second term is similar. For this, we may assume that $u$ is defined on all of $\mathbb{R}^d$ by using a suitable extension operator from the half-space to $\mathbb{R}^d$. Then we write as usual $u_{<j}$ to mean $P_{<j}u$ and $u_{\geq j}:=u-u_{<j}$. By the Besov characterization of $C^{\alpha}$, we need to estimate 
\begin{equation*}
\sup_{j>0}2^{j\alpha}\|P_j((D\Psi_i)^{*}(D^2u)(\Psi_i)D\Psi_i)\|_{L^{\infty}}.
\end{equation*}
By the standard Littlewood-Paley trichotomy, we first obtain,
\begin{equation*}
2^{j\alpha}\|P_j((D\Psi_i)^{*}(D^2u)(\Psi_i)D\Psi_i)\|_{L^{\infty}}\lesssim_A \|u\|_{C^{2,\alpha}}+2^{j\alpha}\|D^2u\|_{L^{\infty}}\|\tilde{P}_jB(D\Psi_i,D\Psi_i)\|_{L^{\infty}} ,
\end{equation*}
where $B$ is a suitable bilinear form.
For the latter term, we split $u=u_{<j}+u_{\geq j}$ and estimate using Bernstein's inequality,
\begin{equation*}
\begin{split}
2^{j\alpha}\|D^2u\|_{L^{\infty}}\|\tilde{P}_jB(D\Psi_i,D\Psi_i)\|_{L^{\infty}} &\lesssim_A \|v\|_{W^{1,\infty}(\Omega)}2^{j(1+\alpha)}\|\tilde{P}_jB(D\Psi_i,D\Psi_i)\|_{L^{\infty}}+\|u\|_{C^{2,\alpha}}
\\
&\lesssim_A \|\Gamma\|_{C^{2,\alpha}}\|v\|_{W^{1,\infty}(\Omega)}+\|u\|_{C^{2,\alpha}}.
\end{split}
\end{equation*}
The other term in (\ref{chainruleterms}) is similarly handled. Combining the above, we obtain
\begin{equation*}
\|v_i\|_{C^{2,\alpha}(\Omega)}\lesssim_A \|\Gamma\|_{C^{2,\alpha}}\|v\|_{W^{1,\infty}(\Omega)}+\|h\|_{C^{\alpha}(\Omega)}+\|(\gamma_{*i}\psi)(H_i(\tilde{z},0))\|_{C^{2,\alpha}}.
\end{equation*}
Expanding
\begin{equation*}
h=\Delta(\gamma_{*i}v)=\Delta\gamma_{*i}v+2\nabla\gamma_{*i}\cdot\nabla v+\gamma_{*i}\Delta v
\end{equation*}
we obtain
\begin{equation*}
\|h\|_{C^{\alpha}(\Omega)}\lesssim_A \|\Gamma\|_{C^{2,\alpha}}\|v\|_{W^{1,\infty}(\Omega)}+\|\nabla\gamma_{*i}\cdot\nabla v\|_{C^{\alpha}(\Omega)}+\|g\|_{C^{\alpha}(\Omega)}.
\end{equation*}
The second term on the right-hand side can be estimated crudely by
\begin{equation*}
\|\nabla\gamma_{*i}\cdot\nabla v\|_{C^{\alpha}(\Omega)}\lesssim_A \|v\|_{C^{1,\alpha}(\Omega)}+\|\Gamma\|_{C^{2,\alpha}}\|v\|_{W^{1,\infty}(\Omega)}.
\end{equation*}
Finally, by estimating the term $\|v\|_{C^{1,\alpha}(\Omega)}\lesssim \delta_0\|v\|_{C^{2,\alpha}(\Omega)}+C(\delta_0)\|v\|_{C^0(\Omega)}$ for some $\delta_0$ sufficiently small and absorbing the first term into the left-hand side of the estimate, we conclude the proof.
\end{proof}
By very similar reasoning and the corresponding estimate in the half-space (see Theorem 8.33 in \cite{MR1814364}) we also have a $C^{1,\alpha}$ variant if the source term $g$ is replaced by $\nabla\cdot g$. More precisely, we have the following.
\begin{proposition}[$C^{1,\alpha}$ estimates for the Dirichlet problem]\label{Gilbarg} Let $\Omega$ be a bounded $C^{1,\alpha}$ domain with $0<\alpha<1$ and with boundary $\Gamma\in\Lambda_*$. Consider the boundary value problem
\begin{equation*}\label{dirprob}
\begin{cases}
&\Delta v=\nabla\cdot g_1+g_2\hspace{5mm}\text{in }\Omega,
\\
&\hspace{3mm}v=\psi\hspace{10mm}\text{on }\partial\Omega.
\end{cases}
\end{equation*}
Then $v$ satisfies the estimate 
\begin{equation*}\label{C1a}
\|v\|_{C^{1,\alpha}(\Omega)}\lesssim_A \|\Gamma\|_{C^{1,\alpha}}(\|v\|_{W^{1,\infty}(\Omega)}+\|g_1\|_{L^{\infty}(\Omega)})+\|g_1\|_{C^{\alpha}(\Omega)}+\|g_2\|_{L^{\infty}(\Omega)}+\|\psi\|_{C^{1,\alpha}(\Gamma)}.
\end{equation*}
Interpolating and using the straightforward estimate 
\begin{equation*}
\|v\|_{L^{\infty}(\Omega)}\lesssim_A \|g_1\|_{L^{\infty}(\Omega)}+\|g_2\|_{L^{\infty}(\Omega)}+\|\psi\|_{L^{\infty}(\Gamma)},
\end{equation*}
 we deduce also
\begin{equation}\label{C1eps}
\|v\|_{C^{1,\epsilon}(\Omega)}\lesssim_A \|g_1\|_{C^{\epsilon}(\Omega)}+\|g_2\|_{L^{\infty}(\Omega)}+\|\psi\|_{C^{1,\epsilon}(\Gamma)}
\end{equation}
and
\begin{equation*}\label{C1est2}
\|v\|_{C^{1,\alpha}(\Omega)}\lesssim_A \|\Gamma\|_{C^{1,\alpha}}(\|g_1\|_{C^{\epsilon}(\Omega)}+\|g_2\|_{L^{\infty}(\Omega)}+\|\psi\|_{C^{1,\epsilon}(\Gamma)})+\|g_1\|_{C^{\alpha}(\Omega)}+\|g_2\|_{L^{\infty}(\Omega)}+\|\psi\|_{C^{1,\alpha}(\Gamma)}   . 
\end{equation*}
\end{proposition}
\begin{proof}
 Much of the proof is similar to the $C^{2,\alpha}$ estimate. We only outline the slight changes. First, we note that
 \begin{equation*}
 \begin{split}
 \Delta v_i&=\partial_j(\partial_j\gamma_{*i}v)+\partial_j\gamma_{*i}\partial_jv+\gamma_{*i}\nabla\cdot g_1+\gamma_{*i}g_2  
 \\
 &=\partial_j(\partial_j\gamma_{*i}v)+\nabla\cdot (\gamma_{*i}g_1)+\partial_j\gamma_{*i}\partial_jv-\nabla\gamma_{*i}\cdot g_1+\gamma_{*i}g_2=:\nabla\cdot h_1+h_2.
 \end{split}
 \end{equation*}
 Hence, localizing with $\gamma_{*i}$ preserves the divergence source term to leading order. More precisely, $h_2$ will be suitable for estimating in $L^{\infty}$ in the sense that $\|h_2\|_{L^{\infty}}\lesssim_A \|v\|_{W^{1,\infty}(\Omega)}+\|g_1\|_{L^{\infty}(\Omega)}+\|g_2\|_{L^{\infty}(\Omega)}$. The next step is to perform the domain flattening procedure. The most important point here is that since the Jacobian determinant of $F_i$ is 1, the corresponding equation for $u$ (using the notation from the proof of \Cref{C2est}) becomes
 \begin{equation*}
 \begin{cases}
&\partial_k(a^{jk}\partial_ju)=\nabla\cdot \tilde{h}_1+\tilde{h}_2\hspace{7mm}\text{in }\Omega,
\\
&\hspace{3mm}u_{|z_d=0}=(\gamma_{*i}\psi)(H_i(\tilde{z},0))\hspace{5mm}\text{on }\partial\Omega,
\end{cases}   
 \end{equation*}
 where
 \begin{equation*}
 \tilde{h}_1:=(h_1\cdot D\Psi_i)(F_i),\hspace{10mm}\tilde{h}_2:=h_2(F_i).
 \end{equation*}
 In other words, the divergence structure of the equation is preserved. From this point, the proof follows the same line of reasoning as the $C^{2,\alpha}$ estimates by writing an equation for $\Delta u$. The difference is that we use the $C^{1,\alpha}$ norm and the corresponding estimate for the Laplace equation in the half-space when the equation has the above divergence form. 
\end{proof}
When $g_1$ and $g_2$ are zero in the above proposition, we can interpolate using the maximum principle for $\mathcal{H}$ and  the $C^{1,\epsilon}$ bound above to obtain $C^{\alpha}$ bounds for the harmonic extension with constant depending only on $A_{\Gamma}$.
\begin{corollary}\label{Hboundlow}
Let $0\leq \alpha<1$. The following low regularity bound for $\mathcal{H}$ holds uniformly for domains $\Omega$ with boundary $\Gamma\in\Lambda_*$,
\begin{equation*}
\|\mathcal{H}g\|_{C^{\alpha}(\Omega)}\lesssim_A \|g\|_{C^{\alpha}(\Gamma)}.
\end{equation*}
\end{corollary}
\begin{proof}
By the above and the maximum principle, we have $C^{1,\epsilon}(\Gamma)\to C^{1,\epsilon}(\Omega)$ and $C^0(\Gamma)\to C^0(\Omega)$ bounds for $\mathcal{H}$ that are uniform in $\Lambda_*$. By \cite[Example 5.15]{MR3753604} we also know that  $(C^0(\mathbb{R}^n),C^{1,\epsilon}(\mathbb{R}^n))_{\theta,\infty}=C^\alpha(\mathbb{R}^n)$ for an appropriate choice of $\theta.$ Therefore, we just have to transfer the interpolation properties on $\mathbb{R}^n$ for $n=d$ and $n=d-1$ to $\Omega$ and $\Gamma$, respectively, with constants uniform in the collar. For $\Omega,$ we argue as in \Cref{Stein}, and on $\Gamma$ we simply unravel the definition of our function spaces via the partition of unity.
\end{proof}
\begin{remark}
    Of course, we note that \Cref{Hboundlow} avoids  $C^1$ and Lipschitz regularity, as these do not fall into the interpolation scale.
\end{remark}
\subsection{\texorpdfstring{$L^2$}{} based balanced elliptic estimates}
In this subsection, we will prove $H^s$ type estimates for various elliptic problems. In the following analysis, we will always be using the coordinate maps $H_i$ and $G_i$ (as opposed to $F_i$ and $\Psi_i$ from the pointwise estimates) to flatten the boundary since we will now need the $\frac{1}{2}$ gain of regularity on $\Omega$ in $H^s$ based norms given by this flattening.
\subsubsection{The Dirichlet problem}
We begin our analysis by proving estimates for the inhomogeneous Dirichlet problem
\begin{equation*}\label{direst1}
\begin{cases}
&\Delta v=g\hspace{5.5mm}\text{in }\Omega,
\\
&\hspace{3mm}v=\psi\hspace{5mm}\text{on }\Gamma.
\end{cases}
\end{equation*}
We first recall  two baseline estimates which will be used heavily in the derivation of the higher regularity bounds below. The first  is when $\psi=0$, in which case  $v$ satisfies the $H^1$ estimate 
\begin{equation}\label{H1base}
\|v\|_{H^1(\Omega)}\lesssim_A \|g\|_{H^{-1}(\Omega)}.    
\end{equation}
On the other hand, for $\frac{1}{2}<s\leq 1$ and $g=0$, we have
\begin{equation}\label{harmonicbase}
\|v\|_{H^s(\Omega)}\lesssim_A \|\psi\|_{H^{s-\frac{1}{2}}(\Gamma)}.    
\end{equation}
The bound \eqref{H1base} is completely standard. The bound \eqref{harmonicbase} was established by Jerison and Kenig in \cite{MR1331981}, and even holds, in an appropriate sense, at the endpoint $s=\frac{1}{2}$. For our purposes, we will only need the range $\frac{1}{2}<s\leq 1$, but we do need to quantify the dependence of the implicit constant in \cite{MR1331981} on the domain. As noted in \cite{MR2274539}, the implicit domain dependent constant is, as expected, solely dependent on the Lipschitz character of $\Omega$, so is controlled uniformly in the collar. Formally, \cite{MR2274539} only  quantifies the domain dependence for the inhomogeneous problem $g\neq 0$, $\psi=0$, but the analogous homogeneous estimate follows immediately from this and the existence of an extension operator $E: H^{s-\frac{1}{2}}(\Gamma)\to H^s(\Omega)$ for $\frac{1}{2}<s\leq 1$ with norm uniform in $\Lambda_*$. In this low regularity range of $s$, such an operator can be constructed by using the partition of unity for $\Omega$ and the construction in \cite{sz}. We omit the details. 
\\

In a small number of places in the higher energy bounds, the following elliptic estimates which hold on $C^{1,\epsilon_0}$ (but not quite Lipschitz) domains will be convenient for simplifying the analysis.
\begin{proposition}\label{desiredelliptic}
For every $0<s<\frac{1}{2}+\epsilon_0$, there holds
\begin{equation*}
\|\Delta^{-1}g\|_{H^{s+1}(\Omega)}\lesssim_A \|g\|_{H^{s-1}(\Omega)},\hspace{5mm}\|\mathcal{H}\psi\|_{H^{s+1}(\Omega)}\lesssim_A \|\psi\|_{H^{s+\frac{1}{2}}(\Gamma)}.
\end{equation*}
\end{proposition}
\Cref{desiredelliptic} is well-known to specialists; see, e.g.,~\cite{mitrea2022geometric}. We remark that bounds of this type hold in the range $s<\frac{1}{2}$ when the domain is Lipschitz; the excess regularity given by a $C^{1,\epsilon_0}$ domain is required to extend the range to $s<\frac{1}{2}+\epsilon_0$.
\\

Next, we move to the higher regularity estimates for the Dirichlet problem.
\begin{proposition}[Higher regularity bounds for the inhomogeneous Dirichlet problem]\label{direst}  Let $\Omega$ be a bounded domain with boundary $\Gamma\in\Lambda_*$. Suppose that $v$ solves the Dirichlet problem
\begin{equation*}
\begin{cases}
&\Delta v=g\hspace{5.5mm}\text{in }\Omega,
\\
&\hspace{3mm}v=\psi\hspace{5mm}\text{on }\partial\Omega,
\end{cases}
\end{equation*}
and let $s\geq 2$. Then for $r\geq 0$, $\alpha\in [0,1],\beta\in [0,1]$ and any sequence of partitions $v:=v_j^1+v_j^2$, we have
\begin{equation*}
    \begin{split}
        \|v\|_{H^s(\Omega)}\lesssim_A\|g\|_{H^{s-2}(\Omega)}+\|\psi\|_{H^{s-\frac{1}{2}}(\Gamma)}+\|\Gamma\|_{H^{s+r-\frac{1}{2}}}\sup_{j>0}2^{-j\left(\alpha-1+r\right)}\|v_j^1\|_{C^\alpha(\Omega)}+\sup_{j>0}2^{j(s-1+\beta-\epsilon)}\|v_j^2\|_{H^{1-\beta}(\Omega)}.
    \end{split}
\end{equation*}
\end{proposition}
\begin{proof}
Using the partition of unity, it suffices to estimate $v_i:=\gamma_{*i}v$ for each $i\geq 0$. Since the case $i=0$ is essentially an interior regularity estimate, we focus on the case $i\geq 1$. We define
\begin{equation*}
    h:=\Delta v_i=g\gamma_{*i}+v\Delta\gamma_{*i}+2\nabla v\cdot\nabla \gamma_{*i}.
\end{equation*}
Using the map $H_i=G_i^{-1}$, we can write a variable coefficient equation for $u:=v_i\circ H_i$,
\begin{equation*}
  \begin{cases}
    &-\Delta u=(a^{ij}-\delta^{ij})\partial_i\partial_ju+b_j\partial_ju-f,
    \\  
    &u|_{\{z_d=0\}}=(\gamma_{*i}\psi)(H_i(\tilde{z},0)).
    \end{cases}
\end{equation*}
Here (dropping the $i$ index from the partition and now using it as a dummy index), we wrote $a^{lm}:=(G^l_{x_k}G^m_{x_k})\circ H$ (where $k$ is summed over), $b_j:=(\Delta G^j)\circ H$ and $f=h\circ H.$ As a first step, we prove the following estimate for $u$:
\begin{equation}\label{uest}
    \begin{split}
        \|u\|_{H^s}\lesssim_A \|f\|_{H^{s-2}}+\|\psi\|_{H^{s-\frac{1}{2}}(\Gamma)}+\|\Gamma\|_{H^{s+r-\frac{1}{2}}}\sup_{j>0}2^{-j\left(\alpha-1+r\right)}\|v_j^1\|_{C^\alpha(\Omega)}+\sup_{j>0}2^{j(s-1+\beta-\epsilon)}\|v_j^2\|_{H^{1-\beta}(\Omega)}.
    \end{split}
\end{equation}
For this, we use the standard elliptic regularity for the half-space to obtain
\begin{equation}\label{firstest}
\|u\|_{H^s}\lesssim_A \|u\|_{L^2}+\|f\|_{H^{s-2}}+\|b_i\partial_iu\|_{H^{s-2}}+\|(a^{ij}-\delta^{ij})\partial_i\partial_ju\|_{H^{s-2}}+\|(\gamma_{*i}\psi)(H_i(\tilde{z},0))\|_{H^{s-\frac{1}{2}}(\mathbb{R}^{d-1})}.
\end{equation}
By definition, the last term on the right-hand side is controlled by $\|\psi\|_{H^{s-\frac{1}{2}}(\Gamma)}$. Moreover, by a change of variables and the baseline estimates (\ref{H1base}) and (\ref{harmonicbase}), we can control, crudely,
\begin{equation}\label{crudelowerorder}
\|u\|_{L^2}\lesssim_A \|v_i\|_{L^2(\Omega)}\lesssim_A\|h\|_{L^2(\Omega)}+\|\psi\|_{H^{\frac{1}{2}}(\Gamma)}\lesssim_A \|f\|_{L^2}+\|\psi\|_{H^{\frac{1}{2}}(\Gamma)}\lesssim_A \|f\|_{H^{s-2}}+\|\psi\|_{H^{s-\frac{1}{2}}(\Gamma)}.  
\end{equation} 
For the purpose of estimating the third and fourth terms on the right-hand side, we may assume that $u\in H^s(\mathbb{R}^d)$ with compact support instead of just $u\in H^s(\mathbb{R}_+^d)$ by using any suitable extension for the half-space. We then recall that in a suitably refined collar, we have
\begin{equation*}
\|a^{ij}-\delta^{ij}\|_{L^{\infty}}+\|DG-I\|_{L^{\infty}}\ll_A 1.
\end{equation*}
Next, we define a partition of $u$ as follows: First write $v_i=\gamma_{*i}v_j^1+\gamma_{*i}v_j^2$ and then $u=v_i\circ H_i=(\gamma_{*i}v_j^1)\circ H_i+(\gamma_{*i}v_j^2)\circ H_i=:u_j^1+u_j^2$. To prove (\ref{uest}), it suffices now by interpolation and the above estimates to prove the estimate 
\begin{equation}\label{bestimatefordir}
\|b_i\partial_iu\|_{H^{s-2}}+\|(a^{ij}-\delta^{ij})\partial_i\partial_j u\|_{H^{s-2}}\lesssim_A \|u\|_{H^{s-\epsilon}}+\|DG-I\|_{L^{\infty}}\|u\|_{H^s}+\text{RHS}(\ref{uest}).
\end{equation}
We show the details for $b_i\partial_iu$ since it is the more difficult of the two terms to deal with (as it involves two derivatives applied to the domain flattening map) and because the estimate for $(a^{ij}-\delta^{ij})\partial_j\partial_iu$ follows from a similar analysis. Our first aim is to establish the bound
\begin{equation}\label{changeofvarbounddirichlet}
\|b_i\partial_i u\|_{H^{s-2}}\lesssim_A \|(\nabla u)(G)\cdot\Delta G\|_{H^{s-2}}+\text{RHS}(\ref{bestimatefordir}),  
\end{equation}
which,  to leading order, is essentially like doing an $H^{s-2}$ ``change of variables". This bound follows immediately from \Cref{McLean} for $2\leq s\leq 3$, so we restrict to $s\geq 3$. To simplify notation a bit, we write $w:=b_i\partial_i u$. We begin by applying \Cref{Moservariant} to obtain
\begin{equation}\label{intermedwest}
\|w\|_{H^{s-2}}\lesssim_A \|(\nabla u)(G)\cdot\Delta G\|_{H^{s-2}}+\|\Gamma\|_{H^{s+r-\frac{1}{2}}}\sup_{j>0}2^{-j(1+r)}\|w_j^1\|_{L^{\infty}}+\sup_{j>0}2^{j(s-2-\epsilon)}\|w_j^2\|_{L^2},
\end{equation}
where $w=w_j^1+w_j^2$ is a well-chosen partition which needs to be picked so that we can estimate the latter two terms above by \text{RHS}(\ref{bestimatefordir}). We take
\begin{equation*}
\begin{split}
&w_j^1:=(\Delta P_{<j}G\cdot (\nabla P_{<j}u_j^1)(G))(H),
\\
&w_j^2:=(\Delta P_{<j}G\cdot (\nabla P_{<j}u_j^2)(G)+\Delta P_{<j}G\cdot (\nabla P_{\geq j}u)(G)+\Delta P_{\geq j}G\cdot (\nabla u)(G))(H).    
\end{split}
\end{equation*}
It is then easily verified using the above and (\ref{intermedwest}) that we have
\begin{equation*}
\|w\|_{H^{s-2}}\lesssim_A \|(\nabla u)(G)\cdot\Delta G\|_{H^{s-2}}+\sup_{j>0}2^{j(s-2-\epsilon)}\|\Delta P_{\geq j}G\cdot (\nabla u)(G)\|_{L^2}+\text{RHS}(\ref{bestimatefordir}).   
\end{equation*}
To estimate the latter term on the right, we use that $s-2-\epsilon>0$ to estimate
\begin{equation*}
2^{j(s-2-\epsilon)}\|\Delta P_{\geq j}G\cdot (\nabla u)(G)\|_{L^2}\leq \sup_{l\geq 0}2^{l(s-2-\epsilon)}\|\Delta P_lG\cdot (\nabla u)(G)\|_{L^2}.    
\end{equation*}
Then splitting $u=P_{<l}u_l^1+\left(P_{<l}u_l^2+P_{\geq l}u\right)$, a change of variables and a simple application of the Bernstein inequalities allows us to control the above term by the right-hand side of (\ref{bestimatefordir}). This establishes (\ref{changeofvarbounddirichlet}) for $s\geq 3$. Finally, for each $s\geq 2$, it remains to estimate $\|(\nabla u)(G)\Delta G\|_{H^{s-2}}$ by the right-hand side of (\ref{bestimatefordir}). From a simple paradifferential analysis as in \Cref{productestref}, we have 
\begin{equation*}
\begin{split}
\|(\nabla u)(G)\cdot\Delta G\|_{H^{s-2}}&\lesssim_A \|(\nabla u)(G)\|_{H^{s-1-\epsilon}}+\|T_{(\nabla u)(G)}\Delta G\|_{H^{s-2}} 
\\
&\lesssim_A \|(\nabla u)(G)\|_{H^{s-1-\epsilon}}+\text{RHS}(\ref{bestimatefordir}) ,
\end{split}
\end{equation*}
where, above, to estimate the latter term in the first line, we estimated each summand $P_{<j-4}(\nabla u(G))P_j\Delta G$ in the paradifferential expansion of $T_{(\nabla u)(G)}\Delta G$ using the partition $u=P_{<j}u_j^1+\left(P_{<j}u_j^2+P_{\geq j}u\right)$ and Bernstein's inequality. Then, using \Cref{Moservariant} and this same partition, we have easily
\begin{equation*}
\|(\nabla u)(G)\|_{H^{s-1-\epsilon}}\lesssim_A \text{RHS}(\ref{bestimatefordir}).    
\end{equation*}
This establishes the bound (\ref{bestimatefordir}) for $b_i\partial_i u$. The bound for $(a^{ij}-\delta^{ij})\partial_i\partial_ju$ follows similar reasoning, but is easier because it involves only one derivative applied to the domain flattening map, and therefore the initial change of variables performed above is not needed. This concludes the estimate (\ref{uest}). Our next step is replace $u$ on the left-hand side of (\ref{uest}) with $v_i$ and replace $f$ on the right-hand side with $g$. Recall first that $v_i=u\circ G_i$ and $f=h\circ H_i.$  We may assume that $v_i$ and $u$ are defined on $\mathbb{R}^d$ using Stein's extension or a suitable half-space extension in the case of $u$. Therefore, using the partition $u=u_j^1+u_j^2$ as defined earlier and \Cref{Moservariant} we obtain 
\begin{equation*}
\begin{split}
\|v_i\|_{H^s(\Omega)}&\lesssim_A \|u\|_{H^s}+\|\Gamma\|_{H^{s+r-\frac{1}{2}}}\sup_{j>0}2^{-j(\alpha-1+r)}\|v_j^1\|_{C^{\alpha}(\Omega)}+\sup_{j>0}2^{j(s+\beta-1-\epsilon)}\|v_j^2\|_{H^{1-\beta}(\Omega)},
\end{split}
\end{equation*}
where we used that $\|G-Id\|_{H^{s+r}}\lesssim_A \|\Gamma\|_{H^{s+r-\frac{1}{2}}}$.
\\

To conclude we now need only show that
\begin{equation}\label{fest}
\|f\|_{H^{s-2}}\lesssim_A \|g\|_{H^{s-2}(\Omega)}+\|\Gamma\|_{H^{s+r-\frac{1}{2}}}\sup_{j>0}2^{-j(\alpha-1+r)}\|v_j^1\|_{C^{\alpha}(\Omega)}+\sup_{j>0}2^{j(s+\beta-1-\epsilon)}\|v_j^2\|_{H^{1-\beta}(\Omega)}+\sup_i\|v_i\|_{H^{s-\epsilon}(\Omega)}.
\end{equation}
Expanding out $h=\Delta (v\gamma_{*i})$ and using again a paradifferential expansion similar to \Cref{productestref}, the identity $g:=\Delta v$ and  the splitting $v=v_j^1+v_j^2$ we observe first that
\begin{equation*}
\|h\|_{H^{s-2}(\Omega)}\lesssim_A \|g\|_{H^{s-2}(\Omega)}+\|\Gamma\|_{H^{s+r-\frac{1}{2}}}\sup_{j>0}2^{-j(\alpha-1+r)}\|v_j^1\|_{C^{\alpha}(\Omega)}+\sup_{j>0}2^{j(s+\beta-1-\epsilon)}\|v_j^2\|_{H^{1-\beta}(\Omega)}+\sup_i\|v_i\|_{H^{s-\epsilon}(\Omega)}.
\end{equation*}
 Therefore, we need to only show (\ref{fest}) with $g$ replaced by $h$. For this, we first extend $h$ to a function $\tilde{h}:=\mathcal{E}\Delta (\gamma_{*i}v)$ on $\mathbb{R}^d$ using Stein's extension. Then, using the partition $\tilde{h}=h_j^1+h_j^2$ with $h_j^1=\mathcal{E}\Delta P_{<j}(v_j^1\gamma_{*i})$ and $h_j^2=\mathcal{E}\Delta P_{<j}(v_j^2\gamma_{*i})+\mathcal{E}\Delta P_{\geq j}(v\gamma_{*i})$ together with \Cref{Moservariant}, we obtain (\ref{fest}) and conclude the proof.
\end{proof}
We also note a much cruder variant of the above estimate which will be useful for constructing regularization operators later on. As with the corresponding Moser bound in \Cref{crudermoser}, the proposition below could be optimized considerably, but such optimizations will not be needed in this article.  
\begin{proposition}[Cruder variant of the Dirichlet estimates]\label{direst2}  Let $\Gamma$, $v$, $\psi$, $g$ and $s\geq 2$  be as in \Cref{direst}, and assume  that $\psi=0$. Then for every $\delta>0$, we have the estimate 
\begin{equation*}
\|v\|_{H^s(\Omega)}\lesssim_{A,\delta} \|g\|_{H^{s-2}(\Omega)}+\|\Gamma\|_{C^{s+\delta}}\|v\|_{H^{1}(\Omega)}.    
\end{equation*}
\end{proposition}
\begin{proof}
We only give a sketch of the proof since it is essentially a much simpler version of \Cref{direst}. One starts by using the cruder flattening (\ref{classicalflattening}) as in the pointwise elliptic estimates and writing the corresponding equation for $u$ 
(using the notation in (\ref{frozeneqn})). This flattening is a bit more convenient for this estimate because the source terms in (\ref{frozeneqn}) are simpler. Moreover, we will only need to measure $\Gamma$ in pointwise norms, and therefore will not need the $\frac{1}{2}$ gain of regularity from the flattening in \Cref{direst}. As in the proof of \Cref{direst}, we then obtain the preliminary bound
\begin{equation*}
\|u\|_{H^s}\lesssim_A \|f\|_{H^{s-2}}+\|(\delta^{jk}-a^{jk})\partial_ju\|_{H^{s-1}}.    
\end{equation*}
Using simple paraproduct type estimates and a change of variables, it is straightforward to then estimate
\begin{equation}
\|u\|_{H^s}\lesssim_{A,\delta} \|f\|_{H^{s-2}}+\|\Gamma\|_{C^{s+\delta}}\|v\|_{H^1(\Omega)}.    
\end{equation}
Then, to conclude, one estimates using \Cref{crudermoser} with $r=0$ and $r=2$,
\begin{equation*}
\|v_i\|_{H^s}\lesssim_{A,\delta} \|u\|_{H^s}+\|\Gamma\|_{C^{s+\delta}}\|v\|_{H^1(\Omega)},\hspace{5mm}\|f\|_{H^{s-2}}\lesssim_A \|h\|_{H^{s-2}(\Omega)}+\|\Gamma\|_{C^{s+\delta}}\|v\|_{H^1(\Omega)} ,    
\end{equation*}
and then performs a simple paraproduct analysis to finally estimate
\begin{equation*}
\|h\|_{H^{s-2}(\Omega)}\lesssim_A \|g\|_{H^{s-2}(\Omega)}+\|\Gamma\|_{C^{s+\delta}}\|v\|_{H^1(\Omega)}+\|v\|_{H^{s-\epsilon}(\Omega)}.     
\end{equation*}
Combining the above and interpolating finishes the proof.
\end{proof}

\subsubsection{Harmonic extension bounds}
By taking $g=0$ in \Cref{direst}, we obtain the following corollary for the harmonic extension operator $\mathcal{H}$.
\begin{proposition}[Harmonic extension bounds]\label{Hbounds}Let $\Omega$ be a bounded domain with boundary $\Gamma\in\Lambda_*$. Then the following bound holds for the harmonic extension operator $\mathcal{H}$ when $s\geq 2$, $r\geq 0$, $\beta\in [0,\frac{1}{2})$ and $\alpha\in [0,1)$,
\begin{equation*}
 \|\mathcal{H}\psi\|_{H^s(\Omega)}\lesssim_A \|\psi\|_{H^{s-\frac{1}{2}}(\Gamma)}+\|\Gamma\|_{H^{s+r-\frac{1}{2}}}\sup_{j>0}2^{-j(\alpha-1+r)}\|\psi_j^1\|_{C^\alpha(\Gamma)}+\sup_{j>0}2^{j(s-1+\beta-\epsilon)}\|\psi_j^2\|_{H^{\frac{1}{2}-\beta}(\Gamma)}.
\end{equation*}
Here, $\psi=\psi_j^1+\psi_j^2$ is any sequence of partitions.
\end{proposition}
\begin{proof}
First, \Cref{direst} yields the estimate
\begin{equation*}
\|\mathcal{H}\psi\|_{H^s(\Omega)}\lesssim_A\|\psi\|_{H^{s-\frac{1}{2}}(\Gamma)}+\|\Gamma\|_{H^{s+r-\frac{1}{2}}}\sup_{j>0}2^{-j(\alpha-1+r)}\|\phi_j^1\|_{C^\alpha(\Omega)}+\sup_{j>0}2^{j(s-1-\epsilon)}\|\phi_j^2\|_{H^{1}(\Omega)},
\end{equation*}
where $\phi_j^1=P_{<j}\mathcal{H}\psi_j^1$ and $\phi_j^2=P_{<j}\mathcal{H}\psi_{j}^2+P_{\geq j}\mathcal{H}\psi$. From the $C^{\alpha}$ bounds for $\mathcal{H}$ in \Cref{Hboundlow} (which hold only for $\alpha\in [0,1)$), we have $\|\phi_j^1\|_{C^{\alpha}(\Omega)}\lesssim\|\psi_j^1\|_{C^{\alpha}(\Gamma)}$. On the other hand, from \eqref{harmonicbase}, we obtain
\begin{equation*}
\sup_{j>0}2^{j(s-1-\epsilon)}\|\phi_j^2\|_{H^1(\Omega)}\lesssim_A \|\mathcal{H}\psi\|_{H^{s-\epsilon}(\Omega)}+\sup_{j>0}2^{j(s-1+\beta-\epsilon)}\|\psi_j^2\|_{H^{\frac{1}{2}-\beta}(\Gamma)}.
\end{equation*}
The proof then concludes by interpolation and again (\ref{harmonicbase}).
\end{proof}
\subsubsection{Curvature estimate}
With the above local coordinates, we can control  the surface regularity in terms of the mean curvature. The following estimate is a slight refinement of Lemma 4.7 as well as Propositions A.2 and A.3 in \cite{sz}.
\begin{proposition}[Curvature estimate]\label{curvaturebound} Let $s\geq 2$. The following estimates for $\|\Gamma\|_{H^s}$ and the normal $n_\Gamma$ hold:
\begin{equation*}
\|\Gamma\|_{H^s}+\|n_\Gamma\|_{H^{s-1}(\Gamma)}\lesssim_A 1+\|\kappa\|_{H^{s-2}(\Gamma)}.
\end{equation*}
\end{proposition}
\begin{proof} We only sketch the details as the proof is similar to \cite{sz}. As in their proof, let $\{f_i\in H^s(\tilde{R}_i(2r_i))\}$ be the local coordinate functions associated to $\Gamma$ defined earlier. Let $\gamma:[0,\infty)\to [0,1]$ be a smooth cutoff function supported on $[0,\frac{3}{2}]$ with $\gamma=1$ on $[0,\frac{5}{4}]$. On each $\tilde{R}_i(2r_i)$, we let
\begin{equation*}
\gamma_i(\tilde{z})=\gamma(\frac{|\tilde{z}|}{r_i}),\hspace{10mm}\kappa_i(\tilde{z})=\gamma_i(\tilde{z})\kappa(\tilde{z},f_i(\tilde{z})),\hspace{10mm}g_i=\gamma_if_i.    
\end{equation*}
Using the mean curvature formula
\begin{equation*}
\kappa(\tilde{z},f(\tilde{z}))=-\partial_j(\frac{\partial_jf}{\sqrt{1+|\nabla f|^2}})=-\frac{\Delta f}{(1+|\nabla f|^2)^{\frac{1}{2}}}+\frac{\partial_jf\partial_kf\partial_{jk}f}{(1+|\nabla f|^2)^{\frac{3}{2}}} ,  
\end{equation*}
we obtain the following elliptic equation for $g_i$:
\begin{equation*}
\begin{split}
-\Delta g_i&=-\frac{\partial_{j_1}f_i\partial_{j_2}f_i}{(1+|\nabla f_i|^2)}\partial_{j_1j_2}g_i+(1+|\nabla f_i|^2)^{\frac{1}{2}}\kappa_i-\Delta \gamma_if_i-2D\gamma_i\cdot Df_i 
\\
&\quad \, +\frac{\partial_{j_1}f_i\partial_{j_2}f_i}{1+|\nabla f_i|^2}(\partial_{j_1j_2}\gamma_if_i+\partial_{j_1}\gamma_i\partial_{j_2}f_i+\partial_{j_2}\gamma_i\partial_{j_1}f_i).
\end{split}
\end{equation*}
As $\|Df_i\|_{L^{\infty}}\ll 1$ the first term on the right-hand side can be viewed perturbatively. A paradifferential type analysis similar to the estimate for $u$ in \Cref{direst} together with standard Moser and product type estimates then gives
\begin{equation*}
\|g_i\|_{H^s}\lesssim_A \delta\|g_i\|_{H^s}+\|f_i\|_{H^{s-\epsilon}}+\|\kappa\|_{H^{s-2}(\Gamma)}
\end{equation*}
for some $\delta>0$ small enough (depending on $\Lambda_*$). We then obtain
\begin{equation*}
\|g_i\|_{H^s}\lesssim_A \|f_i\|_{H^{s-\epsilon}}+\|\kappa\|_{H^{s-2}(\Gamma)},  
\end{equation*}
and so, we obtain,
\begin{equation*}
\sup_{i}\|f_i\|_{H^s}\lesssim_A 1+\|\kappa\|_{H^{s-2}(\Gamma)},   
\end{equation*}
which completes the proof.
\end{proof}
\subsubsection{Estimates for the Dirichlet-to-Neumann operator} Here, we use the above estimates to prove refined bounds for the Dirichlet-to-Neumann operator which is defined by $\mathcal{N}:=n_\Gamma\cdot (\nabla\mathcal{H})_{|\Gamma}$. We begin with the following baseline ellipticity estimate.
\begin{lemma}\label{baselineDN}
    The Dirichlet-to-Neumann map on $\Gamma$ satisfies 
    \begin{equation*}
        \|\psi\|_{H^1(\Gamma)}\lesssim_A \|\mathcal{N}\psi\|_{L^2(\Gamma)}+\|\psi\|_{L^2(\Gamma)}.
    \end{equation*}
\end{lemma}
\begin{proof}
   Let $v=\mathcal{H}\psi$. We begin by proving the  standard estimate 
\begin{equation}\label{boundarygrad}
\int_{\Gamma}|\nabla v|^2dS\lesssim_A \|\mathcal{N}\psi\|_{L^2(\Gamma)}^2+\|\psi\|_{L^2(\Gamma)}\|\psi\|_{H^1(\Gamma)}. 
\end{equation}
Let $X$ be a smooth vector field on $\mathbb{R}^d$  which is uniformly transversal to all hypersurfaces in $\Lambda_*$. That is, $X\cdot n_\Gamma\gtrsim_A 1$ and $|DX|\lesssim_A 1$. Integration by parts then gives
\begin{equation*}
\begin{split}
\int_{\Gamma}|\nabla v|^2\,dS&\lesssim_A \int_{\Gamma}n_\Gamma\cdot X|\nabla v|^2\, dS
\\
&\lesssim_A \|\nabla v\|_{L^2(\Omega)}^2+2\int_{\Omega}X_j\partial_j\nabla v\cdot\nabla v\, dx
\\
&\lesssim_A\|\nabla v\|_{L^2(\Omega)}^2+2\int_{\Gamma}(X\cdot\nabla v)\mathcal{N}\psi \, dS.
\end{split}
\end{equation*}
For the first term, we have from the $H^{\frac{1}{2}}\to H^1$ harmonic extension bound and straightforward interpolation, 
\begin{equation*}
\|v\|_{H^1(\Omega)}^2\lesssim_A \|\psi\|_{H^{\frac{1}{2}}(\Gamma)}^2\lesssim_A \|\psi\|_{L^2(\Gamma)}\|\psi\|_{H^1(\Gamma)}  .  
\end{equation*}
 Combining this with the Cauchy Schwarz inequality for the second term, we obtain (\ref{boundarygrad}). Using the partition of unity $(\gamma_{*i})_i$, it straightforward to then estimate 
\begin{equation*}
\|\psi\|_{H^1(\Gamma)}\lesssim_A \|\psi\|_{L^2(\Gamma)}+\|\nabla^{\top}v\|_{L^2(\Gamma)}\lesssim_A \|\psi\|_{L^2(\Gamma)}+\|\nabla v\|_{L^2(\Gamma)},
\end{equation*}   
where $\nabla^{\top}$ denotes the projection of $\nabla$ onto the tangent space of $\Gamma$. Combining this with (\ref{boundarygrad}) and Cauchy Schwarz concludes the proof.
\end{proof}
We will also need the reverse inequality.
\begin{lemma}\label{baselineDN2}
The Dirichlet-to-Neumann map on $\Gamma$ satisfies 
    \begin{equation*}
        \|\mathcal{N}\psi\|_{L^2(\Gamma)}\lesssim_A \|\psi\|_{H^1(\Gamma)}.
    \end{equation*}  
\end{lemma}
\begin{proof}
Using the same notation as in the above lemma and essentially the same argument, we have the estimate 
\begin{equation*}
\begin{split}
&\int_{\Gamma}(X\cdot n_{\Gamma})|\nabla^{\top} \psi|^2\, dS+\int_{\Gamma}(X\cdot n_{\Gamma})|\mathcal{N}\psi|^2\, dS=\int_{\Gamma}(X\cdot n_{\Gamma})|\nabla v|^2\, dS
\\
&\geq -C \|\psi\|_{H^1(\Gamma)}^2+2\int_{\Gamma} (X\cdot\nabla v)\mathcal{N}\psi \, dS
\end{split}
\end{equation*}
for some constant $C$ depending only on $A$. Writing $X^{\top}:=X-(X\cdot n_{\Gamma})n_{\Gamma}$, we obtain
\begin{equation*}
\begin{split}
\int_{\Gamma}(X\cdot n_{\Gamma})|\mathcal{N}\psi|^2\, dS&\leq C\|\psi\|_{H^1(\Gamma)}^2+\int_{\Gamma}(X\cdot n_{\Gamma})|\nabla^{\top}\psi|^2\, dS-2\int_{\Gamma}X^{\top}\cdot\nabla v\mathcal{N}\psi\, dS  ,
\end{split}
\end{equation*}
which by Cauchy Schwarz completes the proof.
\end{proof}
Next, we prove higher regularity versions of these bounds. The first bound below amounts essentially to elliptic regularity estimates for the Neumann boundary value problem.
\begin{proposition}[Ellipticity for the Dirichlet-to-Neumann operator I]\label{balneum} Let $s\geq \frac{3}{2}$, $\alpha\in [0,1)$ and $\beta\in [0,\frac{1}{2})$. Then we have
\begin{equation}\label{Neumanaestbal}
\|\psi\|_{H^{s}(\Gamma)}\lesssim_A \|\psi\|_{L^2(\Gamma)}+\|\mathcal{N}\psi\|_{H^{s-1}(\Gamma)}+\|\Gamma\|_{H^{s+r}}\sup_{j>0}2^{-j(r+\alpha-1)}\|\psi_j^1\|_{C^{\alpha}(\Gamma)}+\sup_{j>0}2^{j(s+\beta-\frac{1}{2}-\epsilon)}\|\psi_j^2\|_{H^{\frac{1}{2}-\beta}(\Gamma)}.
\end{equation}
\end{proposition}
\begin{proof}
The proof of this is very similar to the Dirichlet problem, so we only sketch the details. Indeed, write $v:=\mathcal{H}\psi$. By \Cref{baltrace}, \eqref{harmonicbase} and the $C^{\alpha}\to C^{\alpha}$ bound for $\mathcal{H}$, it suffices to control $v$ in $H^{s+\frac{1}{2}}(\Omega)$ by the right-hand side of (\ref{Neumanaestbal}). As with the Dirichlet problem, the procedure  is to write the Laplace equation for $u=v_i\circ H_i$ and  to reduce matters to the standard estimate for the Neumann problem on the half-space (which is available since $s>1$). The only added technicality is that there are extra source terms coming from the Neumann data (in contrast to the source terms which do not appear for the Dirichlet problem with zero boundary data). By using \Cref{baltrace} and an analysis similar to \Cref{direst}, it is straightforward to obtain the preliminary estimate
\begin{equation*}
\|\psi\|_{H^{s}(\Gamma)}\lesssim_A \|v\|_{H^1(\Omega)}+\|\mathcal{N}\psi\|_{H^{s-1}(\Gamma)}+\|\Gamma\|_{H^{s+r}}\sup_{j>0}2^{-j(r+\alpha-1)}\|v_j^1\|_{C^{\alpha}(\Omega)}+\sup_{j>0}2^{j(s+\beta-\frac{1}{2}-\epsilon)}\|v_j^2\|_{H^{1-\beta}(\Omega)},
\end{equation*}
where $v:=v_j^1+v_j^2$ is any partition of $v$. The first term $\|v\|_{H^{1}(\Omega)}$ is harmless and can be controlled by $\|\psi\|_{L^2(\Gamma)}+\|\mathcal{N}\psi\|_{L^2(\Gamma)}$ using the $H^{\frac{1}{2}}\to H^1$ bound for $\mathcal{H}$ and \Cref{baselineDN}. We then take $v_j^1=\mathcal{H}\psi_j^1$ and $v_j^2=\mathcal{H}\psi_j^2$ and use again the $C^{\alpha}\to C^{\alpha}$ bounds for $\mathcal{H}$ and (\ref{harmonicbase}) to conclude.
\end{proof}
We will also need the following iterated version of the ellipticity bound above.
\begin{proposition}[Ellipticity for the Dirichlet-to-Neumann operator II]\label{ellipticity} Let $s\geq \frac{1}{2}$ and let $k\geq 1$ be an integer. Then using the same notation as the previous proposition, we have the bound
\begin{equation*}
\|\psi\|_{H^{s+k}(\Gamma)}\lesssim_A \|\psi\|_{L^2(\Gamma)}+\|\mathcal{N}^k\psi\|_{H^{s}(\Gamma)}+\|\Gamma\|_{H^{s+k+r}}\sup_{j>0}2^{-j(\alpha-1+r)}\|\psi_j^1\|_{C^{\alpha}(\Gamma)}+\sup_{j>0}2^{j(s+k-\frac{1}{2}+\beta-\epsilon)}\|\psi_j^2\|_{H^{\frac{1}{2}-\beta}(\Gamma)}.
\end{equation*}
\end{proposition}
\begin{proof}
 \Cref{baselineDN} and \Cref{balneum} give us this bound for $k=1$. For $k\geq 2$, we may assume inductively that the corresponding estimate holds for all $1\leq m\leq k-1$. We begin by applying \Cref{balneum} to obtain
\begin{equation}\label{psiintermed}
\|\psi\|_{H^{s+k}(\Gamma)}\lesssim_A \|\psi\|_{L^2(\Gamma)}+\|\mathcal{N}\psi\|_{H^{s+k-1}(\Gamma)}+\|\Gamma\|_{H^{s+k+r}}\sup_{j>0}2^{-j(\alpha-1+r)}\|\psi_j^1\|_{C^{\alpha}(\Gamma)}+\sup_{j>0}2^{j(s+k-\frac{1}{2}+\beta -\epsilon)}\|\psi_j^2\|_{H^{\frac{1}{2}-\beta}(\Gamma)}.
\end{equation}
Using the inductive hypothesis, we have
\begin{equation*}
\|\mathcal{N}\psi\|_{H^{s+k-1}(\Gamma)}\lesssim_A \|\mathcal{N}\psi\|_{L^2(\Gamma)}+\|\mathcal{N}^{k}\psi\|_{H^s(\Gamma)}+\|\Gamma\|_{H^{s+k+r}}\sup_{j>0}2^{-jr}\|\phi_j^1\|_{L^{\infty}(\Gamma)}+\sup_{j>0}2^{j(s+k-1-2\epsilon)}\|\phi_j^2\|_{H^{\epsilon}(\Gamma)},
\end{equation*}
where $\mathcal{N}\psi:=\phi_j^1+\phi_j^2$ is any partition of $\mathcal{N}\psi$. By \Cref{baselineDN2}, the first term on the right can be controlled by $\|\psi\|_{H^1(\Gamma)}$ which can be dispensed with by interpolation (between $L^2$ and $H^{1+\epsilon}$ to ensure the domain dependent contributions in the estimate are harmless). Therefore, to conclude, we need to choose $\phi_j^1$ and $\phi_j^2$ so that the latter two terms on the right-hand side of the above are controlled by the right-hand side of (\ref{psiintermed}). Using $v$, $v_j^1$ and $v_j^2$ from the previous proposition, we can take $\phi_j^1=\nabla_n P_{<j}v_j^1$ and $\phi_j^2=\nabla_n P_{<j}v_j^2+\nabla_n P_{\geq j}v$. The proof then concludes in a similar way to \Cref{balneum}. We omit the details.
\end{proof}
For our energy estimates, we will also need good bounds for the following div-curl system.
\begin{proposition}[div-curl estimate with Neumann type data]\label{Balanced div-curl} Let $v\in H^s(\Omega)$ be a vector field defined on $\Omega$ and let $s>\frac{3}{2}$, $\alpha,\beta\in [0,1]$. Let $v:=v_j^1+v_j^2$ be any partition of $v$. Moreover, let $\mathcal{B}v$ denote either the Neumann trace of $v$, $n_{\Gamma}\cdot \nabla v$ or the boundary value $\nabla^{\top}v\cdot n_\Gamma$. Then if $v$ solves the  div-curl system,
\begin{equation*}
\begin{cases}
&\nabla\cdot v=f,
\\
&\nabla\times v=\omega,
\\
&\mathcal{B}v=g,
\end{cases}
\end{equation*}
then $v$ satisfies the estimate,
\begin{equation*}
\begin{split}
\|v\|_{H^{s}(\Omega)}&\lesssim_A\|f\|_{H^{s-1}(\Omega)}+\|\omega\|_{H^{s-1}(\Omega)}+\|g\|_{H^{s-\frac{3}{2}}(\Gamma)}+\|v\|_{L^2(\Omega)}+\|\Gamma\|_{H^{s+r-\frac{1}{2}}}\sup_{j>0}2^{-j(r+\alpha-1)}\|v_j^1\|_{C^{\alpha}(\Omega)}
\\
&\qquad +\sup_{j>0}2^{j(s-1+\beta-\epsilon)}\|v_j^2\|_{H^{1-\beta}(\Omega)}.
\end{split}
\end{equation*}
\end{proposition}
\begin{proof}
The proof is very similar to the Dirichlet and Neumann problems in that one flattens the boundary and reduces to the corresponding estimate on the half-space with source terms depending on essentially $f$, $\omega$,  $g$ and the domain regularity. We omit the details of the domain flattening as it is similar to \Cref{direst}. However, for the sake of clarity, it is instructive to explain the div-curl estimate in the case when $\Omega$ is the half-space $\{z_d<0\}$ (particularly in the case of the latter boundary condition involving $\nabla^{\top}v\cdot n_{\Gamma}$). We show that it is in essence a statement about elliptic regularity for the Neumann problem. In such a setting, $n_{\Gamma}$ takes the form $e_d$. We compute for each (Euclidean) component $v_j$ of a vector field $v$ on $\Omega$,
\begin{equation*}
\Delta v_j=\partial_i\omega_{ij}+\partial_jf.    
\end{equation*}
Therefore, in the case of boundary data given by $\mathcal{B}v=n_{\Gamma}\cdot\nabla v$, the div-curl estimate is simply given by elliptic regularity for the Neumann problem. To understand the case of the other boundary value $\nabla^{\top}v\cdot n_{\Gamma}$, we note that the full Neumann data for $v$ is determined by this boundary value and the curl and divergence of $v$. If $j\neq d$, this is seen from the identity
\begin{equation*}
\partial_dv_j=\partial_jv_d+\omega_{dj}.    
\end{equation*}
So, by the trace theorem and elliptic regularity for the Neumann problem, we have the desired control of $v_j$ for $j\neq d$. If $j=d$, we have
\begin{equation*}
\partial_dv_d=f-\sum_{i=1}^{d-1}\partial_iv_i   , 
\end{equation*}
which by the trace theorem and the estimate for $v_i$ with $i\neq d$ gives us the estimate for $v_d$.
\end{proof}
We importantly do not claim that the above div-curl system is well-posed. In fact, the problem is generally over-determined (as, for instance, the curl and divergence fix $\Delta v$, which forbids certain choices of Neumann data). Fortunately, we will only need the above estimate in our analysis later when we prove energy estimates and to a lesser extent in our construction of regular solutions. We will not need any existence type statement for the above system, however.
\\

Next, to complement the ellipticity estimates for $\mathcal{N}$, we will also need the reverse estimates which control powers of $\mathcal{N}$ applied to a function in terms of the corresponding Sobolev norms of that function. As a preliminary step, we state the following proposition.
\begin{proposition}[Normal derivative trace bound]\label{L1bound} Let $s>0$, $r\geq 0$ and $\alpha,\beta\in [0,1]$. The normal trace operator $\nabla_n:=n_\Gamma\cdot (\nabla)_{|\Gamma}$ satisfies the bound
\begin{equation*}
    \|\nabla_nv\|_{H^{s}(\Gamma)}\lesssim_A\|v\|_{H^{s+\frac{3}{2}}(\Omega)}+\|\Gamma\|_{H^{s+r+1}}\sup_{j>0}2^{-j\left(r-1+\alpha\right)}\|v_j^1\|_{C^\alpha(\Omega)}+\sup_{j>0}2^{j(s+\beta+\frac{1}{2}-\epsilon)}\|v_j^2\|_{H^{1-\beta}(\Omega)}.
\end{equation*}
\end{proposition}
\begin{proof} Using the partition $\nabla v=w_j^1+w_j^2$ where $w_j^1:=\nabla P_{<j}v_j^1$ and $w_j^2=\nabla P_{<j}v_j^2+\nabla P_{\geq j}v$ together with the inequalities $\|n_\Gamma\|_{H^{s+r}(\Gamma)}\lesssim_A \|\Gamma\|_{H^{s+r+1}}$ and $\|n_\Gamma\|_{C^{\epsilon}(\Gamma)}\lesssim_A 1$, we obtain from \Cref{boundaryest} and \Cref{baltrace} (after possibly relabelling $\epsilon$),
\begin{equation*}
    \begin{split}
    \|\nabla_nv\|_{H^{s}(\Gamma)}&\lesssim_A \|(\nabla v)_{|\Gamma}\|_{H^{s}(\Gamma)}+\|\Gamma\|_{H^{s+r+1}}\sup_{j>0}2^{-jr}\|w_j^1\|_{L^\infty(\Omega)}+\sup_{j>0}2^{j(s-2\epsilon)}\|w^2_{j|\Gamma}\|_{L^2(\Gamma)}
    \\
    &\lesssim_A \|v\|_{H^{s+\frac{3}{2}}(\Omega)}+\|\Gamma\|_{H^{s+r+1}}\sup_{j>0}2^{-jr}\|w_j^1\|_{L^\infty(\Omega)}+\sup_{j>0}2^{j(s-2\epsilon)}\|w_j^2\|_{H^{\frac{1}{2}+\epsilon}(\Omega)}.
    \end{split}
\end{equation*}
By estimating 
\begin{equation*}
\|w_j^1\|_{L^\infty(\Omega)}\lesssim_A 2^{j(1-\alpha)}\|v_j^1\|_{C^{\alpha}(\Omega)}
\end{equation*}
and 
\begin{equation*}
\begin{split}
2^{j(s-2\epsilon)}\|w_j^2\|_{H^{\frac{1}{2}+\epsilon}(\Omega)}&\lesssim_A \|v\|_{H^{s+\frac{3}{2}}(\Omega)}+2^{j(s+\frac{1}{2}+\beta-\epsilon)}\|v_j^2\|_{H^{1-\beta}(\Omega)},    
\end{split}
\end{equation*}
we complete the proof.
\end{proof}
We can use \Cref{L1bound} and the balanced bounds for $\mathcal{H}$ to prove a refined version of the $H^{s+1}(\Gamma)\to H^s(\Gamma)$ bound for $\mathcal{N}$. 
\begin{proposition}[Dirichlet-to-Neumann operator bound I]\label{DNpower1} Let $s\geq \frac{1}{2}$, $r\geq 0$, $\alpha\in [0,1)$ and $\beta\in [0,\frac{1}{2})$. Then
\begin{equation*}
    \|\mathcal{N}\psi\|_{H^{s}(\Gamma)}\lesssim_A \|\psi\|_{H^{s+1}(\Gamma)}+\|\Gamma\|_{H^{s+1+r}}\sup_{j>0}2^{-j(r-1+\alpha)}\|\psi_j^1\|_{C^\alpha(\Gamma)}+\sup_{j>0}2^{j(s+\frac{1}{2}+\beta-\epsilon)}\|\psi_j^2\|_{H^{\frac{1}{2}-\beta}(\Gamma)}
\end{equation*}
for any sequence of partitions $\psi=\psi_j^1+\psi_j^2$.
\end{proposition}
\begin{proof} The proof begins by writing $\mathcal{N}=\nabla_n\mathcal{H}$ and  applying \Cref{L1bound} to obtain 
\begin{equation*}
 \|\mathcal{N}\psi\|_{H^{s}(\Gamma)}\lesssim_A \|\mathcal{H}\psi\|_{H^{s+\frac{3}{2}}(\Omega)}+\|\Gamma\|_{H^{s+1+r}}\sup_{j>0}2^{-j(r-1+\alpha)}\|\mathcal{H}\psi_j^1\|_{C^\alpha(\Omega)}+\sup_{j>0}2^{j(s+\frac{1}{2}+\beta-\epsilon)}\|\mathcal{H}\psi_j^2\|_{H^{1-\beta}(\Omega)}.    
\end{equation*}
Using the $C^{\alpha}\to C^{\alpha}$ bounds for $\mathcal{H}$, (\ref{harmonicbase}) and \Cref{Hbounds}, we conclude the proof.
\end{proof}
Similarly to the ellipticity estimate for $\mathcal{N}$, we will need a higher order version of the above estimate as well.
\begin{proposition}[Dirichlet-to-Neumann operator bound II]\label{higherpowers} Let $m\geq 1$ be an integer, let $s\geq \frac{1}{2}$ and let $r\geq 0$, $\alpha\in [0,1)$ and $\beta\in [0,\frac{1}{2})$. Then we have the  bound
\begin{equation*}
\label{Iterated N}
    \|\mathcal{N}^m\psi\|_{H^{s}(\Gamma)}\lesssim_A \|\psi\|_{H^{s+m}(\Gamma)}+\|\Gamma\|_{H^{s+r+m}}\sup_{j>0}2^{-j(r+\alpha-1)}\|\psi_j^1\|_{C^\alpha(\Gamma)}+\sup_{j>0}2^{j(s-\frac{1}{2}+m+\beta-\epsilon)}\|\psi_j^2\|_{H^{\frac{1}{2}-\beta}(\Gamma)}
\end{equation*}
and the closely related bound when $s\geq\frac{3}{2}$,
\begin{equation}\label{iteratedN2}
\|\mathcal{H}\mathcal{N}^m\psi\|_{H^{s+\frac{1}{2}}(\Omega)}\lesssim_A \|\psi\|_{H^{s+m}(\Gamma)}+\|\Gamma\|_{H^{s+r+m}}\sup_{j>0}2^{-j(r+\alpha-1)}\|\psi_j^1\|_{C^\alpha(\Gamma)}+\sup_{j>0}2^{j(s-\frac{1}{2}+m+\beta-\epsilon)}\|\psi_j^2\|_{H^{\frac{1}{2}-\beta}(\Gamma)}    
\end{equation}
for any partition $\psi=\psi_j^1+\psi_j^2$.
\end{proposition}
\begin{proof}
We begin with the first bound. The previous proposition handles the case $m=1$. Suppose $m> 1$ and let us suppose inductively that the bound holds for all integers greater than or equal to 1 and strictly less than $m$. Then we have from the inductive hypothesis,
\begin{equation}\label{powers2}
\|\mathcal{N}^m\psi\|_{H^s(\Gamma)}\lesssim_A \|\mathcal{N}\psi\|_{H^{s+m-1}(\Gamma)}+\|\Gamma\|_{H^{s+m+r}}\sup_{j>0}2^{-jr}\|\phi_j^1\|_{L^{\infty}(\Gamma)}+\sup_{j>0}2^{j(s-1+m-\epsilon)}\|\phi_j^2\|_{H^{\epsilon}(\Gamma)},
\end{equation}
where $\mathcal{N}\psi:=\phi_j^1+\phi_j^2$ is the same partition of $\mathcal{N}\psi$ as in the proof of \Cref{ellipticity}. Applying the inductive hypothesis again to the first term on the right and arguing the same way as in \Cref{ellipticity} to control the latter two terms in favour of $\psi$, $\psi_j^1$ and $\psi_j^2$ concludes the proof of the first estimate. To obtain the latter estimate, we proceed in a similar way as above. For the case $m=1$, we can use \Cref{Hbounds} to control $\|\mathcal{H}\mathcal{N}\psi\|_{H^{s+\frac{1}{2}}(\Omega)}$ by the right-hand side of (\ref{powers2}). Then one concludes the bound for all $m\geq 1$ by induction as above.
\end{proof}
Next, we note a bound for the operator $\nabla^{\top}$ which follows from similar reasoning to the above.
\begin{proposition}\label{tangradientbound}
Let $s\geq \frac{1}{2}$, $r\geq 0$, $\alpha\in [0,1)$ and $\beta\in [0,\frac{1}{2})$. Then
\begin{equation}
\|\nabla^{\top}\psi\|_{H^s(\Gamma)}\lesssim_A \|\psi\|_{H^{s+1}(\Gamma)}+\|\Gamma\|_{H^{s+1+r}}\sup_{j>0}2^{-j(r-1+\alpha)}\|\psi_j^1\|_{C^{\alpha}(\Gamma)}+\sup_{j>0}2^{j(s+\frac{1}{2}+\beta-\epsilon)}\|\psi_j^2\|_{H^{\frac{1}{2}-\beta}(\Gamma)}    
\end{equation}
for any sequence of partitions $\psi=\psi_j^1+\psi_j^2$.
\end{proposition}
\begin{proof}
By writing 
\begin{equation*}
\nabla^{\top}\psi=\nabla\mathcal{H}\psi-n_\Gamma\mathcal{N}\psi,    
\end{equation*}
the proof follows essentially the same line of reasoning as the proofs of \Cref{L1bound} and \Cref{DNpower1}. We omit the details.    
\end{proof}
Finally, we note a bound for $\mathcal{N}^m\nabla_n$ which will be needed frequently in the higher energy bounds.
\begin{corollary}\label{EEcorollary} Let $\alpha,\beta\in [0,1]$, $s\geq \frac{1}{2}$ and $r\geq 0$. We have
\begin{equation*}
\|\mathcal{N}^m\nabla_nv\|_{H^{s}(\Gamma)}\lesssim_A \|v\|_{H^{s+m+\frac{3}{2}}(\Omega)}+\|\Gamma\|_{H^{s+1+m+r}}\sup_{j>0}2^{-j(r+\alpha-1)}\|v_j^1\|_{C^{\alpha}(\Omega)}+\sup_{j>0} 2^{j(s+\beta+\frac{1}{2}+m-\epsilon)}\|v_j^2\|_{H^{1-\beta}(\Omega)}
\end{equation*}
where $v=v_j^1+v_j^2$ is any sequence of partitions of $v$.
\end{corollary}
\begin{proof}
We omit most of the details. The proof proceeds by first using \Cref{higherpowers} with the partition $\nabla_nv=n_\Gamma \cdot w^1_{j|\Gamma}+n_\Gamma\cdot w_{j|\Gamma}^2$ in $L^{\infty}(\Gamma)+H^{\epsilon}(\Gamma)$ where $w_j^1$ and $w_j^2$ are as in the proof of \Cref{L1bound} and then using  \Cref{L1bound} to estimate $\nabla_n v$ in $H^{s+m}$.    
\end{proof}

\subsection{Moving surface identities}\label{Movingsurfid} In this section, we suppose that $\Omega_t$ is a one parameter family of domains with boundaries $\Gamma_t\in\Lambda_*$ which flow with a velocity vector field $v$ that is not necessarily divergence free. Our purpose is to collect various identities and commutator estimates involving the material derivative $D_t:=\partial_t+v\cdot\nabla$ and functions on $\Gamma_t$. We begin by recalling several algebraic identities, many of which were proven in \cite{sz}.
\begin{enumerate}
    \item (Material derivative of the normal). \begin{equation}\label{Moving normal}
    D_tn_{\Gamma_t}=-\left((\nabla v)^*(n_{\Gamma_{t}})\right)^\top.
\end{equation}
\item (Leibniz rule for $\mathcal{N}$).
\begin{equation}\label{DNLeibniz}
    \mathcal{N}(fg)=f\mathcal{N}g+g\mathcal{N}f-2\nabla_n\Delta^{-1}(\nabla \mathcal{H}f\cdot \nabla \mathcal{H}g).
\end{equation}
\item(Commutator with $\nabla$).
\begin{equation}\label{Dt commutator grad}
    [D_t,\nabla]g=-(\nabla v)^*(\nabla g).
\end{equation}
\item(Commutator with $\Delta^{-1}$).
\begin{equation}\label{Dt commutator laplacian-1}
    [D_t,\Delta^{-1}]g=\Delta^{-1}\left(2\nabla v\cdot \nabla^2\Delta^{-1}g+\Delta v\cdot\nabla\Delta^{-1}g\right).
\end{equation}
\item(Commutator with $\mathcal{H}$).
\begin{equation}\label{H-commutator}
\begin{split}
S_0f:=[D_t,\mathcal{H}]f=\Delta^{-1}(2\nabla v\cdot \nabla^2\mathcal{H}f+\nabla\mathcal{H}f\cdot\Delta v).
\end{split}
\end{equation}
\item(Commutator with $\mathcal{N}$).
\begin{equation}\label{N-commutator}
\begin{split}
S_1f:=[D_t,\mathcal{N}]f=D_t n_{\Gamma_t}\cdot \nabla \mathcal{H}f-n_{\Gamma_t}\cdot((\nabla v)^*(\nabla \mathcal{H}f))+n_{\Gamma_t}\cdot\nabla ([D_t,\mathcal{H}]f).
\end{split}    
\end{equation}
\end{enumerate}
We also have the general Leibniz type formula,
\begin{equation}\label{Leibniz on moving hypersurfaces general}
\frac{d}{dt}\int_{\Gamma_{t}}fdS=\int_{\Gamma_{t}}D_{t}f+f(\mathcal{D}\cdot v^{\top}-\kappa v^{\perp})\,dS,
\end{equation} 
where $\mathcal{D}$ is the covariant derivative. 
\subsubsection{Balanced commutator estimates}
Using the above identities, we now establish refined estimates for commutators involving $D_t$ and the Dirichlet-to-Neumann operator. If we assume that $v$ is divergence free, it is a straightforward calculation to verify that $S_0\psi$ can be rewritten in the form
\begin{equation}\label{Sid}
 S_0\psi=\Delta^{-1}\nabla\cdot\mathcal{B}(\nabla v,\nabla \mathcal{H}\psi)   ,
\end{equation}
where $\mathcal{B}$ is an $\mathbb{R}^d$-valued bilinear form.  Using \eqref{N-commutator}, we can write the commutator $[D_t,\mathcal{N}]$ as follows:
\begin{equation*}
S_1\psi:=[D_t,\mathcal{N}]\psi=\nabla_n S_0\psi-\nabla\mathcal{H}\psi\cdot (\nabla_n v)-\nabla^{\top}\psi\cdot\nabla v\cdot n_{\Gamma_t}    .
\end{equation*}
In the higher energy bounds, we will need an estimate for higher order commutators $S_k$, given by
\begin{equation}\label{commutatorexpansion}
S_k\psi:=[D_t,\mathcal{N}^k]\psi=\sum_{l+m=k-1}\mathcal{N}^l[D_t,\mathcal{N}]\mathcal{N}^{m}\psi,  
\end{equation}
where $l,m$ are non-negative integers and $k\in\mathbb{N}$. From now on, let us  write $A=\|v\|_{C^{\frac{1}{2}+\epsilon}(\Omega)}+\|\Gamma\|_{C^{1,\epsilon}}$. For $s\geq \frac{1}{2}$, we have the following refined estimates for $S_k$ when $v$ is divergence free, which will be useful for estimating $S_kD_ta$ and $S_ka$, respectively, in the higher energy bounds.
\begin{proposition}\label{materialcom}
 Suppose that the flow velocity $v$ is divergence free and let $s\geq \frac{1}{2}$, $k\geq 1$. Then we have the following bounds for $S_k$. 
 \begin{enumerate}
 \item (Variant 1). For any sequence of partitions $\psi=\psi_j^1+\psi_j^2$, there holds
 \begin{equation*}\label{RHScom}
 \begin{split}
 \|S_k\psi\|_{H^s(\Gamma)}\lesssim_A &\|v\|_{W^{1,\infty}(\Omega)}\|\psi\|_{H^{s+k}(\Gamma)}+\|v\|_{H^{s+\frac{3}{2}+k}(\Omega)}\|\psi\|_{L^{\infty}(\Gamma)}+\|\Gamma\|_{H^{s+\frac{3}{2}+k}}\|\psi\|_{L^{\infty}(\Gamma)}
 \\
 +&\|v\|_{W^{1,\infty}(\Omega)}\|\Gamma\|_{H^{s+k+\frac{3}{2}}}\sup_{j>0}2^{-\frac{j}{2}}\|\psi_j^1\|_{L^{\infty}(\Gamma)}+\|v\|_{W^{1,\infty}(\Omega)}\sup_{j>0} 2^{j(s+k-\epsilon)}\|\psi_j^2\|_{H^{\epsilon}(\Gamma)}.
 \end{split}
 \end{equation*}
 \item (Variant 2).
 \begin{equation*}
 \begin{split}
 \|S_k\psi\|_{H^s(\Gamma)}\lesssim_A &\|v\|_{W^{1,\infty}(\Omega)}\|\psi\|_{H^{s+k}(\Gamma)}+\|\Gamma\|_{H^{s+k+1}}(\|\psi\|_{C^{\frac{1}{2}}(\Gamma)}+\|v\|_{W^{1,\infty}(\Omega)}\|\psi\|_{L^{\infty}(\Gamma)})
 \\
 +&\|v\|_{H^{s+k+1}(\Omega)}\|\psi\|_{C^{\frac{1}{2}}(\Gamma)}.
 \end{split}
 \end{equation*}
 \end{enumerate}
\end{proposition}
\begin{proof}
We will focus on the first estimate as the second one is similar. From (\ref{commutatorexpansion}), we need to prove the estimate in (\ref{RHScom}) with the left-hand side replaced with $\mathcal{N}^l[D_t,\mathcal{N}]\mathcal{N}^m\psi$ where $l+m=k-1$. We will focus first on the term $\mathcal{N}^l(\nabla_n S_0\mathcal{N}^m\psi)$ which is the most difficult to deal with. Let us write $G:=\mathcal{B}(\nabla v,\nabla \mathcal{H}\mathcal{N}^m\psi)$ for notational convenience. We begin by applying \Cref{EEcorollary} and then \Cref{direst} to obtain (using the identity (\ref{Sid})),
\begin{equation*}
\begin{split}
\|\mathcal{N}^l(\nabla_n S_0\mathcal{N}^m\psi)\|_{H^s(\Gamma)}&\lesssim_A \|G\|_{H^{s+l+\frac{1}{2}}(\Omega)}+\|\Gamma\|_{H^{s+\frac{3}{2}+k}}\sup_{j>0}2^{-j(m+\frac{3}{2})}\|\Delta^{-1}\nabla\cdot G_j^1\|_{W^{1,\infty}(\Omega)}
\\
&\quad +\sup_{j>0} 2^{j(s+l+\frac{1}{2}-\epsilon)}\|\Delta^{-1}\nabla\cdot G_j^2\|_{H^1(\Omega)}  ,
\end{split}
\end{equation*}
where $G=G_j^1+G_j^2$ is a partition of $G$ defined by taking $G_j^1=\mathcal{B}(\nabla P_{<j}v,\nabla P_{<j}\mathcal{H}\mathcal{N}^m_{<j}\psi)$, where $\mathcal{N}_{<j}:=\nabla_n P_{<j}\mathcal{H}$. Using the $C^{1,\epsilon}$ estimate for $\Delta^{-1}$ and the maximum principle for $\mathcal{H}$, it is straightforward to control
\begin{equation*}
2^{-j(m+\frac{3}{2})}\|\Delta^{-1}\nabla\cdot G_j^1\|_{W^{1,\infty}(\Omega)}\lesssim_A \|v\|_{C^{\frac{1}{2}+\epsilon}(\Omega)}\|\psi\|_{L^{\infty}(\Gamma)}\lesssim_A \|\psi\|_{L^{\infty}(\Gamma)}.   
\end{equation*}
Moreover, using the $H^{-1}\to H^1_0$ estimate for $\Delta^{-1}$, we can control the other term by
\begin{equation*}
\begin{split}
2^{j(s+l+\frac{1}{2}-\epsilon)}\|\Delta^{-1}\nabla\cdot G_j^2\|_{H^1(\Omega)}\lesssim_A 2^{j(s+l+\frac{1}{2}-\epsilon)}&\|v\|_{W^{1,\infty}(\Omega)}\|\nabla P_{<j}\mathcal{H}\mathcal{N}_{<j}^m\psi-\nabla \mathcal{H}\mathcal{N}^m\psi\|_{L^2(\Omega)} 
\\
+&\|v\|_{H^{s+\frac{3}{2}+k}(\Omega)}\|\psi\|_{L^{\infty}(\Gamma)}.
\end{split}
\end{equation*}
Finally, it is straightforward (albeit somewhat technical) to verify that the terms on the right-hand side above can be controlled by the right-hand side of (\ref{RHScom}) using the $H^{\epsilon}\to H^{\frac{1}{2}+\epsilon}$ bound (\ref{harmonicbase}), \Cref{higherpowers}, \Cref{boundaryest} with $g_j^2=g$ (and the fact that $\|n_\Gamma\|_{C^\epsilon(\Gamma)}\lesssim_A1$) as well as the $H^{\frac{1}{2}+\epsilon}\to H^\epsilon$ trace estimates. Now, we turn to estimating $\|G\|_{H^{s+l+\frac{1}{2}}(\Omega)}$. By performing a paradifferential expansion as in \Cref{productestref}, it is easy to see that
\begin{equation*}
\|G\|_{H^{s+1+\frac{1}{2}}(\Omega)}\lesssim_A \|v\|_{W^{1,\infty}(\Omega)}\|\mathcal{H}\mathcal{N}^{m}\psi\|_{H^{s+l+\frac{3}{2}}(\Omega)}+\|T_{\nabla \mathcal{H}\mathcal{N}^{m}\psi}\nabla v\|_{H^{s+l+\frac{1}{2}}(\Omega)}.    
\end{equation*}
Using \Cref{Hbounds} and \Cref{higherpowers}, the first term on the right can be controlled by the right-hand side of (\ref{RHScom}). For the latter term, we need to control the $l^2$ sum of
\begin{equation*}
2^{j(s+l+\frac{1}{2})}\|P_j\nabla vP_{<j-4}\nabla\mathcal{H}\mathcal{N}^m\psi\|_{L^2(\Omega)}    .
\end{equation*}
For this, we estimate
\begin{equation*}
\begin{split}
2^{j(s+l+\frac{1}{2})}\|P_j\nabla vP_{<j-4}\nabla\mathcal{H}\mathcal{N}^m\psi\|_{L^2(\Omega)}&\lesssim_A 2^{j(s+k+\frac{1}{2})}\|P_j\nabla v\|_{L^2(\Omega)}\|\psi\|_{L^{\infty}(\Gamma)}
\\
&\quad +2^{j(s+l+\frac{1}{2})}\|v\|_{W^{1,\infty}(\Omega)}\|P_{<j-4}\nabla\mathcal{H}(\mathcal{N}^m-\mathcal{N}^m_{<j})\psi\|_{L^2(\Omega)}. 
\end{split}
\end{equation*}
The first term on the right  when summed in $l^2$ is controlled by the right-hand side of (\ref{RHScom}). The same is true for the latter term after making use of \eqref{harmonicbase} and \Cref{higherpowers}. This concludes the full estimate for $\mathcal{N}^l(\nabla_n S_0\mathcal{N}^m\psi)$. The other terms in $\mathcal{N}^l[D_t,\mathcal{N}]\mathcal{N}^m\psi$ are dealt with similarly.
\end{proof}
\section{Regularization operators}\label{SSRO}
Let $\Omega_*$ be a smooth, bounded domain with boundary $\Gamma_*$. In the following, we let $\Omega$ be a bounded domain with  boundary $\Gamma\in \Lambda(\Gamma_*,\epsilon,\delta)$ where $\epsilon>0$ and $\delta>0$ are small positive constants. As usual, we will abbreviate the above set of hypersurfaces by $\Lambda_*$ and consider the volume of the associated domains as part of our implicit constants. We recall from \eqref{PhiGamma} that we have the diffeomorphism from $\Gamma_*$ to $\Gamma$ given by
\begin{equation*}
    \Phi_\Gamma(x)=x+\eta_\Gamma(x)\nu(x)
 \end{equation*}
which parameterizes $\Gamma$ as a graph over $\Gamma_*$. When constructing solutions to the free boundary Euler equations (and also when proving refined energy estimates), it will be important to have a good regularization operator at each dyadic scale which preserves divergence free functions. More precisely, beyond the obvious regularization properties (to be outlined below in more detail), our operators will need to have the following properties.
\begin{enumerate}
\item\label{EP} (Extension property). There is a $\delta_0>0$ such that the following holds: If $\Omega_j$ is a domain containing $\Omega$ with boundary $\Gamma_j\in\Lambda_*$ such that $\|\dist(x,\Omega)\|_{L^{\infty}(\Omega_j)}<\delta_0 2^{-j}$ then there is an associated regularization  $\Psi_{\leq j}v$ at the dyadic scale $2^j$, defined on $\Omega_j$.
\item (Regularization is divergence free). Given $\Omega_j$ as above, the regularization $\Psi_{\leq j}v$ satisfies $\nabla\cdot \Psi_{\leq j}v=0$ on $\Omega_j$. Here, $v$ is a divergence free function on $\Omega$.
\end{enumerate}
\begin{remark}
The first point will be convenient later for comparing velocities defined on different domains, which are sufficiently close. The second point is important as our regularization operators will not necessarily commute with derivatives (but will commute with derivatives up to lower order terms). 
\end{remark}
A more precise description of the above regularization operators is given by the following proposition.
\begin{proposition}\label{c reg bounds}
    Fix $\alpha_0$, let $v$, $\Omega$ and $\Omega_j$ be as above and let $A=\|\Gamma\|_{C^{1,\epsilon}}$. Then there exists a regularization operator $\Psi_{\leq j}$ which is bounded from $H^s_{div}(\Omega)\to H^s_{div}(\Omega_{j})$ for every $s\geq 0$ with the following properties.
    \begin{enumerate}
        \item (Regularization bounds).
        \begin{equation*}
        \|\Psi_{\leq j}v\|_{H^{s+\alpha}(\Omega_{j})} \lesssim_{A} 2^{j\alpha }\|v\|_{H^s(\Omega)},\hspace{16mm}0\leq\alpha.
        \end{equation*}
        \item (Difference bounds).
        \begin{equation*}
        \|(\Psi_{\leq j+1}-\Psi_{\leq j})v\|_{H^{s-\alpha}(\Omega_{j+1})}\lesssim_{A} 2^{-j\alpha}\|v\|_{H^s(\Omega)},\hspace{7mm} 0\leq\alpha\leq \min\{s,\alpha_0\}.
        \end{equation*}
        \item (Error bounds).
        \begin{equation*}
        \|(I-\Psi_{\leq j})v\|_{H^{s-\alpha}(\Omega)}\lesssim_{A} 2^{-j\alpha}\|v\|_{H^s(\Omega)},\hspace{15mm} 0\leq\alpha\leq \min\{s,\alpha_0\}.
        \end{equation*}
    \end{enumerate}
\end{proposition}
\begin{proof}
We begin with a preliminary step of constructing a regularization operator $\Phi_{\leq j}$ with the above three properties which maps $H^s(\Omega)$ to $H^s(\tilde{\Omega}_{j})$ where $\tilde{\Omega}_{j}$ is a neighborhood of $\Omega_{j}$, but does not necessarily preserve divergence free functions. To do this, we aim to construct a suitable kernel $K^{j}$ such that
\begin{equation*}
\Phi_{\leq j}v(x)=\int_\Omega K^{j}(x,y)v(y)\, dy.
\end{equation*}
Here, the kernel $K^j(x,y)$ is of the form
\begin{equation*}\label{Form of kern}
    K^j(x,y)=\sum_{k=0}^n K_k^j(x,y)\chi_k(x),
\end{equation*}
where $(\chi_k)_{k=0}^n$ is a partition of unity of a neighborhood of $\Omega$, obtained by selecting an open cover $\{U_k\}_{k=0}^n$ so that there are  vectors $(e_k)_{k=1}^n$ all of the same length with $e_k$  outward oriented and uniformly transversal to $\Gamma\cap U_k$. The remaining set $U_0$ is then chosen to cover the portion  of $\Omega$ away from the boundary.
Let $e_0=0$ and take $e_k$ with $k\in \{1,\dots, n\}$ as above. Such a smooth partition of unity can be constructed with bounds depending only on the properties of $\Lambda_*$. To construct $K^{j}$ we consider a smooth bump function $\phi_k$ with the following properties:
\begin{enumerate}
    \item The support of $\phi_k$ satisfies supp$\phi_k\subseteq B(e_k,\delta_1)$, $\delta_1\ll 1.$
    \item The average of $\phi_k$ is $1$, i.e., $\int_{\mathbb{R}^d}\phi_k(z)\,dz=1$.
    \item $\phi_k$ has zero moments up to some sufficiently large order $N$, i.e., $\int_{\mathbb{R}^d} z^\alpha \phi_k(z)dz=0$, $1\leq |\alpha|\leq N.$
\end{enumerate}
Then, for each $j>0$, we consider a  regularizing kernel
\begin{equation*}
  K_{0,k}^j(z):=2^{jd}\phi_k(2^j z).
\end{equation*}
We then define $K_k^j(x,y):=K_{0,k}^j(x-y)$ for  $y\in \Omega$. Note that for fixed $x\in U_k$, $K_k^j(x,y)$ is non-zero  only if $2^j(x-y)\in B(e_k,\delta_1)$, i.e.,  $y$ is within distance $2^{-j}\delta_1$ of $x-2^{-j} e_k$. This is what will allow us to view our kernel $K^j$ not only for $x\in \Omega$ but also for $x$ in a $\mathcal{O}(2^{-j})$ enlargement of $\Omega$. With this in mind, one can check that  the family of kernels $K^j$ satisfy the following:
\begin{enumerate}
    \item $K^j:\tilde{\Omega}_{j}\times \Omega\to \mathbb{R}$, where $\tilde{\Omega}_{j}:=\{x\in \mathbb{R}^d : d(x,\Omega)\leq c 2^{-j}\}$ with a small universal constant $c$.
    \item $|\partial_x^\alpha\partial_{y}^\beta K^j(x,y)|\lesssim 2^{j(d+|\alpha|+|\beta|)}$,  for multi-indices $\alpha,\beta$.
    \item $\int_\Omega K^j(x,y)\,dy=1$.
    \item $\int_\Omega K^j(x,y)(x-y)^\alpha \,dy=0$, $1\leq |\alpha|\leq N.$
\end{enumerate}
From the definition of $K^{j}$, we see that $\Phi_{\leq j}v$ is defined on a neighborhood of $\Omega_{j}$ if $\delta_0$ from property (\ref{EP}) above is small enough. It is then a straightforward matter to verify that $\Phi_{\leq j}$ satisfies the regularization, difference and error bounds in \Cref{c reg bounds} when $s$ and $\alpha$ are integers (the latter two bounds requiring the moment conditions, with $N=N(\alpha_0)$). The general bound follows by interpolation. 
\\

It remains to construct the regularization operator $\Psi_{\leq j}$ which preserves divergence free functions.  We first note that without loss of generality  we may assume that $\Gamma_j\in \Lambda_*$ with the regularization bound
    \begin{equation}\label{rreg}
\|\Gamma_{j}\|_{C^{k,\beta}}\lesssim_{A,k,\beta} 2^{j(\beta+k-1-\epsilon)}
    \end{equation}
for each integer $k\geq 1$ and real number $0\leq \beta<1$. Indeed, for large enough $j$, by working in local coordinates and using standard mollification techniques we can use the uniform $C^{1,\epsilon}$ regularity of $\eta_{\Gamma}$ to construct a surface $\tilde{\Gamma}_j\in \Lambda_*$ with the bounds (\ref{rreg}) such that $\tilde{\Gamma}_j$ is within distance  $\lesssim_A 2^{-j(1+\epsilon)}$ of $\Gamma$. For some small $c>0$, we can then define a surface $\Gamma_j$ via the parameterization $\eta_{\Gamma_j}:=\eta_{\tilde{\Gamma}_j}+c2^{-j}$. This defines a domain whose boundary has the required regularization bound and which, if $\delta_0$ is small enough, contains all domains within a $\delta_02^{-j}$ neighborhood of $\Omega$. Therefore, it suffices to construct $\Psi_{\leq j}$ in the case when $\Gamma_j$ satisfies (\ref{rreg}). We make this assumption for the remainder of the construction. 
\\

Next, we correct $\Phi_{\leq j}v$ by a gradient potential. We define for $v\in H_{div}^s(\Omega)$,
\begin{equation*}\label{psidef}
\Psi_{\leq j}v:=\Phi_{\leq j}v-\nabla\Delta_{\Omega_{j}}^{-1}(\nabla\cdot \Phi_{\leq j}v),
\end{equation*}
where $\Delta_{\Omega_{j}}^{-1}$ is the solution operator for the Dirichlet problem with zero boundary data associated to the domain $\Omega_{j}$.
\\
\\
To prove the regularization bounds for $\Psi_{\leq j}$, we note that because $v$ is divergence free, we have
\begin{equation}\label{divformula}
\begin{split}
\nabla\cdot \Phi_{\leq j}v(x)&=\sum_{k=0}^n\int \phi_k(y)\nabla\chi_k(x)\cdot (v(x-2^{-j} y)-v(x))\,dy
. 
\end{split}
\end{equation}
In other words, no derivatives fall on $v$ or the kernel when taking the divergence. From the above formula, one can easily verify  the following bounds for $\nabla\cdot \Phi_{\leq j}v$ for every $s_1,s_2\geq 0$:
\begin{equation*}\label{divbound}
\|\nabla\cdot \Phi_{\leq j}v\|_{H^{s_1}(\Omega_{j})}\lesssim_A 2^{-js_2}\|v\|_{H^{s_1+s_2}(\Omega)}.
\end{equation*}
 To establish the regularization property of $\Psi_{\leq j}$, we use this and (\ref{rreg})  together with the balanced Dirichlet estimate \Cref{direst2} to obtain
\begin{equation*}
\|\nabla\Delta_{\Omega_{j}}^{-1}(\nabla\cdot \Phi_{\leq j}v)\|_{H^{s+\alpha}(\Omega_j)}\lesssim_A 2^{j\alpha}\|v\|_{H^{s}(\Omega)}.
\end{equation*}
Therefore, the regularization bound $\|\Psi_{\leq j}v\|_{H^{s+\alpha}(\Omega_{j})}\lesssim_A 2^{j\alpha}\|v\|_{H^s(\Omega)}$ follows immediately. The bounds for $\Psi_{\leq j+1}v-\Psi_{\leq j}v$ and $I-\Psi_{\leq j}v$ are analogous.
\end{proof}
Finally, we note the pointwise analogues of the above estimates.
\begin{proposition}\label{pointwisereg} Given the assumptions of \Cref{c reg bounds}, the regularization operator $\Psi_{\leq j}$ satisfies the following pointwise bounds for $0\leq\alpha<2$:
\begin{equation*}
\|\Psi_{\leq j}v\|_{C^{\alpha}(\Omega_j)}\lesssim_A 2^{j\beta}\|v\|_{C^{\alpha-\beta}(\Omega)},  
\end{equation*}
for $0\leq\beta\leq\alpha$, and
\begin{equation*}
\|(I-\Psi_{\leq j})v\|_{C^{\alpha}(\Omega)}+\|(\Psi_{\leq j+1}-\Psi_{\leq j})v\|_{C^{\alpha}(\Omega_{j+1})}\lesssim_A 2^{-j\beta}\|v\|_{C^{\alpha+\beta}(\Omega)} ,
\end{equation*}
for $\beta\geq 0$.
\end{proposition}
\begin{proof}
The corresponding bounds for $\Phi_{\leq j}$ are straightforward to directly verify. To estimate the gradient correction, we again may assume without loss of generality the bound (\ref{rreg}) and then use the pointwise estimates from \Cref{C2est} and \Cref{Gilbarg}.
\end{proof}
\subsection{Frequency envelopes}
Let $\Gamma\in\Lambda_*$ and let $s>\frac{d}{2}+1$. Suppose that $v\in H^s(\Omega)$ and suppose that $\Gamma\in H^s$ is parameterized in collar coordinates by $x\mapsto x+\eta_\Gamma(x)\nu(x)$. At this point, we define $A:=\|\Gamma\|_{C^{1,\epsilon}}+\|v\|_{C^{\frac{1}{2}}(\Omega)}$. Using the extension operator from \Cref{continuosext}, we have the following Littlewood-Paley decomposition for a function $v$ defined on $\Omega$:
\begin{equation*}
v=\sum_{j\geq 0}P_jv    ,
\end{equation*}
where by abuse of notation $P_jv$ is interpreted to mean $P_jE_{\Omega}v$ where $E_{\Omega}$ is as in \Cref{continuosext} and $P_0$ is to be interpreted as $P_{\leq 0}$. We also have a corresponding Littlewood-Paley type decomposition for functions on $\Gamma_*$. Indeed, denote by $\langle D\rangle_*:=(I-\Delta_{\Gamma_*})^{\frac{1}{2}}$. For functions on $\Gamma_*$, we then write for $j>0$, $P_j:=\varphi (2^{-j}\langle D\rangle_*)-\varphi (2^{-j+1}\langle D\rangle_*)$ and $P_0:=\varphi(\langle D\rangle_*)$ where $\varphi:\mathbb{R}\to\mathbb{R}$ with $\varphi =1$ on the unit ball and with support in $B_2(0)$. We then have from \Cref{continuosext} the almost orthogonality
\begin{equation*}\label{LWP}
\|(v,\Gamma)\|_{\textbf{H}^s}^2\approx_A \sum_{j\geq 0}2^{2js}\left(\|P_jv\|_{L^2(\mathbb{R}^d)}^2+\|P_j\eta_{\Gamma}\|_{L^2(\Gamma_*)}^2\right).  
\end{equation*}
The above equivalence will allow us to define $\mathbf{H}^s$ frequency envelopes for  states $(v,\Gamma)\in\mathbf{H}^s$ with the $l^2$ decay required to establish our continuous dependence result as well as the continuity of solutions with values in $\mathbf{H}^s$ later on.
\begin{remark}
To define the Littlewood-Paley decomposition above, we use the extension $E_{\Omega}$ from \Cref{continuosext} (as opposed to, e.g., the Stein extension) because of its transparent continuous dependence on the domain. This will be important for establishing continuous dependence of solutions to the free boundary Euler equations with respect to the data when we have to compare frequency envelopes for different initial data.     
\end{remark}
\begin{definition}[Frequency envelopes] Let $s> \frac{d}{2}+1$, $\Gamma\in\Lambda_*$ and $(v,\Gamma)\in\mathbf{H}^s$. An \emph{$\mathbf{H}^s$ frequency envelope} for the pair $(v,\Gamma)$ is a positive sequence $c_j$ such that for each $j\geq 0$,
\begin{equation*}
\|P_jv\|_{H^s(\mathbb{R}^d)}+\|P_j\eta_{\Gamma}\|_{H^s(\Gamma_*)}\lesssim_A c_j\|(v,\Gamma)\|_{\mathbf{H}^s},\hspace{10mm}\|c_j\|_{l^2}\lesssim_A 1.    
\end{equation*}
We say that the sequence $(c_j)_j$ is admissible if $c_0\approx_A 1$ and it is slowly varying,
\begin{equation*}
c_j\leq 2^{\delta |j-k|}c_k,\hspace{10mm} j,k\geq 0,\hspace{10mm}0<\delta\ll 1.    
\end{equation*}
We can always define an admissible frequency envelope by the formula
\begin{equation}\label{admissable}
c_j=2^{-\delta j}+(1+\|(v,\Gamma)\|_{\mathbf{H}^{s}})^{-1}\max_k 2^{-\delta |j-k|}\left(\|P_kv\|_{H^s(\mathbb{R}^d)}+\|P_k\eta_{\Gamma}\|_{H^s(\Gamma_*)}\right).
\end{equation}
Unless otherwise stated, we will take this as our formula for $c_j$. The following proposition will be useful in our construction of rough solutions later on as well as for proving continuity of the data-to-solution map.
\end{definition}
\begin{proposition}\label{envbounds}
Let $\Gamma\in\Lambda_*$ and let $s>\frac{d}{2}+1$. Suppose that $(v,\Gamma)\in\mathbf{H}^s$ and let $(c_j)_j$ be its associated admissible frequency envelope. Then there exists a family of regularized domains $\Omega_j$ with boundaries $\Gamma_j\in\Lambda_*$ and $\Gamma_j\in H^s$ along with associated divergence free regularizations $v_j:=\Psi_{\leq j}v$ defined on a $2^{-j}$ enlargement of $\Omega_j\cup\Omega$  such that the following holds.
\begin{enumerate}
    \item\label{iGPA} (Good pointwise approximation).
    \begin{equation*}
 (v_j,\Gamma_j)\to (v,\Gamma)\hspace{5mm}in\hspace{2mm} C^{1}\times C^{1,\frac{1}{2}}\hspace{5mm}as\hspace{2mm}j\to\infty.       
    \end{equation*}
\item\label{2UB} (Uniform bound).
\begin{equation*}
\|(v_j,\Gamma_j)\|_{\mathbf{H}^s}\lesssim_A \|(v,\Gamma)\|_{\mathbf{H}^s}.    
\end{equation*}
\item\label{3HR} (Higher regularity).
\begin{equation*}
 \|(v_j,\Gamma_j)\|_{\mathbf{H}^{s+\alpha}}\lesssim_A 2^{j\alpha}c_j\|(v,\Gamma)\|_{\mathbf{H}^s},\hspace{5mm}\alpha>1.  
\end{equation*}
\item\label{4LFDB} (Low frequency difference bounds). On a $2^{-j}$ enlarged neighborhood of $\Omega_j\cup\Omega_{j+1}$, there holds
\begin{equation*}
\|(v_j,\eta_{\Gamma_j})-(v_{j+1},\eta_{\Gamma_{j+1}})\|_{L^2\times L^2}\lesssim_A 2^{-js}c_j\|(v,\Gamma)\|_{\mathbf{H}^s}.    
\end{equation*}
\end{enumerate}
\end{proposition}
\begin{proof}
We define $\Gamma_j$ by the graph parameterization $\eta_{\Gamma_j}=P_{\leq j}\eta_{\Gamma}$ (using the projections defined above). By Sobolev embedding, we have $|\eta_{\Gamma_j}-\eta_{\Gamma}|\lesssim 2^{-\frac{3}{2}j}$, and so the existence of the required divergence free regularization $v_j:=\Psi_{\leq j}v$ comes from \Cref{c reg bounds}.
\\

Next, we turn to verifying the above four properties. We focus on the bounds for $v_j$ as the bounds for $\Gamma_j$ are similar (and simpler). Properties (\ref{iGPA}) and (\ref{2UB}) are clear from Sobolev embedding and \Cref{c reg bounds}. Next, we turn to property (\ref{3HR}). We begin by establishing this property for $\Phi_{\leq j}v$ and then we will upgrade  to the full divergence free regularization $v_j=\Psi_{\leq j}v$. We write $w^l$ as shorthand for $P_lw$ and begin by splitting
\begin{equation*}
\|\Phi_{\leq j}v\|_{H^{s+\alpha}}\leq \sum_{l\leq j}\|\Phi_{\leq j}v^l\|_{H^{s+\alpha}}+\sum_{l>j}\|\Phi_{\leq j}v^l\|_{H^{s+\alpha}}.    
\end{equation*}
For $l\leq j$, we estimate
\begin{equation*}
\|\Phi_{\leq j}v^l\|_{H^{s+\alpha}}\lesssim_A \|v^l\|_{H^{s+\alpha}}\lesssim_A 2^{l\alpha}c_l\|(v,\Gamma)\|_{\mathbf{H}^s}\lesssim_A 2^{j\alpha}c_j2^{(\alpha-\delta)(l-j)}\|(v,\Gamma)\|_{\mathbf{H}^s}. 
\end{equation*}
For $l>j$, we estimate
\begin{equation*}
\|\Phi_{\leq j}v^l\|_{H^{s+\alpha}}\lesssim_A 2^{j(\alpha+s)}\|v^l\|_{L^2}\lesssim_A 2^{j\alpha}c_j 2^{(j-l)(s-\delta)}\|(v,\Gamma)\|_{\mathbf{H}^s}.    
\end{equation*}
Summing up each contribution gives
\begin{equation*}
\|\Phi_{\leq j}v\|_{H^{s+\alpha}}\lesssim_A 2^{j\alpha}c_j\|(v,\Gamma)\|_{\mathbf{H}^s}.    
\end{equation*}
To obtain the corresponding bound for $\Psi_{\leq j}$, we simply note that by \Cref{direst2}, 
\begin{equation*}
\|\nabla\Delta^{-1}\nabla\cdot \Phi_{\leq j}v\|_{H^{s+\alpha}}\lesssim_A \|\Phi_{\leq j}v\|_{H^{s+\alpha}}+2^{j(s+\alpha-\epsilon)}\|\nabla\cdot\Phi_{\leq j}v\|_{L^2}.    
\end{equation*}
By \eqref{divformula}, we have $2^{j(s+\alpha)}\|\nabla\cdot \Phi_{\leq j}v\|_{L^2}\lesssim_A 2^{j\alpha}\|v\|_{H^s}$. Therefore, if we choose $\delta$ in the definition of $c_j$ so that $2^{-j\epsilon}\leq c_j$, we have
\begin{equation*}
\|\Psi_{\leq j}v\|_{H^{s+\alpha}}\lesssim_A 2^{j\alpha}c_j\|(v,\Gamma)\|_{\mathbf{H}^s}.    
\end{equation*}
This establishes property (\ref{3HR}) for $\Psi_{\leq j}v$. The proof of property (\ref{4LFDB}) is similar except now one can use the difference and error bounds in \Cref{c reg bounds}. We omit the details.
\end{proof}
\section{Higher energy bounds}\label{HEB}
Let $k>\frac{d}{2}+1$ be an integer. Our aim in this section is to establish control of the $\mathbf{H}^k$ norm of $(v,\Gamma)$ in terms of the initial data where the growth of these norms is dictated by the pointwise control parameters $A$ and $B$ below. To accomplish this, we will first construct a coercive energy functional $(v,\Gamma)\mapsto E^k(v,\Gamma)$ associated to each integer $k>\frac{d}{2}+1$ and then we will prove energy estimates for $E^k(v,\Gamma)$ to obtain estimates for $\|(v,\Gamma)\|_{\mathbf{H}^k}$ when $(v,\Gamma)$ is a solution to the free boundary Euler equations. More precisely, we prove the following theorem.
\begin{theorem}\label{Energy est. thm}
Let $s\in\mathbb{R}$ with $s>\frac{d}{2}+1$ and let $k> \frac{d}{2}+1$ be an integer. Fix a collar neighborhood $\Lambda(\Gamma_*, \epsilon,\delta)$ with $\delta>0$ sufficiently small. Then for $\Gamma$ restricted to $\Lambda_*$ there exists an energy functional $(v,\Gamma)\mapsto E^k(v,\Gamma)$ such that
\begin{enumerate}
\item (Energy coercivity).
\begin{equation}\label{Coercivity bound on integers}
E^k(v,\Gamma)\approx_A 1+\|(v,\Gamma)\|_{\mathbf{H}^k}^2.
\end{equation}
\item  (Energy propagation). If, in addition to the above, $(v,\Gamma)=(v(t),\Gamma_t)$ is a solution to the free boundary Euler equations, then $E^k(t):=E^k(v(t),\Gamma_t)$ satisfies 
\begin{equation*}
\frac{d}{dt}E^k\lesssim_A B\log(1+\|(v,\Gamma)\|_{\mathbf{H}^s})E^k.
\end{equation*}
\end{enumerate}
Here, $A:=1+|\Omega|+\|v\|_{C^{\frac{1}{2}+\epsilon}(\Omega)}+\|\Gamma\|_{C^{1,\epsilon}}$ and $B:=1+\|v\|_{W^{1,\infty}(\Omega)}+\|\Gamma\|_{C^{1,\frac{1}{2}}}$.
\end{theorem}
By Gr\"onwall's inequality, this gives the single and double exponential bounds
\begin{equation*}\label{doubleexpbound}
\begin{split}
\|(v(t),\Gamma_t)\|^2_{\mathbf{H}^k}&\lesssim_A\exp\left(\int_{0}^{t}C_AB(s)\log(1+\|(v,\Gamma)\|_{\mathbf{H}^s})ds\right)(1+\|(v_0,\Gamma_0)\|_{\mathbf{H}^k}^2).
\\
\|(v(t),\Gamma_t)\|^2_{\mathbf{H}^k}&\lesssim_{A} 
 \exp\left(\log(C_A(1+\|(v_0,\Gamma_0)\|_{\mathbf{H}^k}^2))\exp{\int_{0}^{t}C_AB(s)\,ds}\right)  
\end{split}
\end{equation*}
for all integers $k>\frac{d}{2}+1$.
\begin{remark}
It is important to note that  the first part of \Cref{Energy est. thm} does not make any reference to the dynamical problem.
\end{remark}
\subsection{Constructing the energy functional}\label{CTEF}
Before establishing the above theorem, we  motivate our choice of energy. At this point, the discussion will be heuristic only. There are two quantities to control; namely, the $H^k$ norms of $v$ and $\Gamma$. However, these are coupled via the nonlinear evolution, so they must be measured in tandem. We achieve this by working instead with well-chosen \emph{good variables}, which are selected as follows:
\begin{enumerate}[label=\roman*)]
    \item The vorticity $\omega$.  If $v$ is a divergence free vector field on $\Omega$, then in Euclidean coordinates, we have the following relation for $\Delta v_i$:
\begin{equation*}
\Delta v_i=-\partial_j\omega_{ij},
\end{equation*}
where $\omega$ denotes the curl of $v$. Therefore, $v$ is controlled by $\omega$ and a suitable boundary value. However, it turns out to be simpler to view  $v$ as the solution to a div-curl  system, again 
with a boundary condition whose choice will be addressed shortly.
 
\item The Taylor coefficient $a$. This variable is used to describe the regularity of the boundary. Precisely, as we will see later, we have the approximate relation
\[
\mathcal N a \approx a \kappa 
\]
where $\kappa$ represents the mean curvature of $\Gamma$. Thus, as long as the Taylor sign condition remains satisfied, the $H^k$ norm of $\Gamma$ should be comparable at leading order to the $H^{k-1}$ norm of $a$.

\item The material derivative of the Taylor coefficient, $D_t a$. At leading order this provides information about $v$ via the approximate
paradifferential relation
\[
D_t a \approx \mathcal N T_n v, 
\]
for a suitable representation of the paraproduct above. This will provide the needed boundary condition for the div-curl system for $v$. 
\end{enumerate}

Thus, at the principal level we have the correspondence 
\[
v \leftrightarrow (\omega, D_t a), \qquad  \Gamma \leftrightarrow a,
\]
which will be the basis for our coercivity property. For the first part, it is better to think of $v$ as solving a div-curl system.
One might try to think of a rotational/irrotational decomposition  $v = v_{rot}+ v_{ir}$, where the two components solve div-curl systems as follows:
\[
\left\{
\begin{aligned}
&\curl v_{rot} = \omega, \\ & \nabla \cdot v_{rot} = 0, \\ & v_{rot} \cdot n_\Gamma = 0 \ \ \text{on } \Gamma ,  
\end{aligned}
\right.
\qquad 
\left\{
\begin{aligned}
& \curl v_{ir} = 0, \\ & \nabla \cdot v_{ir} = 0, \\ & v_{ir} \cdot n_\Gamma = v \cdot n_\Gamma \ \ \text{on } \Gamma . 
\end{aligned}
\right.
\]
Unfortunately, such a decomposition is not well-suited for our present problem, essentially due to the fact that in our setting 
$n_\Gamma$ has less regularity than $v$ on the free boundary; namely, $H^{k-1}$ versus $H^{k-\frac12}$. Hence, we cannot use such a decomposition directly, though  a paradifferential form of it will appear later in our existence proof. Instead, we will bypass this difficulty by associating the $D_t a$ variable with $\nabla^\top v \cdot n_\Gamma$, the normal component of the tangential derivatives on the boundary, which will then play the role of the boundary condition in the div-curl system for $v$. This, in turn, yields the $v$ part of the coercivity bound.
\\

Now we turn our attention to the dynamical side, which ultimately determines the choice of the good variables. There we separate the good variables differently, into the vorticity $\omega\in H^{k-1}(\Omega)$ on one hand, which will provide the interior component of the energy,  and the pair ($a$,$D_ta$) in $H^{k-1}(\Gamma)\times H^{k-\frac{3}{2}}(\Gamma)$, which carries the boundary component of the energy.
For the vorticity, this is immediately clear from the equation 
\begin{equation}\label{vorteq}
D_t\omega_{ij}=-\omega_{ik}\partial_jv_k+\omega_{jk}\partial_iv_k,
\end{equation}
which results from taking curl of \eqref{Euler}. Based on the transport structure of the vorticity, it is natural to include the quantity $\|\omega\|_{H^{k-1}(\Omega)}^2$ as a component of the energy. On the other hand, it turns out that $\|(a,D_ta)\|_{H^{k-1}(\Gamma)\times H^{k-\frac{3}{2}}(\Gamma)}^2$ can be controlled by the linearized energy $E_{lin}(w_k,s_k)$, where $s_k$ and $w_k$ solve the linearized equation to leading order with
\begin{equation*}
\begin{cases}
&w_k=\nabla\mathcal{H}\mathcal{N}^{k-2}D_ta,
\\
&s_k\hspace{1mm}=\mathcal{N}^{k-1}a.
\end{cases}
\end{equation*}
The derivation for this is a bit more involved than for the vorticity and will be handled later.
\\
\\
With the above discussion in mind, we define our energy as follows:
\begin{equation}\label{ENERGYE1}
E^k(v,\Gamma):=1+\|v\|_{L^2(\Omega)}^2+\|\omega\|_{H^{k-1}(\Omega)}^2+E_{lin}(w_k,s_k).
\end{equation}
In the sequel, we will sometimes refer to $\|\omega\|_{H^{k-1}(\Omega)}^2$ as the rotational part of the energy, denoted by $E^k_{r}(v,\Gamma)$, and $E_{lin}(w_k,s_k)$ as the irrotational part of the energy, denoted by $E^k_{i}(v,\Gamma)$.
\begin{remark}
This definition of the energy has to be interpreted in a suitable way when $v$ and $\Gamma$ do not solve the free boundary Euler equations. Indeed, it is important that, a priori, the definition of the energy functional does not depend on the dynamics of the problem. Therefore, for a bounded connected domain $\Omega$ with $(v,\Gamma)\in\mathbf{H}^k$, we define $p$ through the boundary condition $p_{|\Gamma}=0$ and the Laplace equation
\begin{equation*}
\Delta p=-\text{tr}(\nabla v)^2. 
\end{equation*}
The Taylor sign term is then defined via
\begin{equation*}\label{adef}
a:=-n_{\Gamma}\cdot \nabla p_{|\Gamma}.
\end{equation*}
Moreover, we define $D_tp$ through the Dirichlet boundary condition $D_tp_{|\Gamma}=0$ and $\Delta D_tp$ given by
\begin{equation}\label{Dtpdef}
\Delta D_tp=4\text{tr}(\nabla^2p\cdot\nabla v)+2\text{tr}((\nabla v)^3)+\Delta v\cdot\nabla p=:F.
\end{equation}
In other words, $D_tp=\Delta^{-1}F$. This is the definition of $D_tp$ which is compatible with the dynamical problem. We then define $D_t\nabla p$ by
\begin{equation*}\label{DDpdef}
D_t\nabla p:=-\nabla v\cdot\nabla p+\nabla D_tp
\end{equation*}
and then $D_ta$ by
\begin{equation*}\label{Dadef}
D_ta:=-n_{\Gamma}\cdot D_t\nabla p_{|\Gamma}.
\end{equation*}
With these definitions, the energy functional \eqref{ENERGYE1} is well-defined, irrespective of whether the state $(v,\Gamma)$ evolves dynamically.
\end{remark}
\begin{remark}
    We note that the energy functional \eqref{ENERGYE1} is essentially the same as that from \cite{MR3925531}. The main difference, so far, is in the derivation of this energy. Indeed, our approach was to identify Alinhac style good unknowns, whereas \cite{MR3925531} first derives a wave-type equation for $a$ and then applies powers of the Dirichlet-to-Neumann operator to this equation, as if it were a vector field.  However, as can be immediately inferred from the low regularity of our control norms, the way we treat the energy is very different from \cite{MR3925531}. 
\end{remark}
\subsection{Coercivity of the energy functional}
We begin by establishing the coercivity part of \Cref{Energy est. thm}. That is, we want to show that
\begin{equation*}
E^k(v,\Gamma)\approx_A 1+\|(v,\Gamma)\|_{\mathbf{H}^k}^2.
\end{equation*}
We begin by collecting some preliminary estimates for the various quantities that will appear in our analysis. 
\subsection{\texorpdfstring{$L^{\infty}$}{} estimates for coercivity}
Here we will establish some $L^{\infty}$ based estimates for $p$ and $D_tp$ in terms of the control parameter $A$. The $A$ control parameter involves only the physical variables $v$ and $\Gamma$. The variables $p$ and $D_tp$ are related to these variables through solving a suitable Laplace equation. We will therefore need to make use of the Schauder type estimates in \Cref{Gilbarg} to control these terms (in suitable pointwise norms) by $A$. For this, we have the following lemma. 
\begin{lemma} \label{Linfest}
Given the assumptions of \Cref{Energy est. thm}, the following pointwise estimates for  $p$ and $D_tp$ hold.
\begin{enumerate}
\item ($C^{1,\epsilon}$ estimate for $p$).
\begin{equation*} 
\|p\|_{C^{1,\epsilon}(\Omega)}\lesssim_A 1.
\end{equation*}
\item (Partition bound for $D_tp$). There exists a sequence of partitions $D_tp=:F_j^1+F_j^2$ such that 
\begin{equation*}
\|F_j^1\|_{W^{1,\infty}(\Omega)}\lesssim_A 2^{j(\frac{1}{2}-\epsilon)},\hspace{5mm} \|F_j^2\|_{H^1(\Omega)}\lesssim_A 2^{-j(k-1-\epsilon)}(\|v\|_{H^{k-\epsilon}(\Omega)}+\|p\|_{H^{k+\frac{1}{2}-\epsilon}(\Omega)}).
\end{equation*}
\end{enumerate}
\end{lemma}
One can loosely think of the partition of $D_tp$ in the second part of \Cref{Linfest} as a splitting of $D_tp$ into low and high frequency parts at a dyadic scale $2^j$. The high frequency part will typically be best estimated in $L^2$ based norms, and the low frequency part in $L^{\infty}$ based norms. In particular, one can think  of the estimate for $F_j^1$ as an estimate for the ``low frequency part" of $D_tp$ in $C^{\frac{1}{2}+\epsilon}$. This will serve as a substitute for what would be a $C^{\frac{1}{2}+\epsilon}$ estimate for the inhomogenenous Dirichlet problem, which is not available to us (except for harmonic functions). The usefulness of this will be made more transparent later.
\begin{proof}
We begin with some notation. For any integer $l>0$, we write $\Phi_l:=\Phi_{\leq l+1}-\Phi_{\leq l}$ and $\Psi_{l}:=\Psi_{\leq l+1}-\Psi_{\leq l}$. We also write $\Phi_0$ and $\Psi_0$ to mean $\Phi_{\leq 0}$ and $\Psi_{\leq 0}$, respectively. For a vector or scalar valued function $f$ defined on $\Omega$, we write $f^l$ and $f^{\leq l}$ as shorthand for $\Phi_l f$ and $\Phi_{\leq l}f$, respectively. If in addition, $f$ is a divergence free vector field, we instead use $f^l$ and $f^{\leq l}$ to mean $\Psi_lf$ and $\Psi_{\leq l}f$, respectively. This will ensure that the divergence free structure of $f$ is preserved. We abuse notation and write $f^lg^{\leq l}$ to mean 
\begin{equation*}
f^lg^{\leq l}:=\sum_{l\geq 0}\sum_{0\leq m\leq l}f^lg^m-\frac{1}{2}\sum_{l\geq 0}f^lg^l.
\end{equation*}
This definition ensures (with the convention that $\Phi_0=\Phi_{\leq 0}$ and $\Psi_0=\Psi_{\leq 0}$) that we have the decomposition
\begin{equation}\label{bilpara}
fg=f^lg^{\leq l}+f^{\leq l}g^l,
\end{equation}
which can be thought of as a kind of crude bilinear paraproduct decomposition where $f^lg^{\leq l}$ selects the portion of $fg$ where $f$ is at higher or comparable frequency compared to $g$. Likewise, we can define trilinear expressions of the form $f^lg^{\leq l}h^{\leq l}$ in such a way that we have $fgh=f^lg^{\leq l}h^{\leq l}+f^{\leq l}g^{l}h^{\leq l}+f^{\leq l}g^{\leq l}h^{l}$, and similarly for quadrilinear expressions. Now, we begin with the first part of the lemma. Expanding using (\ref{bilpara}) we see that
\begin{equation}\label{decomp}
\begin{split}
p&=-\Delta^{-1}\text{tr}(\nabla v)^2=-2\Delta^{-1}\partial_j(v^l_i\partial_iv^{\leq l}_j).
\end{split}
\end{equation}
Importantly, because $v^l$ is divergence free, we were able to write $\text{tr}(\nabla v)^2$ as the divergence of a bilinear expression in $v$ and $\nabla v$, where the high frequency factor is undifferentiated. This will allow us to make use of the lower regularity $C^{1,\alpha}$ estimates in \Cref{Gilbarg} and simultaneously allow us to rebalance derivatives in the bilinear expression for $v$. This theme of writing multilinear expressions in divergence form with the highest frequency factor undifferentiated will appear several times in the sequel in more complicated forms. In this case, we have from \Cref{Gilbarg},
\begin{equation*}
\begin{split}
\|p\|_{C^{1,\epsilon}(\Omega)}\lesssim_A \|v^l_i\partial_iv^{\leq l}_j\|_{C^{\epsilon}(\Omega)}\lesssim_A \|v\|_{C^{\frac{1}{2}+\epsilon}(\Omega)}^2\lesssim_A 1.
\end{split}
\end{equation*}
Next, we turn to the estimate for $D_tp$, which is the more difficult part. From (\ref{Dtpdef}), we can write in Euclidean coordinates,
\begin{equation}\label{threeterms}
D_tp=4\Delta^{-1}(\partial_i\partial_jp\partial_iv_j)+2\Delta^{-1}(\partial_jv_k\partial_kv_i\partial_iv_j)+\Delta^{-1}(\partial_i\partial_iv_j\partial_jp).
\end{equation}
In order to make full use of \Cref{Gilbarg}, we will again need to write $D_tp$ in the form $\Delta^{-1}\nabla\cdot f$ for some vector field $f$ in a way which allows us to also rebalance derivatives, as we did in the estimate for $p$. We start by estimating the first term in \eqref{threeterms}. We first write $\partial_i\partial_j p\partial_i v_j=\nabla\cdot (\partial_i p\partial_iv)$ and use the partition
\begin{equation*}
\Delta^{-1}\nabla\cdot (\partial_ip\partial_iv)=T_j^1+T_j^2,
\end{equation*}
where $T_j^1=\Delta^{-1}\nabla\cdot (\partial_ip\partial_i\Phi_{<j}v)$. From \Cref{Gilbarg} and the $C^{1,\epsilon}$ estimate for $p$ above, we have
\begin{equation*}
\begin{split}
\|T_j^1\|_{W^{1,\infty}(\Omega)}&\lesssim_A \|\nabla p\|_{C^{\epsilon}(\Omega)}\|\nabla \Phi_{<j}v\|_{L^{\infty}(\Omega)}+\|\nabla p\|_{L^{\infty}(\Omega)}\|\nabla \Phi_{<j}v\|_{C^{\epsilon}(\Omega)}\lesssim_A 2^{j(\frac{1}{2}-\epsilon)}.
\end{split}
\end{equation*}
We also see from (\ref{H1base}),
\begin{equation*}
\begin{split}
\|T_j^2\|_{H^1(\Omega)}&\lesssim_A 2^{-j(k-1-\epsilon)}\|\nabla p\|_{L^{\infty}(\Omega)}\|v\|_{H^{k-\epsilon}(\Omega)}\lesssim_A 2^{-j(k-1-\epsilon)}\|v\|_{H^{k-\epsilon}(\Omega)}.
\end{split}
\end{equation*}
Next, we turn to the second term in (\ref{threeterms}). We start by performing a trilinear frequency decomposition. Using the symmetry of the indices, we have
\begin{equation}\label{cubicdecomp1}
\begin{split}
\partial_jv_k\partial_kv_i\partial_iv_j&=3\partial_jv_k^l\partial_kv_i^{\leq l}\partial_iv_j^{\leq l}.
\end{split}
\end{equation}
To best balance derivatives, we would like to write this in the form $\nabla\cdot \mathcal{T}(v^l,\nabla v^{\leq l},\nabla v^{\leq l})$ where $\mathcal{T}$ is an appropriate trilinear expression. To do this, we can use the symmetry of the expression and the fact that $v$ is divergence free to write
\begin{equation}\label{cubicdecomp2}
\begin{split}
\partial_jv_k^l\partial_kv_i^{\leq l}\partial_iv_j^{\leq l}&=\partial_j(v_k^l\partial_kv_i^{\leq l}\partial_iv_j^{\leq l})-v_k^l\partial_k\partial_jv_i^{\leq l}\partial_i v_j^{\leq l}
\\
&=\partial_j(v_k^l\partial_kv_i^{\leq l}\partial_iv_j^{\leq l})-\frac{1}{2}v_k^l\partial_k(\partial_jv_i^{\leq l}\partial_i v_j^{\leq l})
\\
&=\partial_j(v_k^l\partial_kv_i^{\leq l}\partial_iv_j^{\leq l})-\frac{1}{2}\partial_k(v_k^l\partial_jv_i^{\leq l}\partial_i v_j^{\leq l}).
\end{split}
\end{equation}
We partition the last line above into $Q_j^1+Q_j^2$ where 
\begin{equation*}
Q_j^1:=\partial_m(v_k^l\partial_k\Phi_{<j}v_i^{\leq l}\partial_iv_m^{\leq l})-\frac{1}{2}\partial_k(v_k^l\partial_m\Phi_{<j}v_i^{\leq l}\partial_i v_m^{\leq l}).
\end{equation*}
We then obtain in a straightforward way using \Cref{Gilbarg} and summing in $l$,
\begin{equation*}
\|\Delta^{-1}Q_j^1\|_{W^{1,\infty}(\Omega)}\lesssim_A 2^{j(\frac{1}{2}-\epsilon)}\|v\|_{C^{\frac{1}{2}+\epsilon}(\Omega)}^3\lesssim_A 2^{j(\frac{1}{2}-\epsilon)},
\end{equation*}
and from the $H^{-1}\to H^1$ estimate for the Dirichlet problem and \Cref{c reg bounds},
\begin{equation*}
\|\Delta^{-1}Q_j^2\|_{H^1(\Omega)}\lesssim_A 2^{-j(k-1-\epsilon)}\|v\|_{H^{k-\epsilon}(\Omega)}.
\end{equation*}
Finally, the last term in (\ref{threeterms}) can be handled by writing
\begin{equation*}\label{decomp2}
\partial_i\partial_iv_j\partial_jp=\partial_i(\partial_i v_j\partial_j p)-\partial_i v_j\partial_i \partial_j p
\end{equation*}
and partitioning each term similarly to the first term in (\ref{threeterms}). Collecting all of the above partitions together completes the proof of the lemma.
\end{proof}
The following simple consequence of the above lemma  will be useful for estimating $D_ta$ in pointwise norms.
\begin{corollary}\label{otherdecomp}
Given the assumptions of \Cref{Linfest}, there exists a sequence of partitions $D_t\nabla p=G_j^1+G_j^2$ such that
\begin{equation*}
\|G_j^1\|_{L^{\infty}(\Omega)}\lesssim_A 2^{j(\frac{1}{2}-\epsilon)},\hspace{5mm}\|G_j^2\|_{H^{\frac{1}{2}+\epsilon}(\Omega)}\lesssim_A 2^{-j(k-\frac{3}{2}-2\epsilon)}(\|v\|_{H^{k-\epsilon}(\Omega)}+\|p\|_{H^{k+\frac{1}{2}-\epsilon}(\Omega)}+\|D_t\nabla p\|_{H^{k-1-\epsilon}(\Omega)}).
\end{equation*}
\end{corollary}
\begin{proof}
This follows from \Cref{Linfest} by taking
\begin{equation*}
G_j^1=\Phi_{<j}(-\nabla\Phi_{<j}v\cdot\nabla p+\nabla F_j^1),\hspace{10mm}G_j^2=\Phi_{<j}(-\nabla \Phi_{\geq j}v\cdot\nabla p)+\Phi_{<j}\nabla F_j^2+\Phi_{\geq j}D_t\nabla p.        
\end{equation*}   
\end{proof}
\subsection{\texorpdfstring{$L^2$}{} based estimates for \texorpdfstring{$a$}{} and \texorpdfstring{$D_ta$}{}}
Our next step will be to control $(a,D_ta)$ in $H^{k-1}(\Gamma)\times H^{k-\frac{3}{2}}(\Gamma)$ by the energy plus some lower order terms. Let us define for the rest of this section the lower order quantity
\begin{equation*}
\Lambda_{k-\epsilon}:=\|\Gamma\|_{H^{k-\epsilon}}+\|v\|_{H^{k-\epsilon}(\Omega)}+\|p\|_{H^{k+\frac{1}{2}-\epsilon}(\Omega)}+\|D_t\nabla p\|_{H^{k-1-\epsilon}(\Omega)},
\end{equation*}
where $\epsilon>0$ is any small, but fixed, positive constant.
\begin{lemma}\label{aest}
We have
\begin{equation*}
\|a\|_{H^{k-1}(\Gamma)}+\|D_ta\|_{H^{k-\frac{3}{2}}(\Gamma)}\lesssim_A (E^k)^{\frac{1}{2}}+\Lambda_{k-\epsilon}.
\end{equation*}
\end{lemma}
\begin{proof}
To control $a$ in $H^{k-1}(\Gamma)$, we use the ellipticity estimate for the Dirichlet-to-Neumann operator from \Cref{ellipticity} to obtain
\begin{equation*}
\|a\|_{H^{k-1}(\Gamma)}\lesssim_A 
 \|a\|_{L^2(\Gamma)}+\|\mathcal{N}^{k-1}a\|_{L^2(\Gamma)}+\|\Gamma\|_{H^{k-\epsilon}}\|a\|_{C^{\epsilon}(\Gamma)}\lesssim_A (E^k)^{\frac{1}{2}}+\Lambda_{k-\epsilon}.
\end{equation*}
To estimate $D_ta$ in $H^{k-\frac{3}{2}}(\Gamma)$, we consider the partition $D_t\nabla p:=G_j^1+G_j^2$ from \Cref{otherdecomp} and estimate using \Cref{ellipticity},
\begin{equation*}
\|D_ta\|_{H^{k-\frac{3}{2}}(\Gamma)}\lesssim_A \|\mathcal{N}^{k-2}D_ta\|_{H^{\frac{1}{2}}(\Gamma)}+\|\Gamma\|_{H^{k-\epsilon}}\sup_{j>0}2^{-j(\frac{1}{2}-\epsilon)}\|n_\Gamma\cdot G_j^1\|_{L^{\infty}(\Gamma)}+\sup_{j>0}2^{j(k-2\epsilon-\frac{3}{2})}\|n_\Gamma\cdot G_j^2\|_{H^{\epsilon}(\Gamma)}+\Lambda_{k-\epsilon}.
\end{equation*}
From the trace theorem, 
\begin{equation*}
\|\mathcal{N}^{k-2}D_ta\|_{H^{\frac{1}{2}}(\Gamma)}\lesssim_A \|\mathcal{H}\mathcal{N}^{k-2}D_ta\|_{H^1(\Omega)}.
\end{equation*}
Since $k\geq 3$ and
\begin{equation*}
\int_{\Gamma}\mathcal{N}^{k-2}D_ta\, dS=\int_{\Gamma}n_\Gamma\cdot\nabla \mathcal{H}\mathcal{N}^{k-3}D_ta\, dS=0,
\end{equation*}
we conclude by a Poincare type inequality that 
\begin{equation*}
\|\mathcal{H}\mathcal{N}^{k-2}D_ta\|_{H^1(\Omega)}\lesssim_A \|\nabla\mathcal{H}\mathcal{N}^{k-2}D_ta\|_{L^2(\Omega)}\lesssim_A(E^k)^{\frac{1}{2}}.
\end{equation*}
From \Cref{otherdecomp}, we have
\begin{equation*}
\sup_{j>0}2^{-j(\frac{1}{2}-\epsilon)}\|n_\Gamma\cdot G_j^1\|_{L^{\infty}(\Gamma)}\lesssim_A 1.
\end{equation*}
On the other hand, from the trace theorem and \Cref{otherdecomp},
\begin{equation*}
2^{j(k-\frac{3}{2}-2\epsilon)}\|n_\Gamma\cdot G_j^2\|_{H^{\epsilon}(\Gamma)}\lesssim_A 
 \Lambda_{k-\epsilon},
\end{equation*}
which completes the proof.
\end{proof}
With our preliminary estimates in hand, let us proceed with the proof of the first (and harder) half of the coercivity estimate; namely,
\begin{equation*}
\|(v,\Gamma)\|_{\mathbf{H}^k}\lesssim_A (E^k)^{\frac{1}{2}}.
\end{equation*}
Let us begin by proving the estimate 
\begin{equation}\label{surfcontrol}
\|p\|_{H^{k+\frac{1}{2}}(\Omega)}+\|\Gamma\|_{H^k}\lesssim_A (E^k)^{\frac{1}{2}}+\Lambda_{k-\epsilon}.
\end{equation}
We start by recalling from \Cref{curvaturebound} that we have
\begin{equation*}
\|\Gamma\|_{H^k}+\|n_\Gamma\|_{H^{k-1}(\Gamma)}\lesssim_A 1+\|\kappa\|_{H^{k-2}(\Gamma)}, 
\end{equation*}
where $\kappa$ is the mean curvature of $\Gamma$. Therefore, to establish (\ref{surfcontrol}), it suffices to establish the same estimate except with $\|p\|_{H^{k+\frac{1}{2}}(\Omega)}+\|\kappa\|_{H^{k-2}(\Gamma)}$ on the left-hand side. To do this, we begin by relating the curvature to the pressure via the formula
\begin{equation}\label{curvaturerelation}
\kappa=a^{-1}\Delta p-a^{-1}D^2p(n_\Gamma,n_\Gamma).
\end{equation}
Here, we used the fact that $\Delta_{\Gamma}p=0$ on $\Gamma$. We now estimate each term on the right-hand side of \eqref{curvaturerelation}. For the first term, we use the Laplace equation for $p$ and the bilinear frequency decomposition for $\Delta p=-\text{tr}(\nabla v)^2$ as in \Cref{Linfest} together with  \Cref{boundaryest} to obtain
\begin{equation*}
\begin{split}
\|a^{-1}\Delta p\|_{H^{k-2}(\Gamma)}&\lesssim_A \|\text{tr}(\nabla v)^2\|_{H^{k-2}(\Gamma)}+(\|a^{-1}\|_{H^{k-1-\epsilon}(\Gamma)}+\|\Gamma\|_{H^{k-\epsilon}})\sup_{j>0}2^{-j(1-\epsilon)}\|\Phi_{<j}\partial_k(v_i^l\partial_iv_k^{\leq l})\|_{L^{\infty}(\Omega)}
\\
&+\sup_{j>0}2^{j(k-2-\epsilon)}\|\Phi_{\geq j}\text{tr}(\nabla v)^2\|_{L^2(\Gamma)}.
\end{split}
\end{equation*}
Using the trace theorem,  the product estimates \Cref{boundaryest} and \Cref{productest}, the latter two terms can be controlled by $C_A\Lambda_{k-\epsilon}$ where $C_A$ is a constant depending polynomially on $A$ only. On the other hand, $\|\text{tr}(\nabla v)^2\|_{H^{k-2}(\Gamma)}$ can be controlled using the balanced trace estimate \Cref{baltrace} as well as  \Cref{productest} as follows:
\begin{equation*}
\begin{split}
\|\text{tr}(\nabla v)^2\|_{H^{k-2}(\Gamma)}&\lesssim_A \|\text{tr}(\nabla v)^2\|_{H^{k-\frac{3}{2}}(\Omega)}+\|\Gamma\|_{H^{k-\epsilon}}\sup_{j>0}2^{-j(1-\epsilon)}\|\Phi_{<j}\partial_k(v_i^l\partial_iv_k^{\leq l})\|_{L^{\infty}(\Omega)}
\\
&+\sup_{j>0}2^{j(k-\frac{3}{2}-\epsilon)}\|\Phi_{\geq j}\text{tr}(\nabla v)^2\|_{L^2(\Omega)}
\\
&\lesssim_A \Lambda_{k-\epsilon}. 
\end{split}
\end{equation*}
To estimate $a^{-1}D^2p(n_\Gamma,n_\Gamma)$ in $H^{k-2}(\Gamma)$, we proceed similarly by starting with \Cref{boundaryest} and \Cref{Linfest} to obtain
\begin{equation*}
\begin{split}
\|a^{-1}D^2p(n_\Gamma,n_\Gamma)\|_{H^{k-2}(\Gamma)}&\lesssim_A \|D^2p(n_\Gamma,n_\Gamma)\|_{H^{k-2}(\Gamma)}+\sup_{j>0}2^{j(k-2-\epsilon)}\|\Phi_{\geq j}D^2p\|_{L^2(\Gamma)}
\\
&+(\|a^{-1}\|_{H^{k-1-\epsilon}(\Gamma)}+\|\Gamma\|_{H^{k-\epsilon}})\sup_{j>0}2^{-j(1-\epsilon)}\|\Phi_{<j}D^2p\|_{L^{\infty}(\Omega)}.
\end{split}
\end{equation*}
Similarly to the previous estimate, the latter two terms are controlled by $C_A\Lambda_{k-\epsilon}$. For the term involving $D^2p(n_\Gamma,n_\Gamma)$, we use \Cref{boundaryest} again, combined with the estimates $\|n_\Gamma\|_{H^{k-1-\epsilon}(\Gamma)}\lesssim_A \|\Gamma\|_{H^{k-\epsilon}}$ and $\|n_\Gamma\|_{C^{\epsilon}(\Gamma)}\lesssim_A 1$ to obtain (similarly to the above estimate but with $a^{-1}$ replaced by $n_\Gamma$)
\begin{equation*}
\begin{split}
\|D^2p(n_\Gamma,n_\Gamma)\|_{H^{k-2}(\Gamma)}\lesssim_A \|D^2p\|_{H^{k-2}(\Gamma)}+\Lambda_{k-\epsilon}.
\end{split}
\end{equation*}
\Cref{baltrace} and the same partition of $D^2p$ above then yields
\begin{equation*}
\|D^2p\|_{H^{k-2}(\Gamma)}\lesssim_A \|\nabla p\|_{H^{k-\frac{1}{2}}(\Omega)}+\Lambda_{k-\epsilon}.
\end{equation*}
To complete the proof of (\ref{surfcontrol}), we now only need to control $\nabla p$ in $H^{k-\frac{1}{2}}$. For this, we use the div-curl estimate \Cref{Balanced div-curl} for $\nabla p$ as well as \Cref{productest}, \Cref{boundaryest} and \Cref{tangradientbound} to obtain
\begin{equation}\label{pest}
\begin{split}
\|\nabla p\|_{H^{k-\frac{1}{2}}(\Omega)}&\lesssim_A \|\nabla p\|_{L^2(\Omega)}+\|\nabla^{\top}a\|_{H^{k-2}(\Gamma)}+\|\text{tr}(\nabla v)^2\|_{H^{k-\frac{3}{2}}(\Omega)}+\|\Gamma\|_{H^{k-\epsilon}}\|\nabla p\|_{C^{\epsilon}(\Omega)}
\\
&\lesssim_A (E^k)^{\frac{1}{2}}+\|a\|_{H^{k-1}(\Gamma)}+\Lambda_{k-\epsilon}
\\
&\lesssim_A (E^k)^{\frac{1}{2}}+\Lambda_{k-\epsilon},
\end{split}
\end{equation}
where we used \Cref{aest} to go from the second to third line. From this, we finally obtain the estimate (\ref{surfcontrol}). To close the coercivity estimate, it remains to control $v$ in $H^k(\Omega)$ and $D_t\nabla p$ in $H^{k-1}(\Omega)$ by the energy. We first reduce to the estimate
\begin{equation*}\label{vcoerc}
\|v\|_{H^k(\Omega)}\lesssim_A (E^k)^{\frac{1}{2}}+\|D_t\nabla p\|_{H^{k-1}(\Omega)}+\Lambda_{k-\epsilon}.    
\end{equation*}
For this, we start by relating the boundary term $\nabla^{\top}v\cdot n_\Gamma$ to $D_t\nabla p$. Indeed, we have 
\begin{equation*}
D_t\nabla p=\nabla D_t p-\nabla v\cdot\nabla p.
\end{equation*}
Since $\nabla p=-an_\Gamma$ and $D_tp=0$ on $\Gamma$, we obtain
\begin{equation*}
\nabla^{\top}v\cdot n_\Gamma=a^{-1}(D_t\nabla p)^{\top},
\end{equation*}
and so, since $v$ is divergence free, we have from the div-curl estimate in \Cref{Balanced div-curl},
\begin{equation}\label{754}
\begin{split}
\|v\|_{H^k(\Omega)}&\lesssim_A \|v\|_{L^2(\Omega)}+\|\omega\|_{H^{k-1}(\Omega)}+\|a^{-1}(D_t\nabla p)^{\top}\|_{H^{k-\frac{3}{2}}(\Gamma)}+\|\Gamma\|_{H^{k-\epsilon}}\|v\|_{C^{\frac{1}{2}+\epsilon}(\Omega)}
\\
&\lesssim_A \|a^{-1}(D_t\nabla p)^{\top}\|_{H^{k-\frac{3}{2}}(\Gamma)}+(E^k)^{\frac{1}{2}}+\Lambda_{k-\epsilon}.
\end{split}
\end{equation}
To estimate the first term on the right-hand side of \eqref{754}, we use the decomposition $D_t\nabla p=G_j^1+G_j^2$ from \Cref{otherdecomp}. By the balanced product and trace estimates \Cref{boundaryest} and \Cref{baltrace} and a similar analysis to the estimate for $\|\kappa\|_{H^{k-2}(\Gamma)}$, we obtain
\begin{equation*}\label{div-curD_tp1}
\begin{split}
\|a^{-1}(D_t\nabla p)^{\top}\|_{H^{k-\frac{3}{2}}(\Gamma)}&\lesssim_A \|D_t\nabla p\|_{H^{k-1}(\Omega)}+(\|a^{-1}\|_{H^{k-1-\epsilon}(\Gamma)}+\|\Gamma\|_{H^{k-\epsilon}})\sup_{j>0}2^{-j(\frac{1}{2}-\epsilon)}\|G_j^1\|_{L^{\infty}(\Omega)}
\\
&+\sup_{j>0}2^{j(k-\frac{3}{2}-2\epsilon)}\|G_j^2\|_{H^{\frac{1}{2}+\epsilon}(\Omega)}\lesssim_A \|D_t\nabla p\|_{H^{k-1}(\Omega)}+\Lambda_{k-\epsilon}.
\end{split}
\end{equation*}
Finally, we need to show that
\begin{equation*}
\|D_t\nabla p\|_{H^{k-1}(\Omega)}\lesssim_A (E^k)^{\frac{1}{2}}+\Lambda_{k-\epsilon}.    
\end{equation*}
For this, we will use the div-curl decomposition for $D_t\nabla p$. The divergence and curl are given by
\begin{equation*}\label{Dtpdiv}
\begin{cases}
&\nabla\cdot D_t\nabla p=3\text{tr}(\nabla^2p\cdot \nabla v)+2\text{tr}(\nabla v)^3\hspace{2mm}\text{in $\Omega$},
\\
&\nabla\times D_t\nabla p=\nabla ^2p\cdot \nabla v-(\nabla v)^*\cdot \nabla ^2p\hspace{2mm}\text{in $\Omega$}.
\end{cases}    
\end{equation*}
Hence, using the div-curl estimate and the partition $D_t\nabla p=G_j^1+G_j^2$ from \Cref{otherdecomp} in conjunction with \Cref{productest}, we obtain
\begin{equation*}
\begin{split}
\|D_t\nabla p\|_{H^{k-1}(\Omega)}&\lesssim_A \|p\|_{H^{k+\frac{1}{2}}(\Omega)}\|v\|_{C^{\frac{1}{2}}(\Omega)}+\|p\|_{C^{1,\epsilon}(\Omega)}\|v\|_{H^{k-\epsilon}(\Omega)}+\|\text{tr}(\nabla v)^3\|_{H^{k-2}(\Omega)}+\|\nabla^{\top}(D_t\nabla p)\cdot n_\Gamma\|_{H^{k-\frac{5}{2}}(\Gamma)}
\\
&+\|\Gamma\|_{H^{k-\epsilon}}\sup_{j>0}2^{-j(\frac{1}{2}-\epsilon)}\|G_j^1\|_{L^{\infty}(\Omega)}+\sup_{j>0}2^{j(k-\frac{3}{2}-2\epsilon)}\|G_j^2\|_{H^{\frac{1}{2}+\epsilon}(\Omega)}+\Lambda_{k-\epsilon}.
\end{split}
\end{equation*}
Estimating $G_j^1$ and $G_j^2$ as before and then using (\ref{pest}) gives
\begin{equation*}
\|D_t\nabla p\|_{H^{k-1}(\Omega)}\lesssim_A (E^k)^{\frac{1}{2}}+\|\nabla^{\top}(D_t\nabla p)\cdot n_\Gamma\|_{H^{k-\frac{5}{2}}(\Gamma)}+\|v\|_{H^{k-\epsilon}(\Omega)}+\|\Gamma\|_{H^{k-\epsilon}}+\|\text{tr}(\nabla v)^3\|_{H^{k-2}(\Omega)}+\Lambda_{k-\epsilon}.
\end{equation*}
Using a trilinear frequency decomposition as in \Cref{Linfest}, we obtain easily
\begin{equation*}
\|\text{tr}(\nabla v)^3\|_{H^{k-2}(\Omega)}\lesssim_A \|v\|_{C^{\frac{1}{2}+\epsilon}(\Omega)}^2\|v\|_{H^{k-\epsilon}(\Omega)}\lesssim_A \Lambda_{k-\epsilon}.
\end{equation*}
It remains to estimate the boundary term.  We compute
\begin{equation}\label{top1}
\nabla^{\top}(D_t\nabla p)\cdot n_\Gamma=-\nabla^{\top}D_ta-D_t\nabla p\cdot\nabla^{\top}n_\Gamma.
\end{equation}
By \Cref{boundaryest}, \Cref{tangradientbound} and using the decomposition $D_t\nabla p=G_j^1+G_j^2$, the terms in \eqref{top1} are controlled in a similar fashion to  the above terms by
\begin{equation*}
\|\nabla^{\top}(D_t\nabla p)\cdot n_\Gamma\|_{H^{k-\frac{5}{2}}(\Gamma)}\lesssim_A \|D_ta\|_{H^{k-\frac{3}{2}}(\Gamma)}+\Lambda_{k-\epsilon}\lesssim_A (E^k)^{\frac{1}{2}}+\Lambda_{k-\epsilon},  
\end{equation*}
where we used \Cref{aest} in the last inequality. Combining everything together, we have
\begin{equation*}
\|D_t\nabla p\|_{H^{k-1}(\Omega)}+\|\Gamma\|_{H^k}+\|v\|_{H^k(\Omega)}+\|p\|_{H^{k+\frac{1}{2}}(\Omega)}\lesssim_A (E^k)^{\frac{1}{2}}+\Lambda_{k-\epsilon}.
\end{equation*}
Using the definition of $\Lambda_{k-\epsilon}$ and interpolating gives
\begin{equation*}
\|D_t\nabla p\|_{H^{k-1}(\Omega)}+\|\Gamma\|_{H^k}+\|v\|_{H^k(\Omega)}+\|p\|_{H^{k+\frac{1}{2}}(\Omega)}\lesssim_A (E^k)^{\frac{1}{2}}+\|v\|_{L^2(\Omega)}+\|p\|_{H^1(\Omega)}+\|D_t\nabla p\|_{L^2(\Omega)}.
\end{equation*}
We can use the $H^1$ estimate for the Laplace equation for $p$ to estimate
\begin{equation*}
\|p\|_{H^1(\Omega)}\lesssim_A \|v\|_{H^1(\Omega)}.    
\end{equation*}
Moreover, by writing $D_t\nabla p=\nabla D_tp-\nabla v\cdot\nabla p$, writing $D_tp$ in the form $\Delta^{-1}\nabla\cdot f$ as in the proof of \Cref{Linfest} and using the $H^{-1}\to H^1$ estimate for $\Delta^{-1}$, we have 
\begin{equation*}
\|D_t\nabla p\|_{L^2(\Omega)}\lesssim_A \|v\|_{H^1(\Omega)}.    
\end{equation*}
Therefore, by interpolation we have
\begin{equation}\label{strongercoercivitybound}
\|D_t\nabla p\|_{H^{k-1}(\Omega)}+\|\Gamma\|_{H^k}+\|v\|_{H^k(\Omega)}+\|p\|_{H^{k+\frac{1}{2}}(\Omega)}\lesssim_A (E^k)^{\frac{1}{2}}.
\end{equation}
This finally establishes the desired estimate
\begin{equation*}
\|(v,\Gamma)\|_{\mathbf{H}^k}\lesssim_A (E^k)^{\frac{1}{2}}.
\end{equation*}
Next, we show the easier part of the coercivity bound; namely,
\begin{equation*}\label{easiercoercive}
(E^k)^{\frac{1}{2}}\lesssim_A 1+\|(v,\Gamma)\|_{\mathbf{H}^k}.
\end{equation*}
Clearly, the only nontrivial part is to control the irrotational energy. More precisely, we have to show that
\begin{equation}\label{reversecoercive}
\|\nabla\mathcal{H}\mathcal{N}^{k-2}D_ta\|_{L^2(\Omega)}+\|a^{\frac{1}{2}}\mathcal{N}^{k-1}a\|_{L^2(\Gamma)}\lesssim_A 1+\|(v,\Gamma)\|_{\mathbf{H}^k}.
\end{equation}
To establish this, we will need the following $L^2$ based estimates for $p$ and $D_tp$.
\begin{lemma}\label{Hestimates} The following estimate holds:
\begin{equation*}
\|p\|_{H^{k+\frac{1}{2}}(\Omega)}+\|D_tp\|_{H^{k}(\Omega)}\lesssim_A \|(v,\Gamma)\|_{\mathbf{H}^k}.
\end{equation*}
\end{lemma}
\begin{proof}
 First, from the balanced Dirichlet estimate in \Cref{direst}, as well as \Cref{productest} and \Cref{Linfest}, we have
\begin{equation*}\label{pdirest}
\begin{split}
\|p\|_{H^{k+\frac{1}{2}}(\Omega)}&\lesssim_A \|\text{tr}(\nabla v)^2\|_{H^{k-\frac{3}{2}}(\Omega)}+\|\Gamma\|_{H^k}\|p\|_{W^{1,\infty}(\Omega)}\lesssim_A \|(v,\Gamma)\|_{\mathbf{H}^k}.
\end{split}
\end{equation*}
 To estimate $D_tp$, recall that we can write $D_tp$ in the form $\Delta^{-1}\nabla\cdot f$. Indeed, similarly to \Cref{Linfest}, we can start by writing
\begin{equation}\label{Dtpid}
D_tp=\Delta^{-1}\partial_i(\partial_i v_j\partial_jp)+3\Delta^{-1}\partial_i(\partial_jp\partial_jv_i)+2\Delta^{-1}\text{tr}(\nabla v)^3=:F_1+F_2+F_3.
\end{equation}
We now will use \Cref{direst} to estimate each term. We begin with $F_1$. We use the partition $F_1=H_j^1+H_j^2$ where $H_j^1:=\Delta^{-1}\partial_i(\partial_i\Phi_{\leq j}v_k\partial_kp)$ and \Cref{direst} to obtain, 
\begin{equation*}
\begin{split}
\|F_1\|_{H^{k}(\Omega)}\lesssim_A \|\nabla p\cdot \nabla v\|_{H^{k-1}(\Omega)}+\|\Gamma\|_{H^k}\sup_{j>0}2^{-\frac{j}{2}}\|H_j^1\|_{W^{1,\infty}(\Omega)}+\sup_{j>0}2^{j(k-1)}\|H_j^2\|_{H^1(\Omega)}.
\end{split}
\end{equation*}
Using \Cref{productest} and the $H^{k+\frac{1}{2}}$ estimate for $p$ above, we obtain
\begin{equation*}
\|\nabla p\cdot \nabla v\|_{H^{k-1}(\Omega)}\lesssim_A \|v\|_{C^{\frac{1}{2}+\epsilon}(\Omega)}\|p\|_{H^{k+\frac{1}{2}}(\Omega)}+\|p\|_{C^{1,\epsilon}(\Omega)}\|v\|_{H^{k}(\Omega)}\lesssim_A \|(v,\Gamma)\|_{\mathbf{H}^k}.
\end{equation*}
We also have from \Cref{Gilbarg} and the properties of $\Phi_{\leq j}$,
\begin{equation*}
\sup_{j>0}2^{-\frac{j}{2}}\|H_j^1\|_{W^{1,\infty}(\Omega)}\lesssim_A \|p\|_{C^{1,\epsilon}(\Omega)}\|v\|_{C^{\frac{1}{2}+\epsilon}(\Omega)}\lesssim_A 1,
\end{equation*}
and from the $H^{-1}\to H^1$ estimate for $\Delta^{-1}$ and \Cref{Linfest}, we have
\begin{equation*}
\sup_{j>0}2^{j(k-1)}\|H_j^2\|_{H^1(\Omega)}\lesssim_A \sup_{j>0}2^{j(k-1)}\|\nabla p\|_{L^{\infty}(\Omega)}\|\nabla\Phi_{>j}v\|_{L^2(\Omega)}\lesssim_A \|v\|_{H^k(\Omega)}.
\end{equation*}
Hence,
\begin{equation}\label{F1}
\|F_1\|_{H^{k}(\Omega)}\lesssim_A \|(v,\Gamma)\|_{\mathbf{H}^k}.
\end{equation}
By a very similar analysis, we obtain the same bound (\ref{F1}) for $F_2$. To estimate $F_3$, one uses the decomposition of $\text{tr}(\nabla v)^3$ from (\ref{cubicdecomp1}) and (\ref{cubicdecomp2}) and then partitions one of the factors $\nabla v^{\leq l}=\nabla \Phi_{< j}v^{\leq l}+\nabla \Phi_{\geq j}v^{\leq l}$. After that, an estimate similar to $F_1$ yields the bound (\ref{F1}) for the term $F_3$. Therefore,
\begin{equation*}
\|D_tp\|_{H^k(\Omega)}\lesssim_A \|(v,\Gamma)\|_{\mathbf{H}^k},
\end{equation*}
as desired. 
\end{proof}
Now, returning to the proof of (\ref{reversecoercive}), for the term $\|a^{\frac{1}{2}}\mathcal{N}^{k-1}a\|_{L^2(\Gamma)}$, we have from \Cref{baselineDN2}  and   \Cref{higherpowers},
\begin{equation*}
\|a^{\frac{1}{2}}\mathcal{N}^{k-1}a\|_{L^2(\Gamma)}\lesssim_A \|a\|_{H^{k-1}(\Gamma)}+\|a\|_{L^{\infty}(\Gamma)}\|\Gamma\|_{H^{k}}\lesssim_A \|a\|_{H^{k-1}(\Gamma)}+\|\Gamma\|_{H^{k}}.   
\end{equation*}
Then from \Cref{boundaryest}, \Cref{baltrace} and \Cref{Hestimates}, we have
\begin{equation*}
\|a\|_{H^{k-1}(\Gamma)}\lesssim_A \|p\|_{H^{k+\frac{1}{2}}(\Omega)}+\|\Gamma\|_{H^{k}}\lesssim_A \|(v,\Gamma)\|_{\mathbf{H}^k}.   
\end{equation*}
To control the other part of the energy, we  first note that by (\ref{harmonicbase}) we have 
\[
\|\nabla\mathcal{H}\mathcal{N}^{k-2}D_ta\|_{L^2(\Omega)}\lesssim_A \|\mathcal{N}^{k-2}D_ta\|_{H^{\frac{1}{2}}(\Gamma)}.
\]
Then we apply \Cref{higherpowers}, \Cref{baltrace} and \Cref{boundaryest}, in that order, to obtain
\begin{equation*}
\begin{split}
\|\mathcal{N}^{k-2}D_ta\|_{H^{\frac{1}{2}}(\Gamma)}&\lesssim_A \|D_t\nabla p\|_{H^{k-1}(\Omega)}+\|\Gamma\|_{H^{k}}\sup_{j>0}2^{-\frac{j}{2}}\|G_j^1\|_{L^{\infty}(\Omega)}+\sup_{j>0}2^{j(k-\frac{3}{2}-2\epsilon)}\|G_j^2\|_{H^{\frac{1}{2}+\epsilon}(\Omega)}
\\
&\lesssim_A \|D_t\nabla p\|_{H^{k-1}(\Omega)}+\|\Gamma\|_{H^{k}},
\end{split}
\end{equation*}
where $D_t\nabla p=G_j^1+G_j^2$ is the partition from \Cref{otherdecomp}. We then write $D_t\nabla p=-\nabla v\cdot\nabla p+\nabla D_tp$ and use \Cref{productest} and \Cref{Hestimates} to obtain
\begin{equation*}
\|D_t\nabla p\|_{H^{k-1}(\Omega)}\lesssim_A \|(v,\Gamma)\|_{\mathbf{H}^k}.    
\end{equation*}
This completes the proof of (\ref{reversecoercive}) and thus the proof of part (i) of \Cref{Energy est. thm}. Next, we turn to part (ii), which is the energy propagation bound.
\subsection{\texorpdfstring{$L^{\infty}$}{} estimates for propagation}
Now, we turn to the energy propagation bounds. As in the coercivity estimate, we will need certain $L^{\infty}$ based estimates for $p$ and $D_tp$, but in norms that have essentially $\frac{1}{2}$ more degrees of regularity compared to \Cref{Linfest}.
\begin{lemma} \label{Linfest2}
Given the assumptions of \Cref{Energy est. thm}, the following pointwise estimates for  $p$ and $D_tp$ hold.
\begin{enumerate}
\item ($C^{1,\frac{1}{2}}$ estimate for $p$).
\begin{equation*}
\|p\|_{C^{1,\frac{1}{2}}(\Omega)}\lesssim_A B.
\end{equation*}
\item ($W^{1,\infty}$ estimate for $D_tp$). Let $s\in\mathbb{R}$ with $s>\frac{d}{2}+1$. Then
\begin{equation*}
\|D_tp\|_{W^{1,\infty}(\Omega)}\lesssim_A \log(1+\|(v,\Gamma)\|_{\mathbf{H}^s})B.
\end{equation*}
\end{enumerate}
\end{lemma}
\begin{proof}
We begin with the $C^{1,\frac{1}{2}}$ estimate. We have from \Cref{Gilbarg}, using the decomposition from (\ref{decomp}) and a similar analysis to the $C^{1,\epsilon}$ estimate for $p$,
\begin{equation*}
\begin{split}
\|p\|_{C^{1,\frac{1}{2}}(\Omega)}&\lesssim_A \|\Gamma\|_{C^{1,\frac{1}{2}}}(\|p\|_{C^{1,\epsilon}(\Omega)}+\|v^l_i\partial_iv^{\leq l}_j\|_{C^{\epsilon}(\Omega)})+\|v^l_i\partial_iv^{\leq l}_j\|_{C^{\frac{1}{2}}(\Omega)}\lesssim _A \|\Gamma\|_{C^{1,\frac{1}{2}}}+\|v\|_{W^{1,\infty}(\Omega)}\lesssim_A B.
\end{split}
\end{equation*}
Now, we turn to the more difficult $W^{1,\infty}$ estimate for $D_tp$. Again, we first recall from (\ref{threeterms}) that we have
\begin{equation}\label{threeterms2}
D_tp=4\Delta^{-1}(\partial_i\partial_jp\partial_iv_j)+2\Delta^{-1}(\partial_jv_k\partial_kv_i\partial_iv_j)+\Delta^{-1}(\partial_i\partial_iv_j\partial_jp).
\end{equation}
Using a very similar analysis to \Cref{Linfest} (except without the partition of $D_tp$), we can estimate the second term in \eqref{threeterms2} in $W^{1,\infty}$ by 
\begin{equation*}
\|\Delta^{-1}(\partial_jv_k\partial_kv_i\partial_iv_j)\|_{W^{1,\infty}(\Omega)}\lesssim_A B.
\end{equation*}
For the first term in \eqref{threeterms2} we have the decomposition
\begin{equation}\label{twotermsintermediate}
\Delta^{-1}(\partial_i\partial_jp\partial_iv_j)=\Delta^{-1}(\partial_i\partial_jp^l\partial_iv_j^{\leq l})+\Delta^{-1}(\partial_i\partial_jp^{\leq l}\partial_iv_j^{l}).    
\end{equation}
The first term in (\ref{twotermsintermediate}) can be estimated similarly using \Cref{Gilbarg} by
\begin{equation}
\begin{split}
\|\Delta^{-1}(\partial_i\partial_jp^l\partial_iv_j^{\leq l})\|_{W^{1,\infty}(\Omega)}&= \|\Delta^{-1}\partial_j(\partial_ip^l\partial_iv_j^{\leq l})\|_{W^{1,\infty}(\Omega)}\lesssim_A \|p\|_{C^{1,\epsilon}(\Omega)}\|v\|_{W^{1,\infty}(\Omega)}\lesssim_A B.
\end{split}
\end{equation}
For the latter term in (\ref{twotermsintermediate}), we write
\begin{equation}
\Delta^{-1}(\partial_i\partial_jp^{\leq l}\partial_iv_j^{l})=\Delta^{-1}\partial_i(\partial_i\partial_jp^{\leq l}v_j^{l})-\Delta^{-1}\partial_j(\partial_i\partial_ip^{\leq l}v_j^{l})    
\end{equation}
and use the fact that the pressure term is at low frequency compared to $v$ and a similar analysis to the above to estimate
\begin{equation}
\|\Delta^{-1}(\partial_i\partial_jp^{\leq l}\partial_iv_j^{l})\|_{W^{1,\infty}(\Omega)}\lesssim_A B.    
\end{equation}
We now focus on the last term in (\ref{threeterms2}) which will be responsible for the logarithmic loss in the estimate. We begin by writing
\begin{equation}\label{decomp22}
\partial_i\partial_iv_j\partial_jp=\partial_i\partial_iv^l_j\partial_jp^{\leq l}+\partial_i\partial_iv^{\leq l}_j\partial_jp^{l}.
\end{equation}
For the second term on the right-hand side of (\ref{decomp22}), we write
\begin{equation*}
\partial_i\partial_iv^{\leq l}_j\partial_jp^{l}=\partial_j(\partial_i\partial_iv^{\leq l}_jp^{l}).
\end{equation*}
Again, similarly to the above, we have
\begin{equation}\label{easyest}
\|\Delta^{-1}\partial_j(\partial_i\partial_iv^{\leq l}_jp^{l})\|_{W^{1,\infty}(\Omega)}\lesssim_A B.
\end{equation}
Now, for the first term on the right of (\ref{decomp22}) we have,
\begin{equation}\label{2easy1hard}
\begin{split}
\partial_i\partial_iv^l_j\partial_jp^{\leq l}=\Delta(v^l_j\partial_jp^{\leq l})+\partial_j(v^l_j\partial_i\partial_ip^{\leq l})-2\partial_i(v^l_j\partial_j\partial_ip^{\leq l}).
\end{split}
\end{equation}
The latter two terms in \eqref{2easy1hard} are estimated similarly to (\ref{easyest}). We focus our attention on the first term, which corresponds to estimating $\Delta^{-1}\Delta (v_j^l\partial_jp^{\leq l})$ in $W^{1,\infty}$. We begin by writing
\begin{equation}\label{D-1D}
\Delta^{-1}\Delta(v^l_j\partial_jp^{\leq l})=v^l_j\partial_jp^{\leq l}-\mathcal{H}(v^l_j\partial_jp^{\leq l}).
\end{equation}
For the first term in \eqref{D-1D} we note that
\begin{equation*}
\nabla (v_j^l\partial_jp^{\leq l})=v_j^l\partial_j\nabla p^{\leq l}+\nabla v_j^l\partial_jp^{\leq l}.
\end{equation*}
From the $C^{1,\epsilon}$ bound for $p$ from \Cref{Linfest}, we  clearly have $\|v_j^l\partial_j\nabla p^{\leq l}\|_{L^{\infty}(\Omega)}\lesssim_A B.$ On the other hand, we have the same estimate for $\nabla v_j^l\partial_jp^{\leq l}$ because 
\begin{equation*}
\nabla v_j^l\partial_jp^{\leq l}=\nabla v_j\partial_j p-\nabla v_j^{\leq l}\partial_jp^{l}.
\end{equation*}
This yields the estimate $\|v_j^l\partial_j p^{\leq l}\|_{W^{1,\infty}(\Omega)}\lesssim_A B$. It remains to estimate $\mathcal{H}(v_j^l\partial_j p^{\leq l})$, which is where we incur the logarithmic loss. By the maximum principle, it suffices to estimate $\|\nabla \mathcal{H}(v_j^l\partial_j p^{\leq l})\|_{L^{\infty}(\Omega)}.$ We begin by showing that for each $m\geq 0$
\begin{equation}\label{truncated estimate}
\|\Phi_m \nabla \mathcal{H}(v_j^l\partial_j p^{\leq l})\|_{L^{\infty}(\Omega)}\lesssim_A B,
\end{equation}
 with implicit constant independent of $m$. Indeed, we have
\begin{equation*}
\|\Phi_m \nabla \mathcal{H}(v_j^l\partial_j p^{\leq l})\|_{L^{\infty}(\Omega)}\lesssim \|\Phi_m \nabla \mathcal{H}\Phi_{\leq m}(v_j^l\partial_j p^{\leq l})\|_{L^{\infty}(\Omega)}+\|\Phi_m \nabla \mathcal{H}\Phi_{> m}(v_j^l\partial_j p^{\leq l})\|_{L^{\infty}(\Omega)}.
\end{equation*}
For the first term, we have from the regularization properties of $\Phi_m$ and the $C^{1,\epsilon}$ estimate from \Cref{Gilbarg},
\begin{equation*}
\begin{split}
\|\Phi_m \nabla \mathcal{H}\Phi_{\leq m}(v_j^l\partial_j p^{\leq l})\|_{L^{\infty}(\Omega)}&\lesssim 2^{-\epsilon m}\|\mathcal{H}\Phi_{\leq m}(v_j^l\partial_j p^{\leq l})\|_{C^{1,\epsilon}(\Omega)}\lesssim_A 2^{-\epsilon m}\|\Phi_{\leq m}(v_j^l\partial_j p^{\leq l})\|_{C^{1,\epsilon}(\Omega)}
\\
&\lesssim_A \|v_j^l\partial_j p^{\leq l}\|_{W^{1,\infty}(\Omega)}.
\end{split}
\end{equation*}
Therefore, similarly to the estimate for $\nabla (v_j^l\partial_j p^{\leq l})$, we have
\begin{equation*}
\|\Phi_m \nabla \mathcal{H}\Phi_{\leq m}(v_j^l\partial_j p^{\leq l})\|_{L^{\infty}(\Omega)}\lesssim_A B.
\end{equation*}
For the other term, we have from the regularization properties of $\Phi_{\leq m}$ and $\Phi_{\geq m}$ and the maximum principle,
\begin{equation*}
\begin{split}
\|\Phi_m \nabla \mathcal{H}\Phi_{> m}(v_j^l\partial_j p^{\leq l})\|_{L^{\infty}(\Omega)}&\lesssim_A 2^{m}\|\mathcal{H}\Phi_{> m}(v_j^l\partial_j p^{\leq l})\|_{L^{\infty}(\Omega)}\leq 2^m\|\Phi_{> m}(v_j^l\partial_j p^{\leq l})\|_{L^{\infty}(\Omega)}
\\
&\lesssim_A \|v_j^l\partial_j p^{\leq l}\|_{W^{1,\infty}(\Omega)}.
\end{split}
\end{equation*}
Combining everything gives (\ref{truncated estimate}). Now, to prove the full estimate, we fix an integer $m_0>0$ to be chosen later and estimate using (\ref{truncated estimate}),
\begin{equation}\label{splitting}
\|\nabla \mathcal{H}(v_j^l\partial_j p^{\leq l})\|_{L^{\infty}(\Omega)}\lesssim_A m_0B+\|\Phi_{\geq m_0}\nabla \mathcal{H}(v_j^l\partial_j p^{\leq l})\|_{L^{\infty}(\Omega)}.
\end{equation}
For the latter term, since $s>\frac{d}{2}+1$, we obtain by Sobolev embedding, the regularization properties of $\Phi_{\geq m_0}$ and the elliptic estimate for $\mathcal{H}$, the estimate
\begin{equation*}
\|\Phi_{\geq m_0}\nabla \mathcal{H}(v_j^l\partial_j p^{\leq l})\|_{L^{\infty}(\Omega)}\lesssim_{A} 2^{-m_0\delta_0}\|\mathcal{H}(v_j^l\partial_j p^{\leq l})\|_{H^{s-\epsilon}(\Omega)}\lesssim_A 2^{-m_0\delta_0}\|(v,\Gamma)\|_{\mathbf{H}^s}^r,
\end{equation*}
where $r\geq 1$ is some integer and $\delta_0>0$ is a constant depending on $k$. Taking $m_0\approx r\delta_0^{-1}\log(1+\|(v,\Gamma)\|_{\mathbf{H}^s})$ and combining everything above with (\ref{splitting}) then yields
\begin{equation*}
\|\nabla \mathcal{H}(v_j^l\partial_j p^{\leq l})\|_{L^{\infty}(\Omega)}\lesssim_A B\log(1+\|(v,\Gamma)\|_{\mathbf{H}^s}).
\end{equation*}
This completes the proof of the lemma.
\end{proof}
\begin{remark}\label{logremovablremark}
It is perhaps worth remarking that by using \Cref{Gilbarg} and the maximum principle to estimate $\|\nabla \mathcal{H}(v_j^l\partial_j p^{\leq l})\|_{L^{\infty}(\Omega)}$ in the above proof in $C^{\epsilon}$, we can also easily obtain the bound
\begin{equation*}
\|D_tp\|_{W^{1,\infty}(\Omega)}\lesssim_A \|v\|_{C^{1,\epsilon}(\Omega)}.    
\end{equation*}
Of course, we do not want this in our energy estimates as it would force us to forfeit the scale invariant control parameter $B$. 
\end{remark}
\subsection{Proof of energy propagation}
Now, we turn to the second part of \Cref{Energy est. thm}. Using (\ref{vorteq}) and the coercivity bound \eqref{Coercivity bound on integers} it is straightforward to verify the following energy estimate for the rotational component of the energy:
\begin{equation*}\label{vortprop}
\frac{d}{dt}E^k_r(v(t),\Gamma_t)\lesssim_A BE^k(v(t),\Gamma_t).
\end{equation*}
The main bulk of the work will be in establishing a propagation bound for the irrotational part of the energy. Namely, we want to show that
\begin{equation*}\label{irrEE}
\frac{d}{dt}E_{i}^k(v(t),\Gamma_t)\lesssim_A B\log(1+\|(v,\Gamma)\|_{\mathbf{H}^s})E^k(v(t),\Gamma_t).
\end{equation*}
To do this, we start by deriving a wave-type equation for $a$.  The general procedure for deriving this equation is similar to \cite{MR3925531}. However, we  need to more precisely identify the source terms in order to obtain  estimates with the required pointwise control parameters $A$ and $B$.
\\

We begin our derivation with the simple commutator identity
\begin{equation*}
D_t\nabla p=-\nabla v\cdot\nabla p+\nabla D_tp.
\end{equation*}
Applying $D_t$ and performing some elementary algebraic manipulations gives
\begin{equation*}
\begin{split}
D_t^2\nabla p&=-\nabla D_t v\cdot\nabla p+D_t\nabla D_tp+\nabla v\cdot (\nabla v\cdot \nabla p)-\nabla v\cdot D_t\nabla p
\\
&=\frac{1}{2}\nabla|\nabla p|^2+\nabla D_t^2p+2\nabla v\cdot (\nabla v\cdot \nabla p)-2\nabla v\cdot \nabla D_tp,
\end{split}
\end{equation*}
where in the last line, we used the Euler equations to write $-\nabla D_tv\cdot\nabla p=\frac{1}{2}\nabla |\nabla p|^2$. As $\Delta p=-\text{tr}(\nabla v)^2$ is lower order, it is natural to further split $\nabla |\nabla p|^2$  as
\begin{equation*}
\frac{1}{2}\nabla |\nabla p|^2=\frac{1}{2}\nabla \mathcal{H}|\nabla p|^2+\frac{1}{2}\nabla \Delta^{-1}\Delta |\nabla p|^2.
\end{equation*}
From this, we obtain the equation
\begin{equation}\label{aeqndom}
D_t^2\nabla p-\frac{1}{2}\nabla \mathcal{H}|\nabla p|^2=\frac{1}{2}\nabla \Delta^{-1}\Delta |\nabla p|^2+\nabla D_t^2p+2\nabla v\cdot (\nabla v\cdot \nabla p)-2\nabla v\cdot \nabla D_tp=:g.
\end{equation}
It will be seen later that $g$ can be thought of as a perturbative source term. In an effort to convert  \eqref{aeqndom} into an equation for $D_ta$, we take the normal component of the trace on $\Gamma_t$ to obtain
\begin{equation}\label{qeqdom2}
D_t^2\nabla p\cdot n_{\Gamma_t}-\frac{1}{2}\mathcal{N}(a^2)=g\cdot n_{\Gamma_t},
\end{equation}
where we used the dynamic boundary condition $p|_{\Gamma_t}=0$ to write $|\nabla p_{|\Gamma_t}|^2=a^2$. Since $D_t$ is tangent to $\Gamma_t$, we have
\begin{equation}\label{qeqdom3}
\begin{split}
D_t^2a&=-D_t^2\nabla p\cdot n_{\Gamma_t}-D_t\nabla p\cdot D_t n_{\Gamma_t}=-D_t^2\nabla p\cdot n_{\Gamma_t}+a|D_tn_{\Gamma_t}|^2 .
\end{split}
\end{equation}
Note that for the latter equality in \eqref{qeqdom3}, we wrote $D_t\nabla p=-D_t(an_{\Gamma_t})$ and used that $D_tn_{\Gamma_t}$ is tangent to $\Gamma_t$. Combining \eqref{qeqdom2} and \eqref{qeqdom3}, we obtain the equation
\begin{equation*}\label{adynamics1}
D_t^2a+\frac{1}{2}\mathcal{N}(a^2)=-g\cdot n_{\Gamma_t}+a|D_tn_{\Gamma_t}|^2,
\end{equation*}
which can be further reduced using the Leibniz type formula for $\mathcal{N}$ from (\ref{DNLeibniz}) to the equation
\begin{equation}\label{adynamics}
D_t^2a+a\mathcal{N}a=f,
\end{equation}
where
\begin{equation*}
f:=-g\cdot n_{\Gamma_t}+a|D_tn_{\Gamma_t}|^2+n_{\Gamma_t}\cdot\nabla\Delta^{-1}(|\nabla\mathcal{H}a|^2).
\end{equation*}
To propagate $(a,D_ta)$ in $H^{k-1}(\Gamma_t)\times   H^{k-\frac{3}{2}}(\Gamma_t)$, one natural idea, in view of the ellipticity of $\mathcal{N}$, would be to use the spectral theorem to apply $\mathcal{N}^{k-\frac{3}{2}}$ to the above equation, and then read off the associated energy for the leading order wave-like equation. This is essentially the approach used in \cite{MR3925531}. However, there is a much better choice for our purposes, which comes from instead applying $\nabla\mathcal{H}\mathcal{N}^{k-2}$ to the above equation. The benefit to this is twofold. The most important advantage is that we only have to work with integer powers of $\mathcal{N}$, which will allow us to make use of the balanced elliptic estimates from the previous sections. Secondly, this choice allows us to reinterpret the desired estimate for $(a,D_ta)$ in $H^{k-1}(\Gamma_t)\times   H^{k-\frac{3}{2}}(\Gamma_t)$ as an $L^2$ type estimate for the linearized equation (\ref{DM lin}) with perturbative source terms. Indeed, by defining the variables
\begin{equation*}
\begin{split}
&w:=\nabla\mathcal{H} \mathcal{N}^{k-2}D_ta,
\\
&s:=\mathcal{N}^{k-1}a,
\\
&q:=\mathcal{H} (  a\mathcal{N}^{k-1}a),
\end{split}    
\end{equation*}
we may interpret $(w,s,q)$ to leading order as a solution to the linearized system (\ref{DM lin}). To verify this, note that we clearly have $\nabla\cdot w=0$. Moreover, we observe that $q|_{\Gamma_t}=as$
 and that $w|_{\Gamma_t}\cdot n_{\Gamma_t}=\mathcal{N}^{k-1}D_ta.$ Hence, 
\begin{equation*}
\begin{split}
D_ts-w|_{\Gamma_t}\cdot n_{\Gamma_t}=[D_t,\mathcal{N}^{k-1}]a=:\mathcal{R}.
\end{split}
\end{equation*}
We also note that in $\Omega_t$, by using the equation (\ref{adynamics}) for $a$  and the Leibniz formula for $\mathcal{N}$, 
\begin{equation*}\label{Qdef}
\begin{split}
D_tw +\nabla q&=\mathcal{Q},
\end{split}
\end{equation*}
where
\begin{equation}\label{def of Q}
\mathcal{Q}:=-\nabla v\cdot w+\nabla [D_t,\mathcal{H}](\mathcal{N}^{k-2}D_ta)+\nabla\mathcal{H}[D_t,\mathcal{N}^{k-2}]D_ta+\nabla\mathcal{H}\mathcal{N}^{k-2}f-\nabla\mathcal{H}[\mathcal{N}^{k-2},a]\mathcal{N}a.    
\end{equation}
To summarize the above in a compact form, we can write 
\begin{equation*}
\begin{cases}
&D_tw+\nabla q=\mathcal{Q} \ \text{in} \ \Omega_t,
\\
&\nabla\cdot w=0 \ \text{in} \ \Omega_t,
\\
&D_ts-w\cdot n_{\Gamma_t}=\mathcal{R} \ \text{on} \ \Gamma_t,
\\
&q=as \ \text{on}\ \Gamma_t.
\end{cases}    
\end{equation*}
The linearized energy estimate from \Cref{EE} combined with Cauchy-Schwarz and \Cref{Linfest2} immediately gives the preliminary bound
\begin{equation*}
\frac{d}{dt}E_{i}^k\lesssim_A B\log (1+\|(v,\Gamma)\|_{\mathbf{H}^s})E^k+(\|\mathcal{R}\|_{L^2(\Gamma_t)}+\|\mathcal{Q}\|_{L^2(\Omega_t)})(E^k)^{\frac{1}{2}}. 
\end{equation*}
It remains to control the source terms $\mathcal{Q}$ and $\mathcal{R}$. This will be where the bulk of the work is situated. Our goal is to show that
\begin{equation*}\label{QRest}
\|\mathcal{Q}\|_{L^2(\Omega_t)}+\|\mathcal{R}\|_{L^2(\Gamma_t)}\lesssim_A B\log(1+\|(v,\Gamma)\|_{\mathbf{H}^s})(E^k)^{\frac{1}{2}}.
\end{equation*}
We begin with the estimate for $\mathcal{Q}$. We proceed term by term. Clearly, we have
\begin{equation*}
\|\nabla v\cdot w\|_{L^2(\Omega_t)}\lesssim B(E^k)^{\frac{1}{2}}.
\end{equation*}
To handle the second term in the definition of $\mathcal{Q}$, we begin by recalling the simple commutator identity from (\ref{Sid}),
\begin{equation*}\label{Hcomm}
[D_t,\mathcal{H}]\psi=\Delta^{-1}\nabla\cdot\mathcal{B}(\nabla v,\nabla \mathcal{H}\psi),
\end{equation*}
where $\mathcal{B}$ is an $\mathbb{R}^d$-valued bilinear form.
We then estimate using the $H^{-1}\to H^1$ bound for $\Delta^{-1}$ to obtain 
\begin{equation*}
\|\nabla [D_t,\mathcal{H}](\mathcal{N}^{k-2}D_ta)\|_{L^2(\Omega_t)}\lesssim_A B\|\nabla\mathcal{H}\mathcal{N}^{k-2}D_ta\|_{L^2(\Omega_t)}\lesssim_A B(E^k)^{\frac{1}{2}}. 
\end{equation*}
For the third term in \eqref{def of Q}, we use the $H^{\frac{1}{2}}(\Gamma_t)\to H^1(\Omega_t)$ bound for $\mathcal{H}$ to obtain
\begin{equation*}
\|\nabla\mathcal{H}([D_t,\mathcal{N}^{k-2}]D_ta)\|_{L^2(\Omega_t)}\lesssim_A \|[D_t,\mathcal{N}^{k-2}]D_ta\|_{H^{\frac{1}{2}}(\Gamma_t)}.
\end{equation*}
Then, from the commutator estimate \Cref{materialcom} we obtain 
\begin{equation*}
\begin{split}
\|[D_t,\mathcal{N}^{k-2}]D_ta\|_{H^{\frac{1}{2}}(\Gamma_t)}\lesssim_A & \ \|v\|_{H^k(\Omega_t)}\|D_ta\|_{L^{\infty}(\Gamma_t)}+\|v\|_{W^{1,\infty}(\Omega_t)}\|D_ta\|_{H^{k-\frac{3}{2}}(\Gamma_t)}+\|D_ta\|_{L^{\infty}(\Gamma_t)}\|\Gamma\|_{H^k}
\\
&+\|v\|_{W^{1,\infty}(\Omega_t)}\|\Gamma\|_{H^k(\Omega_t)}\sup_{j>0}2^{-\frac{j}{2}}\|G_j^1\cdot n_{\Gamma_t}\|_{L^{\infty}(\Gamma_t)}
\\
& +\|v\|_{W^{1,\infty}(\Omega_t)}\sup_{j>0}2^{j(k-\frac{3}{2}-2\epsilon)}\|G_j^2\cdot n_{\Gamma_t}\|_{H^{\epsilon}(\Gamma_t)},
\end{split}
\end{equation*}
where $G_j^1$ and $G_j^2$ are as in \Cref{otherdecomp}. Using \Cref{Linfest2}, the energy coercivity, \Cref{aest} and (\ref{strongercoercivitybound}), we have
\begin{equation*}
\|[D_t,\mathcal{N}^{k-2}]D_ta\|_{H^{\frac{1}{2}}(\Gamma_t)}\lesssim_A B\log(1+\|(v,\Gamma)\|_{\mathbf{H}^s})(E^k)^{\frac{1}{2}}.    
\end{equation*}
Next, we turn to the estimate for $\nabla\mathcal{H}\mathcal{N}^{k-2}f$, which involves the most work. We recall that
\begin{equation*}
f:=-g\cdot n_{\Gamma_t}+a|D_tn_{\Gamma_t}|^2+\nabla_n\Delta^{-1}(|\nabla\mathcal{H}a|^2),
\end{equation*}
where $g$ is defined as in (\ref{aeqndom}). Using the identities $D_tn_{\Gamma_t}=-((Dv)^*n_{\Gamma_t})^{\top}=-(Dv)^*n_{\Gamma_t}+n_{\Gamma_t}(n_{\Gamma_t}\cdot (Dv)^*n_{\Gamma_t}) $ and $|\nabla\mathcal{H}a|^2=\frac{1}{2}\Delta |\mathcal{H}a|^2$, we may reorganize $f$ into the expression
\begin{equation}\label{fsimp}
f=\frac{1}{2}\nabla_n \Delta^{-1}\Delta (\mathcal{H}a)^2-\frac{1}{2}\nabla_n \Delta^{-1}\Delta |\nabla p|^2-\nabla_n D_t^2p+M_1+M_2,
\end{equation}
where $M_1$ is a multilinear expression in $n_{\Gamma_t}$, $\nabla p$, $\nabla v$ with exactly two factors of $\nabla v$ (e.g., from (\ref{Moving normal}), the term $a|D_tn_{\Gamma_t}|^2$), and $M_2$ is a multilinear expression in $\nabla p$, $\nabla v$, $\nabla D_tp$ and $n_{\Gamma_t}$ with a single factor of each of $\nabla D_tp$ and $\nabla v$ (e.g.,~the term $n_{\Gamma_t}\cdot \nabla D_tp\cdot \nabla v$). We will abuse notation slightly and refer to terms of the first type as $M_1(\nabla v,\nabla v)$ and terms of the second type as $M_2(\nabla D_tp,\nabla v)$. Next, we estimate each term in $\nabla\mathcal{H}\mathcal{N}^{k-2}f$, with the  expression \eqref{fsimp} for $f$ substituted in.
\\

From \Cref{EEcorollary}, we have
\begin{equation*}
\begin{split}
\|\nabla\mathcal{H}\mathcal{N}^{k-2}\nabla_n\Delta^{-1}\Delta(\mathcal{H}a)^2\|_{L^2(\Omega_t)}&\lesssim_A \|\mathcal{N}^{k-2}\nabla_n\Delta^{-1}\Delta(\mathcal{H}a)^2\|_{H^{\frac{1}{2}}(\Gamma_t)}
\\
&\lesssim_A \|\Gamma_t\|_{H^k}\|\Delta^{-1}\Delta(\mathcal{H}a)^2\|_{C^{\frac{1}{2}}(\Omega_t)}+\|\Delta^{-1}\Delta(\mathcal{H}a)^2\|_{H^{k}(\Omega_t)}.
\end{split}
\end{equation*}
By writing $\Delta^{-1}\Delta(\mathcal{H}a)^2=(\mathcal{H}a)^2-\mathcal{H}(\mathcal{H}a)^2$ and using the $C^{\frac{1}{2}}$ estimate for $\mathcal{H}$ from \Cref{Hboundlow} twice together with the maximum principle, we have
\begin{equation*}
\|\Delta^{-1}\Delta(\mathcal{H}a)^2\|_{C^{\frac{1}{2}}(\Omega_t)}\lesssim_A \|\mathcal{H}a\|_{L^{\infty}(\Omega_t)}\|\mathcal{H}a\|_{C^{\frac{1}{2}}(\Omega_t)}\lesssim_A \|a\|_{C^{\frac{1}{2}}(\Gamma_t)}\lesssim_A B.
\end{equation*}
 From \Cref{direst}, we obtain also
\begin{equation*}
\|\Delta^{-1}\Delta(\mathcal{H}a)^2\|_{H^{k}(\Omega_t)}\lesssim_A B\|\Gamma_t\|_{H^k}+\|\Delta(\mathcal{H}a)^2\|_{H^{k-2}(\Omega_t)}.
\end{equation*}
Then using that $\Delta(\mathcal{H}a)^2=2|\nabla\mathcal{H}a|^2$, we obtain from \Cref{productest},
\begin{equation*}
\begin{split}
\|\Delta(\mathcal{H}a)^2\|_{H^{k-2}(\Omega_t)}&\lesssim \|\mathcal{H}a\|_{C^{\frac{1}{2}}(\Omega_t)}\|\mathcal{H}a\|_{H^{k-\frac{1}{2}}(\Omega_t)}\lesssim_A B\|\mathcal{H}a\|_{H^{k-\frac{1}{2}}(\Omega_t)}.
\end{split}
\end{equation*}
Then from \Cref{Hbounds}, \Cref{aest} and the energy coercivity bound (\ref{strongercoercivitybound}), we obtain
\begin{equation*}
\begin{split}
\|\mathcal{H}a\|_{H^{k-\frac{1}{2}}(\Omega_t)}\lesssim_A \|a\|_{H^{k-1}(\Gamma_t)}+\|\Gamma\|_{H^k}\|a\|_{L^{\infty}(\Omega_t)}\lesssim_A (E^k)^{\frac{1}{2}}.
\end{split}
\end{equation*}
Therefore, 
\begin{equation*}
\|\Delta(\mathcal{H}a)^2\|_{H^{k-2}(\Omega_t)}\lesssim_A B(E^k)^{\frac{1}{2}}.
\end{equation*}
Next, we turn to the term $\nabla_n\Delta^{-1}\Delta |\nabla p|^2$ in \eqref{fsimp}. The procedure here is similar. Like with the previous estimate, we obtain
\begin{equation}\label{EstimateDp2}
\|\nabla\mathcal{H}\mathcal{N}^{k-2}\nabla_n\Delta^{-1}\Delta |\nabla p|^2\|_{L^2(\Omega_t)}\lesssim_A \|\Gamma_t\|_{H^k}\|\Delta^{-1}\Delta(|\nabla p|^2)\|_{C^{\frac{1}{2}}(\Omega_t)}+\|\Delta(|\nabla p|^2)\|_{H^{k-2}(\Omega_t)}
\end{equation}
and also
\begin{equation*}
\|\Delta^{-1}\Delta(|\nabla p|^2)\|_{C^{\frac{1}{2}}(\Omega_t)}\lesssim_A B.
\end{equation*}
Moreover, by expanding $\Delta|\nabla p|^2$ (and some simple manipulations), we have 
\begin{equation*}
\begin{split}
\|\Delta(|\nabla p|^2)\|_{H^{k-2}(\Omega_t)}&\lesssim \||\nabla^2p|^2\|_{H^{k-2}(\Omega_t)}+\||\Delta p|^2\|_{H^{k-2}(\Omega_t)}+\|\nabla p\Delta p\|_{H^{k-1}(\Omega_t)}.
\end{split}
\end{equation*}
Using \Cref{productest} and \Cref{Linfest2}, we have for the first two terms
\begin{equation*}
\||\nabla^2p|^2\|_{H^{k-2}(\Omega_t)}+\||\Delta p|^2\|_{H^{k-2}(\Omega_t)}\lesssim_A \|\nabla p\|_{C^{\frac{1}{2}}(\Omega_t)}\|p\|_{H^{k+\frac{1}{2}}(\Omega_t)}\lesssim_A B \|p\|_{H^{k+\frac{1}{2}}(\Omega_t)}.    
\end{equation*}
To handle the other term, we use the Laplace equation for $p$ to write
\begin{equation}
\|\nabla p\Delta p\|_{H^{k-1}(\Omega_t)}=\|\nabla p\partial_iv_j\partial_jv_i\|_{H^{k-1}(\Omega_t)}.    
\end{equation}
Then from \eqref{algebraprop}, \Cref{productest}, \Cref{Linfest} and \Cref{Hestimates}, we have
\begin{equation}
\begin{split}
\|\nabla p\partial_iv_j\partial_jv_i\|_{H^{k-1}(\Omega_t)}&\lesssim_A \|v\|_{W^{1,\infty}(\Omega_t)}\|\nabla p\partial_iv_j\|_{H^{k-1}(\Omega_t)}+\|\nabla p\partial_iv_j\|_{L^{\infty}(\Omega_t)}\|v\|_{H^k(\Omega_t)}
\\
&\lesssim_A \|v\|_{W^{1,\infty}(\Omega_t)}(\|v\|_{C^{\frac{1}{2}}(\Omega_t)}\|p\|_{H^{k+\frac{1}{2}}(\Omega_t)}+\|p\|_{W^{1,\infty}(\Omega_t)}\|v\|_{H^k(\Omega_t)})
\\
&\lesssim_A B\|(v,\Gamma)\|_{\mathbf{H}^k}.
\end{split}
\end{equation}
Combining the above with the energy coercivity (\ref{strongercoercivitybound}), we obtain
\begin{equation*}
\begin{split}
\|\Delta(|\nabla p|^2)\|_{H^{k-2}(\Omega_t)}\lesssim_A B(E^k)^{\frac{1}{2}}.
\end{split}
\end{equation*}
Next, we turn to the estimate for $M_1$. We first write $M_1=M_1'\mathcal{B}$ where $M_1'$ is an $\mathbb{R}$-valued multilinear expression in $n_{\Gamma_t}$ and $\nabla p$ and $\mathcal{B}$ is an $\mathbb{R}$-valued bilinear expression in $\nabla v$. We use the bilinear frequency decomposition $\mathcal{B}(\nabla v,\nabla v)=\mathcal{B}(\nabla v^l,\nabla v^{\leq l})+\mathcal{B}(\nabla v^{\leq l},\nabla v^l)$ and consider the partition $\mathcal{B}=\mathcal{B}_j^1+\mathcal{B}_j^2$ where $\mathcal{B}_j^1:=\mathcal{B}(\nabla \Phi_{<j}v^l,\nabla v^{\leq l})+\mathcal{B}(\nabla v^{\leq l},\nabla \Phi_{<j}v^l)$. Then using this partition, the trace inequality, energy coercivity and \Cref{higherpowers}, we have
\begin{equation}\label{firstM1}
\begin{split}
\|\nabla\mathcal{H}\mathcal{N}^{k-2}M_1\|_{L^2(\Omega_t)}&\lesssim_A \|M_1\|_{H^{k-\frac{3}{2}}(\Gamma_t)}+\|\Gamma_t\|_{H^k}\sup_{j>0}2^{-\frac{j}{2}}\|\mathcal{B}_j^1\|_{L^{\infty}(\Omega_t)}+\sup_{j>0}2^{j(k-\frac{3}{2}-2\epsilon)}\|\mathcal{B}_j^2\|_{H^{\frac{1}{2}+\epsilon}(\Omega_t)}
\\
&\lesssim_A \|M_1\|_{H^{k-\frac{3}{2}}(\Gamma_t)}+\|v\|_{W^{1,\infty}(\Omega_t)}\|v\|_{C^{\frac{1}{2}+\epsilon}(\Omega_t)}\|\Gamma_t\|_{H^k}+\|v\|_{W^{1,\infty}(\Omega_t)}\|v\|_{H^{k}(\Omega_t)}
\\
&\lesssim_A \|M_1\|_{H^{k-\frac{3}{2}}(\Gamma_t)}+B(E^k)^{\frac{1}{2}}.
\end{split}
\end{equation}
Using the same partition as above and \Cref{boundaryest}, \Cref{baltrace} and \Cref{Linfest}, we have 
\begin{equation*}
\begin{split}
\|M_1\|_{H^{k-\frac{3}{2}}(\Gamma_t)}&\lesssim_A \|\nabla v\|_{L^{\infty}(\Omega_t)}\|v\|_{H^k(\Omega_t)}+(\|\Gamma_t\|_{H^k}+\|M_1'(\nabla p,n_{\Gamma_t})\|_{H^{k-1}(\Gamma_t)})\sup_{j>0}2^{-\frac{j}{2}}\|\mathcal{B}_j^1\|_{L^{\infty}(\Omega_t)}
\\
&+\sup_{j>0}2^{j(k-\frac{3}{2}-2\epsilon)}\|\mathcal{B}_j^2\|_{\mathcal{H}^{\frac{1}{2}+\epsilon}(\Omega_t)}.
\end{split}
\end{equation*}
Estimating as in (\ref{firstM1}), this simplifies to 
\begin{equation*}
\|M_1\|_{H^{k-\frac{3}{2}}(\Gamma_t)}\lesssim_A B(E^k)^{\frac{1}{2}}+B\|M_1'(\nabla p,n_{\Gamma_t})\|_{H^{k-1}(\Gamma_t)}.    
\end{equation*}
By \Cref{boundaryest}, \Cref{baltrace}, \Cref{Hestimates} and the energy coercivity, we have also
\begin{equation*}
\|M_1'(\nabla p,n_{\Gamma_t})\|_{H^{k-1}(\Gamma_t)}\lesssim_A (E^k)^{\frac{1}{2}},    
\end{equation*}
from which we deduce
\begin{equation*}
\|\nabla\mathcal{H}\mathcal{N}^{k-2}M_1\|_{L^2(\Omega_t)}\lesssim_A B(E^k)^{\frac{1}{2}}.
\end{equation*}
Next, we estimate $M_2$. This estimate is similar to $M_1$. One starts by writing $M_2=M_2'\mathcal{B}$ where $M_2'$ is multilinear in $\nabla p$ and $n_{\Gamma_t}$ while $\mathcal{B}$ is bilinear in $\nabla v$ and $\nabla D_tp$. Using the partition $\mathcal{B}=\mathcal{B}_j^1+\mathcal{B}_j^2$ with $\mathcal{B}_j^1:=\mathcal{B}(\nabla \Phi_{<j}v^l,\nabla (D_tp)^{\leq l})+\mathcal{B}(\nabla v^{\leq l},\nabla\Phi_{<j} (D_tp)^{l})$ and a similar analysis to $M_1$, we have
\begin{equation*}
\begin{split}
\|\nabla\mathcal{H}\mathcal{N}^{k-2}M_2\|_{L^2(\Omega_t)}&\lesssim_A \| v\|_{W^{1,\infty}(\Omega_t)}\|D_tp\|_{H^{k}(\Omega_t)}+\|D_tp\|_{W^{1,\infty}(\Omega_t)}(\|\Gamma_t\|_{H^k}+\|M_2'(\nabla p,n_{\Gamma_t})\|_{H^{k-1}(\Gamma_t)})
\\
&+\|D_tp\|_{W^{1,\infty}(\Omega_t)}\|v\|_{H^k(\Omega_t)}.
\end{split}
\end{equation*}
Then using the $W^{1,\infty}$ bound for $D_tp$ from \Cref{Linfest2} and the $H^{k}$ bound for $D_tp$ from \Cref{Hestimates}, we have
\begin{equation*}
\|\nabla\mathcal{H}\mathcal{N}^{k-2}M_2\|_{L^2(\Omega_t)}\lesssim_A B\log(1+\|(v,\Gamma)\|_{\mathbf{H}^s})(E^k)^{\frac{1}{2}}.
\end{equation*}
Now we turn to the estimate for the term involving $D_t^2p$. As usual, we first aim to write it in the form $\Delta^{-1}\nabla\cdot f$ but in such a way that $f$ involves favorable frequency interactions. This presents some mild technical challenges as $D_t^2p$ will have terms which are up to quadrilinear in $\nabla v$. To deal with this, we have the following lemma. 
\begin{lemma}\label{goodform}
There exist bilinear, trilinear and quadrilinear expressions $\mathcal{B}$, $\mathcal{T}$ and $\mathcal{M}$ taking values in $\mathbb{R}^d$ such that
\begin{equation*}
\Delta D_t^2p=-2\Delta|\nabla p|^2+\nabla\cdot \mathcal{B}(\nabla D_tp,\nabla v)+\nabla\cdot \mathcal{T}(\nabla p,\nabla v,\nabla v)+\nabla\cdot \mathcal{M}(v^{m},\nabla v^{\leq m},\nabla v^{\leq m},\nabla v^{\leq m}).
\end{equation*}
\end{lemma}
\begin{proof}
First, using that $v$ is divergence free, it is straightforward to verify
\begin{equation*}
\Delta D_t^2p=\partial_i(\partial_jD_tp\partial_jv_i)+\partial_i(\partial_iv_j\partial_jD_tp)+D_t\Delta D_tp=\nabla\cdot \mathcal{B}+D_t\Delta D_tp.
\end{equation*}
Next, we expand $D_t\Delta D_tp$. We start with the Laplace equation for $D_tp$ from (\ref{Dtpid}),
\begin{equation*}
\begin{split}
\Delta D_tp=3\partial_j(\partial_ip\partial_iv_j)+\partial_i(\partial_iv_j\partial_jp)+2\partial_jv_k\partial_kv_i\partial_iv_j.
\end{split}
\end{equation*}
Using that $v$ is divergence free, we have the commutator identity $[\partial_i,D_t]f=\partial_j(\partial_iv_jf)$. Combining this with the Euler equations, we obtain
\begin{equation*}
\begin{split}
D_t(3\partial_j(\partial_ip\partial_iv_j)+\partial_i(\partial_iv_j\partial_jp))&=\nabla\cdot \mathcal{B}+\nabla\cdot \mathcal{T}-4\partial_j(\partial_ip\partial_i\partial_jp)
\\
&=\nabla\cdot \mathcal{B}+\nabla\cdot \mathcal{T}-2\Delta |\nabla p|^2.
\end{split}
\end{equation*}
It remains to expand $2D_t(\partial_j v_k\partial_k v_i\partial_i v_j)$. From the Euler equation and symmetry, we have
\begin{equation*}\label{quad}
\begin{split}
2D_t(\partial_j v_k\partial_k v_i\partial_i v_j)&=6D_t(\partial_j v_k)\partial_k v_i\partial_i v_j=-6\partial_j\partial_kp\partial_kv_i\partial_iv_j-6\partial_j v_l\partial_lv_k\partial_kv_i\partial_iv_j.
\end{split}
\end{equation*}
We rearrange the first term as
\begin{equation}\label{identity1}
\begin{split}
-6\partial_j\partial_kp\partial_kv_i\partial_iv_j&=-6\partial_j(\partial_kp\partial_kv_i\partial_iv_j)+6\partial_kp\partial_j\partial_kv_i\partial_i v_j=-6\partial_j(\partial_kp\partial_kv_i\partial_iv_j)+3\partial_kp\partial_k(\partial_jv_i\partial_i v_j)
\\
&=-6\partial_j(\partial_kp\partial_kv_i\partial_iv_j)+3\partial_k(\partial_kp\partial_jv_i\partial_i v_j)-3\partial_k\partial_kp\partial_jv_i\partial_i v_j
\\
&=\nabla\cdot \mathcal{T}+3|\Delta p|^2,
\end{split}
\end{equation}
where in the last line we used the Laplace equation for $p$. On the other hand, for the second term, by symmetry of the indices, we have the quadrilinear frequency decomposition, 
\begin{equation*}
\begin{split}
-6\partial_j v_l\partial_lv_k\partial_kv_i\partial_iv_j&=-24\partial_j v_l^{m}\partial_lv_k^{\leq m}\partial_kv_i^{\leq m}\partial_iv_j^{\leq m}
\\
&=\nabla\cdot \mathcal{M}+24v_l^{m}\partial_l\partial_jv_k^{\leq m}\partial_kv_i^{\leq m}\partial_iv_j^{\leq m}+24v_l^{m}\partial_lv_k^{\leq m}\partial_j\partial_kv_i^{\leq m}\partial_iv_j^{\leq m}.
\end{split}
\end{equation*}
By symmetry and the fact that $v$ is divergence free, the second term on the right-hand side can be rearranged as
\begin{equation*}
24v_l^{m}\partial_l\partial_jv_k^{\leq m}\partial_kv_i^{\leq m}\partial_iv_j^{\leq m}=8v_{l}^{m}\partial_l(\partial_jv_k^{\leq m}\partial_kv_i^{\leq m}\partial_iv_j^{\leq m})=\nabla\cdot \mathcal{M}.
\end{equation*}
For the third term on the right-hand side, we have
\begin{equation}\label{identity2}
\begin{split}
24v_l^{m}\partial_lv_k^{\leq m}\partial_j\partial_kv_i^{\leq m}\partial_iv_j^{\leq m}&=12v_l^{m}\partial_lv_k^{\leq m}\partial_k(\partial_jv_i^{\leq m}\partial_iv_j^{\leq m})=\nabla\cdot \mathcal{M}-12\partial_k v_l^{m}\partial_lv_k^{\leq m}\partial_jv_i^{\leq m}\partial_iv_j^{\leq m}
\\
&=\nabla\cdot \mathcal{M}-3\partial_k v_l\partial_lv_k\partial_jv_i\partial_iv_j
\\
&=\nabla\cdot \mathcal{M}-3|\Delta p|^2,
\end{split}
\end{equation}
where we used the Laplace equation for $p$ in the last line. Combining (\ref{identity1}) and (\ref{identity2}) to cancel the $3|\Delta p|^2$ terms then completes the proof of the lemma.
\end{proof}
Now, we return to the estimate for $\nabla\mathcal{H}\mathcal{N}^{k-2}\nabla_n D_t^2p$. We use \Cref{goodform} and estimate each term separately. The term $-2\nabla\mathcal{H}\mathcal{N}^{k-2}\nabla_n\Delta^{-1}\Delta|\nabla p|^2$ can be estimated identically to (\ref{EstimateDp2}). Let us then turn to the estimate for $\nabla\mathcal{H}\mathcal{N}^{k-2}\nabla_n\Delta^{-1}(\nabla\cdot \mathcal{B})$. We use a partition $\mathcal{B}=\mathcal{B}_j^1+\mathcal{B}_j^2$ where $\mathcal{B}_j^1$ is defined as follows: First, we perform the frequency decomposition,
\begin{equation*}
\mathcal{B}=\mathcal{B}(\nabla (D_tp)^{l},\nabla v^{\leq l})+\mathcal{B}(\nabla (D_tp)^{\leq l},\nabla v^{l})
\end{equation*}
and then define
\begin{equation*}
\mathcal{B}_1^j:=\mathcal{B}(\nabla \Phi_{\leq j}(D_tp)^{l},\nabla v^{\leq l})+\mathcal{B}(\nabla (D_tp)^{\leq l},\nabla \Phi_{\leq j} v^{l}).
\end{equation*}
Then \Cref{EEcorollary} and \Cref{direst} gives 
\begin{equation*}
\begin{split}
\|\nabla\mathcal{H}\mathcal{N}^{k-2}\nabla_n\Delta^{-1}(\nabla\cdot \mathcal{B})\|_{L^2(\Omega_t)}&\lesssim_A \|\mathcal{B}\|_{H^{k-1}(\Omega_t)}+\|\Gamma_t\|_{H^k}\sup_{j>0}2^{-\frac{j}{2}}\|\Delta^{-1}(\nabla\cdot\mathcal{B}_1^j)\|_{W^{1,\infty}(\Omega_t)}
\\
&+\sup_{j>0}2^{j(k-1-\epsilon)}\|\Delta^{-1}(\nabla\cdot\mathcal{B}_2^j)\|_{H^1(\Omega_t)}.
\end{split}
\end{equation*}
From Sobolev product estimates and the $H^k$ and $L^{\infty}$ estimates for $D_tp$,
\begin{equation*}
\|\mathcal{B}\|_{H^{k-1}(\Omega_t)}\lesssim_A \|v\|_{W^{1,\infty}(\Omega_t)}\|D_tp\|_{H^k(\Omega_t)}+\|D_tp\|_{W^{1,\infty}(\Omega_t)}\|v\|_{H^k(\Omega_t)}\lesssim_A B\log(1+\|(v,\Gamma)\|_{\mathbf{H}^s})(E^k)^{\frac{1}{2}}. 
\end{equation*}
Using \Cref{Gilbarg}, we also estimate
\begin{equation*}
2^{-\frac{j}{2}}\|\Delta^{-1}(\nabla\cdot\mathcal{B}_1^j)\|_{C^{1,\epsilon}(\Omega_t)}\lesssim_A \|D_tp\|_{W^{1,\infty}(\Omega_t)}\|v\|_{C^{\frac{1}{2}+\epsilon}(\Omega_t)}\lesssim_A B\log(1+\|(v,\Gamma)\|_{\mathbf{H}^s}).
\end{equation*}
Finally,  using the error bounds for $\Phi_{>j}$ and the $L^{\infty}$ and $H^k$ estimates for $D_tp$ from \Cref{Hestimates} we see that
\begin{equation*}
2^{j(k-1-\epsilon)}\|\Delta^{-1}(\nabla\cdot\mathcal{B}_2^j)\|_{H^1(\Omega_t)}\lesssim_A \|v\|_{W^{1,\infty}(\Omega_t)}\|D_tp\|_{H^k(\Omega_t)}+\|D_tp\|_{W^{1,\infty}(\Omega_t)}\|v\|_{H^k(\Omega_t)}\lesssim_A B\log(1+\|(v,\Gamma)\|_{\mathbf{H}^s})(E^k)^{\frac{1}{2}}.
\end{equation*}
Hence,
\begin{equation*}
\|\nabla\mathcal{H}\mathcal{N}^{k-2}\nabla_n\Delta^{-1}(\nabla\cdot \mathcal{B})\|_{L^2(\Omega_t)}\lesssim_A B\log(1+\|(v,\Gamma)\|_{\mathbf{H}^s})(E^k)^{\frac{1}{2}}.
\end{equation*}
The estimates for $\nabla\mathcal{H}\mathcal{N}^{k-2}\nabla_n\Delta^{-1}(\nabla\cdot \mathcal{T})$ and $\nabla\mathcal{H}\mathcal{N}^{k-2}\nabla_n\Delta^{-1}(\nabla\cdot \mathcal{M})$ are very similar. The main difference is that we use the partition $\mathcal{T}=\mathcal{T}_1^j+\mathcal{T}_2^j$ with
\begin{equation*}
\mathcal{T}_1^j=2\mathcal{T}(\nabla p,\nabla\Phi_{\leq j}v^{l},\nabla v^{\leq l})   
\end{equation*}
and the partition $\mathcal{M}=\mathcal{M}_1^j+\mathcal{M}_2^j$ with
\begin{equation*}
\mathcal{M}_1^j:=\mathcal{M}(v^{m},\nabla \Phi_{\leq j}v^{\leq m},\nabla v^{\leq m},\nabla v^{\leq m}).
\end{equation*}
Ultimately, we obtain
\begin{equation*}
\|\nabla\mathcal{H}\mathcal{N}^{k-2}\nabla_n D_t^2p\|_{L^2(\Omega_t)}\lesssim_A B\log(1+\|(v,\Gamma)\|_{\mathbf{H}^s})(E^k)^{\frac{1}{2}}
\end{equation*}
which when combined with the previous analysis gives
\begin{equation*}
\|\nabla\mathcal{H}\mathcal{N}^{k-2}f\|_{L^2(\Omega_t)}\lesssim_A B\log(1+\|(v,\Gamma)\|_{\mathbf{H}^s})(E^k)^{\frac{1}{2}}
\end{equation*}
as desired. The last term in the estimate for $\mathcal{Q}$ that we need to control is $\nabla\mathcal{H}[\mathcal{N}^{k-2},a]\mathcal{N}a$. For this, we have the following technical lemma.
\begin{lemma}\label{NLeibniz} We have the following estimate:
\begin{equation*}
\|\nabla\mathcal{H}[\mathcal{N}^{k-2},a]\mathcal{N}a\|_{L^2(\Omega_t)}\lesssim_A B\log(1+\|(v,\Gamma)\|_{\mathbf{H}^s})(E^k)^{\frac{1}{2}}.    
\end{equation*}
\end{lemma}
\begin{proof}
Thanks to the $H^{\frac{1}{2}}(\Gamma_t)\to H^1(\Omega_t)$ bound for $\mathcal{H}$, it suffices to estimate $\|[\mathcal{N}^{k-2},a]\mathcal{N}a\|_{H^{\frac{1}{2}}(\Gamma_t)}$. We begin by using the Leibniz formula (\ref{DNLeibniz}) to expand the commutator,
\begin{equation}\label{commutator2}
\begin{split}
[\mathcal{N}^{k-2},a]\mathcal{N}a=\sum_{n+m=k-3}\mathcal{N}^n(\mathcal{N}a\mathcal{N}^{m+1}a)-2\mathcal{N}^n\nabla_n \Delta^{-1}(\nabla\mathcal{H}a\cdot \nabla \mathcal{H}\mathcal{N}^{m+1}a).
\end{split}
\end{equation}
We focus on the latter term in (\ref{commutator2}) first as it is a bit more delicate to deal with. To simplify notation slightly, we write
\begin{equation*}
a_j:=\mathcal{H}\mathcal{N}^{j}a,\hspace{5mm} F:=\nabla a_0\cdot\nabla a_{m+1},\hspace{5mm} \mathcal{N}_{<j}:=n_{\Gamma_t}\cdot\nabla \Phi_{<j}\mathcal{H},\hspace{5mm} \mathcal{N}_{\geq j}:=n_{\Gamma_t}\cdot\nabla \Phi_{\geq j}\mathcal{H}.    
\end{equation*}
Using \Cref{EEcorollary} and then \Cref{direst}, we have
\begin{equation*}
\begin{split}
\|\mathcal{N}^n(\nabla_n\Delta^{-1}F)\|_{H^{\frac{1}{2}}(\Gamma_t)}&\lesssim_A \|F\|_{H^{n}(\Omega_t)}+\|\Gamma\|_{H^{k}}\sup_{j>0}2^{-j(m+\frac{3}{2})}\|\Delta^{-1} F_j^1\|_{W^{1,\infty}(\Omega_t)}
\\
&+\sup_{j>0}2^{j(n+1-\epsilon)}\|\Delta^{-1}F_j^2\|_{H^1(\Omega_t)},  
\end{split}
 \end{equation*} 
where $F=F_j^1+F_j^2$ is a suitable partition of $F$ to be chosen. To find a suitable partition, we start with a bilinear frequency decomposition similar to before. We define $a_j^l:=\Phi_{l}a_{j}$ and $a_j^{\leq l}=\Phi_{\leq 
 l}a_{j}$. 
\begin{remark}\label{phicommutator}
 We note that the regularization operator $\Phi_{\leq l}$ does not preserve the harmonic property of $a_j$. However, using the definition of $\Phi_{\leq l}$ (see \Cref{SSRO}), the operator defined by $C_{\leq l}:=[\Delta,\Phi_{\leq l}]$ is readily seen to satisfy the bounds,
\begin{equation}\label{phicombounds}
\begin{split}
\|C_{\leq l}\|_{C^{\alpha}\to L^{\infty}}\lesssim_A 2^{l(1-\alpha)}\hspace{5mm}\|C_{\leq l}\|_{H^{\alpha}\to L^2}&\lesssim_A 2^{l(1-\alpha)},\hspace{5mm} 0\leq \alpha\leq 1 
\end{split}
\end{equation}
for $\alpha,l\geq 0$.  That is, $C_{\leq l}$ behaves like a differential operator of order $1$ localized at dyadic scale $\lesssim 2^l$.
 \end{remark}
 Now, using the same convention as before in this section (where repeated indices are summed over) we have
\begin{equation*}\label{freqdecompGH}
\begin{split}
F&=\nabla a_0^{l}\cdot\nabla a_{m+1}^{\leq l}+\nabla a_0^{\leq l}\cdot\nabla a_{m+1}^{l}=:F'+F''.
\end{split}
\end{equation*}
We can write $F'$ and $F''$ to leading order as the divergence of some vector field. Using that $a_0$ and $a_{m+1}$ are harmonic, we have
\begin{equation}
\begin{split}
F'&=\nabla\cdot (a_0^l\nabla a_{m+1}^{\leq l})-a_0^lC_{\leq l}a_{m+1}=:G'+H',
\\
F''&=\nabla\cdot (a_{m+1}^{l}\nabla a_0^{\leq l})-a_{m+1}^lC_{\leq l}a_{0}=:G''+H''.
\end{split}
\end{equation}
We will focus on $F'$ first. To choose a partition of $F'$, we need to choose a suitable partition of $G'$ and $H'$. We show the details for $G'$ and remark later on the minor changes needed to deal with $H'$. We write $G'=(G')_j^1+(G')_j^2$ with
\begin{equation*}
(G')_j^1=\nabla\cdot (a_0^l\nabla \Phi_{\leq l}a_{m+1,\leq j}),\hspace{5mm} a_{m+1,\leq j}:=\Phi_{\leq j}(\mathcal{H}\mathcal{N}^{m+1}_{<j}a).    
\end{equation*}
From \Cref{Gilbarg},  iterating the maximum principal and using the $C^{\alpha}$ bounds for $\mathcal{H}$ and the properties of $\Phi_{<j}$, we have
 \begin{equation*}
 2^{-j(m+\frac{3}{2})}\|\Delta^{-1}(G')_j^1\|_{W^{1,\infty}(\Omega_t)}\lesssim_A \|a\|_{C^{\epsilon}(\Gamma_t)}\|a\|_{C^{\frac{1}{2}}(\Gamma_t)}\lesssim_A B,
 \end{equation*}
 where we used \Cref{Linfest} and \Cref{Linfest2} in the last inequality. For $(G')_j^2$, we can write
 \begin{equation*}
 (G')_j^2=\nabla\cdot (a_0^l\nabla b_{m+1,j}^{\leq l})+\sum_{0\leq i\leq m}\nabla\cdot (a_0^l\nabla b_{i,j}^{\leq l})
 \end{equation*}
 where
 \begin{equation*}
b_{m+1,j}^{\leq l}:=\Phi_{\leq l}\Phi_{\geq j}a_{m+1},\hspace{5mm}   b_{i,j}^{\leq l}:=\Phi_{\leq l}\Phi_{<j}\mathcal{H}\mathcal{N}^i_{<j}\mathcal{N}_{\geq j}\mathcal{N}^{m-i}a.  
 \end{equation*}
 Using \Cref{Hboundlow}, the properties of the kernel $\Phi$ and the $H^{-1}\to H^1$  bound for $\Delta^{-1}$, we obtain for each $0\leq i\leq m$,
\begin{equation}\label{commutatordetials}
\begin{split}
 2^{j(n+1-\epsilon)}\|\Delta^{-1}\nabla\cdot(a_0^l\nabla b_{i,j}^{\leq l}) \|_{H^1(\Omega_t)}&\lesssim_A 2^{j(n+1-\epsilon)}\|a_0^l\|_{L^{\infty}(\Omega_t)}\|b_{i,j}^{\leq l}\|_{H^1(\Omega_t)}
 \\
 &\lesssim_A 2^{j(n+1-\epsilon)}\|a\|_{C^{\frac{1}{2}}(\Gamma_t)}\|\mathcal{H}\mathcal{N}^i_{<j}\mathcal{N}_{\geq j}\mathcal{N}^{m-i}a\|_{H^{\frac{1}{2}+\epsilon}(\Omega_t)}.
\end{split}
\end{equation}
  Repeatedly using the $H^{\epsilon}\to H^{\frac{1}{2}+\epsilon}$ estimate (\ref{harmonicbase}), the properties of $\Phi$, the bound  $\|n_{\Gamma_t}\|_{C^{\epsilon}(\Gamma_t)}\lesssim_A 1$ and the trace inequality, we can estimate
  \begin{equation*}
  \begin{split}
  2^{j(n+1-\epsilon)}\|\mathcal{H}\mathcal{N}^i_{<j}\mathcal{N}_{\geq j}\mathcal{N}^{m-i}a\|_{H^{\frac{1}{2}+\epsilon}(\Omega_t)}&\lesssim_A 2^{j (n+1+i-\epsilon)}\|\nabla\Phi_{\geq j}\mathcal{H}\mathcal{N}^{m-i}a\|_{H^{\frac{1}{2}+\epsilon}(\Omega_t)} 
  \\
  &\lesssim_A \|\mathcal{H}\mathcal{N}^{m-i}a\|_{H^{n+i+\frac{5}{2}}(\Omega_t)} .
  \end{split}
  \end{equation*}
  Using \Cref{higherpowers}, \Cref{Linfest}, \Cref{aest} and (\ref{strongercoercivitybound}), we have
  \begin{equation*}
  \begin{split}
  \|\mathcal{H}\mathcal{N}^{m-i}a\|_{H^{n+i+\frac{5}{2}}(\Gamma_t)}\lesssim_A\|a\|_{H^{k-1}(\Gamma_t)}+\|\Gamma\|_{H^k}\|a\|_{C^{\epsilon}(\Gamma_t)}\lesssim_A (E^k)^{\frac{1}{2}}.   
  \end{split}
  \end{equation*} 
  If $n\geq 1$, then doing a similar analysis for the term $\nabla\cdot (a_0^l\nabla b_{m+1,j}^{\leq l})$ and combining this with \eqref{commutatordetials} and the bound $\|a\|_{C^{\frac{1}{2}}(\Gamma_t)}\lesssim_A B$, we obtain
  \begin{equation*}
  2^{j(n+1-\epsilon)}\|\Delta^{-1}(G')_j^2\|_{H^1(\Omega_t)}\lesssim_A B(E^k)^{\frac{1}{2}}.    
  \end{equation*}
If $n=0$, the term $\nabla\cdot (a_0^l\nabla b_{m+1,j}^{\leq l})$ is instead treated slightly differently. For this, we estimate similarly to before,
\begin{equation*}
2^{j(1-\epsilon)}\|\Delta^{-1}\nabla\cdot (a_0^l\nabla b_{m+1,j}^{\leq l})\|_{H^1(\Omega_t)}\lesssim_A \|a\|_{C^{\frac{1}{2}}(\Gamma_t)}\|\mathcal{H}\mathcal{N}^{m+1}a\|_{H^{\frac{3}{2}}(\Omega_t)}.   
\end{equation*}
Then we use \Cref{desiredelliptic} to estimate the last term as
\begin{equation*}
\|\mathcal{H}\mathcal{N}^{m+1}a\|_{H^{\frac{3}{2}}(\Omega_t)}\lesssim_A \|\mathcal{N}^{m+1}a\|_{H^1(\Gamma_t)}  ,  
\end{equation*}
and then estimate this term by $(E^k)^{\frac{1}{2}}$ similarly to the above. Next, one readily verifies analogous bounds for $H'$, $G''$ and $H''$ by using the similar decompositions,
\begin{equation}
\begin{split}
(H')_j^1=-a_0^lC_{\leq l}(a_{m+1,\leq j}),\hspace{5mm}(G'')_j^1=\nabla\cdot (\Phi_l(a_{m+1,\leq j})\nabla a_0^{\leq l}) ,\hspace{5mm}(H'')_j^1=-C_{\leq l}a_0\Phi_l(a_{m+1,\leq j}).
\end{split}
\end{equation}
From these bounds, ultimately, we obtain
 \begin{equation*}
 \|\mathcal{N}^n(\nabla_n\Delta^{-1}F)\|_{H^{\frac{1}{2}}(\Gamma_t)}\lesssim_A \|F\|_{H^{n}(\Omega_t)}+B(E^k)^{\frac{1}{2}}.    
 \end{equation*}
  It remains to estimate $F$ in $H^n$. We begin by looking at each summand in the bilinear frequency decomposition for $F$,
 \begin{equation*}\label{Fsummand}
F_l:=\nabla \Phi_{l} a_0\cdot\nabla \Phi_{\leq l} a_{m+1}+\nabla \Phi_{\leq l} a_0\cdot\nabla \Phi_{l} a_{m+1}.
 \end{equation*}
For the latter term, we have
\begin{equation*}
\|\nabla \Phi_{< l} a_0\cdot\nabla \Phi_{l} a_{m+1}\|_{H^n(\Omega_t)}\lesssim_A \|a\|_{C^{\frac{1}{2}}(\Gamma_t)}\|a_{m+1}\|_{H^{n+\frac{3}{2}}(\Omega_t)}, 
\end{equation*}
which when $n\geq 1$, we know from the above can be controlled by $B(E^k)^{\frac{1}{2}}$. For $n=0$, we have the same bound by simply using \Cref{desiredelliptic}. For the other term, we can further decompose 
\begin{equation}\label{lohi}
a_{m+1}=a_{m+1,l}^1+a_{m+1,l}^2  
\end{equation}
where $a_{m+1,l}^1=\mathcal{H}\mathcal{N}^{m+1}_{<l}a$. We then have from the properties of $\Phi_{\leq l}$ and the control of $\|\mathcal{H}a\|_{H^{n+m+\frac{5}{2}}(\Omega_t)}$ by the energy (as above),
\begin{equation*}
\|\nabla \Phi_{l} a_0\cdot\nabla \Phi_{\leq l} a_{m+1}\|_{H^{n}(\Omega_t)}\lesssim_A \|a\|_{C^{\frac{1}{2}}(\Gamma_t)}(E^k)^{\frac{1}{2}}.    
\end{equation*}
As $\nabla a_{m+1}$ is not at top order, we can easily verify using the decomposition above that we also have the following cruder bound for each $l$
\begin{equation}
\|F_l\|_{H^n(\Omega_t)}\lesssim_A 2^{-\delta l}\|(v,\Gamma)\|_{\mathbf{H}^s}^r(E^k)^{\frac{1}{2}},
\end{equation}
for some integer $r>1$ and small constant $\delta>0$.  Arguing as in \Cref{Linfest2}, we can combine the above two bounds to estimate 
\begin{equation*}
\|F\|_{H^n(\Omega_t)}\lesssim_A B\log(1+\|(v,\Gamma)\|_{\mathbf{H}^s})(E^k)^{\frac{1}{2}}.    
\end{equation*}
This handles the latter term in (\ref{commutator2}). Now, we turn to the first term. We  have to estimate $\|\mathcal{N}^n(\mathcal{N}a\mathcal{N}^{m+1}a)\|_{H^{\frac{1}{2}}(\Gamma_t)}$ where $n,m\geq 0$ and $n+m=k-3$. Here, we only sketch the details as the procedure for this estimate is relatively similar to the previous term. We start by writing
\begin{equation*}
\mathcal{N}a\mathcal{N}^{m+1}a=(\mathcal{H}n_{\Gamma_t}\cdot\nabla a_0)(\mathcal{H}n_{\Gamma_t}\cdot\nabla a_{m})=:K_{|\Gamma_t}.
\end{equation*}
Then we apply \Cref{higherpowers} and \Cref{baltrace} to estimate
\begin{equation*}
\|\mathcal{N}^nK_{|\Gamma_t}\|_{H^{\frac{1}{2}}(\Gamma_t)}\lesssim_A \|K\|_{H^{n+1}(\Omega_t)}+\|\Gamma\|_{H^k}\sup_{j>0}2^{-j(m+\frac{3}{2})}\|K_j^1\|_{L^{\infty}(\Omega_t)}+\sup_{j>0}2^{j(n+\frac{1}{2}-2\epsilon)}\|K_{j}^2\|_{H^{\frac{1}{2}+\epsilon}(\Omega_t)}
\end{equation*}
 where $K=K_j^1+K_j^2$ and
 \begin{equation}
 K_j^1:=\Phi_{<j}((\mathcal{H}n_{\Gamma_t}\cdot\nabla \Phi_{<j}a_0)(\mathcal{H}n_{\Gamma_t}\cdot\nabla\Phi_{<j}\mathcal{H}\mathcal{N}^{m}_{<j}a)).
 \end{equation}
 Similarly to the above, we can estimate
\begin{equation*}
\begin{split}
 2^{-j(m+\frac{3}{2})}\|K_j^1\|_{L^{\infty}(\Omega_t)}\lesssim_A B.
\end{split}
\end{equation*}
We also have an estimate of the form
\begin{equation*}
\begin{split}
 2^{j(n+\frac{1}{2}-2\epsilon)}\|K_j^2\|_{H^{\frac{1}{2}+\epsilon}(\Omega_t)}&\lesssim_A \|K\|_{H^{n+1}(\Omega_t)}+ B(E^k)^{\frac{1}{2}}+2^{j(n+1-\epsilon)}\|\mathcal{B}(\nabla\Phi_{\geq j}a_0,\nabla a_m)\|_{L^2(\Omega_t)}    
 \\
 &\lesssim_A \|K\|_{H^{n+1}(\Omega_t)}+ B(E^k)^{\frac{1}{2}}+\sup_{l>0}2^{l(n+1-\epsilon)}\|\mathcal{B}(\nabla\Phi_{l}a_0,\nabla a_m)\|_{L^2(\Omega_t)}
\end{split}
\end{equation*}
for some bilinear expression $\mathcal{B}$. Using a decomposition of $a_m$ similar to (\ref{lohi}), we have
\begin{equation*}
2^{l(n+1-\epsilon)}\|\mathcal{B}(\nabla\Phi_{l}a_0,\nabla a_m)\|_{L^2(\Omega_t)}\lesssim_A B\log(1+\|(v,\Gamma)\|_{\mathbf{H}^s})(E^k)^{\frac{1}{2}}.
\end{equation*}
Therefore, we have
\begin{equation*}
\|\mathcal{N}^nK_{|\Gamma_t}\|_{H^{\frac{1}{2}}(\Gamma_t)}\lesssim_A B\log(1+\|(v,\Gamma)\|_{\mathbf{H}^s})(E^k)^{\frac{1}{2}}+\|K\|_{H^{n+1}(\Omega_t)}.    
\end{equation*}
To estimate $K$ in $H^{n+1}(\Omega_t)$, the starting point is similar (but slightly more technical) than the estimate for $F$ in $H^n$ from above. The idea is to do a quadrilinear frequency decomposition for $K$ and study each summand individually. The relevant terms correspond to terms essentially of the form $(\Phi_l\mathcal{H}n_{\Gamma_t}\cdot\nabla\Phi_{\leq l} a_0)(\Phi_{\leq l}\mathcal{H}n_{\Gamma_t}\cdot\nabla \Phi_{\leq l}a_{m})$ and $(\Phi_{\leq l}\mathcal{H}n_{\Gamma_t}\cdot\nabla\Phi_{l} a_0)(\Phi_{\leq l}\mathcal{H}n_{\Gamma_t}\cdot\nabla \Phi_{\leq l}a_{m})$ and $(\Phi_{\leq l}\mathcal{H}n_{\Gamma_t}\cdot\nabla\Phi_{\leq l} a_0)(\Phi_{l}\mathcal{H}n_{\Gamma_t}\cdot\nabla \Phi_{\leq l}a_{m})$ and $(\Phi_{\leq l}\mathcal{H}n_{\Gamma_t}\cdot\nabla\Phi_{\leq l} a_0)(\Phi_{\leq l}\mathcal{H}n_{\Gamma_t}\cdot\nabla \Phi_{l}a_{m})$. The second and fourth terms can be handled almost identically to the estimate for $F$ in $H^n$ (as by the maximum principle, one can dispense with the factors of $\mathcal{H}n_{\Gamma_t}$). The first and third terms are handled similarly by decomposing $a_0$ and $a_m$ into low and high frequency parts as in (\ref{lohi}) and using \Cref{Hbounds} when $\Phi_l\mathcal{H}n_{\Gamma_t}$ is at high frequency compared to the other factors. One then obtains the desired estimate similarly to the estimate for $F$ in $H^n$ above. We omit the remaining details.
\end{proof}
We  now turn to the estimate for the final source term, $\mathcal{R}=[D_t,\mathcal{N}^{k-1}]a$ in $L^2(\Gamma_t)$. To control this term, we first write
\begin{equation*}
[D_t,\mathcal{N}^{k-1}]a=[D_t,\mathcal{N}]\mathcal{N}^{k-2}a+\mathcal{N}[D_t,\mathcal{N}^{k-2}]a.    
\end{equation*}
For the latter term, we have by \Cref{baselineDN2},
\begin{equation*}
\|\mathcal{N}[D_t,\mathcal{N}^{k-2}]a\|_{L^2(\Gamma_t)}\lesssim_A \|[D_t,\mathcal{N}^{k-2}]a\|_{H^1(\Gamma_t)}.     
\end{equation*}
Then using \Cref{materialcom} and the coercivity bound, we estimate
\begin{equation*}
\begin{split}
\|[D_t,\mathcal{N}^{k-2}]a\|_{H^1(\Gamma_t)}&\lesssim_A \|v\|_{W^{1,\infty}(\Omega_t)}\|a\|_{H^{k-1}(\Gamma_t)}+\|a\|_{C^{\frac{1}{2}}(\Gamma_t)}(\|\Gamma\|_{H^k}+\|v\|_{H^k(\Omega_t)})+\|a\|_{L^{\infty}(\Gamma_t)}\|v\|_{W^{1,\infty}(\Omega_t)}\|\Gamma\|_{H^k}
\\
&\lesssim_A B(E^k)^{\frac{1}{2}}.    
\end{split}
\end{equation*}
To conclude the proof of \Cref{Energy est. thm}, it remains to estimate $[D_t,\mathcal{N}]\mathcal{N}^{k-2}a$ in $L^2(\Gamma)$. This term is rather delicate due to the lack of a trace estimate in $L^2(\Gamma)$. To deal with this term, we have the following proposition. 
\begin{proposition}\label{comm prop}
Let $s\in\mathbb{R}$ with $s>\frac{d}{2}+1$. Then we have,
\begin{equation}\label{H1commutator}
\|[\mathcal{N},D_t]f\|_{L^2(\Gamma)}\lesssim_A B\log(1+\|(v,\Gamma)\|_{\mathbf{H}^s})\|f\|_{H^1(\Gamma)}.
\end{equation}
\end{proposition}
Our proof requires the following short lemma which is essentially a consequence of \Cref{desiredelliptic}. 
\begin{lemma}\label{desiredellipticlemma} For each $l=1,...,d$, we have
\begin{equation}\label{desiredelliptic2}
\|n_{\Gamma}\cdot(\nabla\Delta^{-1}\partial_l-e_l)\|_{H^{\frac{1}{2}}(\Omega)\to L^2(\Gamma)}\lesssim_A 1.   
\end{equation}
\end{lemma}
\begin{proof}
This will follow by interpolation if we can prove
\begin{equation}\label{desiredeliptic3}
\|n_{\Gamma}\cdot(\nabla\Delta^{-1}\partial_l-e_l)\|_{L^2(\Omega)\to H^{-\frac{1}{2}}(\Gamma)}+\|n_{\Gamma}\cdot(\nabla\Delta^{-1}\partial_l-e_l)\|_{H^{\frac{1}{2}+\delta}(\Omega)\to H^{\delta}(\Gamma)}\lesssim_A 1,
\end{equation}
for some $0<\delta<\epsilon$. The $H^{\frac{1}{2}+\delta}\to H^{\delta}$ bound follows easily from the trace inequality, the bound $\|n_{\Gamma_t}\|_{C^{\epsilon}(\Gamma_t)}\lesssim_A 1$ and \Cref{desiredelliptic}. For the $L^2\to H^{-\frac{1}{2}}$ bound we use duality. Indeed, let $f\in L^2(\Omega)$. Since $(\nabla\Delta^{-1}\partial_l-e_l)f$ is divergence free, we have 
\begin{equation*}
\int_{\Gamma}gn_{\Gamma}\cdot(\nabla\Delta^{-1}\partial_l-e_l)f\, dS=\int_{\Omega}\nabla\mathcal{H}g\cdot (\nabla\Delta^{-1}\partial_l-e_l)f\,dx\lesssim_A \|g\|_{H^{\frac{1}{2}}(\Gamma)}\|f\|_{L^2(\Omega)},
\end{equation*}
for every $g\in H^{\frac{1}{2}}(\Gamma)$. Therefore, we obtain (\ref{desiredeliptic3}) and thus also (\ref{desiredelliptic2}).
\end{proof}
\begin{proof}[Proof of \Cref{comm prop}]
Now, returning to the proposition, we expand using (\ref{N-commutator}),
\begin{equation*}
[D_t,\mathcal{N}]f=D_tn_{\Gamma}\cdot\nabla\mathcal{H}f-n_{\Gamma}\cdot((\nabla v)^*(\nabla \mathcal{H}f))+n_{\Gamma}\cdot\nabla \Delta^{-1}\Delta(v\cdot\nabla \mathcal{H}f).
\end{equation*}
The first two terms on the right can easily be estimated in $L^2$ by the right-hand side of (\ref{H1commutator}) by using (\ref{Moving normal}) and \Cref{baselineDN2}. Now, we turn to the latter term. We write for simplicity $u:=\mathcal{H}f$. We then split $u$ as 
\begin{equation*}
u=\sum_{l\leq l_0}\Phi_lu+\Phi_{> l_0}u=:\sum_{l\leq l_0}u_l+u_{\geq l_0},
\end{equation*}
where $l_0$ is a parameter to be chosen. Note that $u_l$ is not harmonic anymore, but it is to leading order. As usual, we also write the corresponding divergence free regularizations for $v$ as $v_l:=\Psi_l v$, $v_{<l}:=\Psi_{<l}v$ and so forth. 
\\

The following lemma shows that we have a suitable estimate when $u$ is replaced by a single dyadic regularization $u_l$.
\begin{lemma}\label{lemma on com est}
For each $l\in\mathbb{N}_0$, we have
\begin{equation*}
\|\nabla_n\Delta^{-1}\Delta (v\cdot\nabla u_l)\|_{L^2(\Gamma)}\lesssim_A B\|f\|_{H^1(\Gamma)},
\end{equation*}
where the implicit constant does not depend on $l$.
\begin{proof}
We write
\begin{equation}\label{lowandhighfrequencyterms}
\nabla_n\Delta^{-1}\Delta (v\cdot\nabla u_l)=\nabla_n\Delta^{-1}\Delta (v_{<l}\cdot\nabla u_l)+\nabla_n\Delta^{-1}\Delta (v_{\geq l}\cdot\nabla u_l).
\end{equation}
For the second term, where $v$ is at high frequency, we use the identity $\Delta^{-1}\Delta=I-\mathcal{H} $ and the $H^1\to L^2$ bound for $\mathcal{N}$ to estimate
\begin{equation}\label{terms in the est}
\|\nabla_n\Delta^{-1}\Delta (v_{\geq l}\cdot\nabla u_l)\|_{L^2(\Gamma)}\lesssim_A \|\nabla(v_{\geq l}\cdot\nabla u_l)\|_{L^2(\Gamma)}+\|v_{\geq l}\cdot\nabla u_l\|_{H^1(\Gamma)}.    
\end{equation}
For the first term in \eqref{terms in the est}, we distribute the derivative to obtain
\begin{equation}\label{terms in the est2}
\|\nabla(v_{\geq l}\cdot\nabla u_l)\|_{L^2(\Gamma)}\lesssim B\|\nabla u_l\|_{L^2(\Gamma)}+\|v_{\geq l}\cdot\nabla^2 u_l\|_{L^2(\Gamma)}.
\end{equation}
For the first term in \eqref{terms in the est2}, we use the variant of the trace theorem leading to \eqref{trace1} and the fact that $u_l$ is frequency localized to obtain
\begin{equation*}\label{terms in the est3}
\|\nabla u_l\|_{L^2(\Gamma)}\lesssim \|\nabla u_l\|_{H^1(\Omega)}^{\frac{1}{2}}\|\nabla u_l\|_{L^2(\Omega)}^{\frac{1}{2}}\lesssim \|u\|_{H^{\frac{3}{2}}(\Omega)}\lesssim_A \|f\|_{H^1(\Gamma)}
\end{equation*}
where in the last estimate we used \Cref{desiredelliptic}. For the second term in \eqref{terms in the est2}, we again use the trace theorem and the fact that $v_{\geq l}$ is higher frequency to obtain
\begin{equation*}
\begin{split}
\|v_{\geq l}\cdot\nabla^2 u_l\|_{L^2(\Gamma)}&\lesssim \|v_{\geq l}\cdot\nabla^2 u_l\|_{L^2(\Omega)}^{\frac{1}{2}}\|v_{\geq l}\cdot\nabla^2 u_l\|_{H^1(\Omega)}^{\frac{1}{2}}\lesssim B \|u\|_{H^{\frac{3}{2}}(\Omega)}\lesssim B\|f\|_{H^1(\Gamma)}.
\end{split}
\end{equation*}
The term $\|v_{\geq l}\cdot\nabla u_l\|_{H^1(\Gamma)}$ in \eqref{terms in the est} is similarly estimated. For this, we only need to estimate $\|\nabla^{\top}(v_{\geq l}\cdot\nabla u_l)\|_{L^2(\Gamma)}$, and this is handled by an almost identical strategy to the above. 
\\

Now, to estimate the term in \eqref{lowandhighfrequencyterms} where $v$ is at low frequency, we distribute the Laplacian and use that $v_{<l}$ is divergence free to write $\nabla_n\Delta^{-1}\Delta (v_{<l}\cdot\nabla u_l)$ as a sum of terms of the form
\begin{equation*}
\begin{split}
\nabla_n\Delta^{-1}\partial_j(Dv_{<l}Du_l)+\nabla_n\Delta^{-1}\partial_j(v_{<l}C_lu),
\end{split}
\end{equation*}
where $C_lu:=[\Delta,\Phi_l]u$. Using \Cref{desiredellipticlemma} we can then estimate
\begin{equation*}
\|\nabla_n\Delta^{-1}\Delta (v_{<l}\cdot\nabla u_l)\|_{L^2(\Gamma)}\lesssim_A \|Dv_{<l}Du_l\|_{L^2(\Gamma)\cap H^{\frac{1}{2}}(\Omega)}+\|v_{<l}C_lu\|_{L^2(\Gamma)\cap H^{\frac{1}{2}}(\Omega)}=:J_1+J_2.
\end{equation*}
Using that $v$ is at low frequency, we can estimate similarly to the above,
\begin{equation*}
J_1\lesssim_A B\|f\|_{H^1(\Gamma)}.    
\end{equation*}
For $J_2$, we note that $C_l$ is an operator of order $1$ and still retains essentially the frequency localization scale of $2^l$. Therefore, we can estimate $J_2$ similarly.
This completes the proof of the lemma.
\end{proof}
\end{lemma}
Returning to the proof of \Cref{comm prop}, we now estimate using \Cref{lemma on com est},
\begin{equation*}\label{logarithmicerror1}
\|\nabla_n\Delta^{-1}\Delta(v\cdot\nabla u)\|_{L^2(\Gamma)}\lesssim_A l_0 B\|f\|_{H^1(\Gamma)}+\|\nabla_n\Delta^{-1}\Delta(v\cdot\nabla u_{\geq l_0})\|_{L^2(\Gamma)}.
\end{equation*}
Again, using that $v$ is divergence free, we can (as above) expand $\nabla_n\Delta^{-1}\Delta(v\cdot\nabla u_{\geq l_0})$ as a sum of terms of the form 
\begin{equation*}
\nabla_n\Delta^{-1}\partial_j(DvDu_{\geq l_0})+\nabla_n\Delta^{-1}\partial_j(vC_{\leq l_0}u),
\end{equation*}
where $C_{\leq l_0}u=[\Delta, \Phi_{\leq l_0}]u$. For the latter term, we can simply estimate as above (since $v$ is undifferentiated),
\begin{equation*}
\|\nabla_n\Delta^{-1}\partial_j(vC_{\leq l_0}u)\|_{L^2(\Gamma)}\leq\sum_{l\leq l_0}\|\nabla_n\Delta^{-1}\partial_j(vC_{l}u)\|_{L^2(\Gamma)}\lesssim_A l_0 B\|f\|_{H^1(\Gamma)}.    
\end{equation*}
For the other term, we use \Cref{desiredellipticlemma} to obtain
\begin{equation*}
\|\nabla_n\Delta^{-1}\partial_j(DvDu_{\geq l_0})\|_{L^2(\Gamma)}\lesssim_A B\|Du_{\geq l_0}\|_{L^2(\Gamma)}+\|DvDu_{\geq l_0}\|_{H^{\frac{1}{2}}(\Omega)}.  
\end{equation*}
Since $u$ is harmonic we have
\begin{equation*}
B\|Du_{\geq l_0}\|_{L^2(\Gamma)}\lesssim_A B\|f\|_{H^1(\Gamma)}+B\|Du_{<l_0}\|_{L^2(\Gamma)}    .
\end{equation*}
Then expanding $u_{<l_0}=\sum_{l< l_0}u_l$ and using the trace theorem leading to \eqref{trace1} for each term as above, we get
\begin{equation*}
B\|Du_{\geq l_0}\|_{L^2(\Gamma)}\lesssim_A B l_0\|f\|_{H^1(\Gamma)}.
\end{equation*}
Finally, by product estimates and Sobolev embedding, it is easy to bound
\begin{equation*}
\|DvDu_{\geq l_0}\|_{H^{\frac{1}{2}}(\Omega)}\lesssim_A B\|f\|_{H^1(\Gamma)}+\|Dv_{\geq l_0}\|_{H^{\frac{d}{2}+\epsilon}(\Omega)}\|f\|_{H^1(\Gamma)}\lesssim_A (B+2^{-l_0\delta}\|(v,\Gamma)\|_{\mathbf{H}^s})\|f\|_{H^1(\Gamma)}    
\end{equation*}
for some $\delta>0$. Then choosing $l_0\approx_{\delta} \log(1+\|(v,\Gamma)\|_{\mathbf{H}^s})$, we conclude the proof of the proposition.
\end{proof}
Finally, we conclude the proof of \Cref{Energy est. thm} by observing first from the above proposition that we have
\begin{equation*}
\|[D_t,\mathcal{N}]\mathcal{N}^{k-2}a\|_{L^2(\Gamma)}\lesssim_A B\log(1+\|(v,\Gamma)\|_{\mathbf{H}^s})\|\mathcal{N}^{k-2}a\|_{H^{1}(\Gamma)}.
\end{equation*}
Then, using \Cref{higherpowers}, \Cref{Linfest}, \Cref{aest} and (\ref{strongercoercivitybound}), we have
\begin{equation*}
\|\mathcal{N}^{k-2}a\|_{H^{1}(\Gamma)}\lesssim_A (E^k)^{\frac{1}{2}}.    
\end{equation*}
This finally concludes the proof of \Cref{Energy est. thm}.
\section{Construction of regular solutions}\label{Existence section} 
In this section, we give a new, direct method for constructing solutions to the free boundary Euler equations in the high regularity regime. Solutions at low regularity will be obtained in the next section as unique limits of these regular solutions. 
\\

Previous approaches to constructing solutions to free boundary fluid equations include using Lagrangian coordinates, Nash Moser iteration or taking the zero surface tension limit in the capillary problem. A more recent  approach in the case of a laterally infinite ocean with flat bottom can be found in \cite{MR4263411}. The article \cite{MR4263411} uses a paralinearization of the Dirichlet-to-Neumann operator and a complicated iteration scheme to construct solutions. In contrast, we propose a new, geometric approach, implemented fully within the Eulerian coordinates.   
\\

Our novel approach is roughly inspired by nonlinear semigroup theory, where one constructs an approximate solution by discretizing the problem in time. To execute this approach successfully, one needs to show that the energy bounds are  uniformly preserved throughout the time steps. In our setting, a classical semigroup approach would require one to solve an elliptic free boundary problem with very precise estimates. However, on the other end of the spectrum, one could try to view our equation as an ODE and use an Euler type iteration. Of course, a na\"ive Euler method cannot work because it loses derivatives. A partial fix to this would be to combine the Euler method with a transport part, which would reduce but not eliminate the loss of derivatives.
\\

Our goal is to retain the simplicity of the Euler plus transport method, while ameliorating the derivative loss by an initial regularization of each iterate in our discretization. In short, we will split the time step into two main pieces:
\begin{enumerate}
    \item Regularization.
    \item Euler plus transport.
\end{enumerate}
To ensure that the uniform energy bounds survive, the regularization  step needs to be done carefully. For this, we will take a modular approach and try to decouple this process into two steps, where we regularize individually the domain and the velocity. We believe that this modular approach will serve as a recipe for a new and relatively simple method for constructing solutions to various free boundary problems. 
\\

The overarching scheme we employ in this section was carried out in the case of a compressible gas in \cite{ifrim2020compressible}. While we follow the same rough roadmap here, we stress that the main difficulties in the incompressible liquid case are quite different than for the gas. One obvious reason for this is that the surface of a liquid carries a non-trivial energy. Also, we introduce another new idea here, which is to begin the iteration with a regularized version of the initial data, and then to partially propagate these regularized bounds through the iteration.

\subsection{Basic setup and simplifications}
We begin by fixing a smooth reference hypersurface $\Gamma_*$ and a collar neighborhood $\Lambda_*:=\Lambda(\Gamma_*,\epsilon_0,\delta).$ Here, as usual, $\epsilon_0$ and $\delta$ are some small but fixed positive constants.  Given $k>\frac{d}{2}+1$ sufficiently large and an initial state $(v_0,\Gamma_0)\in\mathbf{H}^k$, our aim is to construct a local solution $(v(t),\Gamma_t)\in\mathbf{H}^k$ whose lifespan depends only on the size of $\|(v_0,\Gamma_0)\|_{\mathbf{H}^k}$, the lower bound in the Taylor sign condition and the collar neighborhood $\Lambda_*$. We recall from \Cref{Energy est. thm} that we have the coercivity 
\begin{equation*}
1+\|(v,\Gamma)\|_{\mathbf{H}^k}^2\approx_A E^k(v,\Gamma)    
\end{equation*}
for any state $(v,\Gamma)\in\mathbf{H}^k$.
For technical convenience, we will work with the slightly modified energy, 
\begin{equation}\label{energy re weight}
\mathcal{E}^k(v,\Gamma):=\|\nabla\mathcal{H}\mathcal{N}^{k-2}(a^{-1}D_ta)\|_{L^2(\Omega)}^2+\|a^{-\frac{1}{2}}\mathcal{N}^{k-1}a\|_{L^2(\Gamma)}^2+\|\omega\|_{H^{k-1}(\Omega)}^2+\|v\|_{L^2(\Omega)}^2+1.
\end{equation}
This new energy is readily seen to be equivalent to the old one in the sense that 
\begin{equation}\label{modifiedenergy}
    \mathcal{E}^k(v,\Gamma)\approx_A E^k(v,\Gamma).
\end{equation}
The primary reason we modify the energy is that it will allow for cleaner cancellations in the energy when we later regularize the velocity. 
\\

Now, fix $M>0$. Given a small time step $\epsilon>0$ and a suitable pair of initial data $(v_0,\Gamma_0)\in\mathbf{H}^k$ with $\|(v_0,\Gamma_0)\|_{\mathbf{H}^k}\leq M$, we aim to construct a sequence $(v_{\epsilon}(j\epsilon),\Gamma_{\epsilon}(j\epsilon))\in\mathbf{H}^k$ satisfying the following properties:
\begin{enumerate}
\item (Norm bound). There is a uniform constant $c_0>0$ depending only on $\Lambda_*$, $M$ and the lower bound in the Taylor sign condition such that if $j$ is an integer with $0\leq j\leq c_0\epsilon^{-1}$, then 
\begin{equation*}
\|(v_{\epsilon}(j\epsilon),\Gamma_{\epsilon}(j\epsilon))\|_{\mathbf{H}^k}\leq C(M),
\end{equation*}
where $C(M)>0$ is some constant depending on $M$. 
\item (Approximate solution). 
\begin{equation*}
\begin{cases}
&v_{\epsilon}((j+1)\epsilon)=v_{\epsilon}(j\epsilon)-\epsilon( v_{\epsilon}(j\epsilon)\cdot\nabla v_{\epsilon}(j\epsilon)+\nabla p_{\epsilon}(j\epsilon)+ge_d)+\mathcal{O}_{C^{1}}(\epsilon^2) \hspace{5mm}\text{on}\hspace{2mm}\Omega_{\epsilon}((j+1)\epsilon)\cap\Omega_{\epsilon}(j\epsilon),
\\
&\nabla\cdot v_{\epsilon}((j+1)\epsilon)=0 \hspace{5mm}\text{on}\hspace{2mm}\Omega_{\epsilon}((j+1)\epsilon),
\\
&\Omega_{\epsilon}((j+1)\epsilon)=(I+\epsilon v_{\epsilon}(j\epsilon))(\Omega_{\epsilon}(j\epsilon))+\mathcal{O}_{C^1}(\epsilon^2).
\end{cases}
\end{equation*}
\end{enumerate}
 We will not have to concern ourselves too much with the Taylor sign condition in this section as we are working at  high regularity and this is a pointwise property. In particular, we will suppress the lower bound in the Taylor sign condition from our notation. A nice feature about the above iteration scheme is that it suffices to only carry out a single step. For this, we have the following theorem.
\begin{theorem}\label{onestepiteration}
 Let $k$ be a sufficiently large even integer and $M > 0$. Consider an initial data $(v_0,\Gamma_0)\in\mathbf{H}^k$
 so that $\|(v_0,\Gamma_0)\|_{\mathbf{H}^k}\leq M$ and $v_0$ and $\omega_0$ satisfy the initial regularization bounds
\begin{equation}\label{inductiveregbound}
\|v_0\|_{H^{k+1}(\Omega_0)}\leq K(M)\epsilon^{-1},\hspace{5mm}\|\omega_0\|_{H^{k+n}(\Omega_0)}\leq K'(M)\epsilon^{-1-n},
\end{equation}
for $n=0,1$,  where $K(M)$, $K'(M)>0$ are constants, possibly much larger than $M$, such that $K'(M)\ll K(M)$. Then there exists a one step iterate $(v_0,\Gamma_0)\mapsto (v_1,\Gamma_1)$ with the following properties:
\begin{enumerate}
\item (Energy monotonicity).
\begin{equation}\label{EMBITT}
\mathcal{E}^{k}(v_1,\Gamma_1)\leq (1+C(M)\epsilon)\mathcal{E}^{k}(v_0,\Gamma_0).
\end{equation}
\item (Good pointwise approximation). 
\begin{equation}
\label{approx-sln}
\begin{cases}
&v_1=v_0-\epsilon (v_0\cdot\nabla v_0+\nabla p_0+ge_d)+\mathcal{O}_{C^1}(\epsilon^2)\hspace{5mm}\text{on}\hspace{2mm}\Omega_1\cap\Omega_0,
\\
&\nabla\cdot v_1=0\hspace{5mm}\text{on}\hspace{2mm}\Omega_1,
\\
&\Omega_{1}=(I+\epsilon v_{0})(\Omega_{0})+\mathcal{O}_{C^1}(\epsilon^2).
\end{cases}
\end{equation}
\item (Persistence of the regularization bounds). $v_1$ satisfies the regularization bounds
\begin{equation}\label{regboundprop}
\|v_{1}\|_{H^{k+1}(\Omega_1)}\leq K(M)\epsilon^{-1},\hspace{5mm}\|\omega_1\|_{H^{k+n}(\Omega_1)}\leq (K'(M)+C(M)\epsilon)\epsilon^{-1-n},
\end{equation}
for $n=0,1$. 
\end{enumerate}
\end{theorem}
\begin{remark}
Property \eqref{regboundprop} ensures that $v_1$ retains the $H^{k+1}$ regularization bound with the same constant compared to the first iterate, and $\omega_1$ has a regularization bound which can only grow by an amount comparable to $\epsilon$ times the initial regularization bound, which is acceptable over $\approx_M \epsilon^{-1}$ iterations.   The energy monotonicity property, along with the energy coercivity bound from \Cref{Energy est. thm} will ensure that the resulting sequence $(v_{\epsilon}(j\epsilon),\Gamma_{\epsilon}(j\epsilon))$ of approximate solutions we construct remains uniformly bounded in $\mathbf{H}^k$ for $j\ll_M\epsilon^{-1}$. The second property  in \Cref{onestepiteration} will ensure that $(v_{\epsilon}(j\epsilon),\Gamma_{\epsilon}(j\epsilon))$ converges in a weaker topology to a solution of the equation.   
\end{remark}

The assumption (\ref{inductiveregbound}) for $v_0$ is for technical convenience. In the regularization step of the argument, it will allow us to decouple the process of regularizing the domain and regularizing the velocity into separate arguments (see \Cref{MDBL} in the next section).  The condition \eqref{regboundprop} ensures that (\ref{inductiveregbound}) can be propagated from one iterate to the next. Assuming that the initial iterate satisfies (\ref{inductiveregbound})  is harmless in practice. Indeed, by the regularization properties of $\Psi_{\leq\epsilon^{-1}}$, we can replace the first iterate in the resulting sequence $(v_{\epsilon}(j\epsilon),\Gamma_{\epsilon}(j\epsilon))$ with a suitable $\epsilon^{-1}$ scale regularization so that the base case is satisfied. We note crucially that such a regularization is only done once - on the initial iterate - as we  only know that this regularization  is   bounded on $\mathbf H^k$ (it does not necessarily satisfy the more delicate energy monotonicity).  In contrast, we require the much stricter energy monotonicity bound \eqref{EMBITT} for all other iterations as in the above theorem.  The condition on the vorticity in (\ref{inductiveregbound}) can also be harmlessly assumed for the initial iterate. When we later regularize the velocity, we will not regularize the vorticity, but rather only the irrotational component. This is why, in contrast to the $H^{k+1}$ bound for $v_1$, the constant for $\omega_1$ in \eqref{regboundprop} gets slightly worse. Nonetheless, the careful tracking of its bound in \eqref{regboundprop}  ensures that it only grows by an acceptable amount in each iteration. The heuristic reason why the regularization bound on $\omega_1$ is expected is because the vorticity should be essentially transported by the flow, and therefore should not suffer the derivative loss of the full velocity in the iteration step.
\\

\textbf{Outline of the argument.} We now give a brief overview of the section. The first step is selecting a suitable regularization scale. To motivate this, we recall that the evolution of the domain and the irrotational component of the velocity is essentially governed by the following approximate equation for $a$:
\begin{equation}\label{Dt2a}
D_t^2a\approx -a\mathcal{N}a.
\end{equation}
Therefore, heuristically, $D_t$ behaves roughly as a ``spatial" derivative of order $\frac{1}{2}$. To control quadratic errors in the energy monotonicity bound in the Euler plus transport iteration later, it is therefore natural to attempt to regularize the domain and the irrotational part of the velocity on the $\epsilon^{-1}$ scale, as we do in \Cref{onestepiteration}. As the vorticity is essentially transported by the flow, we are able to  leave the rotational part of the velocity alone, and instead track its growth as in \eqref{regboundprop}.
\\

With the above discussion in mind, we begin our analysis in earnest in \Cref{DRS} by regularizing the domain on the $\epsilon^{-1}$ scale. More specifically, given $(v_0,\Gamma_0)\in\mathbf{H}^k$ with $v_0$ satisfying \eqref{inductiveregbound}, we construct for each $0<\epsilon\ll 1$ a domain $\Omega_\epsilon\subseteq \Omega_0$ whose boundary is within $\mathcal{O}_{C^1}(\epsilon^2)$ of $\Gamma_0$ and which satisfies the regularization bound $\|\Gamma_{\epsilon}\|_{H^{k+\alpha}}\lesssim_{M,\alpha} \epsilon^{-\alpha}$ for all $\alpha\geq 0$. This is achieved by performing a parabolic regularization of the graph parameterization $\eta_{0}$ on $\Gamma_*$, together with a slight contraction of the domain. We then define our new velocity $\tilde{v}_0=\tilde{v}_0(\epsilon)$ by restricting the old velocity $v_0$ to the new domain $\Omega_\epsilon$. As will be the case in every step of the argument, the main difficulty is to carefully track the effect of the regularization on the energy growth. The main point in this part of the argument is to show that the parabolic regularization of $\eta_0$ induces a corresponding parabolic gain in the surface component of the energy $\|a^{-\frac{1}{2}}\mathcal{N}^{k-1}a\|_{L^2(\Gamma)}^2$, allowing us to control all of the resulting errors. 
\\

With the domain now regularized, we  move on to regularizing the velocity in \Cref{VR}, which is step 2 of the argument. In this step, we leave the domain and rotational part of the velocity alone, and  regularize the irrotational part of the velocity on the $\epsilon^{-1}$ scale. The way we execute this is by using the functional calculus for the Dirichlet-to-Neumann operator. The main difficulty in this step of the argument is in tracking the effect of this regularization on the $\|\nabla\mathcal{H}\mathcal{N}^{k-2}(a^{-1}D_ta)\|_{L^2(\Omega)}^2$ portion of the energy, which at leading order controls the irrotational component of the velocity. An additional objective in this step of the argument is to  improve the constant in \eqref{inductiveregbound} so that we can ultimately close the bootstrap in the upcoming Euler plus transport phase of the argument.
\\

The final step in our construction is to use an Euler plus transport iteration to flow  the regularized variables $(v_\epsilon,\Gamma_\epsilon)$ along a discrete version of the Euler evolution. It is in this step of the argument that we expect to observe a $\frac{1}{2}$ derivative loss (see the equation \eqref{Dt2a} for $D_t^2a$, for instance), which is why the above regularization procedure is imperative. The Euler plus transport argument we employ is carried out in \Cref{EPTI}. Control of the resulting energy growth is shown by carefully relating the good variables $a$, $D_ta$ and $\omega$ for the new iterate to the corresponding good variables for the regularized data. Then, with the energy uniformly bounded and the variables appropriately iterated, in \Cref{COTS} we conclude  that our scheme converges in a weaker topology, completing the construction of solutions.

\subsection{Step 1: Domain regularization}\label{DRS}
We begin with the domain regularization step. For this, we have the following proposition. 
\begin{proposition} 
 Given $(v_0,\Gamma_0)\in\mathbf{H}^k$ with $v_0$ satisfying \eqref{inductiveregbound}, there exists a domain $\Omega_{\epsilon}$ contained in $\Omega_0$ with boundary $\Gamma_{\epsilon}\in\Lambda_*$ such that the pair $(v_{0_{|\Omega_{\epsilon}}},\Gamma_{\epsilon})$ satisfies 
 \begin{enumerate}
\item (Energy monotonicity).
\begin{equation}\label{Emonsurf}
\mathcal{E}^{k}(v_{0_{|\Omega_{\epsilon}}},\Gamma_{\epsilon})\leq (1+C(M)\epsilon)\mathcal{E}^{k}(v_0,\Gamma_0).
\end{equation}
\item (Good pointwise approximation).
\begin{equation}\label{approxdomain}
\eta_{\epsilon}=\eta_{0}+\mathcal{O}_{C^1}(\epsilon^2)\hspace{2mm}on\hspace{2mm}\Gamma_*.
\end{equation}
\item (Domain regularization bound). For every $\alpha\geq 0$, there holds,
\begin{equation}\label{surfbound}
\|\Gamma_{\epsilon}\|_{H^{k+\alpha}}\lesssim_{M,\alpha} \epsilon^{-\alpha}.
\end{equation}
 \end{enumerate}
\end{proposition}
\begin{proof}
In the sequel, we will use $\tilde{v}_0$ as a shorthand for $v_{0_{|\Omega_{\epsilon}}}$. To regularize $\Gamma_0$, we begin with the  preliminary parabolic regularization of $\eta_0$ given by
\begin{equation*}
\tilde{\eta}_{\epsilon}=e^{\epsilon^2\Delta_{\Gamma_*}}\eta_0,
\end{equation*}
where $\Delta_{\Gamma_*}$ is the Laplace-Beltrami operator for $\Gamma_*$. The rationale for using the operator $e^{\epsilon^2\Delta_{\Gamma_*}}$ instead of, for instance, the operator $e^{-\epsilon|D|}$ is to ensure that when $k$ is large enough, we have $\|\partial_{\epsilon}\tilde{\eta}_{\epsilon}\|_{H^{k-2}(\Gamma_*)}\lesssim_M \epsilon$. This ensures that the hypersurface parameterized by $\tilde{\eta}_{\epsilon}$ in collar coordinates is at a distance on the order of no more than $\mathcal{O}_M(\epsilon^2)$ from $\Gamma_0$ in the $H^{k-2}$ topology (and thus the $C^1$ topology if $k$ is large enough). We would also like to additionally guarantee that $\Omega_{\epsilon}$ is contained in $\Omega_0$, so that we can use the restriction of the velocity $v_0$ to $\Omega_{\epsilon}$ as the velocity on the new domain. Therefore, we  slightly correct the above parabolic regularization by  defining  our regularized hypersurface $\Gamma_{\epsilon}$ through the collar parameterization
\begin{equation*}
\eta_{\epsilon}=\tilde{\eta}_{\epsilon}-C\epsilon^2,
\end{equation*}
where $C$ is some positive constant depending on $M$ only, imposed to ensure that the domain  $\Omega_{\epsilon}$ associated to $\Gamma_{\epsilon}$ is contained in $\Omega_0$.  Clearly, $\eta_{\epsilon}$ satisfies (\ref{surfbound}) and the required pointwise approximation property in (\ref{approxdomain}).  The main bulk of the work in this step of the argument will therefore be in understanding how the above parabolic regularization of the surface (and also the restriction of the velocity to $\Omega_{\epsilon}$) affects the energy.
\\

 
 Given $(\tilde{v}_0,\Gamma_\epsilon)$ as above, we define the associated quantities $\tilde{\omega}_0:=\nabla\times \tilde{v}_0$ and $\tilde{p}_0$, $D_t\tilde{p}_0$, $\tilde{a}_0$ and $D_t\tilde{a}_0$ on $\Omega_{\epsilon}$ and $\Gamma_{\epsilon}$ by using the relevant Poisson equations, as in \Cref{CTEF}. We will use the notation $\mathcal{N}_{\epsilon}$ to refer to the Dirichlet-to-Neumann operator for $\Gamma_{\epsilon}$. Before proceeding to the proof of energy monotonicity, we note that the above construction gives rise to a flow velocity $V_{\epsilon}$ in the parameter $\epsilon$ for the family of hypersurfaces $\Gamma_{\epsilon}$ by composing $\partial_{\epsilon}\eta_{\epsilon}\nu$ with the inverse of the collar coordinate parameterization $x\mapsto x+\eta_{\epsilon}(x)\nu(x)$. We may harmlessly assume that $V_{\epsilon}$ is defined on $\Omega_{\epsilon}$ by harmonically extending it to $\Omega_{\epsilon}$. We use $D_{\epsilon}:=\partial_{\epsilon}+V_{\epsilon}\cdot\nabla$ to denote the associated material derivative, which will be tangent to the family of hypersurfaces $\Gamma_{\epsilon}$. 
\\

We also importantly make note of the fact that for every $s\in\mathbb{R}$, we have 
\begin{equation}\label{norm-nonincreasing}
\|\tilde{\omega}_0\|_{H^s(\Omega_{\epsilon})}\leq \|\omega_0\|_{H^s(\Omega)},\hspace{5mm} \|\tilde{v}_0\|_{H^s(\Omega_{\epsilon})}\leq \|v_0\|_{H^s(\Omega)}.  
\end{equation}
Therefore, the bounds in (\ref{inductiveregbound}) are retained from the initial data and, moreover, the rotational component of the energy does not increase.
\\

Now we turn to the energy monotonicity bound \eqref{Emonsurf}. We will need the following two lemmas.
\begin{lemma}[Material derivative bounds]\label{MDBL} 
The following bound holds uniformly in $\epsilon$:
\begin{equation}\label{Bound on super}
\|D_{\epsilon}\nabla \tilde{v}_0\|_{H^{k-1}(\Omega_{\epsilon})}\lesssim_M 1.
\end{equation}
\end{lemma}
\begin{lemma}[Variation of the surface energy]\label{arepresentation}
 Let $k$ be a sufficiently large even integer. Then we have the following estimate for the $\tilde{a}_{0}$ component  of the energy:
\begin{equation*}
\frac{d}{d\epsilon}\|\tilde{a}_0^{-\frac{1}{2}}\mathcal{N}_\epsilon^{k-1}\tilde{a}_0\|_{L^2(\Gamma_\epsilon)}^2\lesssim_M -\epsilon \|\Gamma_{\epsilon}\|_{H^{k+1}}^2+\mathcal{O}_M(1).
\end{equation*}
\end{lemma}
\Cref{MDBL} will allow us to essentially ignore any contributions to the  energy coming from the restriction $\tilde{v}_0$, while \Cref{arepresentation} will help in controlling the variation in $\epsilon$ of the irrotational components of the energy.
\\

Before proving the above lemmas, let us see how they imply the energy monotonicity bound \eqref{Emonsurf}. Thanks to  \Cref{arepresentation} and \eqref{norm-nonincreasing}, we only need to study the $D_ta$ component of the energy. For this, we recall from the Laplace equation \eqref{Dtpdef} that we have 
\begin{equation}\label{Dta_ep}
\tilde{a}_0^{-1}D_t\tilde{a}_0=\tilde{a}_0^{-1}n_{\Gamma_\epsilon}\cdot\nabla \tilde{v}_0\cdot\nabla \tilde{p}_0-\tilde{a}_0^{-1}n_{\Gamma_\epsilon}\cdot\nabla\Delta^{-1}_{\Omega_\epsilon}(\Delta \tilde{v}_0\cdot\nabla \tilde{p}_0+4\text{tr}(\nabla^2\tilde{p}_0\cdot\nabla \tilde{v}_0)+2\text{tr}(\nabla \tilde{v}_0)^3) \ \ \ \text{on} \ \Gamma_{\epsilon}.
\end{equation}
We  apply $D_{\epsilon}\nabla\mathcal{H}_{\epsilon}\mathcal{N}_\epsilon^{k-2}$ to \eqref{Dta_ep} and distribute derivatives. We first dispense with the commutator. Using the standard $H^{\frac{1}{2}}(\Gamma_\epsilon)\to H^1(\Omega_\epsilon)$ bound for $\mathcal{H}_\epsilon$, the  $H^{k-\frac{3}{2}}(\Gamma_\epsilon)$ to $H^{\frac{1}{2}}(\Gamma_\epsilon)$ bound for  $\mathcal{N}^{k-2}_\epsilon$ from \Cref{Iterated N} and the $H^{\frac{1}{2}}(\Gamma_{\epsilon})\to H^{1}(\Omega_{\epsilon})$ bound for $[D_{\epsilon},\mathcal{H}_{\epsilon}]$ from (\ref{H-commutator}), we have 
\begin{equation*}
\|[D_{\epsilon},\nabla\mathcal{H}_{\epsilon}\mathcal{N}_\epsilon^{k-2}](\tilde{a}_0^{-1}D_t\tilde{a}_0)\|_{L^2(\Omega_\epsilon)}\lesssim_M \|[D_{\epsilon},\mathcal{N}_{\epsilon}^{k-2}](\tilde{a}_0^{-1}D_t\tilde{a}_0)\|_{H^{\frac{1}{2}}(\Gamma_{\epsilon})}+\|\tilde{a}_0^{-1}D_t\tilde{a}_0\|_{H^{k-\frac{3}{2}}(\Gamma_{\epsilon})}.    
\end{equation*}
Then,  using the formula (\ref{commutatorexpansion}) and the elliptic estimates in \Cref{BEE} as well as the bound $\|V_{\epsilon}\|_{H^{k-1}(\Gamma_{\epsilon})}\lesssim_M 1$, it is straightforward to verify the commutator bound
\begin{equation*}
\|[D_{\epsilon},\mathcal{N}_{\epsilon}^{k-2}]\|_{H^{k-\frac{3}{2}}(\Gamma_\epsilon)\to H^{\frac{1}{2}}(\Gamma_\epsilon)}\lesssim_M 1.    
\end{equation*}
By  elliptic regularity, $\|\tilde{a}_0^{-1}D_t\tilde{a}_0\|_{H^{k-\frac{3}{2}}(\Gamma_\epsilon)}$ is $\mathcal{O}_M(1)$. Hence, we obtain
\begin{equation*}
\|[D_{\epsilon},\nabla\mathcal{H}_{\epsilon}\mathcal{N}_\epsilon^{k-2}](\tilde{a}_0^{-1}D_t\tilde{a}_0)\|_{L^2(\Omega_\epsilon)}\lesssim_M 1.    
\end{equation*}
Using that
\begin{equation*}
\|\nabla\mathcal{H}_{\epsilon}\mathcal{N}_{\epsilon}^{k-2}D_{\epsilon}(\tilde{a}_0^{-1}D_t\tilde{a}_0)\|_{L^2(\Omega_{\epsilon})}\lesssim_M \|D_{\epsilon}(\tilde{a}_0^{-1}D_t\tilde{a}_0)\|_{H^{k-\frac{3}{2}}(\Gamma_{\epsilon})},    
\end{equation*}
it remains now to estimate $\|D_{\epsilon}(\tilde{a}_0^{-1}D_t\tilde{a}_0)\|_{H^{k-\frac{3}{2}}(\Gamma_{\epsilon})}$. For this, we distribute the operator $D_{\epsilon}$  onto the various terms in (\ref{Dta_ep}). To expedite this process, we  collect a few useful bounds. First, using \Cref{MDBL}, the trace theorem ensures that we have the bound
\begin{equation*}
\|D_{\epsilon}\nabla \tilde{v}_0\|_{H^{k-\frac{3}{2}}(\Gamma_{\epsilon})}+\|D_{\epsilon}\nabla \tilde{v}_0\|_{H^{k-1}(\Omega_{\epsilon})}\lesssim_M 1.    
\end{equation*}
Using the identities for $[\Delta^{-1}_{\Omega_\epsilon},D_{\epsilon}]$ and $D_{\epsilon}n_{\Gamma_{\epsilon}}$ in \Cref{Movingsurfid}, the Laplace equation for $p_{\epsilon}$,  and the fact that $V_{\epsilon}$ is harmonic, we also readily verify the bounds
\begin{equation}\label{Depspressureboundlow}
\|D_{\epsilon}\tilde{p}_{0}\|_{H^{k+\frac{1}{2}}(\Omega_{\epsilon})}+\|D_{\epsilon}n_{\Gamma_{\epsilon}}\|_{H^{k-2}(\Gamma_{\epsilon})}+\|D_{\epsilon}\tilde{a}_0\|_{H^{k-2}(\Gamma_{\epsilon})}\lesssim_M 1
\end{equation}
and
\begin{equation*}
\|[D_{\epsilon},\nabla]\|_{H^{k}(\Omega_{\epsilon})\to H^{k-1}(\Omega_{\epsilon})}+\|D_{\epsilon}n_{\Gamma_{\epsilon}}\|_{H^{k-\frac{3}{2}}(\Gamma_{\epsilon})}+\|D_{\epsilon}\tilde{a}_0\|_{H^{k-\frac{3}{2}}(\Gamma_{\epsilon})}\lesssim_M 1+\|V_{\epsilon}\|_{H^{k-\frac{1}{2}}(\Gamma_{\epsilon})}.      
\end{equation*}
From the above bounds and (\ref{Dta_ep}), we  obtain the estimate
\begin{equation*}
\|D_{\epsilon}(\tilde{a}_0^{-1}D_t\tilde{a}_0)\|_{H^{k-\frac{3}{2}}(\Gamma_{\epsilon})}\lesssim_M 1+\|V_{\epsilon}\|_{H^{k-\frac{1}{2}}(\Gamma_{\epsilon})}.
\end{equation*}
The term $\|V_{\epsilon}\|_{H^{k-\frac{1}{2}}(\Gamma_\epsilon)}$ does not contribute an $\mathcal{O}_M(1)$ error, as it ``loses" half a derivative. However, from the definition and regularization properties of $V_{\epsilon}$, we  have
\begin{equation*}
\|V_{\epsilon}\|_{H^{k-\frac{1}{2}}(\Gamma_\epsilon)}\lesssim_M 1+\epsilon^{\frac{1}{2}}\|\eta_{\frac{\epsilon}{2}}\|_{H^{k+1}(\Gamma_*)}. 
\end{equation*}
Hence, using \Cref{Leibniz} and Cauchy-Schwarz, we obtain
\begin{equation*}
\frac{d}{d\epsilon}\|\nabla\mathcal{H}_{\epsilon}\mathcal{N}_{\epsilon}^{k-2}(\tilde{a}_0^{-1}D_t\tilde{a}_0)\|_{L^2(\Omega_\epsilon)}^2\lesssim_M 1+\delta_0\epsilon\|\Gamma_{\frac{\epsilon}{2}}\|_{H^{k+1}}^2,
\end{equation*}
where $\delta_0>0$ is some sufficiently small constant. Using the parabolic gain from \Cref{arepresentation} (and integrating in $\epsilon$), we notice that the latter term on the right-hand side is harmless as long as $\delta_0=\delta_0(M)$ is small enough. 
\\

It remains now to establish the two lemmas. We begin with \Cref{MDBL}, which is quite simple.
\begin{proof}
Since $\partial_{\epsilon}\tilde{v}_0=0$, we have
\begin{equation*}
D_{\epsilon}\nabla\tilde{v}_0=V_{\epsilon}\cdot\nabla \nabla\tilde{v}_0.
\end{equation*}
Then we use $\|V_{\epsilon}\|_{H^{k-\frac{3}{2}}(\Omega_{\epsilon})}\lesssim_M\epsilon$ and $\|V_{\epsilon}\|_{H^{k-\frac{1}{2}}(\Omega_{\epsilon})}\lesssim_M 1$ together with the inductive bound for $v_0$ from \eqref{inductiveregbound}; namely, $\|v_0\|_{H^{k+1}(\Omega_0)}\leq K(M)\epsilon^{-1}$, to estimate
\begin{equation*}
\|D_{\epsilon}\nabla\tilde{v}_0\|_{H^{k-1}(\Omega_{\epsilon})}\lesssim_M \|V_{\epsilon}\|_{H^{k-\frac{3}{2}}(\Omega_{\epsilon})}\|\tilde{v}_0\|_{H^{k+1}(\Omega_{\epsilon})}+\|V_{\epsilon}\|_{H^{k-1}(\Omega_{\epsilon})}\|\tilde{v}_0\|_{H^{k}(\Omega_{\epsilon})}\lesssim_M 1.    
\end{equation*}
This completes the proof of \Cref{MDBL}.
\end{proof}
Finally, we come to establishing \Cref{arepresentation}, which is where the bulk of the work will be. We begin by establishing the following representation formula for the good variable $\mathcal{N}_\epsilon^{k-1}\tilde{a}_0$:
\begin{equation}\label{arep}
\mathcal{N}_\epsilon^{k-1}\tilde{a}_0=(-1)^m\tilde{a}_0\Delta_{\Gamma_{\epsilon}}^m\kappa_\epsilon+R_\epsilon,
\end{equation}
where $\kappa_{\epsilon}$ is the mean curvature for $\Gamma_{\epsilon}$, $2m=k-2$ and $R_\epsilon$ is a remainder term satisfying the bounds
\begin{equation}\label{Rbound}
\|R_\epsilon\|_{H^{\frac{1}{2}}(\Gamma_\epsilon)}+\epsilon^{\frac{1}{2}}\|R_\epsilon\|_{H^1(\Gamma_\epsilon)}+\|D_{\epsilon}R_\epsilon\|_{L^2(\Gamma_\epsilon)}\lesssim_M 1.
\end{equation}
The importance of \eqref{arep} will be clear later. Roughly speaking, \eqref{arep} states that to leading order $\mathcal{N}_\epsilon^{k-1}\tilde{a}_0$ has a convenient local expression. Such an observation will facilitate the  use of local formulas later on, consistent with our choice of domain regularization.
Observe also that in \eqref{Rbound}, we have $D_{\epsilon}R_\epsilon=\mathcal{O}_{L^2(\Gamma_\epsilon)}(1)$. This is stronger than the expected bound $D_{\epsilon}R_\epsilon=\mathcal{O}_{H^{-\frac{1}{2}}(\Gamma_\epsilon)}(1)$. The reason for this improvement is the  bound \eqref{Bound on super} for $D_{\epsilon}\nabla\tilde{v}_0$; this term would have had to have been treated more carefully if we had attempted to regularize  the velocity in this step of the argument. 
\begin{proof}[Proof of \eqref{arep}]
In the following analysis, $R_\epsilon$ will  generically denote a remainder term satisfying (\ref{Rbound}) which is allowed to change from line to line. Likewise, $\tilde{R}_\epsilon$ will denote an analogous remainder term but with 
\begin{equation}
\tilde{R}_\epsilon=\mathcal{O}_{H^{k-\frac{3}{2}}(\Gamma_\epsilon)}(1), \qquad \epsilon^{\frac{1}{2}}\tilde{R}_\epsilon=\mathcal{O}_{H^{k-1}(\Gamma_\epsilon)}(1), \qquad D_{\epsilon}\tilde{R}_\epsilon=\mathcal{O}_{H^{k-2}(\Gamma_\epsilon)}(1).
\end{equation}
To establish (\ref{arep}), we begin by relating $\mathcal{N}_\epsilon \tilde{a}_0$ to the mean curvature. Indeed, from $\Delta_{\Gamma_\epsilon}\tilde{p}_0=0$ and the formula 
\begin{equation*}
\Delta \tilde{p}_0|_{\Gamma_\epsilon}=\Delta_{\Gamma_{\epsilon}}\tilde{p}_0-\kappa_\epsilon n_{\Gamma_\epsilon}\cdot\nabla \tilde{p}_0+D^2\tilde{p}_0(n_{\Gamma_\epsilon},n_{\Gamma_\epsilon}),
\end{equation*}
we have 
\begin{equation*}
\begin{split}
\tilde{a}_0\kappa_\epsilon &=-n_in_j\partial_i\partial_j\tilde{p}_0+\Delta \tilde{p}_0
\\
&=-n_in_j\partial_i\partial_j\tilde{p}_0-\text{tr}(\nabla \tilde{v}_0)^2
\\
&=-n_in_j\partial_i\partial_j\tilde{p}_0+\tilde{R}_\epsilon,
\end{split}
\end{equation*}
where in the last line, we used \Cref{MDBL} to check the remainder property for $D_{\epsilon}\tilde{R}_{\epsilon}$ and the inductive assumption (\ref{inductiveregbound}) and interpolation to control $\epsilon^{\frac{1}{2}}\tilde{R}_{\epsilon}$ in $H^{k-1}(\Gamma_{\epsilon})$. We now further expand using the Laplace equation for $\tilde{p}_0$,
\begin{equation*}
\begin{split}
-n_in_j\partial_i\partial_j\tilde{p}_0 &=n_j\mathcal{N}_\epsilon(n_j\tilde{a}_0)+n_jn_{\Gamma_\epsilon}\cdot\nabla\Delta^{-1}_{\Omega_\epsilon}\partial_j\text{tr}(\nabla \tilde{v}_0)^2
\\
&=n_j\mathcal{N}_\epsilon(n_j\tilde{a}_0)+\tilde{R}_\epsilon.
\end{split}
\end{equation*}
Next, we expand
\begin{equation*}
\begin{split}
n_j\mathcal{N}_\epsilon(n_j\tilde{a}_0)&=\mathcal{N}_\epsilon \tilde{a}_0+\tilde{a}_0 n_j\mathcal{N}_\epsilon n_j-2 n_jn_{\Gamma_\epsilon}\cdot\nabla \Delta^{-1}_{\Omega_\epsilon}(\nabla\mathcal{H}_\epsilon n_j\cdot\nabla\mathcal{H}_\epsilon \tilde{a}_0)
\\
&=\mathcal{N}_\epsilon \tilde{a}_0+\tilde{a}_0 n_j\mathcal{N}_\epsilon n_j+\tilde{R}_\epsilon
\\
&=\mathcal{N}_\epsilon \tilde{a}_0+\tilde{R}_\epsilon,
\end{split}
\end{equation*}
where in the first equality, we used the Leibniz rule \eqref{DNLeibniz} for $\mathcal{N}_\epsilon$. From the second to the third line, we used the Leibniz rule again, and that $\mathcal{N}_\epsilon(n_jn_j)=0$. In summary, what we have so far is the identity
\begin{equation}\label{preliminaryidentitycurvature}
\mathcal{N}_{\epsilon}\tilde{a}_0=\tilde{a}_0\kappa_{\epsilon}+\tilde{R}_{\epsilon}.    
\end{equation}
The next step is to obtain the leading order identity,
\begin{equation}\label{preliminarycommutatoridentity}
\mathcal{N}_{\epsilon}^{k-1}\tilde{a}_0=\tilde{a}_0\mathcal{N}_{\epsilon}^{k-2}(\tilde{a}_0^{-1}\mathcal{N}_{\epsilon}\tilde{a}_0)+R_{\epsilon}    
\end{equation}
by applying $\mathcal{N}_{\epsilon}^{k-2}$ to $\mathcal{N}_{\epsilon}\tilde{a}_0$ and then commuting $\tilde{a}_0^{-1}$ with $\mathcal{N}_{\epsilon}^{k-2}$. Here, $R_{\epsilon}$ can be seen to satisfy the required bounds through the use of the various commutator identities for $D_{\epsilon}$ listed in \Cref{Movingsurfid} as well as the Leibniz rule \eqref{DNLeibniz}, the elliptic estimates in \Cref{BEE} for $\mathcal{N}_{\epsilon}$ and  the estimates in (\ref{Depspressureboundlow}).
\\

Before proceeding further, we recall the formula
\begin{equation}\label{approximationformulaN}
-(\Delta_{\Gamma_{\epsilon}}+\mathcal{N}_{\epsilon}^2)f=\kappa_{\epsilon}\mathcal{N}_{\epsilon}f-2n_{\Gamma_{\epsilon}}\cdot\nabla (-\Delta_{\Omega_\epsilon})^{-1}(\nabla\mathcal{H}_{\epsilon}n_{\Gamma_{\epsilon}}\cdot \nabla^2\mathcal{H}_{\epsilon}f)-\mathcal{N}_{\epsilon}n_{\Gamma_{\epsilon}}\cdot (\mathcal{N}_{\epsilon}fn_{\Gamma_{\epsilon}}+\nabla^{\top}f)
\end{equation}
from \cite[Equation A.13]{sz}. Also, we recall from (4.23) of \cite{sz} the commutator estimate
\begin{equation}\label{com from shatah}
\|[\Delta_{\Gamma_{\epsilon}},D_{\epsilon}]\|_{H^{s}(\Gamma_{\epsilon})\to H^{s-2}(\Gamma_{\epsilon})}\lesssim_M \|V_{\epsilon}\|_{H^{k-\frac{1}{2}}(\Omega_{\epsilon})}\lesssim_M 1,\hspace{5mm}1\leq s\leq k-1  .
\end{equation}
Then, given that $k-2=2m$ is even,  applying (\ref{preliminaryidentitycurvature}), (\ref{preliminarycommutatoridentity}) and iterating (\ref{approximationformulaN}) $m$ times, we have
\begin{equation*}
\begin{split}
\mathcal{N}_{\epsilon}^{k-1}\tilde{a}_0=\tilde{a}_0\mathcal{N}_{\epsilon}^{k-2}(\tilde{a}_0^{-1}\mathcal{N}_{\epsilon}\tilde{a}_0)+R_\epsilon=(-1)^m\tilde{a}_0\Delta_{\Gamma_\epsilon}^m(\tilde{a}_0^{-1}\mathcal{N}_{\epsilon}\tilde{a}_0)+R_\epsilon=(-1)^m\tilde{a}_0\Delta_{\Gamma_\epsilon}^m\kappa_\epsilon+R_\epsilon,
\end{split}
\end{equation*}
where by straightforward (but slightly tedious) computation we verify that the remainder term $R_{\epsilon}$ has the needed bounds through the use of the various commutator identities for $D_{\epsilon}$ listed in \Cref{Movingsurfid} as well as the above estimates (\ref{preliminaryidentitycurvature})-(\ref{com from shatah}), the relevant elliptic estimates in \Cref{BEE} and (\ref{Depspressureboundlow}). 
\end{proof}
Now, we are ready to establish the differential inequality in \Cref{arepresentation}. For the sake of clarity, let us begin by assuming that the reference hypersurface is given by $\{x_d=0\}$ and that $\Gamma_{\epsilon}$ is literally given by $x_d=\eta_{\epsilon}(x_1,...,x_{d-1})$.  Then the mean curvature and Laplace-Beltrami operator take the form 
\begin{equation*}
\kappa_\epsilon=-\frac{\Delta\eta_{\epsilon}}{(1+|\nabla\eta_{\epsilon}|^2)^{\frac{1}{2}}}+\frac{\partial_i\eta_\epsilon\partial_j\eta_{\epsilon}\partial_{i}\partial_j\eta_{\epsilon}}{(1+|\nabla\eta_{\epsilon}|^2)^{\frac{3}{2}}},
\end{equation*}
and
\begin{equation}\label{laplacecoord}
\Delta_{\Gamma_{\epsilon}}f=\frac{1}{\sqrt{1+|\nabla\eta_{\epsilon}|^2}}\partial_i(g^{ij}_\epsilon\sqrt{1+|\nabla\eta_{\epsilon}|^2}\partial_jf),
\end{equation}
where $(g^{ij}_\epsilon)=(\delta_{ij}+\partial_i\eta_{\epsilon}\partial_j\eta_{\epsilon})^{-1}$. Observe that $g^{ij}_\epsilon$ and $\nabla\eta_{\epsilon}$ are one derivative more regular than $\kappa_\epsilon$. Therefore, by  making  use of the identity $\partial_\epsilon\eta_\epsilon =2\epsilon\Delta_{\Gamma_*}\eta_\epsilon$ and the regularization bound \eqref{surfbound}, we can differentiate in $\epsilon$ and commute $2\epsilon\Delta_{\Gamma_*}$ with these coefficients to obtain,
\begin{equation}\label{checked for graph}
\begin{split}
(D_{\epsilon}(\mathcal{N}_\epsilon^{k-1}\tilde{a}_0))_* &=2(-1)^{m}\epsilon\Delta_{\Gamma_*}(\tilde{a}_0\Delta^m_{\Gamma_\epsilon}\kappa_{\epsilon})_*+\mathcal{O}_{L^2(\Gamma_*)}(1),
\end{split}
\end{equation}
where we define $f_*(x):=f(x+\eta_{\epsilon}(x)\nu(x))$ for a function $f$ defined on $\Gamma_{\epsilon}$. Moreover, by an exercise in local coordinates, the reader may check that \eqref{checked for graph}, as written, is valid for general reference hypersurfaces $\Gamma_*$.  Now, using \eqref{Leibniz on moving hypersurfaces general}, the bounds for $R_\epsilon$, and Cauchy-Schwarz, it follows that  
\begin{equation*}
\begin{split}
\frac{d}{d\epsilon}\|\tilde{a}_0^{-\frac{1}{2}}\mathcal{N}_\epsilon^{k-1}\tilde{a}_0\|_{L^2(\Gamma_\epsilon)}^2\lesssim_M 1 -\epsilon\||D|_{\Gamma_*}(\Delta^m_{\Gamma_\epsilon}\kappa_{\epsilon})_*\|_{L^2(\Gamma_*)}^2,
\end{split}
\end{equation*}
where $|D|_{\Gamma_*}=(-\Delta_{\Gamma_*})^{\frac{1}{2}}$. To conclude, we now only need to show the coercivity type bound
\begin{equation*}
\|\eta_{\epsilon}\|_{H^{k+1}(\Gamma_*)}\lesssim_M 1+\||D|_{\Gamma_*}(\Delta^m_{\Gamma_\epsilon}\kappa_{\epsilon})_*\|_{L^2(\Gamma_*)}.
\end{equation*}
For this, we begin with  \Cref{curvaturebound} which yields
\begin{equation*}
\|\eta_{\epsilon}\|_{H^{k+1}(\Gamma_*)}\lesssim_M 1+\|\kappa_{\epsilon}\|_{H^{k-1}(\Gamma_\epsilon)}. 
\end{equation*}
Then, using (\ref{laplacecoord}) and the fact that $2m=k-2$ (this being relevant for ensuring domain dependent implicit constants are at most $\mathcal{O}_M(1)$ in size), one can easily verify the ellipticity bound
\begin{equation*}
\|\kappa_{\epsilon}\|_{H^{k-1}(\Gamma_{\epsilon})}\lesssim_M 1+\|\Delta_{\Gamma_{\epsilon}}^m\kappa_{\epsilon}\|_{H^1(\Gamma_{\epsilon})}\lesssim_M 1+\||D|_{\Gamma_*}(\Delta_{\Gamma_{\epsilon}}^m\kappa)_*\|_{L^2(\Gamma_*)}.    
\end{equation*}
This concludes the proof.
\end{proof}
\subsection{Step 2: Velocity regularization}\label{VR}
Now, we aim to regularize the velocity $\tilde{v}_0$ on the $\epsilon^{-1}$ scale, which will help us to improve the regularization constant in (\ref{inductiveregbound}). This will be needed to compensate for the losses in this constant in the upcoming transport step of the argument. Thanks to the previous step, we are reduced to the situation of regularizing on a fixed domain which has boundary regularized at the $\epsilon^{-1}$ scale. To perform this step of the regularization,  we decompose the velocity $\tilde{v}_0$  into a rotational component which is tangent to the boundary and an irrotational component. Roughly speaking, we will then regularize the irrotational component of $\tilde{v}_0$ and leave the rotational component alone. We will then reconstruct the regularized velocity using  the regularized irrotational part and the original (not regularized) rotational part of $\tilde{v}_0$. The precise procedure for doing this will come with some slight technical subtleties due to the fact that the normal to the surface is half a derivative less regular than the trace of the velocity on the boundary. We will outline these nuances  in more detail shortly. Heuristically, the reason  it is unnecessary to regularize the rotational part of $\tilde{v}_0$ in this construction is because the vorticity will not lose derivatives in the transport step of our argument later. In other words, the vorticity bound in (\ref{inductiveregbound}) is expected  to only worsen by an $\mathcal{O}_M(1)$ error when measured in $H^k$ and an $\mathcal{O}_M(\epsilon^{-1})$ error when measured in $H^{k+1}$, which is acceptable. 
\begin{proposition}\label{Prop velo reg} Given the pair $(\tilde{v}_0,\Gamma_{\epsilon})$ from the previous step, there exists a regularization $\tilde{v}_0\mapsto v_{\epsilon}$ defined on $\Omega_{\epsilon}$ which satisfies:
\begin{enumerate}
\item (Energy monotonicity).
\begin{equation*}\label{Emonsurf2}
\mathcal{E}^{k}(v_{\epsilon},\Gamma_{\epsilon})\leq (1+C(M)\epsilon)\mathcal{E}^{k}(\tilde{v}_0,\Gamma_{\epsilon}).
\end{equation*}
\item (Good pointwise approximation).
\begin{equation}\label{approxdomain2}
\begin{cases}
&v_{\epsilon}=\tilde{v}_0+\mathcal{O}_{C^1}(\epsilon^2),
\\
&\nabla\cdot v_{\epsilon}=0.
\end{cases}
\end{equation}
\item (Regularization bounds). For each $n=1,2$ and $K(M)$ large enough, there holds
\begin{equation}\label{8.44}
\|v_{\epsilon}\|_{H^{k+n}(\Omega_\epsilon)}\leq \frac{1}{4}K(M)\epsilon^{-n}.
\end{equation} 
\end{enumerate}
\end{proposition}
\begin{remark}
 The bound in \eqref{8.44} with $n=1$ ensures  that the constant in (\ref{inductiveregbound}) is improved at this stage. The $H^{k+2}$ bound will be needed to close the bootstrap in the final Euler plus transport step of the iteration in the next section because this step loses derivatives for the velocity. 
\end{remark}
\begin{proof}
We begin by recalling the rotational/irrotational decomposition of $\tilde{v}_0$ from Appendix A of \cite{sz}:
\begin{equation*}
\tilde{v}_0:=\tilde{v}_0^{rot}+\tilde{v}_0^{ir},   
\end{equation*}
where for a divergence free function $v$, we have $v^{ir}:=\nabla\mathcal{H}_{\epsilon}\mathcal{N}_\epsilon^{-1}(v\cdot n_{\Gamma_{\epsilon}})$. Na\"ively, we would like to directly regularize  the irrotational part of $\tilde{v}_0$. However, this does not quite work because the normal $n_{\Gamma_\epsilon}$ is half a derivative less regular than the trace of $\tilde{v}_0$ on $\Gamma_{\epsilon}$. To get around this, we will regularize the irrotational part of a suitable high frequency component of $\tilde{v}_0$. More precisely, let us consider a subregularization $v_-$ of $\tilde{v}_0$, defined by $v_{-}:=\Psi_{\leq\epsilon^{-\frac{1}{2}}}\tilde{v}_0$, which lives on an $\epsilon^{\frac{1}{2}}$ enlargement of $\Omega_{\epsilon}$. We then define $w:=\tilde{v}_0-v_{-}$. Loosely speaking, we think of $w$ as the portion of $\tilde{v}_0$ with frequency greater than $\epsilon^{-\frac{1}{2}}$. In contrast to the full irrotational part of $\tilde{v}_0$, it is safe to regularize the irrotational part of $w$. The heuristic reason for this is that at leading order the term $w\cdot n_{\Gamma_\epsilon}$ can be interpreted  as a high-low paraproduct. That is, the contribution of the portion where $n_{\Gamma_\epsilon}$ is at comparable or higher frequency compared to $w$ is lower order as there is still a nontrivial high frequency component of $w$ to compensate for the $\frac{1}{2}$ derivative discrepancy between the trace of $w$ and $n_{\Gamma_\epsilon}$. 
 \\
 
 For the irrotational part of $w$, the regularization we choose has to respect the energy monotonicity bound. We will see below that the spectral multiplier $\mathcal{P}_{\leq\epsilon^{-1}}(\mathcal{N}_\epsilon):=1_{[-\epsilon^{-1},\epsilon^{-1}]}(\mathcal{N}_\epsilon)$ is very convenient for this purpose. We therefore define the irrotational component of our regularization $v_{\epsilon}$ of $\tilde{v}_0$ by removing the high frequency part of $w\cdot n_{\Gamma_\epsilon}$ as follows:
\begin{equation*}
\begin{split}
v_{\epsilon}^{ir}&:=\tilde{v}_0^{ir}-\nabla\mathcal{H}_\epsilon\mathcal{N}_\epsilon^{-1}\mathcal{P}_{>\epsilon^{-1}}(w\cdot n_{\Gamma_\epsilon})
\\
&=v_{-}^{ir}+\nabla\mathcal{H}_\epsilon\mathcal{N}_\epsilon^{-1}\mathcal{P}_{\leq\epsilon^{-1}}(w\cdot n_{\Gamma_\epsilon}).
\end{split}
\end{equation*}
For simplicity, let us write
\begin{equation*}
w_{\epsilon}^{ir}:=\nabla\mathcal{H}_\epsilon\mathcal{N}_\epsilon^{-1}\mathcal{P}_{\leq\epsilon^{-1}}(w\cdot n_{\Gamma_\epsilon}).    
\end{equation*}
We  define the full regularization $v_{\epsilon}$ of $\tilde{v}_0$ by
\begin{equation*}
v_{\epsilon}:=\tilde{v}_0^{rot}+v_{\epsilon}^{ir}    .
\end{equation*}
If $k$ is large enough, the combination of  Sobolev embedding, ellipticity of $\mathcal{N}$ and spectral calculus allows us to easily establish the pointwise approximation property \eqref{approxdomain2}. Next, we establish the regularization bound \eqref{8.44} for $v_{\epsilon}$. We begin by writing
\begin{equation*}
v_{\epsilon}=v_{-}+w_{\epsilon}^{ir}+w^{rot},
\end{equation*}
where $w^{rot}$ is the rotational part of $w$. We then estimate piece by piece. It is first of all clear that the corresponding bound holds for $v_{-}$. So, we turn to estimating $w^{ir}_{\epsilon}$. For this, we note the following preliminary bound for $\mathcal{N}_\epsilon^{-1}$ on the space $\dot{H}^{s}(\Gamma_{\epsilon}):=\{f\in H^{s}(\Gamma_{\epsilon}): \int_{\Gamma_\epsilon}f=0\}$  from Proposition A.5 in \cite{sz}:
\begin{equation}\label{negativepowerbound}
\|\mathcal{N}_\epsilon^{-1}f\|_{\dot{H}^s(\Gamma_\epsilon)}\lesssim_M \|f\|_{H^{s-1}(\Gamma_{\epsilon})},\hspace{5mm}0\leq s\leq 1.
\end{equation}
From this and the functional calculus for $\mathcal{N}_\epsilon$, we deduce in particular the low regularity bound
\begin{equation}\label{lowregw}
\|\mathcal{P}_{\leq \epsilon^{-1}}\mathcal{N}_\epsilon^{-1}(w\cdot n_{\Gamma_\epsilon})\|_{L^2(\Gamma_{\epsilon})}\lesssim_M \|w\cdot n_{\Gamma_\epsilon}\|_{H^{-1}(\Gamma_{\epsilon})}.    
\end{equation}
This will be useful for handling the low frequency errors in the estimate for $w^{ir}_\epsilon$. Next we check that  \eqref{negativepowerbound} and \eqref{lowregw}, in conjunction with \Cref{boundaryest}, \Cref{Hbounds}, \Cref{ellipticity} and   the regularization bounds for $n_{\Gamma_{\epsilon}}$ and $w$, yield
\begin{equation*}
\begin{split}
\|w_{\epsilon}^{ir}\|_{H^{k+n}(\Omega_\epsilon)}&\lesssim_{M,n} \|\Gamma_{\epsilon}\|_{H^{k+\frac{1}{2}+n}}\|w\cdot n_{\Gamma_\epsilon}\|_{H^{k-2}(\Gamma_\epsilon)}+\|\mathcal{P}_{\leq\epsilon^{-1}}(w\cdot n_{\Gamma_\epsilon})\|_{H^{k-\frac{1}{2}+n}(\Gamma_{\epsilon})}\lesssim_M\epsilon^{-n},
\end{split}
\end{equation*}
where the implicit constant can be taken to be much smaller than $K(M)$ since $K(M)\gg M$. Note that in the above estimate, we used the paraproduct structure of $w\cdot n_{\Gamma_\epsilon}$. More specifically, in the case when  $k-\frac{1}{2}$ derivatives fall on  $n_{\Gamma_\epsilon}$, we compensated the half derivative loss by an $\epsilon^{\frac{1}{2}}$ gain from  $w$.
\\

Finally, we move on to showing the regularization bound for $w^{rot}$. Here, we use \Cref{Balanced div-curl} to obtain
\begin{equation*}
\begin{split}
\|w^{rot}\|_{H^{k+n}(\Omega_{\epsilon})}&\lesssim_M \|w^{rot}\|_{L^2(\Omega_{\epsilon})}+\|\nabla\times 
 w\|_{H^{k+n-1}(\Omega_{\epsilon})}+\|\Gamma_{\epsilon}\|_{H^{k+n-\frac{1}{2}}}+\|\nabla^{\top}w^{rot}\cdot n_{\Gamma_\epsilon}\|_{H^{k+n-\frac{3}{2}}(\Gamma_{\epsilon})}
\\
&\lesssim_{M,K'(M)} \epsilon^{-n}+\|w^{rot}\|_{L^2(\Omega_{\epsilon})}+\|\nabla^{\top}w^{rot}\cdot n_{\Gamma_\epsilon}\|_{H^{k+n-\frac{3}{2}}(\Gamma_{\epsilon})} ,
\end{split}
\end{equation*}
where we used (\ref{inductiveregbound}) for $\tilde{\omega}_0$. Again, the implicit constant can be taken to be much smaller than $K(M)$ if $K'(M)$ in (\ref{inductiveregbound}) is small enough compared to $K(M)$. To estimate $\|w^{rot}\|_{L^2(\Omega_{\epsilon})}$, we simply use (\ref{negativepowerbound}), the identity $w^{rot}=w-w^{ir}$ and the $H^{\frac{1}{2}}(\Gamma_\epsilon)\to H^1(\Omega_\epsilon)$ bound for $\mathcal{H}_{\epsilon}$ to crudely estimate 
\begin{equation}\label{rotL2}
\|w^{rot}\|_{L^2(\Omega_{\epsilon})}\lesssim_M \|w\|_{H^1(\Omega_{\epsilon})}.    
\end{equation}
Then, using 
\begin{equation*}
\nabla^{\top}w^{rot}\cdot n_{\Gamma_\epsilon}=-w^{rot}\cdot\nabla^{\top}n_{\Gamma_\epsilon},
\end{equation*}
  \Cref{boundaryest}, \Cref{baltrace} and the regularization bounds for $\Gamma_{\epsilon}$, we have (if $k$ is large enough)
\begin{equation*}
\|w^{rot}\|_{H^{k+n}(\Omega_{\epsilon})}\lesssim_M \epsilon^{-n}+\|w^{rot}\|_{H^{k-1+n}(\Omega_{\epsilon})}+\epsilon^{-\frac{1}{2}-n}\|w^{rot}\|_{H^{k-2}(\Omega_{\epsilon})},
\end{equation*}
which implies by interpolation and (\ref{rotL2}) that 
\begin{equation*}
\|w^{rot}\|_{H^{k+n}(\Omega_{\epsilon})}\lesssim_{M,n} \epsilon^{-n}+\epsilon^{-\frac{1}{2}-n}\|w^{rot}\|_{H^{k-2}(\Omega_{\epsilon})}.
\end{equation*}
From \Cref{Balanced div-curl}, the inequality (\ref{rotL2}) and the fact that $w$ is localized to frequency $\geq \epsilon^{-\frac{1}{2}}$, we easily obtain
\begin{equation*}
\|w^{rot}\|_{H^{k-2}(\Omega_{\epsilon})}\lesssim_M \|w\|_{H^{k-2}(\Omega_{\epsilon})}\lesssim_M \epsilon. 
\end{equation*}
Therefore, we have
\begin{equation*}
\|w^{rot}\|_{H^{k+n}(\Omega_{\epsilon})}\lesssim_{M,n}\epsilon^{-n},
\end{equation*}
with implicit constant much smaller than $K(M)$.
 This yields the desired regularization bounds for $v_{\epsilon}$. 
 \\
 
 Next, we turn to the energy monotonicity. The domain is fixed in this step, so it is advantageous to compare  the difference between $\mathcal{E}^k(v_{\epsilon},\Gamma_{\epsilon})$ and $\mathcal{E}^k(\tilde{v}_0,\Gamma_{\epsilon})$ directly. It will also be convenient to write the first term  in $\mathcal{E}^k(v,\Gamma)$ as a surface integral:
\begin{equation*}
\|\nabla\mathcal{H}\mathcal{N}^{k-2}(a^{-1}D_ta)\|_{L^2(\Omega)}^2=\|\mathcal{N}^{k-\frac{3}{2}}(a^{-1}D_ta)\|_{L^2(\Gamma)}^2   , 
\end{equation*}
using integration by parts and the functional calculus for $\mathcal{N}$.  Moreover, since the vorticity $\omega_{\epsilon}$ is the same as $\tilde{\omega}_0$, we may restrict our attention to the two surface components of the energy in this step of the argument. 
 \\
 
We begin with a simple algebraic identity for the $a_{\epsilon}$ component of the surface energy:
 \begin{equation*}
\begin{split}
\int_{\Gamma_{\epsilon}}a_{\epsilon}^{-1}|\mathcal{N}_\epsilon^{k-1}a_{\epsilon}|^2\,dS&=\int_{\Gamma_{\epsilon}}\tilde{a}_0^{-1}|\mathcal{N}_\epsilon^{k-1}\tilde{a}_0|^2\,dS+2\int_{\Gamma_{\epsilon}}a_{\epsilon}^{-1}\mathcal{N}_\epsilon^{k-1}a_{\epsilon}\mathcal{N}_\epsilon^{k-1}(a_{\epsilon}-\tilde{a}_0)\,dS
\\
&-\|a_\epsilon^{-\frac{1}{2}}\mathcal{N}_\epsilon^{k-1}(a_{\epsilon}-\tilde{a}_0)\|_{L^2(\Gamma_\epsilon)}^2+\mathcal{O}_M(\epsilon).
\end{split}
\end{equation*}
To derive an analogous relation for the other portion of the surface energy, we note that from the integer bounds for $\mathcal{N}$ in \Cref{BEE} and the identity $\|\mathcal{N}^{k-\frac{3}{2}}f\|_{L^2(\Gamma)}=\|\nabla\mathcal{H}\mathcal{N}^{k-2}f\|_{L^2(\Omega)}$, we have the estimate $\|\mathcal{N}_{\epsilon}^{k-\frac{3}{2}}\|_{H^{k-\frac{3}{2}}(\Gamma_\epsilon)\to L^2(\Gamma_\epsilon)}\lesssim_M 1$. On the other hand, we have the elliptic regularity estimate
\begin{equation*}
\|\tilde{a}_0-a_{\epsilon}\|_{H^{k-\frac{3}{2}}(\Gamma_{\epsilon})}\lesssim_M\|\tilde{p}_0-p_{\epsilon}\|_{H^{k}(\Omega_{\epsilon})}\lesssim_M \|\tilde{v}_0-v_{\epsilon}\|_{H^{k-1}(\Omega_{\epsilon})}\lesssim_M \epsilon.
\end{equation*}
Together, these imply that
\begin{equation*}\label{Emonid}
\begin{split}
\int_{\Gamma_{\epsilon}}|\mathcal{N}_\epsilon^{k-\frac{3}{2}}(a_{\epsilon}^{-1}D_ta_{\epsilon})|^2\,dS&=\int_{\Gamma_{\epsilon}}|\mathcal{N}_\epsilon^{k-\frac{3}{2}}(\tilde{a}_0^{-1}D_t\tilde{a}_0)|^2\,dS+2\int_{\Gamma_{\epsilon}}\mathcal{N}_\epsilon^{k-\frac{3}{2}}(a_{\epsilon}^{-1}D_ta_{\epsilon})\mathcal{N}_\epsilon^{k-\frac{3}{2}}(a_{\epsilon}^{-1}(D_ta_{\epsilon}-D_t\tilde{a}_0))\,dS
\\
&-\|\mathcal{N}_\epsilon^{k-\frac{3}{2}}(a_{\epsilon}^{-1}(D_ta_{\epsilon}-D_t\tilde{a}_0))\|_{L^2(\Gamma_\epsilon)}^2+\mathcal{O}_M(\epsilon).
\end{split}
\end{equation*}
Motivated by the identities above, let us define the ``energy" corresponding to $\tilde{v}_0-v_{\epsilon}$ by 
\begin{equation*}
\begin{split}
\mathcal{E}^{k}(\tilde{v}_0-v_{\epsilon})&:=\|\mathcal{N}_\epsilon^{k-\frac{3}{2}}(a_{\epsilon}^{-1}(D_t\tilde{a}_0-D_ta_{\epsilon}))\|_{L^2(\Gamma_\epsilon)}^2+\|a_{\epsilon}^{-\frac{1}{2}}\mathcal{N}_\epsilon^{k-1}(\tilde{a}_0-a_{\epsilon})\|_{L^2(\Gamma_\epsilon)}^2.
\end{split}
\end{equation*}
In light of the above identities, it suffices to show that
\begin{equation*}
\begin{split}
2\int_{\Gamma_{\epsilon}}\mathcal{N}_\epsilon^{k-\frac{3}{2}}(a_{\epsilon}^{-1}D_ta_{\epsilon})\mathcal{N}_\epsilon^{k-\frac{3}{2}}&(a_{\epsilon}^{-1}(D_ta_{\epsilon}-D_t\tilde{a}_0))\,dS+2\int_{\Gamma_{\epsilon}}a_{\epsilon}^{-1}\mathcal{N}_\epsilon^{k-1}a_{\epsilon}\mathcal{N}_\epsilon^{k-1}(a_{\epsilon}-\tilde{a}_0)\,dS
\\
&\leq C(M)\epsilon+\mathcal{E}^{k}(\tilde{v}_0-v_\epsilon).
\end{split}
\end{equation*}
Our starting point is to observe the leading order relation given in the following lemma.
\begin{lemma}\label{Dtarel}
We have the following relation between $D_ta_{\epsilon}-D_t\tilde{a}_0$ and $(v_{\epsilon}-\tilde{v}_0)\cdot n_{\Gamma_\epsilon}$\emph{:}
\begin{equation}
a_{\epsilon}^{-1}(D_ta_{\epsilon}-D_t\tilde{a}_0)=-\mathcal{N}_\epsilon((v_{\epsilon}-\tilde{v}_0)\cdot n_{\Gamma_\epsilon})+\mathcal{O}_{H^{k-\frac{3}{2}}(\Gamma_\epsilon)}(\epsilon).
\end{equation}
\end{lemma}
\begin{proof}
We begin by noting the bound 
\begin{equation}\label{v0diffbound}
\|\tilde{v}_0-v_{\epsilon}\|_{H^{k-1}(\Omega_{\epsilon})}\lesssim_M \epsilon  
\end{equation}
and the elliptic regularity estimate
\begin{equation*}
\|\tilde{p}_0-p_{\epsilon}\|_{H^k(\Omega_{\epsilon})}\lesssim_M \|\tilde{v}_0-v_{\epsilon}\|_{H^{k-1}(\Omega_{\epsilon})}\lesssim_M\epsilon. 
\end{equation*}
Using the equation for $D_tp$ from (\ref{Dtpdef}) we may therefore  write
\begin{equation*}
D_ta_{\epsilon}-D_t\tilde{a}_0=n_{\Gamma_{\epsilon}}\cdot \nabla (v_{\epsilon}-\tilde{v}_0)\cdot\nabla p_{\epsilon}-n_{\Gamma_{\epsilon}}\cdot\nabla\Delta^{-1}(\Delta (v_{\epsilon}-\tilde{v_0})\cdot\nabla p_{\epsilon})+\mathcal{O}_{H^{k-\frac{3}{2}}(\Gamma_{\epsilon})}(\epsilon).    
\end{equation*}
Then, using the standard identity $\mathcal{N}f_{|\Gamma}=n\cdot\nabla f-n\cdot\nabla\Delta^{-1}\Delta f$ and commuting $n_{\Gamma_\epsilon}\cdot\nabla $ in the first term and $\Delta$ in the second term above, we can verify, from (\ref{v0diffbound}),
\begin{equation*}
D_ta_{\epsilon}-D_t\tilde{a}_0=-\mathcal{N}_\epsilon(a_{\epsilon}(v_{\epsilon}-\tilde{v}_0)\cdot n_{\Gamma_\epsilon})+\mathcal{O}_{H^{k-\frac{3}{2}}(\Gamma_{\epsilon})}(\epsilon)    .
\end{equation*}
The conclusion then follows by commuting $\mathcal{N}_\epsilon$ with $a_{\epsilon}$ using the Leibniz rule for $\mathcal{N}_\epsilon$ and (\ref{v0diffbound}). In the case when everything falls on $a_\epsilon$, we also compensate with the surface regularization bound \eqref{surfbound} and the associated improvement in the bound for $\tilde{v}_0-v_{\epsilon}$ when measured in lower regularity Sobolev norms.
\end{proof}
We now turn to the $a_{\epsilon}$ component of the energy, which is straightforward. Indeed, by elliptic regularity,
\begin{equation*}
\begin{split}
2\int_{\Gamma_{\epsilon}}a_{\epsilon}^{-1}\mathcal{N}_\epsilon^{k-1}a_{\epsilon}\mathcal{N}_\epsilon^{k-1}(a_{\epsilon}-\tilde{a}_0)\,dS\lesssim_M \|a_{\epsilon}-\tilde{a}_0\|_{H^{k-1}(\Gamma_{\epsilon})}&\lesssim_M \|v_{\epsilon}-\tilde{v}_0\|_{H^{k-\frac{1}{2}}(\Omega_\epsilon)}.
\end{split}
\end{equation*}
To estimate $v_{\epsilon}-\tilde{v}_0$, we observe the identity $v_{\epsilon}-\tilde{v}_0=\nabla\mathcal{H}_\epsilon\mathcal{N}_\epsilon^{-1}\mathcal{P}_{>\epsilon^{-1}}((v_{\epsilon}-\tilde{v}_0)\cdot n_{\Gamma_\epsilon})$, which follows from the idempotence $\mathcal{P}_{>\epsilon^{-1}}=\mathcal{P}_{>\epsilon^{-1}}^2$. Using this, \Cref{Dtarel} and ellipticity of $\mathcal{N}_\epsilon$, we have
\begin{equation*}
\|v_{\epsilon}-\tilde{v}_0\|_{H^{k-\frac{1}{2}}(\Omega_\epsilon)}\lesssim_M \epsilon^{\frac{1}{2}}\|(v_{\epsilon}-\tilde{v}_0)\cdot n_{\Gamma_\epsilon}\|_{H^{k-\frac{1}{2}}(\Gamma_{\epsilon})}\lesssim_M \epsilon^{\frac{1}{2}}(\mathcal{E}^{k}(\tilde{v}_0-v_\epsilon))^{\frac{1}{2}}+C(M)\epsilon,
\end{equation*}
which suffices by Cauchy-Schwarz. 
\\

Next, we move to the more difficult portion of the energy which involves $D_ta_{\epsilon}$. We start by combining \Cref{Dtarel} with $(v_{\epsilon}-\tilde{v}_0)\cdot n_{\Gamma_\epsilon}=-\mathcal{P}_{>\epsilon^{-1}}(w\cdot n_{\Gamma_\epsilon})$  to obtain the relation
\begin{equation*}
\begin{split}
\int_{\Gamma_{\epsilon}}\mathcal{N}_{\epsilon}^{k-\frac{3}{2}}(a_{\epsilon}^{-1}D_ta_{\epsilon})\mathcal{N}_{\epsilon}^{k-\frac{3}{2}}(a_{\epsilon}^{-1}(D_ta_{\epsilon}-D_t\tilde{a}_0))\,dS&=\int_{\Gamma_{\epsilon}}\mathcal{N}_{\epsilon}^{k-\frac{3}{2}}(a_{\epsilon}^{-1}D_ta_{\epsilon})\mathcal{N}_{\epsilon}^{k-\frac{1}{2}}\mathcal{P}_{>\epsilon^{-1}}(w\cdot n_{\Gamma_\epsilon})\,dS+\mathcal{O}_M(\epsilon).
\end{split}
\end{equation*}
Define $p_{-}$ and $D_ta_{-}$ in the usual way using the relevant Laplace equations. We  split the above integral into the two components,
\begin{equation}\label{split}
\begin{split}
\int_{\Gamma_{\epsilon}}\mathcal{N}_{\epsilon}^{k-\frac{3}{2}}(a_{\epsilon}^{-1}D_ta_{\epsilon})\mathcal{N}_{\epsilon}^{k-\frac{1}{2}}\mathcal{P}_{>\epsilon^{-1}}(w\cdot n_{\Gamma_\epsilon})\,dS&=\int_{\Gamma_{\epsilon}}\mathcal{N}_{\epsilon}^{k-\frac{3}{2}}(a_{\epsilon}^{-1}D_ta_{-})\mathcal{N}_{\epsilon}^{k-\frac{1}{2}}\mathcal{P}_{>\epsilon^{-1}}(w\cdot n_{\Gamma_\epsilon})\,dS
\\
&+\int_{\Gamma_{\epsilon}}\mathcal{N}_{\epsilon}^{k-\frac{3}{2}}(a_{\epsilon}^{-1}(D_ta_{\epsilon}-D_ta_{-}))\mathcal{N}_{\epsilon}^{k-\frac{1}{2}}\mathcal{P}_{>\epsilon^{-1}}(w\cdot n_{\Gamma_\epsilon})\,dS.
\end{split}
\end{equation}
We begin by studying the first term in \eqref{split}. By self-adjointness of $\mathcal{N}_{\epsilon}$, we have
\begin{equation*}
\begin{split}
\int_{\Gamma_{\epsilon}}\mathcal{N}_{\epsilon}^{k-\frac{3}{2}}(a_{\epsilon}^{-1}D_ta_{-})\mathcal{N}_{\epsilon}^{k-\frac{1}{2}}\mathcal{P}_{>\epsilon^{-1}}(w\cdot n_{\Gamma_\epsilon})\,dS&=\int_{\Gamma_{\epsilon}}\mathcal{N}_{\epsilon}^{k-\frac{1}{2}}(a_{\epsilon}^{-1}D_ta_{-})\mathcal{N}_{\epsilon}^{k-\frac{3}{2}}\mathcal{P}_{>\epsilon^{-1}}(w\cdot n_{\Gamma_\epsilon})\,dS
\\
&\lesssim_M \epsilon\|a_{\epsilon}^{-1}D_ta_{-}\|_{H^{k-\frac{1}{2}}(\Gamma_{\epsilon})}\|\mathcal{P}_{>\epsilon^{-1}}\mathcal{N}_{\epsilon}^{k-\frac{1}{2}}(w\cdot n_{\Gamma_\epsilon})\|_{L^2(\Gamma_{\epsilon})}+\mathcal{O}_M(\epsilon),
\end{split}
\end{equation*}
where we used the multiplier $\mathcal{P}_{> \epsilon^{-1}}$ to recover a power of $\mathcal{N}_{\epsilon}$ in the high frequency term.  Next, we show that $\|a_{\epsilon}^{-1}D_ta_{-}\|_{H^{k-\frac{1}{2}}(\Gamma_{\epsilon})}\lesssim_M 
\epsilon^{-\frac{1}{2}}$. By Sobolev product estimates and the fact that $\|a_{\epsilon}^{-1}\|_{H^{k-\frac{1}{2}}(\Gamma_{\epsilon})}\lesssim_M\epsilon^{-\frac{1}{2}}$, it suffices to show the same estimate for $\|D_ta_{-}\|_{H^{k-\frac{1}{2}}(\Gamma_{\epsilon})}$. To see this, recall that, by definition,
\begin{equation*}
D_ta_{-}=n_{\Gamma_\epsilon}\cdot\nabla v_{-}\cdot\nabla p_{-}-n_{\Gamma_\epsilon}\cdot\nabla D_tp_{-}.
\end{equation*}
Note then that by \Cref{baltrace}  we have the estimate $\|\nabla v_{-}\|_{H^{k-\frac{1}{2}}(\Gamma_{\epsilon})}\lesssim_M\epsilon^{-\frac{1}{2}}$, since $v_{-}$ is regularized at the $\epsilon^{-\frac{1}{2}}$ scale. Moreover, as $n_{\Gamma_\epsilon}=\mathcal{O}_{H^{k-1}}(1)$ and $\Gamma_{\epsilon}$ is regularized at the $\epsilon^{-1}$ scale, we have $\|n_{\Gamma_\epsilon}\|_{H^{k-\frac{1}{2}}(\Gamma_{\epsilon})}\lesssim_M \epsilon^{-\frac{1}{2}}$. By \Cref{baltrace} and \Cref{direst}, we also have $\|\nabla p_{-}\|_{H^{k-\frac{1}{2}}(\Gamma_{\epsilon})}\lesssim_M \epsilon^{-\frac{1}{2}}$. Therefore, by \Cref{boundaryest}, we have $\|n_{\Gamma_\epsilon}\cdot\nabla v_{-}\cdot\nabla p_{-}\|_{H^{k-\frac{1}{2}}(\Gamma_{\epsilon})}\lesssim_M\epsilon^{-\frac{1}{2}}$. 
\\
\\
Using \Cref{direst} and the fact that the pressure terms in the Laplace equation for $D_tp_{-}$ always appear to one half derivative lower than top order,  a similar analysis yields $\|n_{\Gamma_\epsilon}\cdot\nabla D_tp_{-}\|_{H^{k-\frac{1}{2}}(\Gamma_{\epsilon})}\lesssim_M \epsilon^{-\frac{1}{2}}$. Therefore, we obtain from \Cref{Dtarel} the bound,
\begin{equation*}
\begin{split}
\int_{\Gamma_{\epsilon}}\mathcal{N}_{\epsilon}^{k-\frac{3}{2}}(a_{\epsilon}^{-1}D_ta_{-})\mathcal{N}_{\epsilon}^{k-\frac{1}{2}}\mathcal{P}_{>\epsilon^{-1}}(w\cdot n_{\Gamma_\epsilon})\,dS&\lesssim_M \epsilon^{\frac{1}{2}}\|\mathcal{P}_{> \epsilon^{-1}}\mathcal{N}_{\epsilon}^{k-\frac{1}{2}}(w\cdot n_{\Gamma_\epsilon})\|_{{L^2}(\Gamma_{\epsilon})}+\mathcal{O}_M(\epsilon)
\\
&\lesssim_M \epsilon^{\frac{1}{2}}(\mathcal{E}^{k}(\tilde{v}_0-v_{\epsilon}))^{\frac{1}{2}}+\mathcal{O}_M(\epsilon),
\end{split}
\end{equation*}
as desired. It remains to deal with the other term in (\ref{split}). For this, we need to expand $D_ta_{\epsilon}-D_ta_{-}$. As a first reduction, we note that we can replace every appearance of $p_{-}$ with $p_{\epsilon}$ in the definition of $D_ta_{-}$ if we allow for $\mathcal{O}_M(\epsilon^{\frac{1}{2}})$ errors. This is because $\|p_{\epsilon}-p_{-}\|_{H^{k}(\Omega_{\epsilon})}\lesssim_M \|v_{\epsilon}-v_{-}\|_{H^{k-1}(\Omega_{\epsilon})}\lesssim_M \epsilon^{\frac{1}{2}}$. Hence, we have
\begin{equation*}
\begin{split}
D_ta_{-}&=n_{\Gamma_\epsilon}\cdot\nabla v_{-}\cdot\nabla p_{\epsilon}-n_{\Gamma_\epsilon}\cdot \nabla \Delta^{-1}_{\Omega_\epsilon}(4\text{tr}(\nabla^2p_{\epsilon}\cdot \nabla v_{-})+2\text{tr}(\nabla v_{-})^3+\Delta v_{-}\cdot \nabla p_{\epsilon})+\mathcal{O}_{H^{k-\frac{3}{2}}(\Gamma_{\epsilon})}(\epsilon^{\frac{1}{2}}).
\end{split}
\end{equation*}
We may also replace the lower order terms involving $v_{-}$ by $v_{\epsilon}$. Arguing similarly to \Cref{Dtarel}, we then obtain the key identity 
\begin{equation*}
\begin{split}
D_ta_{-}-D_ta_{\epsilon}&=a_{\epsilon}\mathcal{N}_{\epsilon}((v_{\epsilon}-v_{-})\cdot n_{\Gamma_\epsilon})+\mathcal{O}_{H^{k-\frac{3}{2}}(\Gamma_{\epsilon})}(\epsilon^{\frac{1}{2}})
\\
&=a_{\epsilon}\mathcal{P}_{\leq\epsilon^{-1}}\mathcal{N}_{\epsilon}(w\cdot n_{\Gamma_\epsilon})+\mathcal{O}_{H^{k-\frac{3}{2}}(\Gamma_{\epsilon})}(\epsilon^{\frac{1}{2}}).
\end{split}
\end{equation*}
Hence, we have
\begin{equation*}
\begin{split}
\int_{\Gamma_{\epsilon}}\mathcal{N}_{\epsilon}^{k-\frac{3}{2}}(a_{\epsilon}^{-1}(D_ta_{\epsilon}-D_ta_{-}))\mathcal{N}_{\epsilon}^{k-\frac{1}{2}}\mathcal{P}_{>\epsilon^{-1}}(w\cdot n_{\Gamma_\epsilon})\,dS\lesssim_M &-\int_{\Gamma_{\epsilon}}\mathcal{N}_{\epsilon}^{k-\frac{1}{2}}\mathcal{P}_{\leq\epsilon^{-1}}(w\cdot n_{\Gamma_\epsilon})\mathcal{N}_{\epsilon}^{k-\frac{1}{2}}\mathcal{P}_{>\epsilon^{-1}}(w\cdot n_{\Gamma_\epsilon})\,dS
\\
&+\epsilon^{\frac{1}{2}}\|\mathcal{N}_{\epsilon}^{k-\frac{1}{2}}\mathcal{P}_{>\epsilon^{-1}}(w\cdot n_{\Gamma_\epsilon})\|_{L^2(\Gamma_\epsilon)}.
\end{split}
\end{equation*}
The first term on the right-hand side above vanishes by orthogonality (this term is the reason we reweighted the energy in the first place) and the latter term is controlled by $\epsilon^{\frac{1}{2}}(\mathcal{E}^{k}(\tilde{v}_0-v_{\epsilon}))^{\frac{1}{2}}$. Therefore, we obtain the desired bound for the $D_ta_{\epsilon}$ portion of the energy. This completes the proof of \Cref{Prop velo reg}.
\end{proof}
\subsection{Step 3: Euler plus transport iteration}\label{EPTI} In this subsection, we construct the iterate $(v_1,\Gamma_1)$ from the regularized data $(v_{\epsilon},\Gamma_{\epsilon})$. Intuitively, what remains to be done is to carry out something akin to the Euler iteration 
\begin{equation*}
v_{1}:=v_{\epsilon}-\epsilon (v_{\epsilon}\cdot\nabla v_{\epsilon}+\nabla p_{\epsilon}+ge_d)    
\end{equation*}
and then the domain transport
\begin{equation*}
x_1(x):=x+\epsilon v_{\epsilon}(x).    
\end{equation*}
Unfortunately, performed individually, these steps lose a full derivative in each iteration. Therefore, it is important that these two steps be carried out together. This will reduce the derivative loss and allow us to exploit a discrete version of the energy cancellation seen in the energy estimates. We will then use the regularization bounds from the previous subsections to control any remaining errors in the iteration. To carry out this process, we have the following proposition.
\begin{proposition}\label{Final transport +Euler}
Given $(v_{\epsilon},\Gamma_{\epsilon})$ as in  the previous step, there exists an iteration $(v_{\epsilon},\Gamma_{\epsilon})\mapsto (v_1,\Gamma_1)$ 
such that the following properties hold:
\begin{enumerate}
    \item (Approximate solution). 
\begin{equation*}
\begin{cases}
&v_1 = v_{\epsilon}-\epsilon (v_{\epsilon}\cdot\nabla v_{\epsilon}+\nabla p_\epsilon +ge_d)+ \mathcal{O}_{C^1}(\epsilon^2)\hspace{5mm}\text{on}\hspace{2mm}\Omega_1\cap\Omega_{\epsilon},
\\
&\nabla\cdot v_1=0 \hspace{5mm}\text{on}\hspace{2mm}\Omega_1,
\\
& \Omega_1 = (I+\epsilon v_{\epsilon})\Omega_{\epsilon}.
\end{cases} 
\end{equation*}
\item (Energy monotonicity bound).
\begin{equation*}
\mathcal{E}^{k}(v_1,\Gamma_1)\leq (1+C(M)\epsilon)\mathcal{E}^{k}(v_{\epsilon},\Gamma_{\epsilon}).
\end{equation*}
\end{enumerate}
Moreover, $v_1$ and $\omega_1$ satisfy the inductive bounds \eqref{regboundprop}.
\end{proposition}
We define the change of coordinates $x_1(x) := x + \epsilon v_{\epsilon}(x)$
and the iterated domain $\Omega_1$ by
\begin{equation*}
\Omega_1:=(I+\epsilon v_{\epsilon})\Omega_{\epsilon}.    
\end{equation*}
To define $v_1$, we proceed in two steps. First, we define 
\begin{equation}\label{Iterated v_1}
\tilde{v}_1(x_1):=v_{\epsilon}-\epsilon(\nabla p_{\epsilon}+ge_d).
\end{equation}
We note that $\tilde{v}_1$ is not divergence free, so we define the full iterate $v_1$ by correcting the divergence of $\tilde{v}_1$ by a gradient potential:
\begin{equation*}
v_1:=\tilde{v}_1-\nabla\Delta^{-1}_{\Omega_1}(\nabla\cdot \tilde{v}_1).
\end{equation*}
At this point, we can verify the inductive bound (\ref{regboundprop}) for $v_1$ and $\omega_1$. We start with $v_1$. We recall that we have to show that
\begin{equation*}\label{v1indbound}
\|v_1\|_{H^{k+1}(\Omega_1)}\leq K(M)\epsilon^{-1}.    
\end{equation*}
As a first step, using the regularization bound \eqref{8.44} for $v_{\epsilon}$ from the previous section, we have from the definition of $\tilde{v}_1$, the regularization bounds \eqref{surfbound} for $\Gamma_{\epsilon}$ and the balanced elliptic estimate \Cref{direst}, 
\begin{equation}\label{tildev1bound}
\|\tilde{v}_1\|_{H^{k+n}(\Omega_1)}\leq \frac{1}{3}K(M)\epsilon^{-n},
\end{equation}
for $n=0,1,2$. Next, we aim to control the error between $v_1$ and $\tilde{v}_1$ in $H^{k}(\Omega_1)$ and $H^{k+1}(\Omega_1)$ (but not $H^{k+2}(\Omega_1)$).  We have for $n=0,1$ from the balanced elliptic estimate \Cref{direst},
\begin{equation*}
\begin{split}
\|v_1-\tilde{v}_1\|_{H^{k+n}(\Omega_1)}&\lesssim_M \|\Gamma_1\|_{H^{k+\frac{1}{2}+n}}\|\nabla\cdot\tilde{v}_1\|_{H^{k-2}(\Omega_1)}+\|\nabla\cdot\tilde{v}_1\|_{H^{k-1+n}(\Omega_1)}
\\
&\lesssim_M\epsilon^{-\frac{1}{2}-n}\|\nabla\cdot\tilde{v}_1\|_{H^{k-2}(\Omega_1)}+\|\nabla\cdot\tilde{v}_1\|_{H^{k-1+n}(\Omega_1)}.
\end{split}
\end{equation*}
Above, we used the $H^{k+1}$ and $H^{k+2}$ (depending on if $n$ is $0$ or $1$) regularization bounds for $v_{\epsilon}$, Moser estimates, the bounds for $\Gamma_{\epsilon}$ and the relation $\Gamma_1=(I+\epsilon v_{\epsilon})(\Gamma_{\epsilon})$ to control $\|\Gamma_1\|_{H^{k+\frac{1}{2}+n}}\lesssim_M \epsilon^{-\frac{1}{2}-n}$. By using the definition of $\tilde{v}_1$ and the regularization bounds for $v_{\epsilon}$, it is straightforward to see that the divergence, $\nabla\cdot \tilde{v}_1$, contributes an error of size $\mathcal{O}_{H^{k-1+n}(\Omega_1)}(\epsilon^{\frac{3}{2}-n})$ and also $\mathcal{O}_{H^{k-2}(\Omega_1)}(\epsilon^{2})$. Note that for this computation, one must use the cancellation between the velocity and the pressure in \eqref{Iterated v_1} in order to see the desired gain. Therefore, we have
\begin{equation*}
\|\nabla\Delta^{-1}_{\Omega_1}(\nabla\cdot\tilde{v}_1)\|_{H^{k+n}(\Omega_1)}=\|v_1-\tilde{v}_1\|_{H^{k+n}(\Omega_1)}\lesssim_M \epsilon^{\frac{3}{2}-n}.
\end{equation*}
From this and (\ref{tildev1bound}), we conclude the inductive bound
\begin{equation*}
\|v_1\|_{H^{k+1}(\Omega_1)}\leq K(M)\epsilon^{-1} ,   
\end{equation*}
and the leading order expansion for $v_1(x_1)$ in $H^k(\Omega_{\epsilon})$,
\begin{equation*}\label{v1ident}
v_1(x_1)=v_{\epsilon}-\epsilon(\nabla p_{\epsilon}+ge_d)+\mathcal{O}_{H^{k}(\Omega_\epsilon)}(\epsilon^\frac{3}{2}).    
\end{equation*}
If $k$ is large enough, then the  leading order expansion (\ref{v1ident}) with $\mathcal{O}_{C^1}(\epsilon^2)$ error can be seen by slightly modifying the above argument. Now, we verify the inductive bound $\|\omega_1\|_{H^{k+n}(\Omega_1)}\leq \epsilon^{-1-n}(K'(M)+\epsilon C(M))$ for $n=0,1$. It  suffices to establish this for $\tilde{\omega}_1$ since $v_1$ and $\tilde{v}_1$ agree up to a gradient. Taking curl in the definition of $\tilde{v}_1$ and using that $\omega_{\epsilon}=\tilde{\omega}_0$, we have
\begin{equation}\label{curltildev_1}
\|\nabla\times (\tilde{v}_1(x_1))\|_{H^{k+n}(\Omega_{\epsilon})}\leq \|\tilde{\omega}_0\|_{H^{k+n}(\Omega_\epsilon)}\leq K'(M)\epsilon^{-1-n}.
\end{equation}
By chain rule, using (\ref{tildev1bound}) and the regularization bounds for $v_{\epsilon}$, we have
\begin{equation*}
\|\tilde{\omega}_1(x_1)\|_{H^{k+n}(\Omega_{\epsilon})}\leq \|\nabla\times (\tilde{v}_1(x_1))\|_{H^{k+n}(\Omega_{\epsilon})}+C(M)\epsilon^{-n},    
\end{equation*}
which by a change of variables and (\ref{curltildev_1}) yields
\begin{equation*}\label{desiredcurlbound}
\|\tilde{\omega}_1\|_{H^{k+n}(\Omega_1)}\leq \epsilon^{-1-n}(K'(M)+\epsilon C(M))    ,
\end{equation*}
as desired. Note that in the above two lines, we treated $C(M)$ as an arbitrary constant, and relabelled it from line to line. Importantly, we did not do this for $K(M)$ and $K'(M)$.
\\

Next, we work towards establishing the energy monotonicity bound for the transport part of the argument. As a first step, we aim to relate the good variables associated to the iterate $v_1$ to the good variables associated to $v_{\epsilon}$ at the regularity level of the energy. We have the following lemma.
\begin{lemma}[Relations between the good variables]\label{goodvarrel} The following relations hold:
\begin{enumerate}
\item (Relation for $\omega_1$).
\begin{equation*}
\omega_1(x_1)=\omega_{\epsilon}+\mathcal{O}_{H^{k-1}(\Omega_\epsilon)}(\epsilon).
\end{equation*}
\item (Relation for $p_1$).
\begin{equation}\label{prelation}
p_{1}(x_1)-p_{\epsilon}-\epsilon D_tp_{\epsilon}=\mathcal{O}_{H^{k+\frac{1}{2}}(\Omega_\epsilon)}(\epsilon).    
\end{equation}
\item (Relation for $a_1$).
\begin{equation}\label{arelation}
a_1(x_1)=a_{\epsilon}+\epsilon D_ta_{\epsilon}+\mathcal{O}_{H^{k-1}(\Gamma_\epsilon)}(\epsilon).
\end{equation}
\item (Relation for $D_ta_1$). 
\begin{equation*}
D_ta_1(x_1)=D_ta_{\epsilon}-\epsilon a_{\epsilon}\mathcal{N}_{\epsilon}a_{\epsilon}+\mathcal{O}_{H^{k-\frac{3}{2}}(\Gamma_\epsilon)}(\epsilon).
\end{equation*}
\end{enumerate}
\end{lemma}
\begin{proof} 
The relation for $\omega_1$  is immediate. Next, we move to the relations for $p_1$ and $a_1$. By the chain rule and the Laplace equation (\ref{Dtpdef}) for $D_tp_{\epsilon}$, we have 
\begin{equation*}
\begin{split}
\Delta (p_1(x_1))&=(\Delta p_1)(x_1)+\epsilon\Delta v_{\epsilon}\cdot (\nabla p_1)(x_1)+2\epsilon \nabla v_{\epsilon}\cdot (\nabla^2p_1)(x_1)+\mathcal{O}_{H^{k-\frac{3}{2}}(\Omega_\epsilon)}(\epsilon)
\\
&=\Delta p_{\epsilon}+\epsilon \Delta D_tp_{\epsilon}+\epsilon\Delta v_{\epsilon}\cdot ((\nabla p_1)(x_1)-\nabla p_{\epsilon})+\mathcal{O}_{H^{k-\frac{3}{2}}(\Omega_\epsilon)}(\epsilon)
\\
&=\Delta p_{\epsilon}+\epsilon \Delta D_tp_{\epsilon}+\mathcal{O}_{H^{k-\frac{3}{2}}(\Omega_\epsilon)}(\epsilon),
\end{split}
\end{equation*}
where in the last line, we controlled $\epsilon\Delta v_{\epsilon}\cdot ((\nabla p_1)(x_1)-\nabla p_{\epsilon})=\mathcal{O}_{H^{k-\frac{3}{2}}(\Omega_{\epsilon})}(\epsilon)$ by using the regularization bounds for $v_{\epsilon}$ as well as the error bound $(\nabla p_1)(x_1)-\nabla p_{\epsilon}=\mathcal{O}_{L^{\infty}(\Omega_\epsilon)}(\epsilon)$, which is gotten by performing an $H^{k}(\Omega_\epsilon)$ elliptic estimate  in the second line, using the fact that $p_1(x_1)-p_{\epsilon}$ vanishes on $\Gamma_{\epsilon}$ and that each of the source terms can be estimated directly in $H^{k-2}(\Omega_\epsilon)$ (but not in $H^{k-\frac{3}{2}}(\Omega_\epsilon)$). Therefore, since $p_1(x_1)-p_{\epsilon}-\epsilon D_tp_{\epsilon}$ vanishes on $\Gamma_{\epsilon}$, we may now do a $H^{k+\frac{1}{2}}(\Omega_\epsilon)$  elliptic estimate to obtain the finer bound,
\begin{equation}\label{ellipticregpdiff}
p_{1}(x_1)-p_{\epsilon}-\epsilon D_tp_{\epsilon}=\mathcal{O}_{H^{k+\frac{1}{2}}(\Omega_\epsilon)}(\epsilon),
\end{equation}
which gives (\ref{prelation}). We also deduce from this that
\begin{equation*}
\begin{split}
(\nabla p_{1})(x_1)&=\nabla p_{\epsilon}+\epsilon\nabla D_t p_{\epsilon}-\epsilon\nabla v_{\epsilon}\cdot(\nabla p_{1})(x_1)+\mathcal{O}_{H^{k-\frac{1}{2}}(\Omega_\epsilon)}(\epsilon)
\\
&=\nabla p_{\epsilon}+\epsilon D_t\nabla p_{\epsilon}+\mathcal{O}_{H^{k-\frac{1}{2}}(\Omega_\epsilon)}(\epsilon).
\end{split}
\end{equation*}
From this we see that
\begin{equation*}
\begin{split}
a_1(x_1)&=a_{\epsilon}+\epsilon D_ta_{\epsilon}-(n_{\Gamma_1}(x_1)-n_{\Gamma_\epsilon})\cdot(\nabla p_{1})(x_1)+\mathcal{O}_{H^{k-1}(\Gamma_\epsilon)}(\epsilon)
\\
&=a_{\epsilon}+\epsilon D_ta_{\epsilon}+\mathcal{O}_{H^{k-1}(\Gamma_\epsilon)}(\epsilon),
\end{split}
\end{equation*}
where in the last line we used $$(n_{\Gamma_1}(x_1)-n_{\Gamma_\epsilon})\cdot (\nabla p_1)(x_1)=-a_1(x_1)(n_{\Gamma_1}(x_1)-n_{\Gamma_\epsilon})\cdot n_{\Gamma_1}(x_1)=-a_1(x_1)\frac{1}{2}|n_{\Gamma_1}(x_1)-n_{\Gamma_\epsilon}|^2=\mathcal{O}_{H^{k-1}(\Gamma_\epsilon)}(\epsilon).$$ This gives the relation (\ref{arelation}).
\\
\\
Next, we prove the relation for $D_ta_1$. First, we see that
\begin{equation}\label{Dtaeqndiff}
\begin{split}
-(D_t\nabla p_{1})(x_1)+D_t\nabla p_{\epsilon}&=((\nabla v_1\cdot\nabla p_{1})(x_1)-\nabla v_{\epsilon}\cdot\nabla p_{\epsilon})-((\nabla D_tp_{1})(x_1)-\nabla D_tp_{\epsilon})
\\
&=((\nabla v_1)(x_1)-\nabla v_{\epsilon})\cdot\nabla p_{\epsilon}-((\nabla D_t p_{1})(x_1)-\nabla D_t p_{\epsilon})+\mathcal{O}_{H^{k-1}(\Omega_\epsilon)}(\epsilon).
\end{split}
\end{equation}
To control the second term on the right-hand side above, we write out the Laplace equation for $D_tp_1(x_1)$:
\begin{equation*}
\begin{split}
\Delta (D_tp_1(x_1))&=(\Delta D_tp_1)(x_1)+\mathcal{O}_{H^{k-2}(\Omega_\epsilon)}(\epsilon).
\end{split}
\end{equation*}
By a similar analysis to the proof of \eqref{arelation} and the relation
\begin{equation*}
\begin{split}
(\Delta v_1)(x_1)&=\Delta (v_1(x_1))+\mathcal{O}_{H^{k-2}(\Omega_\epsilon)}(\epsilon)=\Delta v_{\epsilon}-\epsilon\nabla\Delta p_{\epsilon}+\mathcal{O}_{H^{k-2}(\Omega_\epsilon)}(\epsilon)=\Delta v_{\epsilon}+\mathcal{O}_{H^{k-2}(\Omega_\epsilon)}(\epsilon),
\end{split}
\end{equation*}
we obtain
\begin{equation*}
\begin{split}
(\Delta D_tp_1)(x_1)&=\Delta D_tp_{\epsilon}+(\Delta v_1\cdot\nabla p_1)(x_1)-\Delta v_{\epsilon}\cdot\nabla p_{\epsilon}+4\text{tr}(\nabla v_1\cdot\nabla ^2 p_1)(x_1)-4tr\nabla v_{\epsilon}\cdot\nabla ^2 p_{\epsilon}+\mathcal{O}_{H^{k-2}(\Omega_\epsilon)}(\epsilon)
\\
&=\Delta D_tp_{\epsilon}+4tr\left(\nabla v_{\epsilon}\cdot ((\nabla^2 p_1)(x_1)-\nabla^2p_{\epsilon})\right)+\mathcal{O}_{H^{k-2}(\Omega_\epsilon)}(\epsilon)
\\
&=\Delta D_tp_{\epsilon}+\mathcal{O}_{H^{k-2}(\Omega_\epsilon)}(\epsilon),
\end{split}    
\end{equation*}
where in the last line, we used (\ref{ellipticregpdiff}) and that $\epsilon D_tp_{\epsilon}=\mathcal{O}_{H^{k}(\Omega_\epsilon)}(\epsilon)$. Combining the above with (\ref{Dtaeqndiff}), one obtains by elliptic regularity,
\begin{equation*}
-D_t\nabla p_1(x_1)+D_t\nabla p_{\epsilon}=((\nabla v_1)(x_1)-\nabla v_{\epsilon})\cdot\nabla p_{\epsilon}+\mathcal{O}_{H^{k-1}(\Omega_\epsilon)}(\epsilon).
\end{equation*}
Then, noting from (\ref{arelation}) that
\begin{equation*}
(D_t\nabla p_1)(x_1)\cdot (n_{\Gamma_1}(x_1)-n_{\Gamma_\epsilon})=(D_t\nabla p_1)(x_1)\cdot (a_\epsilon^{-1}\nabla p_\epsilon-(a_1^{-1}\nabla p_1)(x_1))=\mathcal{O}_{H^{k-\frac{3}{2}}(\Gamma_\epsilon)}(\epsilon)    
\end{equation*}
and using the fact that $\Delta p_{\epsilon}$ is lower order, we obtain
\begin{equation*}
\begin{split}
D_ta_{1}(x_1)-D_ta_{\epsilon}&=-a_{\epsilon}n_{\Gamma_\epsilon}\cdot\nabla (v_1(x_1)-v_{\epsilon})\cdot n_{\Gamma_\epsilon}-(D_t\nabla p_1)(x_1)\cdot (n_{\Gamma_1}(x_1)-n_{\Gamma_\epsilon})+ \mathcal{O}_{H^{k-\frac{3}{2}}(\Gamma_\epsilon)}(\epsilon)
\\
&=\epsilon a_{\epsilon}n_{\Gamma_\epsilon}\cdot\nabla \nabla p_{\epsilon}\cdot n_{\Gamma_\epsilon}+\mathcal{O}_{H^{k-\frac{3}{2}}(\Gamma_\epsilon)}(\epsilon)
\\
&=\epsilon a_{\epsilon}\mathcal{N}_{\epsilon}\nabla p_{\epsilon}\cdot n_{\Gamma_\epsilon}+\mathcal{O}_{H^{k-\frac{3}{2}}(\Gamma_\epsilon)}(\epsilon).
\end{split}
\end{equation*}
Finally, noting that $\mathcal{N}_{\epsilon}n_{\Gamma_\epsilon}\cdot n_{\Gamma_\epsilon}$ is lower order, we have, thanks to the Leibniz rule for $\mathcal{N}_{\epsilon}$,
\begin{equation*}
\begin{split}
\epsilon a_{\epsilon}\mathcal{N}_{\epsilon}\nabla p_{\epsilon}\cdot n_{\Gamma_\epsilon} &=-\epsilon a_{\epsilon}\mathcal{N}_{\epsilon}(n_{\Gamma_\epsilon} a_{\epsilon})\cdot n_{\Gamma_\epsilon}=-\epsilon a_{\epsilon}\mathcal{N}_{\epsilon}a_{\epsilon}+\mathcal{O}_{H^{k-\frac{3}{2}}(\Gamma_\epsilon)}(\epsilon).
\end{split}
\end{equation*}
Therefore, we have the desired relation for $D_ta_1$. This completes the proof of the lemma.
\end{proof}
\textbf{Energy monotonicity}. To finish the proof of \Cref{Final transport +Euler}, it remains to establish energy monotonicity. The following lemma will allow us to more easily work with the relations in \Cref{goodvarrel}.
\begin{lemma}\label{changeofvar}
Define the ``pulled-back" energy $\mathcal{E}^k_*(v_1,\Gamma_{1})$ by 
\begin{equation*}
\begin{split}
\mathcal{E}^k_*(v_1,\Gamma_{1}):=1&+\|\mathcal{N}_{\epsilon}^{k-\frac{3}{2}}(a_{1}^{-1}(x_1)D_ta_1(x_1))\|_{L^2(\Gamma_{\epsilon})}^2+\|a_1^{-\frac{1}{2}}(x_1)\mathcal{N}_{\epsilon}^{k-1}(a_1(x_1))\|_{L^2(\Gamma_{\epsilon})}^2
\\
&+\|\omega_1(x_1)\|_{H^{k-1}(\Omega_{\epsilon})}^2+\|v_1(x_1)\|_{L^2(\Omega_{\epsilon})}^2.
\end{split}
\end{equation*}
Then we have the relation
\begin{equation*}
\mathcal{E}^{k}(v_1,\Gamma_1)\leq \mathcal{E}_*^{k}(v_1,\Gamma_{1})+\mathcal{O}_M(\epsilon).
\end{equation*}
\end{lemma}
Before proving the above lemma, we show how it easily implies the desired energy monotonicity bound. In light of \Cref{changeofvar}, it suffices to establish the bound
\begin{equation*}
\mathcal{E}^{k}_*(v_1,\Gamma_1)\leq (1+C(M)\epsilon)\mathcal{E}^{k}(v_{\epsilon},\Gamma_{\epsilon}).
\end{equation*}
The monotonicity bound for the vorticity is immediate from \Cref{goodvarrel}. For the surface components of the energy, we first use \Cref{goodvarrel}, the fact that $\|\mathcal{N}_{\epsilon}^{k-\frac{3}{2}}\|_{H^{k-\frac{3}{2}}(\Gamma_{\epsilon})\to L^2(\Gamma_{\epsilon})}\lesssim_M 1$ and the regularization bounds for $\Gamma_{\epsilon}$ and $v_{\epsilon}$ to obtain  
\begin{equation}\label{iterateDtaenergy}
\begin{split}
&\int_{\Gamma_{\epsilon}}|\mathcal{N}_{\epsilon}^{k-\frac{3}{2}}(a_{1}^{-1}(x_1)D_ta_1(x_1))|^2\, dS-\int_{\Gamma_{\epsilon}}|\mathcal{N}_{\epsilon}^{k-\frac{3}{2}}(a_{\epsilon}^{-1}D_ta_{\epsilon})|^2\,dS
\\
&=2\int_{\Gamma_{\epsilon}}\mathcal{N}_{\epsilon}^{k-\frac{3}{2}}(a_{\epsilon}^{-1}D_ta_{\epsilon})\mathcal{N}_{\epsilon}^{k-\frac{3}{2}}(a_{\epsilon}^{-1}((D_ta_1)(x_1)-D_ta_{\epsilon}))\,dS+\mathcal{O}_M(\epsilon)
\\
&=-2\epsilon \int_{\Gamma_{\epsilon}}a_{\epsilon}^{-1}\mathcal{N}_{\epsilon}^{k-1}D_ta_{\epsilon}\mathcal{N}_{\epsilon}^{k-1}a_{\epsilon}\,dS+\mathcal{O}_M(\epsilon),
\end{split}
\end{equation}
where in the last line, we used the commutator estimate $\|[\mathcal{N}_\epsilon^{k-1},a_{\epsilon}^{-1}]D_ta_{\epsilon}\|_{L^{2}(\Gamma_\epsilon)}\lesssim_M 1$ to shift a factor of $\mathcal{N}_\epsilon^{\frac{1}{2}}$ onto $\mathcal{N}_{\epsilon}^{k-\frac{3}{2}}D_ta_{\epsilon}$. We similarly observe the leading order relation for the other component of the energy by using (\ref{arelation}) to obtain,
\begin{equation*}
\int_{\Gamma_{\epsilon}}a_1^{-1}(x_1)|\mathcal{N}_{\epsilon}^{k-1}(a_1(x_1))|^2\,dS-\int_{\Gamma_{\epsilon}}a_{\epsilon}^{-1}|\mathcal{N}_{\epsilon}^{k-1}a_{\epsilon}|^2\, dS=2\epsilon\int_{\Gamma_{\epsilon}}a_{\epsilon}^{-1}\mathcal{N}_{\epsilon}^{k-1}D_ta_{\epsilon}\mathcal{N}_{\epsilon}^{k-1}a_{\epsilon}\,dS+\mathcal{O}_M(\epsilon).
\end{equation*}
The first term on the right-hand side of the above relation cancels the main term on the right-hand side of (\ref{iterateDtaenergy}). Combining everything together then gives
\begin{equation*}
\begin{split}
\mathcal{E}^k(v_1,\Gamma_1)\leq (1+C(M)\epsilon)\mathcal{E}^k(v_{\epsilon},\Gamma_{\epsilon}),
\end{split}
\end{equation*}
as desired. It remains now to establish \Cref{changeofvar}. 
\begin{proof}[Proof of \Cref{changeofvar}]
By a simple change of variables, it is clear that the difference between $\|\omega_1(x_1)\|_{H^{k-1}(\Omega_{\epsilon})}^2$ and $\|\omega_{1}\|_{H^{k-1}(\Omega_1)}^2$ contributes only $\mathcal{O}_M(\epsilon)$ errors. This is likewise true for the $L^2$ component of the velocity. The main difficulty is in dealing with the surface components of the energy. For this, we need the following proposition.
\begin{proposition}\label{changeofvar2}
Let $-\frac{1}{2}\leq s\leq k-2$ and let $f\in H^{s+1}(\Gamma_1)$. Then we have the following bound on $\Gamma_{\epsilon}$\emph{:}
\begin{equation*}
\|(\mathcal{N}_{1}f)(x_1)-\mathcal{N}_{\epsilon}(f(x_1))\|_{H^s(\Gamma_{\epsilon})}\lesssim_M \epsilon\|f\|_{H^{s+1}(\Gamma_1)}.   
\end{equation*}
\end{proposition}
\begin{proof}
First, we handle the case $s=-\frac{1}{2}$. If $g\in C^{\infty}(\Gamma_{\epsilon})$, we write $h=g(x_1^{-1})\mathcal{H}_1J$ where $J$ is the Jacobian corresponding to the change of variables $y=x_1(x)$. Then we have by the divergence theorem,
\begin{equation*}
\begin{split}
\int_{\Gamma_{\epsilon}}g((\mathcal{N}_1f)(x_1)-\mathcal{N}_{\epsilon}(f(x_1)))\,dS&=\int_{\Gamma_{1}}h\mathcal{N}_1f\,dS-\int_{\Gamma_{\epsilon}}g\mathcal{N}_{\epsilon}(f(x_1))\,dS
\\
&=\int_{\Omega_1}\nabla\mathcal{H}_1h\cdot\nabla \mathcal{H}_1f\,dx-\int_{\Omega_{\epsilon}}\nabla\mathcal{H}_{\epsilon}g\cdot\nabla \mathcal{H}_{\epsilon} (f(x_1))\,dx.
\end{split}
\end{equation*}
Using again the change of variables $x\mapsto x_1$ for the first term in the second line above, together with the estimates $$\|\mathcal{H}_1h\|_{H^{1}(\Omega_1)}\lesssim_M \|g\|_{H^{\frac{1}{2}}(\Gamma_{\epsilon})}\ \ \text{and}\  \ \|(\nabla\mathcal{H}_1 f)(x_1)\|_{L^2(\Omega_{\epsilon})}\lesssim_M \|f\|_{H^{\frac{1}{2}}(\Gamma_1)},$$ it is easy to verify
\begin{equation}\label{twoH-12terms}
\begin{split}
\int_{\Gamma_{\epsilon}}g((\mathcal{N}_1f)(x_1)-\mathcal{N}_{\epsilon}(f(x_1)))\,dS&\lesssim_M \int_{\Omega_{\epsilon}}\nabla((\mathcal{H}_1h)(x_1)-\mathcal{H}_{\epsilon}g)\cdot (\nabla\mathcal{H}_1f)(x_1)\,dx
\\
&+\int_{\Omega_{\epsilon}}\nabla H_{\epsilon}g\cdot \nabla ((\mathcal{H}_1f)(x_1)-\mathcal{H}_{\epsilon}(f(x_1)))\,dx+\epsilon\|g\|_{H^{\frac{1}{2}}(\Gamma_\epsilon)}\|f\|_{H^{\frac{1}{2}}(\Gamma_1)}   .
\end{split}
\end{equation}
We label the first and second terms on the right-hand side above by $I_1$ and $I_2$. For $I_1$, we use the fact that on $\Gamma_{\epsilon}$ we have
\begin{equation*}
(\mathcal{H}_1h)(x_1)-\mathcal{H}_{\epsilon}g=(J(x_1)-I)g   
\end{equation*}
to obtain the following simple elliptic estimate
\begin{equation*}
I_1\lesssim_M \epsilon\|f\|_{H^{\frac{1}{2}}(\Gamma_1)}\|g\|_{H^{\frac{1}{2}}(\Gamma_\epsilon)}+\|f\|_{H^{\frac{1}{2}}(\Gamma_1)}\|\Delta ((\mathcal{H}_1h)(x_1))\|_{H^{-1}(\Omega_\epsilon)}\lesssim_M \epsilon\|f\|_{H^{\frac{1}{2}}(\Gamma_1)}\|g\|_{H^{\frac{1}{2}}(\Gamma_\epsilon)},   
\end{equation*}
where we used the chain rule and that $\mathcal{H}_1h$ is harmonic to estimate $\Delta ((\mathcal{H}_1h)(x_1))$. A similar elliptic estimate yields the same bound for $I_2$. This establishes the case $s=-\frac{1}{2}$. By interpolation, we only need to handle the remaining cases when $\frac{1}{2}\leq s\leq k-2$. As a starting point, we have from some simple manipulations with the chain rule and the trace inequality,
\begin{equation*}
\begin{split}
\|(\mathcal{N}_1f)(x_1)-\mathcal{N}_{\epsilon}(f(x_1))\|_{H^s(\Gamma_\epsilon)}&\lesssim \epsilon\|f\|_{H^{s+1}(\Gamma_1)}+ \|(n_{\Gamma_1}(x_1)-n_{\Gamma_{\epsilon}})\cdot(\nabla\mathcal{H}_1f)(x_1)\|_{H^s(\Gamma_\epsilon)}
\\
&+\|(\mathcal{H}_1f)(x_1)-\mathcal{H}_{\epsilon}(f(x_1))\|_{H^{s+\frac{3}{2}}(\Omega_{\epsilon})}.
\end{split}
\end{equation*}
By writing $n_{\Gamma_1}(x_1)-n_{\Gamma_\epsilon}=a_{\epsilon}^{-1}\nabla p_{\epsilon}-a_1^{-1}(x_1)(\nabla p_1)(x_1)$ and using the relations in \Cref{goodvarrel} and that $s\leq k-2$, the second term on the right is straightforward to control by $\epsilon\|f\|_{H^{s+1}(\Gamma_{\epsilon})}$. For the third term, we do an elliptic estimate analogous to the $s=-\frac{1}{2}$ case (using that $(\mathcal{H}_1f)(x_1)-\mathcal{H}_{\epsilon}(f(x_1))=0$ on $\Gamma_{\epsilon}$) to obtain
\begin{equation*}
\|(\mathcal{H}_1f)(x_1)-\mathcal{H}_{\epsilon}(f(x_1))\|_{H^{s+\frac{3}{2}}(\Omega_{\epsilon})}\lesssim_M \|\Delta ((\mathcal{H}_1 f)(x_1))\|_{H^{s-\frac{1}{2}}(\Omega_{\epsilon})}\lesssim_M \epsilon\|f\|_{H^{s+1}(\Gamma_1)}.    
\end{equation*}
This completes the proof.
\end{proof}
Now we return to the proof of \Cref{changeofvar}. We note first that
\begin{equation*}
\begin{split}
\|(\mathcal{N}^{k-1}_1a_1)(x_1)-\mathcal{N}^{k-1}_{\epsilon}(a_1(x_1))\|_{L^2(\Gamma_{\epsilon})}&\lesssim \|\mathcal{N}_{\epsilon}(\mathcal{N}^{k-2}_1a_1)(x_1)-\mathcal{N}^{k-1}_{\epsilon}(a_1(x_1))\|_{L^2(\Gamma_{\epsilon})} 
\\
&+\|\mathcal{N}_{\epsilon}(\mathcal{N}^{k-2}_1a_1)(x_1)-(\mathcal{N}^{k-1}_1a_1)(x_1)\|_{L^2(\Gamma_{\epsilon})}.
\end{split}
\end{equation*}
Applying \Cref{changeofvar2} to the term in the second line and using the $H^1\to L^2$ bound for $\mathcal{N}$, we have
\begin{equation*}
\|(\mathcal{N}^{k-1}_1a_1)(x_1)-\mathcal{N}^{k-1}_{\epsilon}(a_1(x_1))\|_{L^2(\Gamma_{\epsilon})}\lesssim_M \|(\mathcal{N}^{k-2}_1a_1)(x_1)-\mathcal{N}^{k-2}_{\epsilon}(a_1(x_1))\|_{H^1(\Gamma_{\epsilon})}+\mathcal{O}_M(\epsilon).     
\end{equation*}
Iterating this procedure and applying \Cref{changeofvar2} $k-2$ times, we see that we have 
\begin{equation*}
\|(\mathcal{N}^{k-1}_1a_1)(x_1)-\mathcal{N}^{k-1}_{\epsilon}(a_1(x_1))\|_{L^2(\Gamma_{\epsilon})}\lesssim_M\epsilon   . 
\end{equation*}
It follows from the above and a change of variables that we have
\begin{equation*}
\|a_1^{-\frac{1}{2}}\mathcal{N}_1^{k-1}a_1\|_{L^2(\Gamma_1)}^2\leq \|a_1^{-\frac{1}{2}}(x_1)\mathcal{N}_{\epsilon}^{k-1}(a_1(x_1))\|_{L^2(\Gamma_{\epsilon})}^2+\mathcal{O}_M(\epsilon).
\end{equation*}
To conclude the proof of \Cref{changeofvar}, we need to show that
\begin{equation*}\label{Dtachangebound}
\|\nabla\mathcal{H}_1(\mathcal{N}_1^{k-2}(a_1^{-1}D_ta_1))\|_{L^2(\Omega_{1})}^2\leq \|\nabla\mathcal{H}_{\epsilon}\mathcal{N}_{\epsilon}^{k-2}(a_1^{-1}(x_1)D_ta_1(x_1))\|_{L^2(\Omega_{\epsilon})}^2+\mathcal{O}_M(\epsilon).
\end{equation*}
From a change of variables, we see that
\begin{equation*}
\|\nabla\mathcal{H}_1(\mathcal{N}_1^{k-2}(a_1^{-1}D_ta_1))\|_{L^2(\Omega_{1})}^2- \|\nabla\mathcal{H}_{\epsilon}\mathcal{N}_{\epsilon}^{k-2}(a_1^{-1}(x_1)D_ta_1(x_1))\|_{L^2(\Omega_{\epsilon})}^2\lesssim_M \mathcal{J}+\mathcal{O}_M(\epsilon),
\end{equation*}
where
\begin{equation*}
\mathcal{J}:=\|(\nabla\mathcal{H}_1\mathcal{N}_1^{k-2}(a_1^{-1}D_ta_1))(x_1)-\nabla\mathcal{H}_{\epsilon}\mathcal{N}_{\epsilon}^{k-2}(a^{-1}(x_1)D_ta_1(x_1))\|_{L^2(\Omega_{\epsilon})}.  
\end{equation*}
By elliptic regularity, it is easy to verify the bound
\begin{equation*}
\mathcal{J}\lesssim_M \|(\mathcal{N}_1^{k-2}(a_1^{-1}D_ta_1))(x_1)-\mathcal{N}_{\epsilon}^{k-2}(a_1^{-1}(x_1)D_ta_1(x_1))\|_{H^{\frac{1}{2}}(\Gamma_{\epsilon})}+\mathcal{O}_M(\epsilon).    
\end{equation*}
From here, we use \Cref{changeofvar2} similarly to the other surface term in the energy to estimate
\begin{equation*}
\|(\mathcal{N}_1^{k-2}(a_1^{-1}D_ta_1))(x_1)-\mathcal{N}_{\epsilon}^{k-2}(a_1^{-1}(x_1)D_ta_1(x_1))\|_{H^{\frac{1}{2}}(\Gamma_{\epsilon})}\lesssim_M\epsilon.    
\end{equation*}
This completes the proof.
\end{proof}

\subsection{Convergence of the iteration scheme}\label{COTS} 

We have now arrived at the final step of the existence proof, where we use our one step 
iteration result in Theorem~\ref{onestepiteration} 
in order to prove the existence of regular solutions.
Precisely, we aim to establish the following theorem.

\begin{theorem}\label{t:existence}
Let $k$ be a sufficiently large even integer 
and $M > 0$.
Let $(v_0,\Gamma_0)\in\mathbf{H}^k$ be an initial data set so that $\|(v_0,\Gamma_0)\|_{\mathbf{H}^k}\leq M$. Then there exists 
$T = T(M)$ and a solution $(v,\Gamma)$ to the free boundary incompressible Euler equations on $[0,T]$ with this initial data and the following regularity properties:
\begin{equation*}
(v,\Gamma) \in L^\infty([0,T]; \mathbf H^k) \cap  C([0,T]; \mathbf H^{k-1})   
\end{equation*}
with the uniform bound
\begin{equation*}
\|(v,\Gamma)(t)\|_{ \mathbf H^k} \lesssim_M 1, \qquad t \in [0,T].  
\end{equation*}
\end{theorem}

We remark that the solution we construct is unique 
by the result in Theorem~\ref{t:unique}.  One missing piece 
here is the lack of continuity in $\mathbf H^k$,
which does not follow from the proof below. However,
this will be rectified in the next section.
We now turn to  the proof of the theorem.

\begin{proof}

Starting from the initial data $(v_0,\Gamma_0) \in \mathbf H^k$ with 
$\Gamma_0 \in \Lambda_*:= \Lambda(\Gamma_*,\epsilon_0,\delta)$,
for each small time scale $\epsilon$ we construct a discrete
approximate solution $(v_\epsilon,\Gamma_\epsilon)$
which is defined at discrete times $t = 0, \epsilon, 2\epsilon, \dots$, as follows:

\begin{enumerate}
\item We define $(v_\epsilon(0), \Gamma_\epsilon(0))$
by directly regularizing $(v_0,\Gamma_0)$ at scale $\epsilon$. Such a regularization is provided by Proposition~\ref{c reg bounds} with $\epsilon = 2^{-j}$. In view of the higher regularity bound there, these regularized data will satisfy the 
hypothesis of our one step Theorem~\ref{onestepiteration}, with $M$ 
replaced by $\tilde M = C(A)M$.

\item We inductively define the approximate solutions $(v_\epsilon(j \epsilon), \Gamma_\epsilon(j\epsilon))$
by repeatedly applying the iteration step in Theorem~\ref{onestepiteration}.
\end{enumerate}

 To control the growth of the $\mathbf H^k$ norms
 of $(v_\epsilon,\Gamma_\epsilon)$
we rely on the energy monotonicity relation, together with the coercivity property in Theorem~\ref{Energy est. thm} (and also the relation (\ref{modifiedenergy})). We use the energy coercivity in both ways.
At time $t = 0$ we have 
\[
\mathcal{E}^k(v_\epsilon(0),\Gamma_\epsilon(0)) \leq C_1(A) M. 
\]
We let our iteration continue for as long as 
\begin{equation}\label{stop-iteration}
\begin{aligned}
& \mathcal{E}^k(v_\epsilon(j\epsilon),\Gamma_\epsilon(j\epsilon)) \leq 2 C_1(A) M,
\\ 
& \Gamma_\epsilon(j\epsilon) \in 2\Lambda_*:= \Lambda(\Gamma_*,\epsilon_0,2\delta).
\end{aligned}
\end{equation}
As long as this happens, using the coercivity in the other direction we get
\[
\| (v_\epsilon(j\epsilon),\Gamma_\epsilon(j\epsilon))\|_{\mathbf H^k} \leq C_2(A) M. 
\]
Now by the energy monotonicity bound \eqref{EMBITT} we conclude that
\[
\mathcal{E}^k(v_\epsilon(j\epsilon),\Gamma_\epsilon(j\epsilon))
\leq (1+ C( C_2(A) M)\epsilon )^j \mathcal{E}^k(v_\epsilon(0),\Gamma_\epsilon(0))
\leq e^{ C(C_2(A) M)\epsilon j}  \mathcal{E}^k(v_\epsilon(0),\Gamma_\epsilon(0)).
\]
Hence we can reach the cutoff given by the first inequality in \eqref{stop-iteration} no earlier than at time
\[
t = \epsilon j<T(M): = C(C_2(A) M)^{-1},
\]
which is a bound that does not depend on $\epsilon$. Similarly, for the second requirement in \eqref{stop-iteration}, the relations \eqref{approx-sln} ensure that at each step the boundary only moves by $\mathcal{O}(\epsilon)$, so by step $j$ it moves at most by $\mathcal{O}(j \epsilon)$. This leads to a similar constraint as above on the number of steps. Analogous reasoning shows that the vorticity growth in \eqref{regboundprop} is also harmless on this time scale.
\\

To summarize, we have proved that the discrete approximate solutions $(v_\epsilon,\Gamma_\epsilon)$
are all defined up to the above time $T(M)$, and satisfy the uniform bound
\begin{equation*}
\|   (v_\epsilon,\Gamma_\epsilon)\|_{\mathbf H^k}
\lesssim_M 1 \qquad \text{in } [0,T],
\end{equation*}
with $\Gamma_\epsilon \in 2\Lambda_*$.
Since $k$ is large enough, by Sobolev embeddings, this yields uniform bounds, say, in $C^3$,
\begin{equation}\label{c3}
\| v_\epsilon\|_{C^3} + \|\eta_\epsilon\|_{C^3} \lesssim_M 1 \qquad \text{in } [0,T],
\end{equation}
where $\eta_\epsilon:=\eta_{\Gamma_\epsilon}$ is the defining function for $\Gamma_\epsilon \in 2\Lambda_*$.
\\


The other piece of information we have about $v_\epsilon$ comes from \eqref{approx-sln}. However, this only tells us what happens over a single time step of size $\epsilon$, so we need to iterate it over multiple steps. 
We begin with the first relation for the velocity in \eqref{approx-sln}, which implies that
\[
|v_\epsilon(t,x) - v_\epsilon(s,y)|+
|\nabla v_\epsilon(t,x) - \nabla v_\epsilon(s,y)| \lesssim_M |t-s| + |x-y|, \qquad t-s = \epsilon.
\]
Iterating this we arrive at
\begin{equation}\label{Lip-ve}
|v_\epsilon(t,x) - v_\epsilon(s,y)|+
|\nabla v_\epsilon(t,x) - \nabla v_\epsilon(s,y)| \lesssim_M |t-s| + |x-y|, \qquad t,s \in \epsilon \N \cap [0,T].
\end{equation}
A similar reasoning based on the last part of \eqref{approx-sln} yields 
\begin{equation}\label{Lip-etae}
   \|\eta_\epsilon(t) - \eta_\epsilon(s)\|_{C^1} \lesssim_M |t-s|, \qquad t,s \in \epsilon \N \cap [0,T].
\end{equation}
Similarly, from (\ref{prelation}) in \Cref{goodvarrel} and the elliptic estimate $\|D_tp_{\epsilon}\|_{H^k}\lesssim_M 1$ for each time, we also get
a difference bound for the pressure; namely,
\begin{equation}\label{Lip-pe}
|\nabla p_\epsilon(t,x) - \nabla p_\epsilon(s,y)| \lesssim_M |t-s|+|x-y|, \qquad t,s \in \epsilon \N \cap [0,T].
\end{equation}
Equipped with the last three Lipschitz bounds 
in time, we are now able to return to 
\eqref{approx-sln} and reiterate in order to 
obtain second order information.
As above, we begin with the first relation in 
\eqref{approx-sln}. Here we reiterate directly,
using the bounds \eqref{Lip-ve} and \eqref{Lip-pe} in order to compare the expressions on the right at different times in the uniform norm. This yields 
\begin{equation}\label{Euler-app}
v_\epsilon(t) = v_\epsilon(s) -(t-s) (v_\epsilon(s)\cdot\nabla v_\epsilon(s)+\nabla p_\epsilon(s)+ge_d)+\mathcal{O}((t-s)^2),
\qquad t,s \in \epsilon \N \cap [0,T].
\end{equation}
The same procedure applied to the last component of \eqref{approx-sln} yields
\begin{equation}\label{kinenatic-app}
    \Omega_\epsilon(t) = (I + (t-s)v_\epsilon(s))\Omega_\epsilon(s) + \mathcal{O}((t-s)^2),  \qquad t,s \in \epsilon \N \cap [0,T].
\end{equation}

\medskip

We now have enough information about our approximate solutions $(v_\epsilon,\Gamma_\epsilon)$, and we seek to obtain the desired solution $(v,\Gamma)$ by taking the limit of 
$(v_\epsilon,\Gamma_\epsilon)$ on a subsequence as $\epsilon \to 0$. For this 
it is convenient to take $\epsilon$ of the form $\epsilon = 2^{-m}$, where we let $m \to \infty$.
Then the time domains of the corresponding approximate solutions  $v_m$ are nested.
\\

Starting from the Lipschitz bounds \eqref{Lip-ve}, \eqref{Lip-etae} and \eqref{Lip-pe},
a careful application of the Arzela-Ascoli theorem yields uniformly  convergent subsequences 
\begin{equation}\label{AA}
\eta_m \to \eta, \qquad v_m \to v, \qquad \nabla v_m \to \nabla v,   \qquad \nabla p_m \to \nabla p,
\end{equation}
whose limits still satisfies the bounds \eqref{Lip-ve}, \eqref{Lip-etae} and \eqref{Lip-pe}. It remains to show that $(v,\Gamma)$ is the desired solution
to the free boundary incompressible Euler equations, with 
$\Gamma$ defined by $\eta$ and $p$, where $p$  is the associated pressure.
\\

We begin by upgrading the spatial regularity 
of $v$ and $\eta$. For this we observe that for $t \in 2^{-j} \N \cap [0,T]$ we can pass to the limit  as 
$m \to \infty$ in \eqref{c3} to obtain the uniform bound 
\begin{equation*}\label{c3-lim}
\| v\|_{C^3} + \|\eta\|_{C^3} \lesssim_M 1 .
\end{equation*}
Since both $v$ and $\eta$ are Lipschitz continuous in $t$, this extends easily to all $t \in [0,T]$. A similar argument applies to the $\mathbf H^k$ norm of $(v,\Gamma)$. 
\\

Next we show that $(v,\Gamma)$ solves the free boundary incompressible Euler equations, which we 
do in several steps:

\medskip

\emph{i) The initial data.} The fact that at the initial time we have $(v(0),\Gamma(0)) = (v_0,\Gamma_0)$ follows directly from the construction of $(v_\epsilon(0), \Gamma_\epsilon(0))$; namely, by Proposition~\ref{c reg bounds}.

\medskip

\emph{ii) The pressure equation.} To verify that 
$p$ is the pressure associated to $v$ and $\Gamma$ 
we simply use the uniform convergence of $\nabla v_m$,
$\eta_m$ and $\nabla p_m$ in order to pass to the limit in the pressure equation \eqref{Euler-pressure}.

\medskip
\emph{iii) The incompressible Euler equations.} 
Here we directly use the uniform convergence \eqref{AA} in order to pass to the limit in \eqref{Euler-app}. This implies that $v$ is differentiable in time, and that the incompressible Euler equations are verified.

\medskip
\emph{iv) The kinematic boundary condition.}
Arguing as above, this time  we directly use the uniform convergence \eqref{AA} in order to pass to the limit in \eqref{kinenatic-app}.
\\

Finally, the $C(\mathbf H^{k-1})$ regularity of $(v,\Gamma)$ follows directly from the incompressible Euler equations and the kinematic boundary condition.
\end{proof}

\section{Rough solutions} 
In this section, we aim to construct solutions in the state space $\mathbf{H}^s$ as limits of regular solutions for $s>\frac{d}{2}+1$. The general procedure for executing this construction will be as follows.
\begin{enumerate}
    \item We regularize the initial data.
    \item\label{STEP2} We prove uniform bounds for the corresponding regularized solutions.
    \item We show convergence of the regularized solutions in a weaker topology.
    \item We combine the difference estimates and the uniform $\mathbf{H}^s$ bounds from step (\ref{STEP2}) to obtain convergence in the $\mathbf{H}^s$ topology.
\end{enumerate}
As will be seen below, this procedure carries with it various subtleties since it involves comparing functions defined on different domains. In addition, we must carefully address the fact that our control parameters in the difference and energy estimates are not entirely consistent.
\subsection{Initial data regularization} 
Let $(v_0,\Gamma_0)\in \mathbf{H}^s$ be an initial data. The first step is to place $\Gamma_0$ within a suitable collar
$\Lambda_*=\Lambda(\Gamma_*,\epsilon,\delta)$ with $\delta \ll 1$.
Since $\Gamma_0 \in H^s \subseteq C^{1,\epsilon+}$,
$\Gamma_*$ is easily obtained by regularizing $\Gamma_0$ on a 
small enough spatial scale. We remark that the price to pay for a small enough regularization scale is that the higher Sobolev norms $H^k$ of $\Gamma_*$ will be large; but this is acceptable, as explained in Remark~\ref{r:delta}.
\\

Let $M:=\|(v_0,\Gamma_0)\|_{\mathbf{H}^s}$ denote the data size 
measured relative to the collar $\Lambda_*$, and write $c_0$ for the lower bound on the Taylor term. We begin by constructing regularized data at each dyadic scale $2^j$. For this, we define $\Gamma_{0,j}$ (along with $\Omega_{0,j}$) by regularizing the collar parameterization $\eta_0$. More specifically, we define  $\eta_{0,j}:=P_{\leq j}\eta_{0}$, where  the meaning of $P_{\leq j}$ is as in \Cref{SSRO}. Then, we define the regularized velocity $v_{0,j}:=\Psi_{\leq j}v_0$. Here, we recall that, as long as $j$ is much larger than $M$, $v_{0,j}$ is defined on some $2^{-j}$ enlargement of both $\Omega_{0,j}$ and $\Omega_0$. Indeed, by Sobolev embeddings, we have the distance bound 
\begin{equation*}
|\eta_{0,j}-\eta_0|\lesssim_{M} 2^{-\frac{3}{2}j}.    
\end{equation*}
Moreover, for such $j$, we stay in the collar and have a uniform lower bound on the Taylor term.

\subsection{Uniform bounds and lifespan of regular solutions} By \Cref{t:existence}, the regularized data $(v_{0,j},\Gamma_{0,j})$ from the previous step generate corresponding smooth solutions $(v_j,\Gamma_j)$. Our goal now is to establish uniform bounds for these regular solutions and, in particular, show that they have a lifespan which depends only on the size of the initial data $(v_0,\Gamma_0)$ in $\mathbf{H}^s$, Taylor sign and the collar.  To do this, we carry out a bootstrap argument with the $\mathbf{H}^s$ norm of $(v_j,\Gamma_j)$. 
\\

In the argument below, we will be working with the enlarged control parameter $\tilde{B}_j(t):=\|v_j\|_{W^{1,\infty}(\Omega_j)}+\|\Gamma_j\|_{C^{1,\frac{1}{2}}}+\|D_tp_j\|_{W^{1,\infty}(\Omega_j)}$ for the corresponding solution $(v_j,\Gamma_j)$. Note that the reason we work for now with  $\tilde{B}_j$ instead of just $B_j(t):=\|v_j\|_{W^{1,\infty}(\Omega_j)}+\|\Gamma_j\|_{C^{1,\frac{1}{2}}}$ is because we will make use of the difference estimates which require control of $D_tp_j$. By elliptic regularity and Sobolev embeddings, it is easy to see that $\tilde{B}_j$ is controlled by some polynomial in $\|(v_j,\Gamma_j)\|_{\mathbf{H}^s}$. 
\\

Fix some large  parameters $A_0$ and $B_0$ depending only on the numerical constants for the data ($M$, $c_0$ and so forth) such that $A_0\ll B_0$. As alluded to above, we 
make the bootstrap assumption
\begin{equation*}\label{boostrap}
\|(v_j,\Gamma_j)(t)\|_{\mathbf{H}^s}\leq 2B_0,\hspace{5mm}A_j(t)\leq 2A_0,\hspace{5mm} a_j(t)\geq \frac{c_0}{2},\hspace{5mm}\Gamma_j(t)\in 2\Lambda_*,\hspace{5mm}t\in [0,T],\hspace{5mm} j(M)=: j_0\leq j\leq j_1,
\end{equation*}
 with $j(M)$ sufficiently large depending on $M$, in a time interval $[0,T]$ where all the $(v_j,\Gamma_j)$ are defined as smooth solutions with boundaries in the collar.
Above, $j_1$ is some finite but arbitrarily large parameter, introduced for technical convenience to ensure  that we run the bootstrap on only finitely many solutions at a time.
Our aim will be to show that we can improve this bootstrap assumption as long as $T\leq T_0$ for some time $T_0>0$ which is independent of $j_1$.  
 \\
 
 For any large integer $k > s > \frac{d}{2}+1$ 
 as in \Cref{t:existence}, we may  consider
 the solutions $(v_j,\Gamma_j)$ as solutions in $\mathbf H^k$. In light of Theorems~\ref{Energy est. thm} and \ref{t:existence}, for each $j\geq j_0$, the solution $(v_j,\Gamma_j)$ can be continued past time $T$ in $\mathbf{H}^k$ (and therefore $\mathbf{H}^s$)  as long as the bootstrap is satisfied.  Morally speaking, our choice for $T_0$ will be
\begin{equation*}
T_0\ll\frac{1}{P(B_0)},    
\end{equation*}
for some fixed polynomial $P$, though this is not entirely accurate, as $T_0$ will also depend on the collar and $c_0$. Thanks to the energy bound in \Cref{Energy est. thm}, if the bootstrap could be extended to such a $T_0$, it would guarantee uniform $\mathbf{H}^k$ bounds for $(v_j,\Gamma_j)$ for any integer $k>\frac{d}{2}+1$ in terms of its initial data in $\mathbf{H}^k$. 
The main difficulty we face is that, a priori, the $\mathbf{H}^s$ bounds for $(v_j,\Gamma_j)$ do not necessarily propagate for noninteger $s$. The goal, therefore, is to establish $\mathbf{H}^s$ bounds for noninteger $s$. We will do this by working solely with the energy estimates for integer indices and the difference estimates.
\\

We begin by letting $c_j$ be the $\mathbf{H}^s$ admissible frequency envelope for the initial data $(v_{0},\Gamma_{0})$ given by (\ref{admissable}). We let $\alpha\geq 1$ be such that $k = s+\alpha$ is an integer. From \Cref{envbounds}
we know that the regularized data $(v_{0,j},\Gamma_{0,j})$ satisfy the bounds
\begin{equation}\label{hiregbound0}
\|(v_{0,j},\Gamma_{0,j})\|_{\mathbf{H}^{s+\alpha}}\lesssim_{A_0} 2^{\alpha j}c_j\|(v_{0},\Gamma_0)\|_{\mathbf{H}^s}.
\end{equation}
From the energy bounds in \Cref{Energy est. thm} and the bootstrap hypothesis, we deduce from \eqref{hiregbound0} and the definition of $c_j$ that 
\begin{equation}\label{hiregbound}
\|(v_j,\Gamma_j)(t)\|_{\mathbf{H}^{s+\alpha}}\lesssim_{A_0} 2^{\alpha j}c_j(1+\|(v_{0},\Gamma_0)\|_{\mathbf{H}^s}), \qquad t \in [0,T],
\end{equation}
as long as $T\leq T_0\ll\frac{1}{P(B_0)}$. One may think of this as a high frequency bound, which roughly speaking allows us 
to control frequencies $\gtrsim 2^j$ in $(v_j,\Gamma_j)$. Note that in \eqref{hiregbound} we suppressed the implicit dependence on the Taylor term and the collar. We will do  this throughout the subsection except when these terms are of primary importance, as it will be clear that our argument can handle these minor technicalities. 
\\

To estimate low frequencies we use the difference estimates. Precisely, at the initial time  we claim that we have the difference bound
\begin{equation}\label{regularizeddiffbound}
D((v_{0,j},\Gamma_{0,j}),(v_{0,j+1}, \Gamma_{0,j+1}))\lesssim_{A_0} 2^{-2js}c_j^2\|(v_{0},\Gamma_0)\|_{\mathbf{H}^s}^2    .
\end{equation}
This bound is clear by \Cref{envbounds}  for the first term in (\ref{diff functional candidate}). To see this for the surface integral, we use that on $\tilde{\Gamma}_{0,j}:=\partial (\Omega_{0,j}\cap\Omega_{0,j+1})$, the pressure difference $p_{0,j}-p_{0,j+1}$ is proportional (with implicit constant depending on $A_0$) to the distance between $\Gamma_{0,j}$ and $\Gamma_{0,j+1}$, measured using the displacement function (\ref{distancefunction}). Combining this with a change of variables, we have  
\begin{equation*}
\int_{\tilde{\Gamma}_{0,j}}|p_{0,j}-p_{0,j+1}|^2\,dS\approx_{A_0} \|\eta_{0,j+1}-\eta_{0,j}\|_{L^2(\Gamma_*)}^2\lesssim_{A_0} 2^{-2js}c_j^2\|(v_0,\Gamma_0)\|_{\mathbf{H}^s}^2,   
\end{equation*}
from which \eqref{regularizeddiffbound} follows. By \Cref{Difference}, we can propagate the difference bound (\ref{regularizeddiffbound}) to obtain 
\begin{equation}\label{propagateddbound}
D((v_j,\Gamma_j)(t),(v_{j+1},\Gamma_{j+1})(t))\lesssim_{A_0} 2^{-2js}c_j^2\|(v_{0},\Gamma_0)\|_{\mathbf{H}^s}^2,   \qquad t \in [0,T],
\end{equation}
as long as $T\leq T_0\ll\frac{1}{P(B_0)}$. In particular, this gives by a similar argument to the above, 
\begin{equation}\label{diffbounds}
\|v_{j+1}-v_j\|_{L^2(\Omega_j\cap\Omega_{j+1})},\hspace{3mm}\|\eta_{j+1}-\eta_j\|_{L^2(\Gamma_*)}\lesssim_{A_0} 2^{-js}c_j\|(v_{0},\Gamma_0)\|_{\mathbf{H}^s}.    
\end{equation}
Now, the goal is to combine the high frequency bound (\ref{hiregbound}) and the 
$L^2$ difference bound (\ref{diffbounds}) in order to obtain a uniform $\mathbf{H}^s$ bound of the form 
\begin{equation*}
\|(v_j,\Gamma_j)\|_{\mathbf{H}^s}\lesssim_{A_0}1+ \|(v_0,\Gamma_0)\|_{\mathbf{H}^s},
\end{equation*}
for $T\leq T_0$. To establish such a bound for $\Gamma_j$, we consider the telescoping series on $\Gamma_*$ given by
\begin{equation}\label{etadecomp}
\eta_j=\eta_{j_0}+\sum_{j_0\leq l\leq j-1}(\eta_{l+1}-\eta_l)   . 
\end{equation}
From the higher energy bound  (\ref{hiregbound}), we have for each $j_0\leq l\leq j-1$,
\begin{equation}\label{etaboundreg}
\|\eta_{l+1}-\eta_{l}\|_{H^{s+\alpha}(\Gamma_*)}\lesssim_{A_0} 2^{l\alpha}c_l(1+\|(v_0,\Gamma_0)\|_{\mathbf{H}^s}).    
\end{equation}
Using the telescoping sum and interpolation, it is straightforward to verify from (\ref{diffbounds}), (\ref{etaboundreg}) and an argument similar to \Cref{envbounds} (see also \cite{IT-primer}) that for each $k\geq 0$,
\begin{equation}\label{P-keta}
\|P_k\eta_j\|_{H^s(\Gamma_*)}\lesssim_{A_0} c_k(1+\|(v_0,\Gamma_0)\|_{\mathbf{H}^s}).    
\end{equation}
As a consequence, by almost orthogonality, we obtain the uniform bound
\begin{equation}\label{surfenvbound}
\|\Gamma_j\|_{H^s}\lesssim_{A_0}1+ \|(v_0,\Gamma_0)\|_{\mathbf{H}^s}.    
\end{equation}
 Next, we turn to the bound for $v_j$. We first note that the analogous decomposition to (\ref{etadecomp}) for $v_j$ does not work because for each $l\leq j-1$, $v_l$ and $v_{l+1}$ are defined on different domains. However, we can compare $v_l$ and $v_{l+1}$ by first regularizing each function $v_l\mapsto \Psi_{\leq l}v_l$ which is defined on a $2^{-l}$ enlargement of $\Omega_l$. For this comparison to work, we need to know that $\Gamma_j$ and $\Gamma_{j+1}$ are sufficiently close. By interpolating using (\ref{diffbounds}) and (\ref{surfenvbound}) we have
 \begin{equation}\label{surfdistance}
 \|\eta_{j+1}-\eta_j\|_{L^{\infty}(\Gamma_*)}\lesssim_{A_0} 2^{-\frac{3}{2}j}, \hspace{5mm} \|\eta_{j+1}-\eta_j\|_{C^{1,\frac{1}{2}}(\Gamma_*)}\lesssim_{A_0} 2^{-\delta j},
 \end{equation}
 for some $\delta>0$. Now, we return to the uniform bound for $v_j$. Thanks to (\ref{surfdistance}), we can safely consider the decomposition on $\Omega_j$, 
\begin{equation}\label{vtelescope}
v_j=\Psi_{\leq j_0}v_{j_0}+\sum_{j_0\leq l\leq j-1}\Psi_{\leq l+1}v_{l+1}-\Psi_{\leq l}v_l+(I-\Psi_{\leq j})v_j. 
\end{equation}
 The first term in the telescoping decomposition is trivial to bound. We therefore focus our attention on the remaining terms. First, define for $l\geq j_0$
\begin{equation*}
\tilde{\Omega}_{l}=\bigcap_{k=l}^{j}\Omega_k.   
\end{equation*}
Thanks again to (\ref{surfdistance}), for $j_0$ large enough (independent of $j$ and only depending on the data parameters), we can arrange for the regularization operator $\Psi_{\leq l}$ to be bounded from $H^s(\tilde{\Omega}_l)$ to $H^s(\tilde{\Omega}_l')$ where $\tilde{\Omega}_l'$ is some $2^{-l}$ enlargement of the union of all of the $\Omega_k$ for $k\geq l$. We will use this fact to establish the following lemma which will help us to estimate the intermediate terms in (\ref{vtelescope}).
\begin{lemma}\label{L^2regbounds}Let $j_0\leq l\leq j-1$, where $j_0$ is some universal parameter depending only on the numerical constants for the data. Then given the above decomposition for $v_j$, we have
\begin{equation}\label{formerbound}
\|\Psi_{\leq l+1}v_{l+1}-\Psi_{\leq l}v_l\|_{L^2(\Omega_j)}\lesssim_{A_0} 2^{-ls}c_l(1+\|(v_0,\Gamma_0)\|_{\mathbf{H}^s}),   
\end{equation}
\begin{equation}\label{latterbound}
\|\Psi_{\leq l+1}v_{l+1}-\Psi_{\leq l}v_l\|_{H^{s+\alpha}(\Omega_j)}\lesssim_{A_0} 2^{l\alpha}c_l(1+\|(v_0,\Gamma_0)\|_{\mathbf{H}^s}).    
\end{equation}
\end{lemma}
By Sobolev embedding, a corollary of this lemma is the following pointwise bound at the $C^1$ regularity.
\begin{corollary}\label{pointwiseclose} We have the estimate
\begin{equation*}
\|\Psi_{\leq l+1}v_{l+1}-\Psi_{\leq l}v_l\|_{C^1(\Omega_j)}\lesssim_{A_0} 2^{-l\delta}(1+\|(v_0,\Gamma_0)\|_{\mathbf{H}^s}),\hspace{5mm}\delta>0.    
\end{equation*}    
\end{corollary}
\begin{proof}
The latter bound (\ref{latterbound}) is clear from the $H^{s+\alpha}$ boundedness of $\Psi_{\leq l}$ and (\ref{hiregbound}). For the first bound, we split
\begin{equation*}
\Psi_{\leq l+1}v_{l+1}-\Psi_{\leq l}v_l=(\Psi_{\leq l+1}-\Psi_{\leq l})v_{l+1}+\Psi_{\leq l}(v_{l+1}-v_l).    
\end{equation*}
Using \Cref{c reg bounds} and (\ref{hiregbound}), we have
\begin{equation*}
\|(\Psi_{\leq l+1}-\Psi_{\leq l})v_{l+1}\|_{L^2(\Omega_j)}\lesssim_{A_0} 2^{-ls}c_l(1+\|(v_0,\Gamma_0)\|_{\mathbf{H}^s}).    
\end{equation*}
For the remaining term, we use the difference bound and the $L^2$ boundedness of $\Psi_{\leq l}$ to obtain
\begin{equation*}
\|\Psi_{\leq l}(v_{l+1}-v_l)\|_{L^2(\Omega_j)}\lesssim_{A_0} D((v_{l},\Gamma_l),(v_{l+1},\Gamma_{l+1}))^{\frac{1}{2}}\lesssim_{A_0} 2^{-ls}c_l\|(v_0,\Gamma_0)\|_{\mathbf{H}^s}.    
\end{equation*}
\end{proof}
We also observe that the same bounds in \Cref{L^2regbounds}  hold for the third term in (\ref{vtelescope}) but with the parameter $l$ replaced by $j$ in the corresponding estimates. This is immediate for \eqref{latterbound}
and follows by telescopic summation from 
Proposition~\ref{envbounds} in the case of \eqref{formerbound}.
\\

We can use the above lemma (and the corresponding bounds for $(I-\Psi_{\leq j})v_j$) to estimate similarly to \eqref{P-keta} that for each $k\geq 0$,
\begin{equation*}
\|P_k\mathcal{E}_{\Omega_j}v_j\|_{H^s(\mathbb{R}^{d})}\lesssim_{A_0} c_k(1+\|(v_0,\Gamma_0)\|_{\mathbf{H}^s}).  
\end{equation*}
From this observation and almost orthogonality, we obtain the desired uniform bound, 
\begin{equation*}
\|(v_j,\Gamma_j)(t)\|_{\mathbf{H}^s}\lesssim_{A_0} 1+\|(v_0,\Gamma_0)\|_{\mathbf{H}^s},  
\end{equation*}
for $t \in [0,T_0]$. In particular, if the constant $B_0$ is chosen to be sufficiently large relative to $A_0$ and the data size, this improves the bootstrap assumption for $\|(v_j,\Gamma_j)\|_{\mathbf{H}^s}$.  It remains to improve the bootstrap assumption for $A_j$ and at the same time the Taylor term and the collar neighborhood size. For this we rely on a computation similar to \cite{MR3925531,sz} for the Lagrangian flow map $u_j(t,\cdot):\Omega_{0,j}\to \Omega_j(t)$, defined as the solution to the ODE 
\begin{equation*}
\partial_tu_j(t,y)=v_j(t,u_j(t,y)),\hspace{5mm} y\in \Omega_{0,j},\hspace{5mm}u_j(0)=I.   
\end{equation*}
Since $s>\frac{d}{2}+1$, if $T_0$ is small enough, then for any $0\leq t\leq T\leq T_0$ we have the bound
\begin{equation*}
\begin{split}
\|u_j(t,\cdot)-I\|_{H^{s}(\Omega_{0,j})}&\lesssim\int_{0}^{t}\|v_j(t',\cdot)\|_{H^{s}(\Omega_j(t'))}\|u_j(t',\cdot)\|_{H^s(\Omega_{0,j})}^s\,dt'    
\\
&\lesssim_{A_0} t\|(v_0,\Gamma_0)\|_{\mathbf{H}^s}.
\end{split}
\end{equation*}
If $A_0$ is large enough relative to the data size, this  easily implies simultaneously
\begin{equation*}
\Gamma_j(t)\in \frac{3}{2}\Lambda_*,\hspace{5mm}\|\Gamma_{j}(t)\|_{C^{1,\epsilon}}\ll A_0  ,
\end{equation*}
 as long as $T_0$ is small enough. Doing a similar computation with $u_t$ in place of $u$ and using the equation 
\begin{equation*}
\partial_t^2u_j(t,y)=\partial_t(v_j(t,u_j(t,y)))=-(\nabla p_j+ge_d)(t,u_j(t,y))    
\end{equation*}
together with the elliptic estimates for the pressure, we obtain also
\begin{equation*}
\|v_j(t)\|_{C^{\frac{1}{2}+\epsilon}(\Omega_j)}\ll A_0.    
\end{equation*}
This improves the bootstrap assumption for $A_j$. Finally, a similar argument but instead with the pressure gradient and the $H^s$ bound for $D_tp$ allows one to close the bootstrap for $a_j$ as long as $T_0$ is sufficiently small depending on $M$ and $c_0$.
\subsection{The limiting solution} Here we show that for $T\leq T_0$,
\begin{equation*}
(v,\Gamma)=\lim_{j\to\infty}(v_j,\Gamma_j)    \ \ \text{in}\ \ C([0,T];\mathbf{H}^s).
\end{equation*}
First, we show domain convergence in $H^s$, which is more straightforward. Indeed, from \eqref{surfdistance} we see that the limiting domain $\Omega$ exists and has Lipschitz boundary $\Gamma$. Next, we let $j\geq j_0$ and consider the telescoping sum
\begin{equation*}
\eta-\eta_j=\sum_{l=j}^{\infty}\eta_{l+1}-\eta_l. 
\end{equation*}
An analysis similar to the previous subsection, using the difference bounds and the higher energy bounds, yields
\begin{equation}\label{limitdistance}
\|\eta-\eta_j\|_{L^{\infty}(\Gamma_*)}\lesssim_{A_0} 2^{-\frac{3}{2}j}    
\end{equation}
and
\begin{equation*}
\|\eta-\eta_j\|_{C([0,T];H^s(\Gamma_*))}\lesssim_{A_0} \|c_{\geq j}\|_{l^2}(1+\|(v_0,\Gamma_0)\|_{\mathbf{H}^s}),
\end{equation*}
which in particular shows convergence of $\Gamma_j\to\Gamma$ in $C([0,T];H^s(\Gamma_*))$. Next, we turn to showing the convergence $v_j\to v$ in $C([0,T];\mathbf{H}^s)$. We, formally, define $v$ through the  telescoping sum
\begin{equation*}
v=\Psi_{\leq j_0}v_{j_0}+\sum_{l\geq j_0}\Psi_{\leq {l+1}}v_{l+1}-\Psi_{\leq l}v_l  ,  
\end{equation*}
where, as usual,  $j_0$ ensures that all the terms in the sum are defined on $\Omega$. Thanks to (\ref{limitdistance}), this is possible.  We begin by showing that $\Psi_{\leq j}v_j\to v$ in $H^s(\Omega_t)$ uniformly in $t$ (which is again unambiguous thanks to (\ref{limitdistance})). We have
\begin{equation*}
v-\Psi_{\leq j}v_j=\sum_{l\geq j}\Psi_{\leq {l+1}}v_{l+1}-\Psi_{\leq l}v_l.    
\end{equation*}
From this we see that
\begin{equation*}
\|v-\Psi_{\leq j}v_j\|_{H^s(\Omega_t)}\lesssim_{A_0} \|c_{\geq j}\|_{l^2}(1+\|(v_0,\Gamma_0)\|_{\mathbf{H}^s}),
\end{equation*}
which establishes the desired uniform convergence in $H^s(\Omega_t)$. To show convergence of $v_j$ in the sense of \Cref{Def of convergence}, we consider the regularization $\tilde{v}=\Psi_{\leq m}v_m$. We then have as above, 
\begin{equation*}
\|v-\Psi_{\leq m}v_m\|_{H^s(\Omega)}\lesssim_{A_0} \|c_{\geq m}\|_{l^2}(1+\|(v_0,\Gamma_0)\|_{\mathbf{H}^s})  ,
\end{equation*}
which goes to $0$ as $m\to\infty$. On the other hand, for $j>m$, we have
\begin{equation*}
\begin{split}
\|v_j-\Psi_{\leq m}v_m\|_{H^s(\Omega_j)}\lesssim_{A_0} &\|(1-\Psi_{\leq j})v_j\|_{H^s(\Omega_j)}+\|\Psi_{\leq j}(v_j-v)\|_{H^s(\Omega_j)}+\|\Psi_{\leq m}(v_m-v)\|_{H^s(\Omega_j)}
\\
&+\|\Psi_{\leq j}v-\Psi_{\leq m}v\|_{H^s(\Omega_j)}.
\end{split}
\end{equation*}
Using (\ref{hiregbound}) for the first term and the difference bounds for $D((v_j,\Gamma_j),(v,\Gamma)),$ $D((v_m,\Gamma_m),(v,\Gamma))$ for the second and third terms, respectively, we obtain 
\begin{equation*}
\|v_j-\Psi_{\leq m}v_m\|_{H^s(\Omega_j)}\lesssim_{A_0} \|c_{\geq m}\|_{l^2}(1+\|(v_0,\Gamma_0)\|_{\mathbf{H}^s})+\|\Psi_{\leq j}v-\Psi_{\leq m}v\|_{H^s(\Omega_j)}.    
\end{equation*}
To estimate the last term above, we have
\begin{equation*}
\begin{split}
\|\Psi_{\leq j}v-\Psi_{\leq m}v\|_{H^s(\Omega_j)}&\lesssim_{A_0} \|(\Psi_{\leq j}-\Psi_{\leq m})(v-\Psi_{\leq m}v_m)\|_{H^s(\Omega_j)}+\|(\Psi_{\leq j}-\Psi_{\leq m})\Psi_{\leq m}v_m\|_{H^s(\Omega_j)}
\\
&\lesssim_{A_0} \|v-\Psi_{\leq m}v_m\|_{H^{s}(\Omega)}+2^{-m\alpha}\|v_m\|_{H^{s+\alpha}(\Omega_m)}
\\
&\lesssim_{A_0} \|c_{\geq m}\|_{l^2}(1+\|(v_0,\Gamma_0)\|_{\mathbf{H}^s}),   
\end{split}
\end{equation*}
where we used (\ref{hiregbound}) to estimate the second term in the last inequality. The combination of  the above estimates establishes strong convergence in $\mathbf{H}^s$. A similar argument shows continuity of $v$ with values in $\mathbf{H}^s$. Finally, one may also check that the limiting solution solves the free boundary Euler equations.
\subsection{Continuous dependence} Given a sequence of initial data $(v_0^n,\Gamma_0^n)\in\mathbf{H}^s$ such that $(v_0^n,\Gamma_0^n)\to (v_0,\Gamma_0)$, we aim to show that we have the corresponding convergence of the solutions $(v^n,\Gamma^n)\to (v,\Gamma)$ in $C([0,T];\mathbf{H}^s)$. First, we note that thanks to the data convergence, the corresponding solutions have a uniform in $n$ lifespan in $\mathbf{H}^s$, and so, on some compact time interval $[0,T]$, we have $\|(v^n,\Gamma^n)\|_{\mathbf{H}^s}+\|(v,\Gamma)\|_{\mathbf{H}^s}\lesssim_M 1$. Let us denote by $c_j^n$ and $c_j$ the admissible frequency envelopes for the data $(v_0^n,\Gamma_0^n)$ and $(v_0,\Gamma_0)$, respectively. Now, let $\epsilon>0$ and let $\delta=\delta(\epsilon)>0$ be a small positive constant to be chosen. Moreover, let $n_0=n_0(\epsilon)$ be some large integer to be chosen.
\\

 By definition of convergence in $\mathbf{H}^s$, there is a divergence free function $v_0^{\delta}\in H^s(\Omega_0^{\delta})$ defined on some enlarged domain $\Omega_0^{\delta}$ such that 
\begin{equation*}\label{v_0approx}
\|v_0-v_0^{\delta}\|_{H^s(\Omega_0)}+\limsup_{n\to\infty}\|v_0^n-v_0^{\delta}\|_{H^s(\Omega_0^n)}<\delta.
\end{equation*}
Moreover, for $n$ large enough, depending only on $\delta$, $v_0^{\delta}$ is defined on a neighborhood of $\Omega_0$ and $\Omega_{0}^n$.  Moreover, we may also assume that $v_0^{\delta}$ belongs to $H^s(\mathbb{R}^d)$. Indeed, for some $\delta'\ll\delta$, $v_0^{\delta}$ is defined on the domain $\Omega_0'$ defined by taking $\eta_0'=\eta_0+\delta'$. Then we can extend $v_0^{\delta}$ to $\mathbb{R}^d$ using \Cref{continuosext}. We note that $v_0^{\delta}$ is not necessarily divergence free on $\mathbb{R}^d$ but is on an enlargement of $\Omega_0$ and $\Omega_0^n$ for $n$ large enough. Now, let $c_j^{\delta}$  denote the admissible frequency envelope for $(v_0^{\delta},\Gamma_0)$ (note that we are using the same domain $\Omega_0$ as $v_0$ for the frequency envelope here; if $\delta$ is small enough, Taylor sign holds for this state) and denote by $(v^{\delta},\Gamma^{\delta})$ the corresponding $\mathbf{H}^s$ solution (which we note has lifespan comparable to $v$ and $v^n$ for $n$ large enough). We begin by choosing $j=j(\epsilon)$ large enough so that
\begin{equation}\label{dataenvsize}
\|c_{\geq j}\|_{l^2}<\epsilon.   
\end{equation}
We next observe that we can choose $\delta(\epsilon)$ and then $n_0(\delta)$ so that
\begin{equation}\label{envcontrol}
\|c_{\geq j}^n\|_{l^2}\lesssim_M \epsilon +\|c_{\geq j}\|_{l^2}\lesssim_M \epsilon,   
\end{equation}
for $n\geq n_0$. One can establish this by estimating the error when comparing terms in $c^{\delta}_j$ and $c_j^n$ and then the error when comparing terms in $c^{\delta}_j$ and $c_j$ by using  (\ref{admissable}) and square summing. The main error in the first comparison is essentially comprised of two terms. The first term to control involves the error between $\eta_0^n$ and $\eta_0$. If $\delta$ is small enough and $n$ is large enough, we have
\begin{equation*}
\|\eta^n_0-\eta_0\|_{H^s(\Gamma_*)}<\delta<\epsilon.    
\end{equation*}
The second source of error  comes from the extensions of the velocity functions,
\begin{equation*}
\|E_{\Omega_0^n}v_0^n-E_{\Omega_0}v_0^{\delta}\|_{H^s(\mathbb{R}^d)}\leq \|E_{\Omega_0^n}v_0^{\delta}-E_{\Omega_0}v_0^{\delta}\|_{H^s(\mathbb{R}^d)}+\|E_{\Omega_0^n}(v_0^n-v_0^{\delta})\|_{H^s(\mathbb{R}^d)}.    
\end{equation*}
If $\delta\ll_M\epsilon$, then the latter term is $\mathcal{O}(\epsilon)$ by (uniform in $n$) boundedness of $E_{\Omega_0^n}$ and the definition of $v_0^{\delta}$. The first term is $\mathcal{O}(\epsilon)$ if $n$ is large enough (relative to $\delta$) thanks to the continuity property of the family $E_{\Omega_0^n}$ in \Cref{continuosext}. Then one establishes (\ref{envcontrol}) by comparing $c_j$ and $c_j^{\delta}$ which just involves controlling essentially the error term $\|E_{\Omega_0}(v_0^{\delta}-v_0)\|_{H^s(\mathbb{R}^d)}$.
\\

Now that we have uniform smallness of the initial data frequency envelopes, the next step is to compare the corresponding solutions. First, thanks to the difference estimates, we observe that for large enough $n$, $\Gamma^n$ and $\Gamma^{\delta}$ are within distance $\ll 2^{-j}$ as long as $\delta$ is chosen small enough relative to $j$ (recall that $j$ was chosen to ensure (\ref{dataenvsize})). Indeed, by interpolating and using the uniform $\mathbf{H}^s$ bound, we have
\begin{equation*}
\|\eta^n-\eta^{\delta}\|_{L^{\infty}(\Gamma_*)}\lesssim_M D((v^n,\Gamma^n),(v^{\delta},\Gamma^\delta))^{\frac{3}{4s}}\lesssim_{M} \delta^{\frac{3}{2s}}.    
\end{equation*}
This ensures that we may compare $\Psi_{\leq j}v^{\delta}$ to $v^n$. Denoting by $(v^n_j,\Gamma^n_j)$ the regular solution corresponding to the regularized data $(v^n_{0,j},\Gamma_{0,j}^n)$ (from the previous section), we have 
\begin{equation*}
\begin{split}
\|\Psi_{\leq j}v^{\delta}-v^n\|_{H^s(\Omega^n)}&\lesssim \|\Psi_{\leq j}(v^{\delta}-v^n)\|_{H^s(\Omega^n)}+\|\Psi_{\leq j}(v^n-v^n_j)\|_{H^s(\Omega^n)}+\|v^n-\Psi_{\leq j}v^n_j\|_{H^s(\Omega^n)}
\\
&\lesssim_M \|c_{\geq  j}^n\|_{l^2}+2^{js}D((v^n,\Gamma^n),(v^n_j,\Gamma^n_j))^{\frac{1}{2}}+2^{js}D((v^n,\Gamma^n),(v^{\delta},\Gamma^\delta))^{\frac{1}{2}}
\\
&\lesssim_M \|c_{\geq j}^n\|_{l^2}+2^{js}D((v^n,\Gamma^n),(v^{\delta},\Gamma^{\delta}))^{\frac{1}{2}},
\end{split}
\end{equation*}
which if $\delta$ is small enough gives
\begin{equation*}
\|\Psi_{\leq j}v^{\delta}-v^n\|_{H^s(\Omega^n)}\lesssim_M \epsilon.    
\end{equation*}
Similarly, we may obtain
\begin{equation*}
\|\eta^n-\eta\|_{H^s(\Gamma_*)}\lesssim_M \epsilon    
\end{equation*}
and
\begin{equation*}
\|\Psi_{\leq j}v^{\delta}-v\|_{H^s(\Omega)}\lesssim_M \epsilon.   
\end{equation*}
This establishes continuous dependence.

\subsection{Lifespan of rough solutions} 
Here, we finally establish the continuation criterion from \Cref{cont crit intro} for $\mathbf{H}^s$ solutions. 
We consider initial data $(v_0,\Gamma_0)\in\mathbf{H}^s$ and the corresponding solution $(v,\Gamma)$ in a time interval $[0,T)$ which has  the property that 
\begin{equation*}
\mathcal{C}:=\sup_{0\leq t<T}A(t)+\int_{0}^{T}B(t)\,dt<\infty,\hspace{5mm}a(t)\geq c_0>0,\hspace{5mm}t\in [0,T), 
\end{equation*}
and whose domains $\Omega_t$  maintain a uniform thickness. Unlike with the construction of rough solutions, we now work with the weaker control parameter 
\[
B(t)=\|v\|_{W^{1,\infty}(\Omega_t)}+\|\Gamma_t\|_{C^{1,\frac{1}{2}}}. 
\]

One starting difficulty we face in this proof is that we do not a priori have a fixed reference collar neighborhood. However, the uniform bound on $A(t)$
guarantees that the free boundaries $\Gamma_t$ are uniformly of class $C^{1,\epsilon}$, and the uniform bound on $v$ guarantees
that they move at most with velocity $\mathcal{O}(1)$. This implies that 
the limiting boundary $\Gamma_T = \lim_{t \to T} \Gamma_t$ exists in the uniform topology, and also belongs to $C^{1,\epsilon}$, with the corresponding domain $\Omega_T$ having positive thickness. Furthermore, by interpolation, it follows that 
\[
\lim_{t \to T} \Gamma_t = \Gamma_T \qquad \text{in } C^{1,\epsilon_1}, \qquad 0 < \epsilon_1 < \epsilon.
\]
This allows us define the reference boundary $\Gamma_*$ as a regularization of $\Gamma_T$, so that $\Gamma_T \in  \Lambda(\Gamma_*, \epsilon/2,\delta/4)$ for an acceptable choice of $\delta$  ensuring that $\Lambda(\Gamma_*,\epsilon/2,\delta/2)$ is also a well-defined collar (cf.~Remark~\ref{r:delta}). Then the above convergence 
implies that $\Gamma_t \in \Lambda_*:= \Lambda(\Gamma_*,\epsilon/2,\delta/2)$ for $t$ close to $T$.
\\


Reinitializing the starting time close to $T$, we arrive at the case where we have the initial data $(v_0,\Gamma_0)\in\mathbf{H}^s$ and the corresponding solution $(v,\Gamma)$ in a time interval $[0,T)$ with the property that 
\begin{equation*}
\Gamma_t\in \Lambda_*,\hspace{5mm}t\in [0,T). 
\end{equation*}
 From the local well-posedness theorem, it suffices to show that 
\begin{equation}\label{blowuptimeunifbound}
\|(v,\Gamma)\|_{L^{\infty}([0,T);\mathbf{H}^s)}<\infty.    
\end{equation}
Similarly to the previous subsections, the strategy we would like to employ will involve showing that the control parameters for a suitable family of regularized solutions $(v_j,\Gamma_j)$ can be controlled to leading order by the control parameters for $(v,\Gamma)$. The main difficulty is that $v_j$ and $v$ are defined on different domains. As in the previous sections, as long as we can ensure that $\Gamma_j$ and $\Gamma$ are within distance $2^{-j(1+\delta)}$ of each other, we can compare $v$ with $\Psi_{\leq j}v_j$. However, there is one added difficulty now. The difference bound, which ensured the closeness of domains in the previous sections, has a stronger control parameter involving the term $\|D_tp\|_{W^{1,\infty}(\Omega_t)}$  in addition to $B(t)$, which from \Cref{Linfest2} has size controlled by $B(t)$ and an additional logarithmic factor.
\\

To overcome this, we will divide $[0,T)$ into two disjoint intervals $[0,\tilde{T}]$ and $[\tilde{T},T)$ where $0<\tilde{T}<T$ and $\tilde{T}$ has the property that
\begin{equation*}
\int_{\tilde{T}}^TB(t)\,dt<\delta_0,   
\end{equation*}
where $\delta_0$ is some parameter to be chosen depending only on $\mathcal{C}$, $c_0,$ the collar and the $\mathbf{H}^s$ norm of $(v_0,\Gamma_0)$. Given such a $\tilde{T}$, we consider the regularized data $(v_{\tilde{T},j},\Gamma_{\tilde{T},j})$ of $(v(\tilde{T}),\Gamma_{\tilde{T}})$ and the corresponding solutions $(v_j,\Gamma_j)$. We remark that  $\tilde{T}$ and $\delta_0$ need to be chosen carefully to not depend on $j$,  but we postpone this choice for now. Their purpose  is to guarantee that the stronger control parameter $D_tp$ in the difference bounds as well as the logarithmic factor in the energy bounds does not cause the distance between $\Gamma_j$ and $\Gamma$ to grow larger than $2^{-j(1+\delta)}$ for times $t<T$ where $(v_j,\Gamma_j)$ is defined. 
\\

From the continuous dependence result, the above regularized solutions converge to $(v,\Gamma)$ in $[\tilde{T},T)$ and their lifespans $T_j$ satisfy 
\begin{equation*}
\liminf_{j\to\infty}T_j\geq T-\tilde{T}.    
\end{equation*}
However, a priori, we do not have a uniform  $L^1_T$ bound on their corresponding control parameters $B_j$, nor a uniform $L_T^{\infty}$ bound on $A_j$, nor a uniform lower bound on the corresponding Taylor terms $a_j$. Arguing similarly to the previous subsections, if such bounds could be established, one could hope to use them to establish a uniform $\mathbf{H}^s$ bound on the regularized solutions $(v_{j},\Gamma_j)$ and hence extend their time of existence by an amount uniform in $j$. To establish such uniform control on these pointwise parameters, we will run a relatively simple bootstrap argument. From here on, we write $M:=\|(v_0,\Gamma_0)\|_{\mathbf{H}^s}$ and $M_{\tilde{T}}:=\|(v(\tilde{T}),\Gamma_{\tilde{T}})\|_{\mathbf{H}^s}$. To set up the bootstrap, we begin by noting that at time $\tilde{T}$, we have by Sobolev embedding and interpolation, the bound
\begin{equation}
\|\eta_j(\tilde{T})-\eta(\tilde{T})\|_{C^{1,\epsilon}(\Gamma_*)}\lesssim 2^{-\frac{j}{2}}M_{\tilde{T}}.     
\end{equation}
Moreover, by the properties of $\Psi_{\leq j}$, we have $\|v_j(\tilde{T})\|_{C^{\frac{1}{2}+\epsilon}}\lesssim_{\mathcal{C}} 1$. Hence, initially we have
\begin{equation}
A_j(\tilde{T})\leq P(\mathcal{C})+2^{-\frac{j}{2}}M_{\tilde{T}}   
\end{equation}
where $P>1$ is some sufficiently large positive polynomial. As long  as the choice of $\tilde{T}$ we make later on depends only on $\mathcal{C}$ and $c_0$ (but not on $j$), we can arrange by taking $j$ large enough, the initial bound
\begin{equation}
A_j(\tilde{T})\leq 2P(\mathcal{C}).    
\end{equation}
Finally, if $j$ is large enough, and $\tilde{T}$ is as above, we also initially have (for instance),
\begin{equation*}
a_j(\Tilde{T})\geq \frac{2}{3}c_0.    
\end{equation*}
Now, we make the bootstrap assumption that on a time interval $[\tilde{T},T_0]$ with $\tilde{T}<T_0<T$ we have the bounds
\begin{equation}\label{ctrlbootstrap}
\int_{\tilde{T}}^{T_0}B_j(t)\,dt< 4C_1(A)\delta_0,\hspace{5mm}A_j(t)\leq 4P(\mathcal{C}),\hspace{5mm} a_j(t)\geq \frac{1}{2}c_0,\hspace{5mm}\Gamma_j
(t)\in 2\Lambda_*  
\end{equation} 
for $j\geq j_0(M,T_0)$ and some large universal constant $C_1\gg 1$ depending only on $A:=\sup_{t\in [0,T)}A(t)$.  Our goal will be to show that the constant $4C_1\delta_0$ can be improved to $2C_1\delta_0$ and the constant $4P(\mathcal{C})$ can be improved to $2P(\mathcal{C})$, with similar improvements on the Taylor term and the collar. After we close this boostrap, we will give a separate argument which uses the uniform bounds on the control parameters to establish a uniform bound for $(v_j,\Gamma_j)$ in $\mathbf{H}^s$, and hence  permit us to continue the solution. To close the above bootstrap, we aim to establish the bounds 
\begin{equation}\label{bootstrap2}
B_j\leq C_1(A)B+C_22^{-\delta j},\hspace{5mm} A_j\leq P(\mathcal{C})+C_22^{-\delta j}  ,\hspace{5mm}a_j\geq \frac{2}{3}c_0,\hspace{5mm}\Gamma_j
(t)\in \frac{3}{2}\Lambda_*  ,
\end{equation}
where $\delta>0$ is some small positive constant and $C_2$ depends on the size of $M_{\tilde{T}}$ as well as the constant $\mathcal{C}$ above. The bootstrap can then be closed by choosing $j_0$ large enough to absorb the contribution of $C_2$.
\\

 As mentioned above, the main difficulty in comparing $B_j$ with $B$ and $A_j$ with $A$ is, as usual, the fact that the corresponding domains $\Omega_j$ and $\Omega$ are different. Our starting point is to select the parameter $\delta_0$ and the time $\tilde{T}(\delta_0)$ to ensure that $\Omega_j$ and $\Omega$ are close enough. As mentioned above, in order for our argument not to be circular, we need to ensure that the choice of $\delta_0$ depends only on $c_0$ and $\mathcal{C}$. Our first aim is to obtain some preliminary bounds for $\eta_j-\eta_{j+1}$ in $L^{\infty}$ and $C^{1,\frac{1}{2}}$. We let $k$ be the smallest integer larger than $s$. First, by the double exponential bound in \Cref{Energy est. thm} and the bootstrap hypothesis, we have for each $j$,
 \begin{equation*}
\|(v_j,\Gamma_j)\|_{\mathbf{H}^k}^2\lesssim_{A} \exp\left(\exp(K\delta_0)\log(K(1+2^{2j(k-s)}\|(v(\tilde{T}),\Gamma_{\tilde{T}})\|_{\mathbf{H}^s}^2)\right)   .
 \end{equation*}
 Above, $K$ is some (possibly large) constant depending on $\mathcal{C}$ and $c_0$ which we will let change from line to line. In the above estimate, if we take $K\delta_0\ll 1$ (in particular, $\delta_0$ does not depend on $j$), then we can arrange for 
 \begin{equation}
 \|(v_j,\Gamma_j)\|_{\mathbf{H}^k}^2\lesssim K2^{2j(k-s)}M_{\tilde{T}}^2(M_{\tilde{T}}2^{j})^{\delta}    
 \end{equation}
 for some small constant $\delta>0$, where we assumed without loss of generality that $M_{\tilde{T}}\geq 1$ to simplify notation. Note here that there is a slight loss compared to \eqref{hiregbound}  coming from the double exponential bound in the energy estimate. On the other hand, the difference estimates, \Cref{Linfest2} and the energy coercivity ensures that by Gr\"onwall and the bootstrap assumption, we have 
\begin{equation*}
D((v_{j},\Gamma_j),(v_{j+1},\Gamma_{j+1}))\lesssim 2^{-2js}KM_{\tilde{T}}^2\exp\left(K\delta_0\mathcal{I}_j\right),
\end{equation*}
where $\mathcal{I}_j=\sup_{\tilde{T}<t\leq T_0}(\log(K+KE^k(v_j,\Gamma_j))+\log(K+KE^k(v_{j+1},\Gamma_{j+1})))$ and $k$ is, again, the smallest integer larger than $s$. By the higher energy bound and the bootstrap assumption, we have 
\begin{equation*}
\mathcal{I}_j\lesssim K\log(1+2^{2jk}\|(v(\tilde{T}),\Gamma_{\tilde{T}})\|_{L^2}^2)\lesssim_{k} Kj,
\end{equation*}
where we used the higher energy bound for the regularized solution to propagate $\log(1+E^k(v_j,\Gamma_j))$ and control $\log(1+E^k(v_j,\Gamma_j))$ by $\log(1+2^{2kj}\|(v(\tilde{T}),\Gamma_{\tilde{T}})\|_{L^2}^2
)$ as well as the fact that the volume of $\Omega_t$ is conserved and H\"older's inequality to estimate $\|(v(\tilde{T}),\Gamma_{\tilde{T}})\|_{L^2}\lesssim_A 1$. Again, we choose $\delta_0$ small enough (and therefore $\tilde{T}$) depending only on $\mathcal{C}$ and $c_0$ so that
\begin{equation*}
\exp(K\delta_0\mathcal{I}_j)\leq 2^{j\delta}  ,  
\end{equation*}
for some sufficiently small $\delta>0$ (depending only on $s$). Next, we pick $j_0$ depending on $M_{\tilde{T}}$, $\mathcal{C}$ and $c_0$  so that if $j\geq j_0$ (after possibly relabelling $\delta$), we have
\begin{equation*}\label{refineddiff1}
D((v_{j},\Gamma_j),(v_{j+1},\Gamma_{j+1}))\lesssim 2^{-2j(s-\delta)},\hspace{5mm}\|(v_j,\Gamma_j)\|_{\mathbf{H}^k}^2\lesssim 2^{2j(k-s)}2^{j\delta}
\end{equation*}
with universal implicit constant. The key point to observe here is that there is now a slight loss in the difference estimates and energy estimates compared to the previous subsections because of the stronger control parameter in the difference bounds and the logarithmic factor in the energy estimates. However, by using these estimates, we still obtain by Sobolev embedding and interpolating, the bounds (after possibly relabelling $\delta$) 
\begin{equation}\label{refineddiff}
\|\eta_j-\eta_{j+1}\|_{C^{1,\frac{1}{2}}(\Gamma_*)}\lesssim 2^{-\delta j},\hspace{5mm}\|\eta_j-\eta_{j+1}\|_{C^{1,\epsilon}(\Gamma_*)}\lesssim 2^{-\frac{1}{2}j} ,\hspace{5mm} \|\eta_j-\eta_{j+1}\|_{L^{\infty}(\Gamma_*)}\lesssim 2^{-\frac{3}{2}j}    ,
\end{equation}
all with universal implicit constant if $j_0$ is large enough. The first bound will give us control of $\|\Gamma_j\|_{C^{1,\frac{1}{2}}}$ in the first estimate in (\ref{bootstrap2}). The second bound above gives us control over $\|\Gamma_j\|_{C^{1,\epsilon}}$ for the second estimate in (\ref{bootstrap2}) and also shows that $\Gamma_j\in \frac{3}{2}\Lambda_*$. The third bound ensures that $\Gamma_j$ and $\Gamma_{j+1}$ are sufficiently close. With this closeness established, we now work towards closing the bootstrap (\ref{bootstrap2}) for the $\|v_j\|_{W^{1,\infty}(\Omega_j)}$ component of $B_j$ and the $\|v_j\|_{C^{\frac{1}{2}+\epsilon}(\Omega_j)}$ component of $A_j$. We show the details for $\|v_j\|_{W^{1,\infty}(\Omega_j)}$ as the other component is very similar. We estimate in three steps. First, we observe that from the bounds for $\Psi_{\leq j}$, we have
\begin{equation}\label{step1}
\|\Psi_{\leq j}v\|_{W^{1,\infty}}\lesssim_A B.    
\end{equation}
We can ensure that the implicit constant in this estimate is less than $C_1(A)$ if $C_1(A)$ is initially chosen large enough. 
Then we compare $\Psi_{\leq j}v$ and $\Psi_{\leq j}v_j$ which is justified thanks to (\ref{refineddiff}). We have
\begin{equation*}
\Psi_{\leq j}v-\Psi_{\leq j}v_j=\sum_{l\geq j}\Psi_{\leq j}v_{l+1}-\Psi_{\leq j}v_l.    
\end{equation*}
By Sobolev embedding and a similar argument to the $C^{1,\frac{1}{2}}$ bound for $\eta_{j+1}-\eta_j$, we see that
\begin{equation*}
\|\Psi_{\leq j}v_{l+1}-\Psi_{\leq j}v_l\|_{W^{1,\infty}}\leq C_22^{-l\delta} ,   
\end{equation*}
which gives by summation
\begin{equation}\label{step2}
\|\Psi_{\leq j}v-\Psi_{\leq j}v_j\|_{W^{1,\infty}}\leq C_22^{-j\delta}.     
\end{equation}
Using the error bound for $I-\Psi_{\leq j}$, Sobolev embedding and the higher energy bounds, we also have
\begin{equation}\label{step3}
\|\Psi_{\leq j}v_j-v_j\|_{W^{1,\infty}}\leq C_22^{-j\delta}. 
\end{equation}
Combining (\ref{step1}), (\ref{step2}) and (\ref{step3}) shows that
\begin{equation*}
\|v_j\|_{W^{1,\infty}(\Omega_j)}\leq C_1(A)B+C_22^{-j\delta}. 
\end{equation*}
Doing a similar estimate for $\|v_j\|_{C^{\frac{1}{2}+\epsilon}(\Omega_j)}$ and taking $j$ large enough allows us to close the bootstrap for $A_j$. 
\\

It remains now to improve the bootstrap assumption for the Taylor term $a_j$. To do this, we need a suitable way of comparing the $C^1$ norms of the pressures $p_j$ and $p$. We begin by defining the shrunken domain $\Omega'$ via $\eta':=\eta-2^{-j_0}$. As $\Omega_j$ is within distance $\mathcal{O}(2^{-\frac{3}{2}j})$ of $\Omega$ for $j\geq j_0$, it follows that
\begin{equation*}
\Omega'\subset \Omega\cap \bigcap_{j\geq j_0}\Omega_j.    
\end{equation*}
We next note the following bound which holds on $\Omega'$ for any $0<\delta<\frac{\epsilon}{2}$,
\begin{equation}\label{pointwisediffbound12}
\|v_j-v\|_{C^{\frac{1}{2}+\delta}(\Omega')}\leq C_2 2^{-j_0\delta}  .
\end{equation}
This follows by similar reasoning to the above. 
Now, we establish the following $C^1$ estimate for $p-p_j$:
\begin{equation}\label{pointwisediffboundpressure12}
\|p-p_j\|_{C^1(\Omega')}\leq C_22^{-j_0\delta}    .
\end{equation}
We begin by splitting $p-p_j$ into an inhomogeneous part plus a harmonic part on $\Omega'$, $$p-p_j=\Delta^{-1}\Delta(p-p_j)+\mathcal{H}(p-p_j).$$
Using \Cref{Gilbarg}, the dynamic boundary condition and the fact that the boundary of $\Omega'$ is within distance $2^{-j_0}$ of the boundaries of $\Omega$ and $\Omega_j$, we have
\begin{equation*}
\|\mathcal{H}(p-p_j)\|_{C^{1,\delta}(\Omega')}\lesssim_{\mathcal{C}} 2^{-j_0\delta}(\|p\|_{C^{1,\epsilon}(\Omega)}+\|p_j\|_{C^{1,\epsilon}(\Omega_j)})    .
\end{equation*}
By \Cref{Linfest} and the bootstrap assumption on $A_j$, this gives
\begin{equation*}
\|\mathcal{H}(p-p_j)\|_{C^{1,\delta}(\Omega')}\lesssim_{\mathcal{C}}2^{-j_0\delta}.    
\end{equation*}
To estimate the inhomogeneous part, we can argue similarly to the proof of \Cref{Linfest} using a bilinear frequency decomposition for $\Delta (p_j-p)$, to obtain
\begin{equation*}
\|\Delta^{-1}\Delta (p-p_j)\|_{C^1(\Omega')}\lesssim_{\mathcal{C}} \|v-v_j\|_{C^{\frac{1}{2}+\delta}(\Omega')}\leq C_2 2^{-j_0\delta}  , 
\end{equation*}
where in the second inequality we used (\ref{pointwisediffbound12}). Finally, to close the bootstrap on the Taylor term $a_j$, we can work in  collar coordinates on $\Gamma_*$ to estimate
\begin{equation*}
\inf_{x\in\Gamma_j}|\nabla p_j(x)|\geq \inf_{x\in\Gamma} |\nabla p(x)|-\|p_j-p\|_{C^1(\Omega')}-2^{-j_0\delta}(\|p_j\|_{C^{1,\epsilon}(\Omega_j)}+\|p\|_{C^{1,\epsilon}(\Omega)})   .
\end{equation*}
In the above, we first estimate the error between $\nabla p_j(x+\eta_j(x)\nu(x))$ and $\nabla p_j(x+\eta'(x)\nu(x))$ (and also $\nabla p(x+\eta'(x)\nu(x))$ and $\nabla p(x+\eta(x)\nu(x))$) using the $C^{1,\epsilon}$ H\"older regularity of $p_j$ and $p$. Then, we estimate the difference between $\nabla p_j(x+\eta'(x)\nu(x))$ and $\nabla p(x+\eta'(x)\nu(x))$ on the common domain using our bounds for $\|p_j-p\|_{C^1(\Omega')}$. 
\\

Taking $j_0$ large enough and using (\ref{pointwisediffboundpressure12}) and \Cref{Linfest}, this gives
\begin{equation*}
a_j\geq \frac{2}{3}c_0   , 
\end{equation*}
which closes the bootstrap for $a_j$.
\\

From the above argument, we see that for $j\geq j_0$,  the regular solutions $(v_j,\Gamma_j)$ are defined on the interval $[\tilde{T},T]$ and satisfy the assumptions (\ref{ctrlbootstrap}). What we do not yet know  is whether we have a uniform in $j$ bound for the $\mathbf{H}^s$ norm of $(v_j,\Gamma_j)$. Once we have this,  (\ref{blowuptimeunifbound}) will follow from our continuous dependence result. From here on, we assume without loss of generality that $M_{\tilde{T}}\gg C(A)$. We let $c_j$ denote the frequency envelope for the data at time $\tilde{T}$. Similarly to the above, on a time interval $[\tilde{T},T_0]$, we make the bootstrap assumption that for finitely many $j\geq j_0$,
\begin{equation}\label{finalbootstrap}
\|(v_j,\Gamma_j)\|_{\mathbf{H}^s}\leq M_{\tilde{T}}^2. 
\end{equation}
As in the previous subsection, we let $\alpha\geq 1$ be such that $s+\alpha$ is an integer. Then the higher energy bounds, (\ref{finalbootstrap}) and (\ref{ctrlbootstrap}) yield
\begin{equation*}
\|(v_j,\Gamma_j)\|_{\mathbf{H}^{s+\alpha}}\lesssim 2^{j\alpha}c_j\exp(K\delta_0\log(M_{\tilde{T}}^2))M_{\tilde{T}}    
\end{equation*}
where $K$ is some constant depending on $\mathcal{C}$. As long as $\delta_0$ is such that $K\delta_0\ll 1$, we obtain
\begin{equation}\label{newhigherenergyboundenv}
\|(v_j,\Gamma_j)\|_{\mathbf{H}^{s+\alpha}}\lesssim 2^{j\alpha}c_jM_{\tilde{T}}^{1+\delta}    
\end{equation}
for some positive constant $\delta\ll 1$. A similar argument with the difference bounds yields
\begin{equation*}
D((v_{j},\Gamma_{j}),(v_{j+1},\Gamma_{j+1}))^{\frac{1}{2}}\lesssim 2^{-js}c_jM_{\tilde{T}}^{1+\delta}.    
\end{equation*}
Arguing as in the local well-posedness result, we can use the above two bounds to estimate
\begin{equation*}
\|(v_j,\Gamma_j)\|_{\mathbf{H}^{s}}\lesssim M_{\tilde{T}}^{1+\delta}  ,  
\end{equation*}
which improves the bootstrap. We are then able to finally conclude the bound \eqref{blowuptimeunifbound} and thus the proof of \Cref{cont crit intro}.
\bibliographystyle{plain}
\bibliography{refs.bib}

\end{document}